\documentclass[12 pt]{amsart}

\usepackage{fullpage}

\usepackage{diagbox}

\usepackage{hyperref}
\usepackage{etex}
\usepackage[shortlabels]{enumitem}
\usepackage{amsmath}
\usepackage{amsxtra}
\usepackage{amscd}
\usepackage{amsthm}
\usepackage{amsfonts}
\usepackage{amssymb}
\usepackage{eucal}
\usepackage{cleveref}
\usepackage[all]{xy}
\usepackage{graphicx}
\usepackage{tikz}
\usetikzlibrary{cd}
\usetikzlibrary{fit, patterns}
\usepackage{mathrsfs}
\usepackage{subfiles}
\usepackage{mathpazo}
\usepackage[colorinlistoftodos, textsize=tiny]{todonotes}
\usepackage{morefloats}
\usepackage{pdfpages}
\usepackage{thmtools}
\usepackage{thm-restate}
\usepackage[utf8]{inputenc}
\usepackage[T1]{fontenc}
\usepackage{epigraph}
\usepackage{csquotes}
\usepackage{moreverb}
\usepackage[margin=1in]{geometry}
\usepackage{adjustbox}
\usepackage{accents}
\usepackage{xparse}

\graphicspath{ {images/} }

\RequirePackage{color}
\definecolor{myred}{rgb}{0.75,0,0}
\definecolor{mygreen}{rgb}{0,0.5,0}
\definecolor{myblue}{rgb}{0,0,0.65}

\usepackage{hyperref}
\hypersetup{citecolor=blue}
\usepackage{tikz}
\usetikzlibrary{matrix,arrows,decorations.pathmorphing}

\newtoggle{draft}
\toggletrue{draft}

\iftoggle{draft}{

	\newcommand{\NB}[1]{\todo[color=gray!40]{#1}}
	\newcommand{\TODO}[1]{\todo[color=red]{#1}}
}{ 
	\newcommand{\NB}[1]{}
	\newcommand{\TODO}[1]{}
	\renewcommand{\todo}[1]{}
	\renewcommand{\todo}[1]{}
}

\theoremstyle{plain}
\newtheorem{theorem}[subsubsection]{Theorem}
\newtheorem{construction}[subsubsection]{Construction}
\newtheorem{proposition}[theorem]{Proposition}
\newtheorem{lemma}[theorem]{Lemma}

\theoremstyle{definition}
\newtheorem{definition}[theorem]{Definition}
\newtheorem{remark}[theorem]{Remark}
\newtheorem{example}[theorem]{Example}

\newtheorem{conjecture}[theorem]{Conjecture} 

\newtheorem{warn}[theorem]{Warning}

\newtheorem{notation}[theorem]{Notation}

\theoremstyle{remark}
\numberwithin{equation}{section}

\newcommand\nc{\newcommand}
\nc\on{\operatorname}
\nc\renc{\renewcommand}

\newcommand*{\shom}{\mathscr{H}\kern -.5pt om}
\newcommand*{\stor}{\mathscr{T}\kern -.5pt or}
\newcommand*{\sext}{\mathscr{E}\kern -.5pt xt}

\makeatletter
\providecommand\@dotsep{5}
\renewcommand{\listoftodos}[1][\@todonotes@todolistname]{%
\@starttoc{tdo}{#1}}
\makeatother

\makeatletter
\def\Ddots{\mathinner{\mkern1mu\raise\p@
		\vbox{\kern7\p@\hbox{.}}\mkern2mu
		\raise4\p@\hbox{.}\mkern2mu\raise7\p@\hbox{.}\mkern1mu}}
\makeatother

\newcommand{\customlabel}[2]{\protected@write \@auxout {}{\string \newlabel {#1}{{#2}{\thepage}{#2}{#1}{}} }\hypertarget{#1}{#2}}

\DeclareMathOperator\id{id}

\DeclareMathOperator\Spc{Spc}

\renewcommand\hom{\mathrm{Hom}}

\DeclareMathOperator\rk{rk}

\DeclareMathOperator\Barc{Bar}

\DeclareMathOperator\spec{Spec}

\DeclareMathOperator\Mod{Mod}

\DeclareMathOperator\sym{Sym}

\renewcommand\drop{\mathrm{Drop}}

\DeclareMathOperator\ab{ab}
\DeclareMathOperator\std{std}
\DeclareMathOperator\CHur{CHur}
\DeclareMathOperator\perm{perm}

\DeclareMathOperator\tr{tr}
\DeclareMathOperator\frob{Frob}

\DeclareMathOperator\Hur{Hur}
\DeclareMathOperator\Conf{Conf}
\DeclareMathOperator\PConf{PConf}
\DeclareMathOperator\ord{ord}

\DeclareMathOperator\inv{inv}

\DeclareMathOperator\aut{Aut}
\DeclareMathOperator\gl{GL}

\newcommand\ZZ{\mathbb{Z}}

\newcommand\NN{\mathbb{N}}

\newcommand\EE{\mathbb{E}}
\DeclareMathOperator\sel{Sel}

\DeclareMathOperator\gr{gr}

\DeclareMathOperator\colim{colim}

\DeclareFontFamily{U}{wncy}{}
\DeclareFontShape{U}{wncy}{m}{n}{<->wncyr10}{}
\DeclareSymbolFont{mcy}{U}{wncy}{m}{n}
\DeclareMathSymbol{\Sha}{\mathord}{mcy}{"58}

\newcommand{\conf}[3]{\operatorname{Conf}_{#1,#2/#3}} 
\newcommand{\qtwist}[3]{\operatorname{QTwist}_{#1, #2/#3}} 
\newcommand{\selspacemoments}[2]{\operatorname{Sel}_{#1}^{#2}} 
\newcommand{\rankselspacemoments}[2]{\operatorname{Sel}^{#2,\rk}_{#1}}
\newcommand{\paritybklpr}[2]{\operatorname{Sel}^{\operatorname{BKLPR},#2}_{#1}} 
\newcommand{\bklpr}[1]{\operatorname{Sel}^{\operatorname{BKLPR}}_{#1}} 

\newcommand{\cphur}[4]{\operatorname{CHur}^{#1,#3}_{#2, #4}}

\newcommand{\cquohur}[5]{[\operatorname{CHur}^{#1,#4}_{#3,#5}/#2]}

\newcommand{\cphurc}[2]{\operatorname{CHur}^{#2}_{#1}}

\DeclareMathOperator\surj{Surj}

\def\listtodoname{List of Todos}
\def\listoftodos{\@starttoc{tdo}\listtodoname}
\title{The stable homology of Hurwitz modules and applications}

\author{Aaron Landesman}
\author{Ishan Levy}

\usepackage{microtype}
\begin{document}

\begin{abstract}
	We show that the homology of modules for Hurwitz spaces stabilizes and
	compute its stable value. 
	As one consequence, we
	compute the moments
	of Selmer groups in quadratic twist families of abelian varieties 
	over suitably large function fields.
	As a second consequence, we deduce a
	version of Bhargava's conjecture, counting the number of $S_d$ degree $d$
	extensions of $\mathbb F_q(t)$, for suitably large $q$.
	As a third consequence, we deduce that the homology of Hurwitz spaces associated to
	racks with a single component satisfy representation stability.
\end{abstract}

\maketitle
\tableofcontents

\section{Introduction}
\label{section:intro}

If $G$ is a group and $c = c_1 \cup \cdots \cup c_\upsilon$ is a union of
$\upsilon$ conjugacy classes in $G$,
we use $\CHur^c_{n_1,
\ldots, n_\upsilon}$ to denote the Hurwitz space associated to $c$.
Roughly, this is a moduli space
parameterizing geometrically connected Galois $G$ covers of $\mathbb A^1$ with with $n_i$
labeled points of branching type $c_i$, together with a labeling of the sheets
of the cover near $\infty$.
These Hurwitz spaces are of much interest in number theory as their $\mathbb
F_q$ points parameterize
covers of global function fields, and they are also some of the most fundamental
moduli spaces appearing in algebraic geometry.
In \cite{landesmanL:homological-stability-for-hurwitz} we showed that the
homology groups 
$H_i(\CHur^c_{n_1, \ldots, n_\upsilon};\mathbb Z)$ stabilize
as $n_1 \to \infty$ with $n_2, \ldots,n_\upsilon$ fixed. We used this to deduce
applications toward a number of conjectures in number theory and algebraic
geometry, including the Cohen-Lenstra heuristics, Malle's conjecture, and the
Picard rank conjecture.
However, in
\cite{landesmanL:homological-stability-for-hurwitz},
we only were able to compute the stable value of 
$H_i(\CHur^c_{n_1, \ldots, n_\upsilon};\mathbb Z[|G|^{-1}|)$
when all
$n_i$ are sufficiently large.
In this paper, we compute the stable value ``in all directions,'' meaning that we
require only $n_1$ to be sufficiently large and remove the restriction that
$n_2, \ldots,n_\upsilon$ be sufficiently large, see
\autoref{theorem:one-large-stable-homology-hurwitz-space}.
For example, in the case $G = S_3$ and $c:= S_3 - \id$, before we were only able
to compute the stable homology of $\CHur^c_{n_1, n_2}$ when there were sufficiently many $3$-cycles and 
transpositions, while one of our main results in this paper enables us to compute the stable
homology when there is a single transposition and many $3$-cycles.
Moreover, in this paper, we show Hurwitz spaces parameterizing covers of 
punctured curves of arbitrary genus also stabilize and we compute their stable value.

As mentioned in the introduction of \cite{landesmanL:homological-stability-for-hurwitz}, we hope
that our papers will give arithmetic statisticians the tools to explore
arithmetic statistics problems over function fields, similarly to the way
Bhargava's thesis allowed arithmetic statisticians to make much progress over
$\mathbb Q$.
While our previous paper \cite{landesmanL:homological-stability-for-hurwitz}
began laying the framework for this, the results of this paper significantly
widen the scope of the types of problems that can be approached.
See \autoref{section:further-questions} for some additional potential 
applications not explored in this paper.

As some sample applications of our results, we describe progress toward the
Poonen-Rains heuristics and Bhargava's conjecture over function fields.
In this paper, we will work with Hurwitz spaces associated to racks, which are
more general than those associated to unions of conjugacy classes in a group.
By applying our results to suitably chosen racks, we are able to
deduce representation stability for Hurwitz spaces associated to a conjugacy
class in a group.
We begin by surveying these applications in
\autoref{subsection:poonen-rains-intro} (toward the Poonen-Rains heuristics in \autoref{theorem:poonen-rains-moments}),
\autoref{subsection:bhargava-intro} (toward Bhargava's conjecture in
\autoref{theorem:bhargava-intro}),
and
\autoref{subsection:intro-rep-stability} (toward representation stability in
\autoref{theorem:representation-stability}), and then discuss our topological
results in \autoref{subsection:homological-stability} (specifically in
	\autoref{theorem:one-large-stable-homology-hurwitz-space} where we
	compute the stable homology of Hurwitz space in all directions,
	\autoref{theorem:homology-stabilizes-intro} showing the homology of
	bijective Hurwitz space modules stabilizes, and
\autoref{theorem:one-large-stable-homology} computing its stable value).

\subsection{The moments of the distribution of Selmer groups in quadratic twist
families}
\label{subsection:poonen-rains-intro}

One of our main results is the verification of the Poonen-Rains conjectures for
the moments of Selmer groups of abelian varieties in quadratic twist families
over function fields over a finite field $\mathbb F_q$. 
Recall that for $\nu$ an integer and $E$ an elliptic curve over a global field
$K$,
one can define the finite $\mathbb Z/\nu \mathbb Z$ module $\sel_\nu(E)$.
This finite set is closely related to rank of $E$, which measures the number of
solutions in $K$ to the equation defining $E$, but it is typically more
computable.
Recall that the Poonen-Rains conjectures were formulated in
\cite{poonenR:random-maximal-isotropic-subspaces-and-selmer-groups}
for prime order Selmer groups and were generalized to composite order Selmer
groups in
\cite[\S5.7]{bhargavaKLPR:modeling-the-distribution-of-ranks-selmer-groups},
see also \cite[\S5.3.3]{fengLR:geometric-distribution-of-selmer-groups}.
These conjectures predict the distribution of the Selmer groups of a family of
elliptic curves.
The moments of this distribution were computed in
\cite[Proposition
2.3.1]{ellenbergL:homological-stability-for-generalized-hurwitz-spaces}.
Although these conjectures were originally stated for the universal family of
all elliptic curves, it is
also common to conjecture them in quadratic twist families of abelian varieties as in \cite[Remark
1.9]{poonenR:random-maximal-isotropic-subspaces-and-selmer-groups}, which is the
context we consider in this paper.
We refer the reader to the introduction of
\cite{ellenbergL:homological-stability-for-generalized-hurwitz-spaces} for a
more leisurely introduction to the Poonen-Rains heuristics in the context of
this paper.
Henceforth, we refer to these predictions as the ``BKLPR heuristics'' and the
moments predicted by the above distribution as the ``BKLPR moments.''

We start with a very special case of our main result. We will be working in the
case that $K$ as above is a global function field, i.e. $K = K(C)$ for $C$ a
curve over a finite field
$\mathbb F_q$. We consider
an elliptic curve over $K(C)$, or equivalently a relative elliptic curve $A$ over some
over $U \subset C$, which is nonconstant with squarefree discriminant.
In this case, the average size of the $\nu$ Selmer group in the associated quadratic twist
family over $\mathbb F_{q^j}$, with $j$ sufficiently large, depending on $\nu$,
is $\sum_{d \mid \nu} d$. In particular, if $\nu = \ell$ is a prime, the
average size is $\ell + 1$.
To our knowledge, this constitutes the first such verification of even this
special case of the BKLPR heuristics over any global field 
with $\nu$ odd and $\nu > 3$.

\begin{notation}
	\label{notation:quadratic-twist-family}
	Fix a smooth proper geometrically connected curve $C$ over a finite
	field $\mathbb F_q$ of odd characteristic. Let $K := K(C)$ be the
	function field of $C$.
	Let $U \subset C$ be a nonempty open subscheme with nonempty complement
	$Z := C - U$.

	Fix an odd integer $\nu$ and a polarized abelian scheme $A \to U$ with
	polarization of degree prime to $\nu$.
	Let $\on{QTwist}_{n,U/\mathbb F_q}(\mathbb F_{q^j})$ denote the groupoid of quadratic
	twists of the base change $A_{\mathbb F_{q^j}}:= A \times_{\spec \mathbb
F_q} \spec \mathbb F_{q^j}$, ramified
	over a degree $n$ divisor contained in $U$ with $n$ even, as defined
	precisely in \cite[Notation
	5.1.4]{ellenbergL:homological-stability-for-generalized-hurwitz-spaces}.
That is, $x \in \on{QTwist}_{n,U/\mathbb F_q}(\mathbb F_{q^j})$ is the data of a double
cover $U' \to U_{\mathbb F_{q^j}}$ with degree $n$ branch locus. The associated quadratic twist of $A$ is
the quotient of the Weil restriction
$A_x := \on{Res}_{U'/U_{\mathbb F_{q^j}}}(A_{\mathbb F_{q^j}} \times_{U_{\mathbb
F_{q^j}}} U')/A_{\mathbb F_{q^j}}.$
If $B$ is an abelian scheme over $U$ with generic fiber $B_K$, we use $\sel_\nu(B)$ as equivalent
notation for $\sel_\nu(B_K):= \ker \left(H^1(K, B_K[\nu]) \to \prod_v H^1(K_v,
B_K)\right),$
where the product is taken over all places $v$ of $K$, or equivalently over closed
points of $C$.
\end{notation}

\begin{theorem}
	\label{theorem:poonen-rains-elliptic-curves}
Choose $q$ with $\on{char} \mathbb F_q  > 3$ and $\nu$ an integer prime
	to $6q$.
	With notation as in \autoref{notation:quadratic-twist-family},
	suppose $A$ is a nonconstant elliptic curve with squarefree
	discriminant.
	There is a constant $C_{\nu}$ depending on $\nu$ (but not on $A$) so
	that if $q^j > C_{\nu}$,
\begin{align}
	\label{equation:elliptic-curve-average-selmer}
		\lim_{\substack{n \to \infty \\ n
		\hspace{.1cm} \mathrm{even}}}
				\frac{\sum_{x \in \qtwist n U
			{\mathbb F_q}(\mathbb F_{q^j})}  
\# \sel_\nu(A_x)}{\sum_{x \in \qtwist n U {\mathbb F_q}(\mathbb F_{q^j})}
		1}
		&= \sum_{d \mid \nu} d.
		\end{align}
\end{theorem}
%
%
%

This can be deduced fairly immediately from the more general result
\autoref{theorem:poonen-rains-moments} below, and we spell out the details of
this deduction in \autoref{subsection:proof-poonen-rains-elliptic}.

We next introduce some notation to state our more general version
of \autoref{theorem:poonen-rains-elliptic-curves}, which works with abelian
varieties of arbitrary dimension and computes arbitrary moments of Selmer
groups, instead of just their average size.

\begin{notation}
	\label{notation:abelian-scheme}
	Retain notation from \autoref{notation:quadratic-twist-family}.
	Assume that $A$ has multiplicative reduction with toric part of
	dimension $1$ over some point of $C$.
	Also assume that 
	$\nu$ is
	prime to $q$, $A[\nu]$ is a tame finite \'etale cover of $U$,
	and
	every prime $\ell \mid \nu$ satisfies $\ell > 2 \dim A
	+ 1$ and that $A[\ell] \times_{\mathbb F_q} \overline{\mathbb F}_q$
	corresponds to an irreducible sheaf of $\mathbb Z/\ell \mathbb Z$
	modules on $U \times_{\mathbb F_q} \overline{\mathbb F}_q$.
	Moreover assume that $\nu$ is relatively prime to the order of the geometric
	component group of the N\'eron model of $A$ over $C$,
	see \cite[Notation
	5.2.2]{ellenbergL:homological-stability-for-generalized-hurwitz-spaces}.

For $X$ and $Y$ two finite groups, we use $\# \surj(X, Y)$ for the number of
surjective group homomorphisms from $X$ to $Y$.
\end{notation}

Our main result toward the Poonen-Rains heuristics is as follows.

\begin{theorem}
	\label{theorem:poonen-rains-moments}
	Assuming $A$ and $\nu$ are as in
	\autoref{notation:quadratic-twist-family} and
	\autoref{notation:abelian-scheme},
there is some constant $C_{H}$ depending on $H$ (but not on $A$) so that if $q^j
	> C_{H}$, we have
\begin{align}
	\label{equation:moment-limit-residue}
		\lim_{\substack{n \to \infty \\ n
		\hspace{.1cm} \mathrm{even}}}
				\frac{\sum_{x \in \qtwist n U
			{\mathbb F_q}(\mathbb F_{q^j})}  \# \surj
		(\sel_\nu(A_x), H)}{\sum_{x \in \qtwist n U {\mathbb F_q}(\mathbb F_{q^j})}
		1}
		&= \#\sym^2 H.
		\end{align}
\end{theorem}

We prove this in \autoref{subsection:proof-poonen-rains}.
The reader may wish to consult
\cite[\S1.6]{ellenbergL:homological-stability-for-generalized-hurwitz-spaces}
for a description of prior related work on this topic.

\begin{remark}
	\label{remark:}
	\autoref{theorem:poonen-rains-moments} is related to \cite[Theorem
	1.1.6]{ellenbergL:homological-stability-for-generalized-hurwitz-spaces},
	where a version of the \eqref{equation:moment-limit-residue} was
	established where one additionally takes a large $j$ limit.
	Here, we improve that result by establishing it for fixed $j$
	sufficiently large, without needing to take such a large $j$ limit.

	We also obtain the improvement over 
\cite[Theorem 1.1.6]{ellenbergL:homological-stability-for-generalized-hurwitz-spaces}
	that the value of the
	constant $C_H$ appearing in \autoref{theorem:poonen-rains-moments} and
	also in \autoref{theorem:symplectic-sheaf-point-counts} can
	be chosen to be independent of the choice of the abelian scheme $A$.
	We thank Jordan Ellenberg for pointing out this independence to us.
\end{remark}

\begin{remark}
	\label{remark:explicit-constant-intro}
The constants $C_\nu$ and $C_H$ appearing in
\autoref{theorem:poonen-rains-elliptic-curves} and
\autoref{theorem:poonen-rains-moments} are explicit and computable. See
\autoref{remark:explicit-constant} for more details.
\end{remark}
\begin{remark}
	\label{remark:}
	The conditions in \autoref{theorem:poonen-rains-elliptic-curves} and
	\autoref{theorem:poonen-rains-moments} that $n$ is even is not
	especially important and can be removed. It is only there so that we can
	more easily cite
	\cite{ellenbergL:homological-stability-for-generalized-hurwitz-spaces},
	where the stack 
	$\on{QTwist}_{n,U/\mathbb F_q}$ was set up to assume $n$ is even.
\end{remark}

\subsection{Bhargava's conjecture}
\label{subsection:bhargava-intro}

Bhargava's conjecture, \cite[Conjecture
1.2]{bhargava:mass-formulae-for-extensions-of-local-fields},
predicts the asymptotic growth of the number of degree $d$ number fields with Galois
group $S_d$, as a function of the discriminant.
For the reader's convenience, before continuing, we recall the statement of
Bhargava's conjecture.

\begin{conjecture}[\protect{\cite[Conjecture
	1.2]{bhargava:mass-formulae-for-extensions-of-local-fields}}]
	\label{conjecture:bhargava}
	Let $N_d(X)$ denote the number of number fields of degree $d$ having
discriminant with absolute value at most $X$. Let $q(n,k)$ denote the number of
partitions of $n$ into at most $k$ parts.
Let $r_2(S_d)$ denote the number of elements of order either $1$ or $2$ in
$S_d$.
Then,
\begin{align}
	\label{equation:bhargava-prediction}
	\lim_{X \to \infty} \frac{N_d(X)}{X} =
	\frac{r_2(S_d)}{2d!} \prod_{p \text{ prime}} \left( \frac{\sum_{k=0}^n
	q(k,n-k)-q(k-1,n-k+1)}{p^k} \right).
\end{align}
\end{conjecture}

One of our main results in this paper is a computation of the constant in the
asymptotic growth of the number of degree
$d$, $S_d$, field extensions of $\mathbb F_q(t)$ for $q$
sufficiently large relative to $d$.
Prior to this paper, for any global field $K$, mathematicians have only been
able to compute this constant when $d \leq 5$.
We now introduce notation to state our results precisely.

\begin{notation}
	\label{notation:twisted-conf}
	For $d \geq 2$, write $S_d - \id = c_1 \cup \cdots \cup c_\upsilon$ as
	a disjoint union of its non-identity conjugacy classes, so that $c_1$
	is the conjugacy class of transpositions.
	We fix $q$ a prime power, relatively prime to $d! = |S_d|$.
	Define 
$\Conf_{n_1, \ldots, n_\upsilon, \mathbb F_q}$ to be the
multi-colored configuration space with $n_i$ points of color $i$,
	see
	\cite[Definition 2.2.1]{landesmanL:homological-stability-for-hurwitz}
	for a precise definition.

If $K/\mathbb F_q(t)$ is a generically separable extension and $\mathscr O_K$ is the normalization of
$\mathbb F_q[t]$ in $K$, we say $K/\mathbb F_q(t)$ has discriminant equal to the
discriminant of $\mathscr O_K$ over $\mathbb F_q[t]$, which we define to be
$q^{\deg \Omega_{\mathscr O_K/\mathbb F_q[t]}}$, where 
$\Omega_{\mathscr O_K/\mathbb F_q[t]}$ is the sheaf of relative differentials.

We use $\Delta(\mathbb F_q(t),S_d-\id, q^n)$ to denote the number of degree $d$, $S_d$
extensions $K/\mathbb F_q(t)$ of discriminant $q^n$. 
Since $S_d$ acts on the set $\{1, \ldots, d\}$, each element $g \in S_d$ acts on
the set $\{1, \ldots, d\}$ and we let $r(g)$ denote the number of orbits of this set under
the action of $g$.
For $c_i \subset S_d$ a conjugacy class, we use $\Delta(c_i) := d - r(g)$,
for any $g \in c_i$.
\end{notation}

\begin{definition}
	\label{definition:sigma-weighting}
	Let $\sigma(n_1, \ldots, n_\upsilon)$ denote the number of conjugacy classes of
$S_d$ whose image in the abelianization
$S_d^{\ab} \simeq \mathbb Z/2 \mathbb Z$ agrees with
the projection of $n_1c_1 + \cdots + n_\upsilon c_\upsilon$ to $S_d^{\ab}$.
\end{definition}

Here is our main result toward Bhargava's conjecture.

\begin{theorem}
	\label{theorem:bhargava-intro}
	Using notation from \autoref{notation:twisted-conf} and
	\autoref{definition:sigma-weighting}, if
	$q$ is sufficiently large depending on $d$, we
	have
	\begin{align}
	\label{equation:total-bound-intro}
\Delta(\mathbb F_q(t),S_d-\id , q^n) = 
	\sum_{\substack{n_1, \ldots, n_v \\ 
			\sum_{i=1}^\upsilon n_i
	\Delta(c_i) = n}}
	\sigma(n_1, \ldots, n_\upsilon)
	\left|\Conf_{n_1, \ldots, n_\upsilon, \mathbb F_q} (\mathbb
	F_q) \right| + o(q^{n}).
\end{align}
\end{theorem}
\autoref{theorem:bhargava-intro} is proven as part of the statement of
\autoref{theorem:bhargava-analog}.

\begin{remark}
	\label{remark:}
	We now describe some prior work toward Bhargava's conjecture.
	The case $d = 3$ over $\mathbb Q$ was due to Davenport-Heilbronn
\cite{davenportH:density-discriminants-cubica} and the cases $d= 4$ and $d = 5$
over $\mathbb Q$
were a substantial part of Bhargava's work leading to his Fields medal
\cite{bhargava:density-quartic,bhargava:enumeration-quintic}.
Over general global fields of characteristic not $2$ or $3$, the $d = 3$ case was handled by work of Datskovsky and Wright
\cite{datskovsky-wright:density-of-discriminants-of-cubic-extensions} while the
cases $d \leq 5$ and characteristic not $2$
was subsequently proven in
\cite{bhargavaSW:geometry-of-numbers-over-global-fields-i}.
\end{remark}

\begin{remark}
	\label{remark:}
	It is also possible to use the methods of this paper to count the number
	of $S_d$
	extensions of $\mathbb F_q(t)$ by other invariants, or variants thereof,
	and not just by discriminant.
	For example,
	one can easily
	adapt the argument to count extensions by discriminant, where one takes the
	discriminant of $K/\mathbb F_q(t)$ defined by $q^{\deg(\Omega_{C_K/\mathbb
	P^1_{\mathbb F_q}})}$, for $C_K$ the normalization of $\mathbb
	P^1_{\mathbb F_q}$ in the extension $K$ (instead of just counting the
		contribution to this from primes over $\mathbb A^1_{\mathbb F_q} \subset
	\mathbb P^1_{\mathbb F_q})$.
	In that case, if $\overline{c_i}$ denotes the image of the conjugacy
	class $c_i$ in $S_d^{\ab} \simeq \mathbb Z/2\mathbb Z$, the count would end up being
\begin{align}
	\label{equation:modified-bound-intro}
	\sum_{\substack{c^\partial \\ \text{ conjugacy classes in }S_d
	}} \left(  \sum_{\substack{n_1, \ldots, n_v \\ 
			\sum_{i=1}^\upsilon n_i
	\Delta(c_i) = n-\Delta(c^\partial) \\
	\overline{c^\partial} = \sum_{i=1}^\upsilon n_i \overline{c_i}
	}}
	\left|\Conf_{n_1, \ldots, n_\upsilon, \mathbb F_q} (\mathbb
F_q)\right| \right) + o(q^{n})
\end{align}
	in place of the right hand side of \eqref{equation:total-bound-intro}.
\end{remark}

The next remark is only intended for those familiar with 
Bhargava's conjecture \autoref{conjecture:bhargava} and we encourage the reader unfamiliar with
Bhargava's conjecture to skip it.

\begin{remark}
	\label{remark:}
	The reader familiar with Bhargava's conjecture may question in what
	sense \autoref{theorem:bhargava-intro} is an
	analog of Bhargava's conjecture 
	in the function field setting, given that the constants in 
	\eqref{equation:bhargava-prediction}
	and
	\eqref{equation:total-bound-intro} 
	look quite different at first glance.
	The reason we call it an analog of Bhargava's conjecture is that both predict the
	constant in the asymptotic growth of the number of $S_d$ extensions.

	We believe it would be interesting to understand the relation between
	the constants more closely.
	For example, the point counts of configuration space have an Euler
	product description which could relate them to the Euler product in 
	\eqref{equation:bhargava-prediction}.
	Also, Galois $S_d$ extensions of $\mathbb Q$ are always ramified to
	order $1$ or $2$ over the infinite place $\mathbb R$ of $\mathbb Q$ and
	consist either of $d!$ copies of
	$\mathbb R$ or $d!/2$ copies of $\mathbb C$. This suggests the constant
	$r_2(S_d)$ from \eqref{equation:bhargava-prediction} may
	be more related to counting the $\mathbb F_q$ points of
	those components of $\CHur^{S_d,c}_{n,\mathbb F_q}$ whose monodromy over $\infty$ has order
	$1$ or $2$, rather than all $\mathbb F_q$ points of 
	$[\CHur^{S_d,c}_{n,\mathbb F_q}/S_d]$ with arbitrary monodromy.
\end{remark}

\begin{remark}
	\label{remark:}
	It should likely be possible to prove a version of
	\autoref{theorem:bhargava-intro} counting extensions of $\mathbb F_q(t)$
	by reduced discriminant (instead of by the usual discriminant)
	using the results of
	\cite{landesmanL:homological-stability-for-hurwitz}.
	However, the results there are insufficient to count extensions by
	discriminant, and it is only through our refined computation of the
	stable homology of Hurwitz spaces ``in all directions,'' proved in
\autoref{theorem:one-large-stable-homology-hurwitz-space}, that we are
	able to count extensions by discriminant.
	See \cite[Remark
	11.1.3]{landesmanL:homological-stability-for-hurwitz} for further
	explanation.
\end{remark}

\subsection{Representation stability}
\label{subsection:intro-rep-stability}

One recent wave of developments in homological stability is that of
representation stability.
There is a natural representation stability question related to Hurwitz spaces.
Namely, let $\PConf_n \to \Conf_n$ 
denote the finite \'etale $S_n$ associated to specifying an ordering on the $n$ points.
That is, 
$\PConf_n \subset (\mathbb A^1_{\mathbb C})^n$ is the open subset parameterizing ordered
tuples of $n$ points in $\mathbb A^1_\mathbb C$.
It is known by 
\cite[Theorem 4.1]{churchF:representation-theory-and-homological-stability}
that $H_i(\PConf_n;\mathbb Q)$ satisfies representation stability as an $S_n$
representation, meaning that the multiplicities of certain $S_n$ representations
stabilize as $n$ grows.
In what follows, we will freely refer to the notion of a rack and its associated
Hurwitz space. For background on this, the reader can consult
\cite[\S2.1 and \S2.2]{landesmanL:homological-stability-for-hurwitz}.
For the reader's convenience, we recall the definition of a rack here.
See also \cite[Definition 2.1.1 and Remark
2.1.2]{landesmanL:homological-stability-for-hurwitz} for why this definition is
equivalent to other more standard definitions.
\begin{definition}
	\label{definition:rack}
	A {\em rack} is a set $c$ with an action map $\triangleright: c \times c
	\to c, (a,b) \mapsto a \triangleright b$ such that for all $n \geq 1$
	and all $1 \leq i \leq n-1$, the operation
	\begin{align*}
		\sigma_i : c^n & \rightarrow c^n \\
		(x_1, \ldots, x_{i-1}, x_i, x_{i+1}, x_{i+2}, \ldots, x_n) & \mapsto
		(x_1, \ldots, x_{i-1}, x_{i+1}, x_{i+1} \triangleright x_i, x_{i+2}, \ldots,
		x_n)
	\end{align*}
	defines an action of the braid group $B_n$, generated by $\sigma_1,
	\ldots, \sigma_{n-1}$, on $c^n$.
\end{definition}

Our results on homological stability for Hurwitz spaces
can be viewed as saying that the multiplicity of the trivial representation in 
$H_i(\CHur^c_n \times_{\Conf_n} \PConf_n;\mathbb Q))$
stabilizes, at least when $c$ is a rack with a single component (meaning the
action of $c$ on itself is transitive).
Given this, it is natural to ask if the multiplicities of other representations
(in the sense of representation stability) stabilize.
We verify this affirmatively in \autoref{theorem:representation-stability} when
$c$ has a single component, and note it is false if $c$ has more than one
component in \autoref{remark:non-splitting-stability}.

Although it may seem like representation stability for
Hurwitz spaces is a stronger statement than homological stability for Hurwitz
spaces, it turns out that, in combination with knowing their stable values, the
two are roughly equivalent. This is a testament to the power of working with
racks, as knowing representation stability for the rack $c$ turns out to be
roughly equivalent to proving usual homological stability for the Hurwitz space
associated to $c^{\boxtimes k}$, a rack consisting
of $k$ copies of $c$,
as defined in \autoref{definition:copies-rack}.

We now introduce notation to state this result precisely.

\begin{definition}
	\label{definition:representation-stability}
	Fix a finite rack $c$ with a single component. For each integer $n$, fix an partition $\lambda =
	(\lambda_1, \ldots, \lambda_p)$ and let $|\lambda| := \lambda_1+ \ldots+
	\lambda_p$. For any $n \geq |\lambda|$, define
	$\rho_{\lambda,n} : S_n \to \gl_{r_n}(\mathbb Q)$ to be the irreducible representation
	associated to the partition $(n - |\lambda|,\lambda_1,
	\lambda_2, \ldots, \lambda_p)$.
	This corresponds to a finite monodromy local system $\mathbb V_{\lambda,n}$ on $\Conf_n$ via the 
	representation $\pi_1(\Conf_n) \simeq B_n \to S_n
	\xrightarrow{\rho_{\lambda,n}}
	\gl_{r_n}(\mathbb C)$.
	Via pullback along the map $f_n: \CHur^c_n \to \Conf_n$,
	we obtain a local system $\mathbb H_{\lambda,n} := f_n^* \mathbb V_{\lambda,n}$.
	Let 
	$\PConf_n \to \Conf_n$ denote the $S_n$ cover associated to ordering the
	marked points.
	We say $H_*(\CHur^c_n \times_{\Conf_n} \PConf_n; \mathbb Q)$ has {\em semi-uniform linear representation stability} if there are
	constants $I$ and $J$ depending only on $c$, but independent of $\lambda$, so that
	$H_i(\CHur^c_n; \mathbb H_{\lambda,n})$ has dimension independent of
	$n$, for
$n- |\lambda| > Ii + J$. (This is equivalent to a more customary definition of
	representation stability, as explained in
\autoref{remark:rep-stability-reformulation}.)
\end{definition}
\begin{remark}
	\label{remark:rep-stability-reformulation}
	Using the notation for $\PConf_n,\rho_{\lambda,n}$, and $\mathbb
	H_{\lambda,n}$
from
\autoref{definition:representation-stability}, we can identify
the $\rho_{\lambda,n}$
isotypic part
of $H_i(\CHur^c_n \times_{\Conf_n} \PConf_n;\mathbb Q)$ with 
$H_i(\CHur^c_n; \mathbb H_{\lambda,n}) \otimes \rho_{\lambda,n}$.
Here, we view $H_i(\CHur^c_n; \mathbb H_{\lambda,n})$ as a trivial
representation of $S_n$.
This identification is explained,
for example, in the proof of \cite[Corollary
4.4]{churchF:representation-theory-and-homological-stability}.
The reason for our name above is that semi-uniform linear representation
stability implies that the multiplicity of
$\rho_{\lambda,n}$
in $H_i(\CHur^c_n \times_{\Conf_n} \PConf_n;\mathbb Q)$
stabilizes as $n$ grows.
\end{remark}

\begin{remark}
	\label{remark:}
	We call the above semi-uniform linear representation stability due to
	the presence of the term
	$|\lambda|$ in the inequality
$n- |\lambda|> Ii + J$. If instead the homology stabilized for
	$n > Ii + J$, one might naturally call this uniform linear
	representation stability.
	We expect uniform linear representation stability should in fact hold,
	but we weren't able to prove it. We think it would be quite interesting
	to do so.
	See also \autoref{remark:past-representation-stability-work}.
\end{remark}

We now state our main result on representation stability, which proves Hurwitz
spaces have semi-uniform linear representation stability and also identifies their
stable value.

\begin{theorem}
	\label{theorem:representation-stability}
	Fix a finite rack $c$ with a single component. With notation as in
	\autoref{definition:representation-stability}, the Hurwitz space $H_*(\CHur^c_n \times_{\Conf_n} \PConf_n; \mathbb Q)$ has 
	semi-uniform linear representation stability. Moreover, for $n$ sufficiently
	large, and every component $Z \subset \CHur^c_n$, the natural projection
	map $Z \subset \CHur^c_n \to \CHur^{c/c}_n \simeq \Conf_n$ 
	induced by $c \to c/c = *$
	induces an isomorphism
	$H_i(Z; \mathbb H_{\lambda,n}) \simeq
	H_i(\Conf_n; \mathbb V_{\lambda,n})$.
\end{theorem}
We prove this in \autoref{subsection:representation-stability-proof}. We note
that since $c$ has a single component above, $c/c$ is simply a point. In general,
if $c$ has $k$ components, $c/c$ is a rack with $k$ elements acting trivially on
itself, so $\CHur^{c/c}$ is a
$k$ colored configuration space.
\begin{remark}
	\label{remark:past-representation-stability-work}
	Before we even started working on this paper, we learned of a
	forthcoming result of Himes-Miller-Wilson, which has now appeared as
	\cite{himesMW:homological-stability-for-hurwitz}.
	They prove a uniform version
	of representation
	stability for Hurwitz spaces associated to a conjugacy class $c \subset
	G$ which has a certain non-splitting property, meaning that the
	intersection of $c$ with a subgroup of $G$ does not split into more than
	one conjugacy class in that subgroup.
	We also learned of related forthcoming work of Ellenberg-Shusterman
	\cite{ellenbergS:averages-of-arithmetic-functions}
	proving a result showing, in some cases, 
	$H_i(\CHur^c_n; \mathbb H_{\lambda,n}) = 0$ when $\lambda$ is a
	partition of the form $(k, 1^{n-k})$, corresponding to a wedge power of
	the standard representation of $S_n$.

	We only thought to consider the question of representation
	stability due to our knowledge of the above mentioned works. In fact, we learned of
	the relevant reference
	\cite[Theorem 2.4]{shusterman:the-tamely-ramified-geometric}
	from Jeremy Miller, and we would like to thank him for his helpful
	correspondence on this matter.
	Since their work only addressed the non-splitting case, we were curious
	whether one could remove this hypothesis and prove it for general racks
	with a single component. 
	Semi-uniform linear representation stability for general racks
	turned out
	to be a
	fairly immediate corollary of the main results of this paper, so we have 
	included a short proof.
	Of course, this does not imply the results of 
	\cite{himesMW:homological-stability-for-hurwitz} because they prove a
	stronger, uniform version representation stability.
\end{remark}

\subsection{Homological stability results}
\label{subsection:homological-stability}

We now discuss our main new results on the stable homology of Hurwitz spaces,
which enable us to deduce the above consequences to Bhargava's conjecture,
the BKLPR conjectures, and representation stability.

Recall the definition of a rack from \autoref{definition:rack}.
The components of a rack are the orbits of $c$ under the $\triangleright$ action of $c$ on
itself.
Let $c$ be a rack with components $c_1, \ldots, c_\upsilon$.
For $n_1, \ldots, n_\upsilon \in \mathbb Z_{\geq 0}$,
we use the notation $\cphurc {n_1, \ldots, n_\upsilon} c$
to denote the pointed Hurwitz scheme over $\mathbb C$ as defined in
\cite[Definition 2.2.2]{landesmanL:homological-stability-for-hurwitz}.
In the case $c$ is a union of conjugacy classes in a group, this 
is homotopic to the topological space parameterizing connected covers of a disc, together with a trivialization of the cover over a point on the
boundary of the disc, whose
inertia at every branch point is contained in $c$,
with $n_i$ branch points whose inertia lies in $c_i$.
In \cite[Theorem 1.4.1]{landesmanL:homological-stability-for-hurwitz}, we showed
that the homology of Hurwitz spaces stabilizes once one of the $n_i$'s is
sufficiently large,
and we computed the stable value of this homology when all $n_i$ were
sufficiently large in \cite[Theorem
1.4.2]{landesmanL:homological-stability-for-hurwitz}.
However, the above leaves open the natural question as to what the stable value
is when only $n_1$ is large, but the other $n_i$ for $i > 1$ are small.
We often refer to this colloquially throughout the paper as computing the stable
homology ``in all directions`` because we can let only a single one of the $n_i$ grow instead
of needing to have all of them be large.
\begin{example}
	\label{example:s3-module}
	For example, if $c = S_3 - \id$, so that $c = c_1 \cup c_2$ where $c_1$ is the
set of transpositions in $S_3$ and $c_2$ is the set of $3$-cycles in $S_3$,
prior to this paper there was no description of the stable homology of the Hurwitz space
$\cphurc {1, n_2} c$ for $n_2$ large.
\end{example}

We next introduce notation to state a result which provides a computation of
this stable homology, after
inverting a suitable set of primes.
\begin{definition}
	\label{definition:quotient-rack}
	If $c$ is a rack and $c' \subset c$ is a
	subrack, we say $c'$ is {\em normal} if its normalizer (see
	\autoref{definition:normalizer}) $N_c(c') = c$.
	If $c' \subset c$ is normal,
	one can form the {\em quotient rack} $c/c'$, as the rack whose underlying
set consists of equivalence classes of elements of $c$ under the equivalence
relation generated by equivalences of the form $x \sim y$ if there is some $w
\in c'$ so that $w \triangleright x = y$.
Using the notation $\overline{x} \in c/c'$ to denote the equivalence class of $x
\in c$, one can give $c/c'$ the structure of a rack by declaring $\overline{x}
\triangleright \overline{y} := \overline{x \triangleright y}$;
we verify this is independent of the choice of lifts $x$
and $y$ later in \autoref{lemma:quotient-well-defined}.
\end{definition}

\begin{definition}
	\label{definition:relative-structure-group}
	Suppose $c$ and $c'$ are two racks and we are given an action of $c'$ on
	$c$. 
	We use $\aut(c)$ to denote the automorphisms of the underlying set of
	$c$ (so these automorphisms do not have any relation to the
	$\triangleright$ operation on $c$).
	Define the {\em relative
	structure group} $G^{c'}_{c}$ to be the subgroup of $\aut(c)$ generated by the action of $c'$ on
		$c$.
\end{definition}
\begin{example}
	\label{example:relative-subrack}
	The reduced structure group of a rack $c$, which is the subgroup of
	automorphisms of $c$ generated by $x \triangleright$ for $x \in c$,
	is often notated $G^0_c$. In the context of this paper, we notate it as
	$G^c_c$.
	If $c' \subset c$ is a subrack, then the relative structure group
	$G^{c'}_c$ is the
	subgroup of $G^c_c$ generated by elements of $c'$.
\end{example}
\begin{example}
	\label{example:normal-relative-subrack}
	If $c' \subset c$ is a normal subrack, then $c$ acts on $c'$, and so we
	can form the relative structure group $G^c_{c'}$. We have $G^{c'}_{c'} \subset G^c_{c'}$ as the subgroup generated by elements of
	$c'$.
\end{example}

The next theorem computes the stable homology of Hurwitz spaces in all
directions.
\begin{theorem}
	\label{theorem:one-large-stable-homology-hurwitz-space}
	Let $c$ be a finite rack whose connected components are $c_1, \ldots, c_\upsilon$.
	Then there are constants $I$ and $J$, depending only on $|c_1|$ and
	the maximum order of an element of $c_1$ acting on $c$, with the
	following property. For any
	$i \geq 0$, any $n_1 > Ii + J$, 
	and any component $Z \subset \cphurc {n_1, \ldots, n_\upsilon} c$
	mapping to a component $Z' \subset \cphurc {n_1, \ldots, n_\upsilon}{c/c_1}$
	under the map $\cphurc {n_1, \ldots, n_\upsilon} c \to 
\cphurc {n_1, \ldots, n_\upsilon}{c/c_1}$
induced by $c \to c/c_1$,
	the map
	$H_i(Z; \mathbb Z[|G^{c'}_{c}|^{-1}]) \to H_i(Z'; \mathbb
	Z[|G^{c'}_{c}|^{-1}])$
	is an isomorphism.
\end{theorem}
\autoref{theorem:one-large-stable-homology-hurwitz-space} is essentially
equivalent to 
\autoref{theorem:stable-homology-ring}
and we spell out the details of this equivalence in 
\autoref{subsubsection:stable-ring-proof}.

\begin{example}
	\label{example:generating-is-conf}
	An important special case of
	\autoref{theorem:one-large-stable-homology-hurwitz-space} occurs when
	$c_1$ generates $c$ so that $c/c_1 = c/c$ and so 
$\cphurc {n_1, \ldots, n_\upsilon}{c/c_1}$ is a multicolored configuration space
on $\upsilon$ colors.
In this case, we are able to identify the stable homology of each component of Hurwitz spaces with
the homology of the corresponding $\upsilon$ colored configuration spaces, which can
in turn be identified with the homology of
the free $\mathbb
E_2$ algebra on $\upsilon$ generators. The homology of this space is completely understood, see
\cite[\S16]{galatius2018cellular} for a modern reference.
\end{example}

One can think about 
\autoref{theorem:one-large-stable-homology-hurwitz-space} as describing what
the homology of the Hurwitz space $\CHur^c$ stabilizes to when we consider it as a module over
the Hurwitz space $\CHur^{c_1}$. For example, in \autoref{example:s3-module} we consider the
Hurwitz space with a single $3$ cycle and an arbitrary number of transpositions
as a module for the Hurwitz with no $3$ cycles and an arbitrary number of
transpositions.
From this perspective, it is natural to consider Hurwitz modules more generally.

In \autoref{definition:hurwitz-module}, we define a notion of Hurwitz module $S$
over $c$,
which is essentially a module for a Hurwitz space. Topologically, this also be viewed as a
union of
covering spaces of configuration space on a genus $g$ surface with $f$
punctures and one boundary component, and we label the corresponding space
$\Hur^{c,S}$, as defined in \autoref{definition:configuration-space}.

We also define a notion of bijective Hurwitz module in
\autoref{definition:bijective-hurwitz-module}, where the relevant actions 
on the sets defining the module are bijective.
If we let $c_1, \ldots, c_\upsilon$ denote the $S$-components of $c$ (i.e.,
	minimal subsets closed under the joint actions coming from $c$ and $S$
as defined in \autoref{definition:s-component},)
$\Hur^{c,S}_{n_1, \ldots, n_\upsilon}$ is the union of components of 
$\Hur^{c,S}$ parameterizing configurations with $n_i$ points labeled by an
element of $c_i$.

We are able to prove 
bijective Hurwitz modules satisfy a certain form of homological stability.
If one works with the whole Hurwitz module $\Hur^{c,S}$ it will not satisfy homological
stability. Indeed, this can already be seen in the case of Hurwitz spaces
$\Hur^c$, when
$c$ comes from a conjugacy class in a group, since
in general one needs to restrict to covers with connected source.
The union of components parameterizing such covers with connected source 
was denoted $\CHur^c$ in \cite[Definition
2.2.2]{landesmanL:homological-stability-for-hurwitz}.
Generalizing this, we define $\CHur^{c,S}$ in
\autoref{construction:S-filtration}, which roughly describes the union of
components of $\CHur^{c,S}$ not contained in any Hurwitz module associated
$\CHur^{c',S'}$ over some subset $(c'',S'') \subsetneq (c,S)$, in the sense of
\autoref{definition:subset}.
We can now state our main result explaining how the homology of these Hurwitz
modules stabilize.
For the next statement, we let $\CHur^{c,S}_{n_1, \ldots, n_\upsilon}$ denote
the union of components of $\Hur^{c,S}_{n_1, \ldots, n_\upsilon}$ 
also lying in 
$\CHur^{c,S}$.

\begin{theorem}
	\label{theorem:homology-stabilizes-intro}
	Let $c$ be a finite rack and let $S$ be a finite bijective Hurwitz module over
	$c$.
	Let $c_1, \ldots, c_\upsilon$ denote the $S$-components of $c$.
	Using notation from \autoref{definition:bijective-hurwitz-module},
	there are constants $I$ and $J$, depending on $|c_1|$ and the
	maximal order of an element of $c_1$ acting on $c$, with the following
	property. For any
	$i \geq 0$ and $n_1 > Ii + J$, 
	any element $x \in c_1$ induces an isomorphism
	$H_i(\cphurc {n_1, \ldots, n_\upsilon} {c,S};\mathbb Z) \to H_i(\cphurc
		{n_1+1, \ldots, n_\upsilon}
	{c,S};\mathbb Z)$.
\end{theorem}
A statement equivalent to \autoref{theorem:homology-stabilizes-intro}, but written in a
slightly different language is proven in 
\autoref{theorem:stable-homology}.
A closely related homological stability theorem covering some some
special cases of \autoref{theorem:homology-stabilizes-intro} was proven in
\cite[Theorem
4.2.6]{ellenbergL:homological-stability-for-generalized-hurwitz-spaces}.

In addition to showing the homology of Hurwitz modules stabilize, we also
describe their stable value.
To state this result, generalizing the notion of quotient rack from
\autoref{definition:quotient-rack}, we also will need to be able to quotient a
bijective Hurwitz module $S$ over $c$ by an $S$-component $c' \subset c$.
We denote this quotient by $S/c'$, which we define precisely in \autoref{definition:quotient-hurwitz-module}.

In addition to the above notion of quotient Hurwitz module, 
we will need the notion of the relative structure group of a
subrack $c' \subset c$, defined in \autoref{definition:relative-structure-group}, which records
the action of $c'$ on $c$. For $S$ a Hurwitz module over $c$, we also need the notion
of the module structure group $G^{c'}_{S}$
from 
\autoref{definition:module-structure-group}, 
which, loosely speaking records the actions
of collections of elements from $c'$ on $S$.
We show $G^{c'}_S$ is a finite group when $c$ and $S$ are finite in
\autoref{lemma:module-structure-group-finite}.
\begin{theorem}
	\label{theorem:one-large-stable-homology}
	Let $c$ be a finite rack and $S$ a finite bijective Hurwitz module over $c$
	as in \autoref{definition:bijective-hurwitz-module}.
	Let $c_1, \ldots, c_\upsilon$ denote the $S$-components of $c$.
	There are constants $I$ and $J$, depending only on $|c_1|$ and the
	minimal order of an element of $c_1$ acting on $c$, so that for any
	$i \geq 0$ and $n_1 > Ii + J$, 
	and any component $Z \subset \cphurc {n_1, \ldots, n_\upsilon} {c,S}$
	mapping to a component $Z' \subset \cphurc {n_1, \ldots, n_\upsilon}
	{c/c_1,S/c_1}$ under the map $\cphurc {n_1, \ldots, n_\upsilon} {c,S} \to \cphurc {n_1, \ldots, n_\upsilon}{c/c_1,S/c_1}$ induced by $c \to c/c_1$,
	the map
	$H_i(Z; \mathbb Z[|G^{c'}_{c}|^{-1}, |G^{c}_{c'}|^{-1}, |G^{c'}_{S}|^{-1}]) \to H_i(Z'; \mathbb
	Z[|G^{c'}_{c}|^{-1},|G^{c}_{c'}|^{-1},|G^{c'}_{S}|^{-1}])$
	is an isomorphism.
\end{theorem}
We prove \autoref{theorem:one-large-stable-homology} in
\autoref{subsubsection:stable-value-proof}.
\begin{remark}
	\label{remark:uncomputable-stable-cohomology}
	The description of
	\autoref{theorem:one-large-stable-homology} relates the stable value of the homology of these Hurwitz spaces
	to the homology of a smaller Hurwitz space.
	In complete generality, this stable homology seems uncomputable as it can, in
	some sense, involve all the unstable homology that appears in arbitrary
	Hurwitz spaces.

	However, in many circumstances, such as in
	\autoref{example:generating-is-conf},
		the smaller
	Hurwitz space may be a configuration space, in
	which case it is relatively manageable. We will see this is the case
	in all three of the main applications of this paper.
\end{remark}

\subsection{Summary of the proofs}
\label{subsection:proof-idea}
We focus on explaining the new ideas in computing the stable homology of
Hurwitz modules. 
One can obtain our applications from 
our topological results without much difficulty using 
prior work.
The general strategy is similar to that used to prove our analogous results for Hurwitz spaces in
\cite{landesmanL:homological-stability-for-hurwitz}.
To show the homology stabilizes, we first need to show the homology of a certain
quotient by all elements of $c$ stabilizes.
This follows by combining a previous result we proved to show such homology
stabilizes in \cite[Theorem 3.1.4]{landesmanL:homological-stability-for-hurwitz}
with various scanning argument similar to those carried out in
\cite[Appendix A]{landesmanL:the-stable-homology-of-non-splitting}.
A key new feature is that we also have to apply scanning arguments to higher
genus curves with punctures, but a point pushing homotopy carried out in
\autoref{lemma:deform-epsilon-distance} allows us to cut such surfaces up into a
union of rectangles, reducing the situation to one similar to the case of a
disc.
Once we show the homology of this quotient stabilizes we need to show the
homology stabilizes before quotienting as well.
To do so, the key input is a comparison between a certain bar construction
related to $c$ and a bar construction related to $c'$, for $c' \subset c$ a
subrack, which we prove in
\autoref{proposition:subrack-comparison}.
The proof of 
\autoref{proposition:subrack-comparison} is similar to \cite[Proposition
4.5.11]{landesmanL:the-stable-homology-of-non-splitting}
though many aspects
are substantially trickier, 
as we have to verify that general bijective Hurwitz modules satisfy certain
desirable properties that are obviously satisfied by racks.

Once we prove homological stability, the remaining task is to compute the stable
value of this homology.
A substantial insight of this paper is that the particularly simple answer can be succinctly
described in terms of racks. 
Although the proof is 
inspired by our proof that the
homology stabilizes,
a number of additional subtleties arise.
The general strategy is to produce a comparison map to the stable homology and
use a descent argument to reduce to verifying that a certain complex is
nullhomotopic.
However, because this nullhomotopy is only true rationally and does not hold
integrally, it is not possible
to produce a nullhomotopy on the level of spaces which will induce one on
chains.
Instead, we argue directly on the level of chains. 
Even after we verify the relevant complexes are nullhomotopic,
to relate this nullhomotopy to our stable homology, we encounter a technical
issue that we need to commute certain tensor products with pullbacks. We verify
this by proving certain relevant maps of simplicial sets are Kan fibrations and applying
a result of Bousfield-Friedlander.

\subsection{Outline}
\label{subsection:outline}
The structure of our paper is as follows.
We first prove some preliminary 
results about bijective Hurwitz modules in \autoref{section:hurwitz-modules}.
We then use scanning arguments to identify explicit models for certain bar
constructions in \autoref{section:scanning}.
Next, in 
\autoref{section:quotient-stability}, we show the quotient of Hurwitz
modules by all element of $c$ has homology which stabilizes.
In \autoref{section:equivalence} we prove a technical result comparing two bar
constructions, which will enable us to undo the above mentioned quotienting
procedure. We carry out this unquotienting in 
\autoref{section:homological-stability} to prove Hurwitz modules satisfy
homological stability.
We compare cohomology of certain tensor products in
\autoref{section:chain-homotopies}, which serves as one of the key technical
ingredients
to 
compute the stable homology of Hurwitz modules in 
\autoref{section:stable-homology}.
We explain our application to the BKLPR conjectures in 
\autoref{section:poonen-rains},
to
Bhargava's conjecture in
\autoref{section:bhargava-conjecture}, 
and to representation stability in
\autoref{section:representation-stability}.
We conclude with some further questions in \autoref{section:further-questions}.

\subsection{Acknowledgements}

Some of this project was completed at the University Copenhagen when the first
author visited the second. We would like to thank University of Copenhagen for
their hospitality and their excellent working
conditions.
We would like to thank Will Sawin for a number of extremely helpful discussions
and Jordan Ellenberg for suggesting improvements
to many aspects of the paper.
We thank Jeremy Miller for helpful discussions related to
representation stability of Hurwitz spaces.
We also thank 
Jef Laga,
Zhao Yu Ma,
Adam Morgan,
Peter Patzt,
Mark Shusterman,
Nathalie Wahl,
Jenny Wilson,
Melanie Wood,
and Wei Zhang
for helpful discussions.
Landesman 
was supported by the National Science
Foundation 
under Award No.
DMS 2102955.
Levy was supported by the Clay Research Fellowship.

\section{Hurwitz modules}
\label{section:hurwitz-modules}

In this section, we will define Hurwitz modules, whose homology is the central
object of study throughout the paper.
These seem to be a fairly general setting for many natural questions in geometric topology and arithmetic statistics over
function fields can be framed.
We first define Hurwitz modules in
\autoref{subsection:hurwitz-module-definition}, we then investigate the notion of
subsets of Hurwitz modules in \autoref{subsection:subsets}, and finally we
discuss quotients of Hurwitz modules in \autoref{subsection:quotients}.

\subsection{Definition of Hurwitz modules}
\label{subsection:hurwitz-module-definition}

Our main results concern the stable homology of Hurwitz modules, which we define
now.
This definition is quite similar to the definition of coefficient system given
in
\cite[Definition
3.1.6]{ellenbergL:homological-stability-for-generalized-hurwitz-spaces}
except that our Hurwitz modules are set valued instead of vector space
valued.

\begin{definition}
	\label{definition:hurwitz-module}
	Let $\Sigma^1_{g,f}$ denote a genus $g$ surface with $f$ punctures and
	$1$ boundary component.
Let $B^{\Sigma^1_{g,f}}_n$ denote the surface braid group associated to $n$
points on
$\Sigma^1_{g,f}$.
	Fix a rack $c$. A {\em Hurwitz module over $c$}
	is a triple
$S = ({\Sigma^1_{g,f}}, \{T_n\}_{n \in \mathbb
Z_{\geq 0}}, \{\psi_n\}_{n \in \mathbb Z_{\geq 0}})$ 
	where $g,f \in \mathbb Z$,
	$T_0$ is a set, $T_n := c^n \times T_0$ is a set,
	and
$\psi_n: B^{\Sigma^1_{g,f}}_n \times T_n \to T_n$
	is a left action of the
	surface braid group on the set $T_n$,
	such that for $0 \leq i \leq n$
	the diagram
	\begin{equation}
		\label{equation:}
		\begin{tikzcd} 
			(B^{\Sigma^1_{0,0}}_i \times B^{\Sigma^1_{g,f}}_{n-i})
			\times (c^i
			\times (c^{n-i} \times T_0)) \ar {r} \ar {d} & c^i
			\times (c^{n-i} \times T_0) \ar {d} \\
			B^{\Sigma^1_{g,f}}_n \times c^n \times T_0 \ar {r} & c^n
			\times T_0
	\end{tikzcd}\end{equation}
	commutes; the maps in the above diagram are defined as follows.
	The top horizontal map is induced by the action of 
$B^{\Sigma^1_{0,0}}_i \simeq B_n$ on $c^i$ from the definition of $c$ (see
\autoref{definition:rack}) and 
the action maps defining the Hurwitz module.
The left vertical map comes from the inclusion
$B_i^{\Sigma^1_{0,0}} \times B_{n-i}^{\Sigma^1_{g,f}} \subset
B_n^{\Sigma^1_{g,f}}$
constructed in 
\cite[Notation
3.1.1]{ellenbergL:homological-stability-for-generalized-hurwitz-spaces},
where we used the notation $B^n_{g,f}$ instead of $B_n^{\Sigma^1_{g,f}}$.

Given a Hurwitz module $S$ as above, we call $T_n$ the $n$-set of $S$.
In particular, when $n = 0$, $T_0$ is the $0$-set of $S$.

We say $S$ is {\em finite} if $c$ is finite and $T_0$ is finite.
\end{definition}

The above notion of Hurwitz modules seems too general for the proofs of many of our main
results, and we will mostly work in the slightly more restricted setting
of bijective Hurwitz modules.
\begin{definition}
	\label{definition:bijective-hurwitz-module}
	Fix a rack $c$. A {\em bijective Hurwitz module} over $c$ is a Hurwitz
	module	
$S = ({\Sigma^1_{g,f}}, \{T_n\}_{n \in \mathbb
Z_{\geq 0}}, \{\psi_n\}_{n \in \mathbb Z_{\geq 0}})$ such that the maps
$B^{\Sigma^1_{g,f}}_1 \times c \times T_0 \xrightarrow{\psi_1} c \times T_0 \to
c$, and
$B^{\Sigma^1_{g,f}}_1 \times c \times T_0 \xrightarrow{\psi_1} c \times T_0 \to
T_0$, 
induce maps
$B^{\Sigma^1_{g,f}}_1 \times T_0 \to \aut(c)$ and
$B^{\Sigma^1_{g,f}}_1 \times c \to \aut(T_0)$.
For $\gamma \in
B^{\Sigma^1_{g,f}}_1$ and $t \in T_0$, we denote
the first map by
$\sigma^\gamma_t:c \to c$
and for
$\gamma \in
B^{\Sigma^1_{g,f}}_1$ and $x \in c$ we denote the second map by $\tau^\gamma_x : T_0 \to T_0$.

We say $S$ is {\em finite} if it the corresponding Hurwitz module is finite in
the sense of \autoref{definition:hurwitz-module}.
\end{definition}

\begin{example}
	\label{example:group-hurwitz-modules}
One important class of examples of bijective Hurwitz modules 
are obtained by taking $G$ to be a finite group, $c \subset G$ a union of
conjugacy classes, and taking its $0$ set $T_0$ to be the set of maps
$\hom(\pi_1(\Sigma^1_{g,f}),G)$.
See \cite[Example
3.1.9]{ellenbergL:homological-stability-for-generalized-hurwitz-spaces} for a
detailed explanation of this example.
\end{example}

Just as it was important to split up racks into components in
\cite{landesmanL:homological-stability-for-hurwitz}, it will also be
convenient to split up Hurwitz modules into their corresponding components, which
we define next.
\begin{definition}
	\label{definition:s-component}
	For $c$ a rack and $S$ a bijective Hurwitz module over $c$, an {\em
	$S$-component}
	of $c$ is a subset $z \subset c$ which is a minimal nonempty subset of $c$ closed
	under the action of $c$ on itself and closed under the action of
	$B_1^{\Sigma^1_{g,f}} \times T_0$ on $c$.
\end{definition}

We next introduce notation for the schemes over the complex numbers which are
naturally associated to Hurwitz modules.

\begin{definition}
	\label{definition:configuration-space}
	Let $c$ be a rack and $S=({\Sigma^1_{g,f}}, \{T_n\}_{n \in \mathbb
Z_{\geq 0}}, \{\psi_n\}_{n \in \mathbb Z_{\geq 0}})$
	be a bijective Hurwitz module over $c$.
	Let $\Conf_n^{\Sigma^1_{g,f}}$ denote the configuration space
	parameterizing $n$ distinct points on the interior of 
	$\Sigma^1_{g,f}$.
	Upon identifying $B^n_{g,f} \simeq \pi_1(\Conf_n^{\Sigma^1_{g,f}})$, we
	can view the bijective Hurwitz module as yielding an action
	$B^n_{g,f} \to \aut(c^n \times T_0)$. Define $\Hur^{c,S}_n$ as the
	topological space which is the unramified covering space of
	$\Conf_n^{\Sigma^1_{g,f}}$ corresponding to the above action. In
	particular, this covering space has degree $|c|^n \cdot |T_0|$.
	Suppose $c$ has $S$-components 
	$c_1, \ldots, c_\upsilon$.
	Suppose $n_1 + \cdots
	+ n_\upsilon = n$ and let let $S^{n_1, \ldots, n_\upsilon} \subset c^n
	\times T_0$ denote
	the subset such that there are 
	$n_i$ points with labels in $c_i$.
	Then let $\Hur^{c,S}_{n_1, \ldots, n_\upsilon}$ denote the 
	unramified covering space of $\Conf_{n_1, \ldots, n_\upsilon}^{\Sigma^1_{g,f}}$
	corresponding to the map $B_n^{\Sigma^1_{g,f}} \to \aut(S^{n_1, \ldots,
	n_\upsilon})$.
\end{definition}
\begin{warn}
	\label{warning:}
	The components $c_1, \ldots, c_\upsilon$ from
	\autoref{definition:hurwitz-module} depend on $S$. In particular, there can
	be fewer components under the joint action of $c$ and 
	$B_1^{\Sigma^1_{g,f}}\times T_0$
	than the number of components of $c$ under only the
	action of $c$ on itself.
\end{warn}

\begin{example}
	\label{remark:}
	In the case $g = f = 0$, we can take $T_0 = *$ and we obtain
	$\Hur^{c,S}$ recovers the usual Hurwitz space
	$\Hur^c$.
\end{example}
\subsection{Subsets of Hurwitz modules}
\label{subsection:subsets}

In this subsection, we define the notion of subsets of Hurwitz modules, which is
the natural notion of an inclusion of Hurwitz modules over an inclusion of
racks.
If $c$ is a rack, $c'$ is a subrack, and $S$ is a bijective Hurwitz module over
$c$, we will define a maximal subset over $c'$, denoted $S_{c'}$.
The main challenge of this subsection, proven in
\autoref{lemma:normalizer-is-hurwitz-module}, will be to show that there is a subset
over $N_c(c')$, the normalizer of $c'$ in $c$, with the same $0$-set as
$S_{c'}$.

\begin{definition}
	\label{definition:subset}
	Let $c$ be a rack and $S$ be a Hurwitz module over $c$.
	Let $c' \subset c$ be a subrack. We say a bijective Hurwitz module $S'$ over
	$c'$ is a {\em subset} of $S$ over $c$ if there is an inclusion $T'_0
	\subset T_0$ which induces commuting diagrams
	\begin{equation}
		\label{equation:}
		\begin{tikzcd} 
			B_{n}^{\Sigma^1_{g,f}} \times T'_n \ar {r} \ar
			{d} & T'_n \ar {d} \\
			B_{n}^{\Sigma^1_{g,f}} \times T_n \ar {r} &
			T_n.
	\end{tikzcd}\end{equation}
	We write $(c',S') \subset (c,S)$
	to indicate that $S'$ is a subset of $S$.
\end{definition}

Here are several equivalent descriptions of the notion of a subset.
\begin{lemma}
	\label{lemma:criterion-to-be-subset}
	Suppose $c$ is a rack and $S = (\Sigma^1_{g,f}, \{T_n\}_{n \in \mathbb
	Z}, \{\psi_n\}_{n \in \mathbb Z_{\geq 0}})$ is a bijective Hurwitz module over $c$.
	Fix a base point $\star$ on the boundary of $\Sigma^1_{g,f}$.
	If $c' \subset c$ is a subrack, and $T'_0 \subset T_0$ is a subset, then
	the following are equivalent:
	\begin{enumerate}
		\item The data $S' = (\Sigma^1_{g,f}, \{(c')^n \times T'_0\}_{n \in \mathbb
				Z}, \{\psi_n|_{(c')^n \times T'_0}\}_{n \in
			\mathbb Z_{\geq 0}})$ forms a bijective Hurwitz module such
			that $S'$ over $c'$ is a subset of $S$ over $c$.
\item For any $x \in c'$, $t \in T'_0$ and any
	$\gamma \in \pi_1(\Sigma^1_{g,f}, \star) = B_1^{\Sigma^1_{g,f}}$,
	$\psi_1(\gamma, x,t) \in c' \times T'_0 \subset c \times T_0$.
\item Fix a set of generators $\{\gamma_i\}$ of $B_1^{\Sigma^1_{g,f}}$.
	for any $x \in c'$, $t \in T'_0$ and any $\gamma_i$,
$\psi_1(\gamma_i, x,t) \in c' \times T'_0 \subset c \times T_0$.
	\end{enumerate}
	\end{lemma}
\begin{proof}
	The final two statements are equivalent since $\psi_1$ defines an action
	of $B_1^{\Sigma^1_{g,f}}$ on $T_1$.
	The first statement easily implies the second, so it remains to check
	the second implies the first.
	That is, we need to show that if 
	$\psi_1(\gamma, x,t) \in c' \times T'_0 \subset c \times T_0$ for all
	$\gamma,x,t$ as above, then 
	$\psi_n(B_n^{\Sigma^1_{g,f}} \times (c')^n \times T'_0)$ has image
	contained in $(c')^n \times T'_0$ for all $n$.
Note that the surface braid group $B_n^{\Sigma^1_{g,f}}$ is
	generated by $B_n^{\Sigma^1_{0,0}} \subset B_n^{\Sigma^1_{g,f}}$ 
	and $B_1^{\Sigma^1_{g,f}} \subset B_n^{\Sigma^1_{g,f}}$.
	The former acts on $(c')^n$ and
		preserves the $T'_0$ coordinate, as follows from
	\autoref{definition:hurwitz-module} and the definition of $c'$
	being a subrack.
	The latter acts on $c' \times
T'_0$ by assumption and preserves the first $(c')^{n-1}$ coordinates.
Combining this shows that 
$\psi_n(B_n^{\Sigma^1_{g,f}} \times (c')^n \times T'_0) \subset (c')^n \times
	T'_0$
	as every generator of $B_n^{\Sigma^1_{g,f}}$ sends $(c')^n \times
	T'_0$ to itself.
\end{proof}

The following lemma can easily be verified, for example, using the second
criterion from \autoref{lemma:criterion-to-be-subset}.
\begin{lemma}
	\label{lemma:largest-preserving-c-prime}
	Let $c$ be a rack, $S$ a bijective Hurwitz module over
	$c$, and $c' \subset c$ a subrack.
	If $(c',S_1) \subset (c,S)$ and $(c',S_2) \subset (c,S)$ are two subsets
	in the sense of \autoref{definition:subset},
	then $(c',S_1 \cup S_2) \subset (c,S)$.
\end{lemma}

With the above lemma, we can now define the notion of a maximal subset
associated to a subrack.
This will later be used to define a notion of the connected Hurwitz space
associated to a subrack.
\begin{notation}
	\label{definition:subsystem-for-subrack}
	Let $c$ be a rack and $S$ be a bijective Hurwitz module over $c$.
	For $c' \subset c$ a subrack, define $S_{c'}$ to be the bijective
	Hurwitz module over $c'$ which is maximal among all subsets,
	$(c', S_{c'}) \subset (c,S)$ in the sense of
	\autoref{definition:subset}.
	We note this is well defined by
	\autoref{lemma:largest-preserving-c-prime}.
\end{notation}

\begin{definition}
	\label{definition:normalizer}
	For $c$ a rack and $c' \subset c$ a subrack, we use $N_c(c')$, the {\em
normalizer} of
$c'$ in $c$, to denote the set of $x \in c$ so that $x \triangleright y \in c'$
for every $y \in c'$.
\end{definition}
\begin{lemma}
	\label{lemma:normalizer-property}
	For $c$ a rack and $c' \subset c$ a subrack,
	if $x \in c'$ and $y \in N_c(c')$ then $x \triangleright y \in N_c(c')$.
\end{lemma}
\begin{proof}
	Note that the set $N_c(c')$ is preserved by rack automorphisms of $c$ preserving $c'$. $x \triangleright (-)$ is such an automorphism, concluding the proof.
\end{proof}

We next aim to show that if $c' \subset c$ is a subrack, $S$ is a coefficient
system for $c$, there is a subset $(N_c(c'), S') \subset (c, S)$ so that $S'$
has the same $0$
set as $S_{c'}$.
The following lemma will be an important stepping stone, which unwinds the
conditions to be a bijective Hurwitz module.

\begin{lemma}
	\label{lemma:three-conditions}
	Suppose $c$ is a rack, $c' \subset c$ is a subrack.
	Let $S = ({\Sigma^1_{g,f}}, \{T_n\}_{n \in \mathbb
	Z_{\geq 0}},\{\psi_n\}_{n \in \mathbb Z_{\geq 0}})$ 
	be a bijective Hurwitz module over $c$.
	Fix two points $p_1$ and $p_2$ in $\Sigma^1_{g,f}$
	and a standard generating set for $\pi_1(\Sigma^1_{g,f})$ 
	of the form $\Delta := \{\alpha_1, \beta_1, \ldots, \alpha_g, \beta_g,\gamma_1,
	\ldots, \gamma_f\}$ as in
	\cite[\S2.2]{bellingeri:on-presentations-of-surface-braid-groups} (see
	\autoref{remark:opposite-gamma-convention}).
For any $\gamma \in \Delta$,
$x, y \in c$ and $t \in T_0$,
	\begin{align}
		\label{equation:first-relation-same}
		(x \triangleright^{-1} y) \triangleright^{-1} \sigma_t^\gamma(x)
		&= (x \triangleright^{-1} \sigma_t^\gamma(y))\triangleright^{-1}
		\sigma^\gamma_{\tau_y^\gamma(t)}(x) \\
		\label{equation:second-relation-same}
\sigma^\gamma_{\tau_x^\gamma(t)}(x \triangleright^{-1} y) &= x \triangleright^{-1}
\sigma_t^\gamma(y) \\
		\label{equation:third-relation-same}
		\tau_{x \triangleright^{-1} y}^\gamma(\tau_x^\gamma(t)) &=
	\tau_x^\gamma(\tau_y^\gamma(t)).
	\end{align}
	If $\gamma = \alpha_i, \phi=\beta_i$ for some $i$, and
$x, y \in c$ and $t \in T_0$,
	\begin{align}
		\label{equation:first-relation-pair}
(x \triangleright^{-1} \sigma_t^\gamma(y))
			\triangleright^{-1} \sigma^\phi_{\tau_y^\gamma(t)}(x)
			&=
y \triangleright^{-1} \sigma_t^\phi(y \triangleright x) \\
		\label{equation:second-relation-pair}
x \triangleright^{-1} \sigma_t^\gamma(y) &=
\sigma^\gamma_{\tau_x^\phi(t)}(y) \\
		\label{equation:third-relation-pair}
\tau_{y}^\gamma(\tau_x^\phi(t))&=
\tau_x^\phi(\tau_y^\gamma(t)).
	\end{align}
Finally, $\gamma \neq \phi \in \Delta$ are two distinct paths with
$\{\gamma,\phi\} \neq \{\alpha_i, \beta_i\} \subset \Delta$,
such that $\phi$ is situated above $\gamma$ in the model $\mathcal
M_{g,f,1}^\epsilon$ of \autoref{notation:point-pushed-quotient},
then, for
$x, y \in c$ and $t \in T_0$,
	\begin{align}
		\label{equation:first-relation-different}
y \triangleright^{-1} \sigma_t^\phi(y \triangleright x)
&=
\sigma_t^\gamma(y)
			\triangleright^{-1} \sigma^\phi_{\tau_y^\gamma(t)}(\sigma_t^\gamma(y)
			\triangleright x)\\
		\label{equation:second-relation-different}
\sigma^\gamma_{\tau_x^\phi(t)}(y) &=
\sigma_t^\gamma(y)
\\
		\label{equation:third-relation-different}
\tau_{y}^\gamma(\tau_x^\phi(t)) &=
\tau_{\sigma_t^\gamma(y)
\triangleright x}^\phi(\tau_y^\gamma(t)).
	\end{align}
\end{lemma}
\begin{remark}
	\label{remark:opposite-gamma-convention}
	We can think of the paths $\alpha_i, \beta_i, \gamma_i$ in 
	\autoref{lemma:three-conditions} in terms of the model $\mathcal
	M_{f,g,1}^\epsilon$ of \autoref{notation:point-pushed-quotient} as
	starting from a lower point on the left boundary and moving horizontally
	until it reaches a higher point.
	In particular, this is the opposite direction of the allowable paths we
	choose later in \autoref{definition:allowable-and-output}.
	However, it is convenient for us to use this opposite convention here to
	be able to directly apply the results of
	\cite{bellingeri:on-presentations-of-surface-braid-groups}.
\end{remark}

\begin{proof}
	Let $\eta \in B_2^{\Sigma_{0,0}^1} \subset B_2^{\Sigma_{g,f}^1}$
	denote the element corresponding to moving $p_1$ (labeled by $x$) counterclockwise under 
	$p_2$ (labeled by $y$), correspond to the map $c^2 \to c^2, (x,y) \mapsto (x
	\triangleright^{-1} y, x)$. 
	(This is notated as $\sigma_1^{-1}$ in \cite[Theorem
	1.1]{bellingeri:on-presentations-of-surface-braid-groups}.)
	Let us begin by computing the result of applying several braid group
	elements to $(x,y,t)$. We view an application of $\gamma$ or $\phi$ as
	taking the base point to be $p_1$ and moving $p_1$ around $\gamma$ or
	$\phi$.
	We compute
	\begin{equation}
	\begin{aligned}
		\label{equation:g-e-p-e}
		\gamma\eta\phi\eta(x,y,t)
		&= \gamma\eta\phi(x \triangleright^{-1} y, x, t) \\
		&= \gamma\eta(x \triangleright^{-1} y, \sigma_t^\phi(x),
		\tau_x^\phi(t)) \\
		&= \gamma\left( (x \triangleright^{-1} y) \triangleright^{-1}
		\sigma_t^\phi(x), x \triangleright^{-1} y, \tau_x^\phi(t)
	\right) \\
	&=  \left( (x \triangleright^{-1} y) \triangleright^{-1}
		\sigma_t^\phi(x), \sigma^\gamma_{\tau_x^\phi(t)}(x
		\triangleright^{-1} y), \tau_{x \triangleright^{-1}
		y}^\gamma(\tau_x^\phi(t))
	\right),
	\end{aligned}
\end{equation}
	\begin{equation}
	\begin{aligned}
		\label{equation:e-p-e-g}
		\eta\phi\eta\gamma(x,y,t) &=
		\eta\phi\eta\left( x, \sigma_t^\gamma(y), \tau_y^\gamma(t) \right) \\
		&= \eta\phi\left( x \triangleright^{-1} \sigma_t^\gamma(y), x,
		\tau_y^\gamma(t)\right) \\
		&= \eta \left( x \triangleright^{-1} \sigma_t^\gamma(y),
			\sigma^\phi_{\tau_y^\gamma(t)}(x),
		\tau_x^\phi(\tau_y^\gamma(t)) \right) \\
		&=\left( (x \triangleright^{-1} \sigma_t^\gamma(y))
			\triangleright^{-1} \sigma^\phi_{\tau_y^\gamma(t)}(x),
			x \triangleright^{-1} \sigma_t^\gamma(y),
		\tau_x^\phi(\tau_y^\gamma(t)) \right),
\end{aligned}
\end{equation}
	\begin{equation}
	\begin{aligned}
		\label{equation:g-e-p-e-1}
		\gamma\eta\phi\eta^{-1}(x,y,t)
		&= \gamma\eta\phi(y, y \triangleright x, t) \\
		&= \gamma\eta(y, \sigma_t^\phi(y\triangleright x),
		\tau_{y \triangleright x}^\phi(t)) \\
		&= \gamma\left( y \triangleright^{-1}
			\sigma_t^\phi(y \triangleright x), y, \tau_x^\phi(t)
	\right) \\
	&=  \left( y \triangleright^{-1}
			\sigma_t^\phi(y \triangleright x), 
			\sigma^\gamma_{\tau_x^\phi(t)}(y), \tau_{y}^\gamma(\tau_x^\phi(t))
		\right),
\end{aligned}
\end{equation}
	\begin{equation}
	\begin{aligned}
		\label{equation:e-p-e-1-g}
		\eta\phi\eta^{-1}\gamma(x,y,t) &=
		\eta\phi\eta^{-1}\left( x, \sigma_t^\gamma(y), \tau_y^\gamma(t) \right) \\
		&= \eta\phi\left( \sigma_t^\gamma(y), \sigma_t^\gamma(y)
			\triangleright x,
		\tau_y^\gamma(t)\right) \\
		&= \eta \left( \sigma_t^\gamma(y),
			\sigma^\phi_{\tau_y^\gamma(t)}(\sigma_t^\gamma(y)
			\triangleright x),
		\tau_{\sigma_t^\gamma(y)\triangleright x}^\phi(\tau_y^\gamma(t)) \right) \\
		&=\left( \sigma_t^\gamma(y)
			\triangleright^{-1} \sigma^\phi_{\tau_y^\gamma(t)}(\sigma_t^\gamma(y)
			\triangleright x),
			\sigma_t^\gamma(y),
		\tau_{\sigma_t^\gamma(y)
	\triangleright x}^\phi(\tau_y^\gamma(t)) \right).
\end{aligned}
\end{equation}

	We have the relation $\gamma \eta \gamma \eta = \eta \gamma \eta \gamma
	\in B_2^{\Sigma_{g,f}^1}$ 
	for $\gamma \in \Delta$ by
	\cite[Theorem 1.1,
	(R2),(R8)]{bellingeri:on-presentations-of-surface-braid-groups}. (Recall
		$\eta$ is notated as $\sigma_1^{-1}$ in 
		\cite[Theorem
	1.1]{bellingeri:on-presentations-of-surface-braid-groups}.)
	Taking $\gamma = \phi$ in \eqref{equation:g-e-p-e} and
	\eqref{equation:e-p-e-g} and equating the three terms yields
	\eqref{equation:first-relation-same},
	\eqref{equation:second-relation-same}, and
	\eqref{equation:third-relation-same}.
	Next, 
\cite[Theorem 1.1,
	(R4)]{bellingeri:on-presentations-of-surface-braid-groups}
	implies that when $\gamma= \alpha_i, \phi = \beta_i$,
we can identify \eqref{equation:e-p-e-g} and \eqref{equation:g-e-p-e-1}.
Identifying the three terms yields \eqref{equation:first-relation-pair},
\eqref{equation:second-relation-pair}, and \eqref{equation:third-relation-pair}.
Finally,
upon comparing the terms of \eqref{equation:g-e-p-e-1} and
\eqref{equation:e-p-e-1-g},
\cite[Theorem 1.1,
(R3), (R6), (R7)]{bellingeri:on-presentations-of-surface-braid-groups}
implies \eqref{equation:first-relation-different}, \eqref{equation:second-relation-different} and
\eqref{equation:third-relation-different}.
\end{proof}

We can next deduce an important relation between $S_{c'}$ and
$S_{N_c(c')}$.
\begin{lemma}
	\label{lemma:normalizer-is-hurwitz-module}
	Suppose $c$ is a rack, $c' \subset c$ is a subrack.
	Let $S = ({\Sigma^1_{g,f}}, \{T_n\}_{n \in \mathbb
	Z_{\geq 0}},\{\psi_n\}_{n \in \mathbb Z_{\geq 0}})$ 
	be a bijective Hurwitz module over $c$
	and
	$S_{c'} = ({\Sigma^1_{g,f}}, \{T'_n\}_{n \in \mathbb
	Z_{\geq 0}}, \psi'_n)$
	be the system over $c'$ defined in
	\autoref{definition:subsystem-for-subrack}.
	Then 
	$S' := ({\Sigma^1_{g,f}}, \{N_c(c')^n \times T'_0\}_{n \in \mathbb
	Z_{\geq 0}}, \{\psi_n|_{N_c(c')^n \times T'_0}\}_{n \in \mathbb Z_{\geq
0}})$
	is a bijective Hurwitz module.
\end{lemma}
\begin{proof}
 	Let $\star$ be a fixed basepoint on
	the boundary of $\Sigma^1_{g,f}$.
	Take $\Delta$ to be the generating set of $\pi_1(\Sigma^1_{g,f})$ from
\autoref{lemma:three-conditions}.
	Let $\widetilde{T'_0} := \{ \tau_x^\gamma(t) : x \in N_c(c'), \gamma \in
	\Delta, t \in T'_0\}.$
		We claim that the action
	of $\psi_1(\delta, \bullet, \bullet) : c \times T_0 \to c\times T_0$
	preserves $c' \times \widetilde{T'_0}$ for every $\delta \in \Delta$.

	Choose 
	$x \in c',\gamma \in\Delta, y \in N_c(c'), t \in T'_0$ so that
	$\tau_y^\gamma(t) \in \widetilde{T'_0}$.
	First, we check $\sigma^\gamma_{\tau^\gamma_y(t)}(x) \in c'$.
	Since $y \in N_c(c')$ we have $y \triangleright x \in c'$, and hence 
	$\sigma_t^\gamma( y \triangleright x) \in c'$ as $t \in T_0'$. Finally 
	$y \triangleright^{-1} \sigma_t^\gamma( y \triangleright x) \in c'$ since $y \in
	N_c(c')$. This implies 
	$\sigma^\gamma_{\tau^\gamma_y(t)}(x) \in c'$ using 
	\eqref{equation:second-relation-same}, where we use $y$ here in place
		of $x$ there and $y \triangleright x$ here in place of $y$
	there.

	Next, we check that for $\gamma \in \Delta$, $y \in N_c(c')$, and $t \in
	T'_0$,
	$\sigma^\gamma_t(y) \in N_c(c')$.
	Indeed, choose $x \in c'$.
	We find $x \triangleright^{-1} y \in N_c(c')$ by
	\autoref{lemma:normalizer-property}, and therefore
	$(x \triangleright^{-1} y) \triangleright^{-1} \sigma_t^\gamma(x) \in c'$.
	It follows from \eqref{equation:first-relation-same} that $(x
	\triangleright^{-1} \sigma_t^\gamma(y)) \triangleright^{-1}
	\sigma^\gamma_{\tau^\gamma_y(t)}(x) \in c'$.
	Since we saw above $\sigma^\gamma_{\tau^\gamma_y(t)}(x) \in c'$, we find
	that $x	\triangleright^{-1} \sigma_t^\gamma(y) \in N_c(c')$ and hence 
	$\sigma_t^\gamma(y) \in N_c(c')$.

	Next, we check $\sigma^{(-)}_{(-)}$, with input in $\Delta \times \widetilde{T'_0}$, takes values in endomorphisms of $c'$.
	That is for $x \in c', \gamma \in \Delta, t' \in \widetilde{T'_0}$, we
	will show $\sigma^\gamma_{t'}(x) \in c'$.
	Let $x \in c',\gamma,\phi \in\Delta, y \in N_c(c'), t \in T'_0$ so that
	$\tau_y^\phi(t) \in \widetilde{T'_0}$.
	We already saw above that when $\gamma = \phi$, 
	$\sigma^\gamma_{\tau^\gamma_y(t)}(x) \in c'$
	above, using \eqref{equation:second-relation-same}.
	One can similarly verify that when $\phi \neq \gamma$, we still have
$\sigma^\gamma_{\tau^\phi_y(t)}(x) \in c'$ using one of
\eqref{equation:first-relation-pair}, \eqref{equation:second-relation-pair},
\eqref{equation:first-relation-different}, or \eqref{equation:second-relation-different},
depending on the case; note that it will be important to know
$\sigma_t^\gamma(y) \in N_c(c')$, as we established above, when we apply
\eqref{equation:first-relation-pair} or
\eqref{equation:first-relation-different}.
Therefore, we have $\sigma^{(-)}_{(-)}$, with input in $\Delta \times \widetilde{T'_0}$, takes values in endomorphisms of $c'$.

	Next, we check $\tau^{(-)}_{(-)}$, with input in $\Delta \times c'$ gets sent to an endomorphism preserving
	$\widetilde{T'_0}$. Indeed, let $x \in c',\gamma,\phi \in
	\Delta, y \in N_c(c'), t \in T'_0$ so that
	$\tau_y^\gamma(t) \in \widetilde{T'_0}$. 
	We first consider the case $\phi = \gamma$.
	Then $\tau^\gamma_x
	(\tau_y^\gamma(t)) = \tau_y^\gamma( \tau^\gamma_{y \triangleright
	x}(t))$
	by \eqref{equation:third-relation-same}.
	We want to show the left hand side lies in $\widetilde{T'_0}$, which indeed holds
	because $y \triangleright x \in c'$ and so 
	$\tau^\gamma_{y \triangleright	x}(t) \in T'_0$, and hence
	$\tau_y^\gamma( \tau^\gamma_{y \triangleright
	x}(t)) \in \widetilde{T'_0}.$
	One can similarly verify the remaining cases that $\phi \neq \gamma$
	using \eqref{equation:third-relation-pair} and
	\eqref{equation:third-relation-different}; in the latter case, one will
	either use that $\sigma^\gamma_t(y) \in N_c(c')$ when $y \in N_c(c')$,
	as shown above, or that $c'$ normalizes $N_c(c')$ as shown in
	\autoref{lemma:normalizer-property}.

	Combining the above, we have shown above that $\psi_1(\delta, \bullet, \bullet)$ preserves $c'
	\times \widetilde{T'_0}$.
	We will next show $\widetilde{T'_0} = T'_0$. 
First, \autoref{lemma:criterion-to-be-subset} implies
	$\widetilde{T'_n} := (c')^n \times \widetilde{T'_0}$ defines a
	bijective Hurwitz module $\widetilde{S}$ over $c'$ containing $S$ as a
	subset. Then, maximality of $S_{c'}$ implies
	$\widetilde{S} = S_{c'}$ so $\widetilde{T'_0} = T'_0$. 

	We can reinterpret the condition that $\widetilde{T'_0} = T'_0$ as
	saying that $\Delta \times N_c(c')$ preserves $T'_0$.
	We also saw above that for $\gamma \in \Delta, t \in T'_0, y \in
	N_c(c')$, we have $\sigma^\gamma_t(y) \in N_c(c')$.
	This means that $\Delta \times T_0'$ preserves $N_c(c')$.
	Therefore, $\psi_1(\delta, \bullet, \bullet)$ preserves
	$N_c(c') \times T_0'$ for each $\delta \in \Delta$.
	Therefore, 
	$S'$ is a bijective Hurwitz module by 
	\autoref{lemma:criterion-to-be-subset}.
\end{proof}

\begin{lemma}
	\label{lemma:normalizer-subsystem-containment}
	Let $c$ be a rack, $S$ be a bijective Hurwitz module over $c$, and $c' \subset
	c$ be a subrack.
	Then $(c', S_{c'}) \subset (N_c(c'), S_{N_c(c')})$.
	In particular,
	$S_{c'} = (S_{N_c(c')})_{c'}$, viewed as bijective Hurwitz modules over $c'$.
\end{lemma}
\begin{proof}
	First, we verify 
$(c', S_{c'}) \subset (N_c(c'), S_{N_c(c')})$.
Let $S = (\Sigma^1_{g,f}, \{T_n\}, \{\psi_n\}_{n \in \mathbb Z_{\geq 0}})$
and let
$S_{c'} = (\Sigma^1_{g,f}, \{T'_n\}, \{\psi'_n\}_{n \in \mathbb Z_{\geq 0}})$.
Using \autoref{lemma:normalizer-is-hurwitz-module},
we find 
$S' := (\Sigma^1_{g,f}, \{N_c(c')^n \times T'_0\}_{n \in \mathbb Z_{\geq 0}}, \{\psi_n|_{N_c(c')^n \times T'_0}\}_{n \in \mathbb
Z_{\geq 0}})$
is a bijective Hurwitz module.
The definition of $S_{N_c(c')}$ implies $(N_c(c'), S') \subset (N_c(c'),
S_{N_c(c')})$.
Therefore, $(c', S_{c'}) \subset (N_c(c'), S_{N_c(c')})$, proving the first
part.

Since
$(c', S_{c'}) \subset (N_c(c'), S_{N_c(c')})$,
it follows that
$(c', S_{c'}) \subset (c', (S_{N_c(c')})_{c'})$
Moreover, since 
$(N_c(c'), S_{N_c(c')}) \subset (c, S)$, we also obtain
$(c', (S_{N_c(c')})_{c'}) \subset (c', S_{c'})$, and so
$S_{c'} = (S_{N_c(c')})_{c'}$.
\end{proof}

\subsection{Quotients of Hurwitz modules}
\label{subsection:quotients}

In this subsection, we discuss quotients of Hurwitz modules by certain subracks.
We start with defining quotients of racks by normal subracks.
Recall the normalizer of a subrack was defined in
\autoref{definition:normalizer}.
For $c' \subset c$ a normal subrack, we defined the quotient rack $c/c'$ in
\autoref{definition:quotient-rack}.
We needed the following lemma to show this notion of quotient is well defined.

\begin{lemma}
	\label{lemma:quotient-well-defined}
	If $c' \subset c$ is a normal subrack, the operation $\overline{x}
\triangleright \overline{y} := \overline{x \triangleright y}$
is independent of the choice of representatives $x$ and $y$.
\end{lemma}
\begin{proof}
	Suppose $u \in c'$ and $x, y \in c$. First, we claim that $\overline{x
	\triangleright(u \triangleright y)} = \overline{x
	\triangleright y}$. Using the definition of a rack,
$x\triangleright(u \triangleright y) = (x \triangleright u) \triangleright (x
\triangleright y)$. The claim then follows since $x \triangleright u \in
	c'$ as $c'$ is normal.
	To conclude, it suffices to show that $\overline{ (u \triangleright x)
\triangleright y} = \overline{x \triangleright y}$.
Suppose $w\in c$ is such that $u \triangleright w = y$.
Then $\overline{(u \triangleright x) \triangleright (u \triangleright w)}
=\overline{u
\triangleright (x \triangleright w)} = \overline{x \triangleright w} =
\overline{x \triangleright y}$.
\end{proof}

We next define quotients of Hurwitz modules by
normal subracks. This was used to express our main result computing the stable
homology of Hurwitz modules in \autoref{theorem:one-large-stable-homology}.

\begin{definition}
	\label{definition:quotient-hurwitz-module}
	If $c$ is a rack and $c' \subset c$ is a
	subrack, and let $S = (\Sigma^1_{g,f}, \{T_n\}_{n \in \mathbb Z_{\geq 0}},
	\{\psi_n\}_{n \in \mathbb Z_{\geq 0}})$ be a bijective Hurwitz module
	over $c$.
	Suppose $c' \subset c$ is normal and closed under the action of
	$B_1^{\Sigma^1_{g,f}}
	\times T_0$ on $c$. 	
	Define the bijective Hurwitz module 
	$S/c' = (\Sigma^1_{g,f}, \{\overline{T}_n\}_{n \in \mathbb Z_{\geq 0}},
	\{\overline{\psi}_n\}_{n \in \mathbb Z_{\geq 0}})$ over $c/c'$
	as follows.
	Take $\overline{T}_0$ to denote the quotient of $T_0$ by the equivalence
	relation generated by 
	$s \sim s'$ if there is some $\gamma \in B_n^{\Sigma^1_{g,f}}$ and $x_1,
	\ldots, x_n \in c'$ with
	$\psi_n(\gamma, x_1, \ldots, x_n, s) \sim (y_1, \ldots, y_n,s')$
	such that $y_i$ and $x_i$ have the same image in $c'/c'$.
	Then, take $\overline{T}_n := (c/c')^n \times \overline{T}_0$.
	Finally, for $x \in T_n$, we use $\overline{x}$ to denote its image in
	$\overline{T}_n$ and for $\gamma \in B_n^{\Sigma^1_{g,f}}$ define
	$\overline{\psi}_n((\gamma, \overline{x})) :=
	\overline{\psi_n((\gamma, x))}$. 
	We will see this is well defined later in
	\autoref{lemma:quotient-module-well-defined}.
\end{definition}

\begin{warn}
	\label{warning:}
	We note that the ``quotient'' $S/c'$ is not a quotient in any
	categorical sense of the word. It is merely a convenient Hurwitz module
	for the proofs of our main results.
\end{warn}

To make sense of the above definition of quotient of Hurwitz modules,
we need to show it is well defined. We do so in the next couple lemmas.

\begin{lemma}
	\label{lemma:dependence-on-image}
	We claim that for $x_1, \ldots, x_n \in c', s \in T_0$,
	$(y_1, \ldots, y_n,t) := \psi_n(\gamma, x_1, \ldots, x_n,s)$
	then the values of $y_1, \ldots, y_n$ in $c/c'$ only depend on the
	values of $x_1, \ldots, x_n$ in $c/c'$.
\end{lemma}
\begin{proof}
	We can write $\gamma$ as a composite of paths in
	$B^{\Sigma^1_{0,0}}_n$ and $B^{\Sigma^1_{g,f}}_1$ in $\Delta$, it
	suffices to show the lemma when $n = 1$.
	Concretely, this means that we wish to show that for $x,y \in c', s \in
	T_0$, so that $x$ and $y$ have the same image in $c'/c'$, then
	$\sigma_s(x) = \sigma_s(y)$.
	To check this, it is enough to verify it for $\gamma \in \Delta$.
	Then, by \eqref{equation:first-relation-same},
(taking $y$ here to denote $x$ there and $z$ here to denote $y$
there,) $\sigma_s^\gamma(y)$ has the same image in $c/c'$ 
	as $\sigma^\gamma_{\tau_z^\gamma(t)}(y)$ for $z \in c'$.
	Hence, we find that
	$\sigma^\gamma_s(z \triangleright^{-1} y)$ lies in the same $c/c'$ component as
	$\sigma_{\tau_z^\gamma(t)}(z \triangleright^{-1} y)$, and by
	\eqref{equation:second-relation-same}, this also lies in
	the same component as $z\triangleright^{-1} \sigma_s^\gamma(y)$, and hence
	in the same component as $\sigma_s^\gamma(y)$.
	Therefore, $\sigma^\gamma_s$ sends $y$ and $z \triangleright^{-1} y$ to
	the same component for any $z \in c'$, so is well defined on $c/c'$ components.
\end{proof}

\begin{lemma}
	\label{lemma:quotient-module-well-defined}
	Suppose that $c$ is a rack, $S$ is a bijective Hurwitz module over $c$
	with $0$ set $T_0$, 
	and $c' \subset c$ is a subrack such that $c'$ is normal in $c$ and $c'$
	is preserved by the $B_1^{\Sigma^1_{g,f}} \times T_0$ action.
	Then set $S/c'$ is a bijective Hurwitz module.
\end{lemma}
\begin{proof}
	The only difficult part to check is that the maps $\overline{\psi}_n$
	are well defined.

	Suppose we have some $(x'_1, \ldots, x'_n, s')$ which is equivalent to
	$(x_1, \dots, x_n,s)$ under the equivalence relation defining $S/c'$;
	that is, we can suppose $x'_i$ agrees with $x_i$ in $c/c'$ and 
	$s$ is equivalent to $s'$ as elements of $T_0$.
	Write $(y_1', \ldots, y_n',t') := \psi_n(\gamma, x'_1, \ldots, x'_n,s')$.
	Then \autoref{lemma:dependence-on-image} implies $y_i'$ agrees with $y_i$ in $c/c'$ for all
	$i$ (since $x'_i$ agrees with $x_i$).
	It remains to check that $t'$ is also equivalent to $t$.
	To simplify matters, by writing $\gamma$ as a composite of elements in
	$B^{\Sigma^1_{0,0}}_n$ and $B^{\Sigma^1_{g,f}}_1$ in $\Delta$,
	we may assume $n= 1$ and moreover that $\gamma \in
	\Delta$ as in \autoref{lemma:three-conditions}, so we just need to show
	that for $x,x' \in c$ with the same image in $c/c'$, that
	$\tau_x^\gamma(s) \sim \tau_{x'}^\gamma(s')$.
	By assumption, we can find $u_1, \ldots, u_j \in c'$ and $\eta \in
	B_j^{\Sigma^1_{g,f}}$ with
	$\psi_j(\eta, u_1, \ldots, u_j,s) = (u'_1, \ldots, u'_j,s')$ for $u_i$
	with the same image as $u_i'$ in $c'/c'$.
	This will allow us to write $s' = \tau(s),$ where $\tau$ is some composite of functions of the
	form $\tau_{u_{i_k}}^{\gamma_{i_k}}$ with $\gamma_{i_k} \in \Delta$.
	Then, using \eqref{equation:third-relation-same},
	\eqref{equation:third-relation-pair}, and
	\eqref{equation:third-relation-different} iteratively, we can rewrite
	$\tau_{x'}^\gamma(s') =
	\tau_x^\gamma(\tau(s))=\tau'(\tau_{x''}^\gamma(s))$ where $x'' \in c$ has the
	same image as $x'$ and $x$ in $c/c'$ and $\tau'$ is a composite of
	functions of the form $\tau_{v_{i_k}}^{\gamma_{i_k}}$ for $v_i \in c'$ elements
	in the same $c'$ component as $u_i$. This reduces us to verifying that 
	$\tau_{x''}^\gamma(s) \sim \tau_x^\gamma(s).$
	Finally, to check this, it suffices to verify the case that $x'' = z
	\triangleright^{-1} x$ for $z \in c'$.
	Hence, we want to show $\tau_{z \triangleright^{-1} x}^\gamma \circ
(\tau_x^\gamma)^{-1}(s) \sim s$, which holds using 
\eqref{equation:third-relation-same} because
\begin{align*}
	\tau_{z \triangleright^{-1} x}^\gamma \circ
(\tau_x^\gamma)^{-1}(s)
&=
(\tau_{(z \triangleright^{-1} x) \triangleright^{-1} z}^\gamma)^{-1} \circ
\tau_{(z \triangleright^{-1} x) \triangleright^{-1} z}^\gamma \circ
\tau_{z \triangleright^{-1} x}^\gamma \circ
(\tau_x^\gamma)^{-1}(s) 
\\
&=
(\tau_{(z \triangleright^{-1} x) \triangleright^{-1} z}^\gamma)^{-1} \circ
\tau_{z \triangleright^{-1} x}^\gamma \circ
\tau_z^\gamma \circ
(\tau_x^\gamma)^{-1}(s) 
\\
&=
(\tau_{(z \triangleright^{-1} x) \triangleright^{-1} z}^\gamma)^{-1} \circ
\tau_z^\gamma \circ
\tau_{x}^\gamma \circ
(\tau_x^\gamma)^{-1}(s) 
\\
&=
(\tau_{(z \triangleright^{-1} x) \triangleright^{-1} z}^\gamma)^{-1} \circ
\tau_z^\gamma(s).
\end{align*}
So it remains to check 
$(\tau_{(z \triangleright^{-1} x) \triangleright^{-1} z}^\gamma)^{-1} \circ
\tau_z^\gamma(s)$ is equivalent to $s$.
To see this, note that using \autoref{lemma:dependence-on-image}, note that $\sigma^\gamma_s(z)$
has the same image in $c'/c'$ as $w := \sigma^\gamma_s((z \triangleright^{-1} x)
\triangleright^{-1} z)$.
Therefore, if $\gamma' \in B_2^{\Sigma^1_{g,f}}$ is the path which first does $\gamma$, then applies the
half twist $\eta$ switching the two elements of $c$ and then applies $\gamma^{-1}$,
we find
\begin{align*}
	\psi_2(\gamma',w, z,s) &=  \psi_2(\gamma^{-1}, \psi_2(\eta, \psi_2(\gamma,
	w, z,s))) = \psi_2(\gamma^{-1}, \psi_2(\eta, w, \sigma_s(z),\tau_z(s))) \\
&=
	\psi_2(\gamma^{-1}, w \triangleright^{-1} \sigma_s(z), w, \tau_z(s)) \\
	&= 
	(w \triangleright^{-1} \sigma_s(z),(z \triangleright^{-1} x) \triangleright^{-1} z , (\tau_{(z \triangleright^{-1} x) \triangleright^{-1} z}^\gamma)^{-1} \circ
\tau_z^\gamma(s))
\end{align*}
so we see that indeed 
$(\tau_{(z \triangleright^{-1} x) \triangleright^{-1} z}^\gamma)^{-1} \circ
\tau_z^\gamma(s)$ is equivalent to $s$ because $w$ lies in the same $c'$ orbit
as 
$w \triangleright^{-1} \sigma_s(z)$
and $z$ lies in the same $c'$ component as 
$(z \triangleright^{-1} x) \triangleright^{-1} z$.
\end{proof}

We conclude the section with a simple lemma that will be important for our
application to the BKLPR heuristics.
\begin{lemma}
	\label{lemma:quotient-components-conf}
	Suppose $c$ is a rack with a single component and $S = (\Sigma^1_{g,f},
	\{T_n\}_{n \in \mathbb Z_{\geq 0}}, \{\psi_n\}_{n \in \mathbb Z_{\geq
0}}\}$ is a bijective Hurwitz module over $c$.
	Then, every component of $\Hur_n^{c/c, S/c}$ maps isomorphically to
	$\Conf_n^{\Sigma^1_{g,f}}$.
\end{lemma}
\begin{proof}
	Start with an element $(x_1, \ldots, x_n, s) \in \Hur_n^{c,S}$ mapping
	to an element $(z_1, \ldots, z_n,t) \in \Hur_n^{c/c,S/c}$. The statement
	of the lemma is equivalent to the statement that every element of $B^{\Sigma^1_{g,f}}_n$ acts trivially
	on $(z_1, \ldots, z_n,t)$. 
	Suppose we have some path $\gamma \in B_n^{\Sigma^n_{g,f}}$ so that
	$\psi_n(\gamma, x_1, \ldots, x_n, s) = (x_1', \ldots, x_n', s')$.
	Then, we wish to show $x_i$ is equivalent to $x_i'$ in $c/c$ and $s$ is
	equivalent to $s'$ in $S/c$.
	Since $c$ has a single component $x_i$ and $x_i'$ lie in the same
	component, so are equivalent in $c/c$.
	Finally, $s$ is equivalent to $s'$ in $S/c$ as is immediate from the
	definition of $S/c$, using that $x_i$ and $x_i'$ lie in
	the same component of $c$.
\end{proof}

\section{Scanning arguments}
\label{section:scanning}

Throughout this paper, it will be convenient to have particular topological
models for certain bar constructions, which are of the form
$M\otimes_{\Hur^c_+}\Hur^{c,S}_+ $,
where $c$ is a rack, $M$ is a discrete module for $\Hur^c_+$
and $S$ is a Hurwitz module over $c$.
Many of the models we will construct will be similar to those
constructed in \cite[Appendix
A]{landesmanL:the-stable-homology-of-non-splitting},
and so we will be somewhat brief.

The main result of this section will be \autoref{proposition:pointed-scanning},
which identifies a certain bar construction with an explicit topological space.
Along the way to proving that, we first introduce notation for a particular model
of Hurwitz spaces in \autoref{subsection:notation-scanning}.
We then relate this to a scanning model for the bar construction in
\autoref{subsection:scanning-model}. 
We next relate this to a quotient model in
\autoref{subsection:quotient-model}.
Finally, we make further refinements of
this quotient model in \autoref{subsection:quotient-refinements} in order to
prove \autoref{proposition:pointed-scanning}.


\subsection{Notation for scanning models}
\label{subsection:notation-scanning}

We begin by producing a topological monoid
modeling 
$\Hur^{c,S}$ and $\Hur^c$,
so that the former is a module over the latter.
To construct these, we define a ``Moore variant'' of
$\Hur^{c,S}$, where we also keep track of a time parameter.
We call this Moore variant $\on{hur}^{c,S}$ to match the notation in
\cite[Notation A.2.1 and Notation
A.2.4]{landesmanL:the-stable-homology-of-non-splitting}.
In order to define this, we first construct
$\Sigma^1_{g,f}$ as a quotient in a particular way,
depending on a time $t$, which we denote $\mathcal M_{g,f,t}$,
which will be useful for describing Hurwitz spaces.
This definition is a generalization of
\cite[\S4.2]{bianchiS:homology-of-configuration-spaces}
(where $t= 1, g = 0$)
and
\cite[Proof of Lemma
4.3.1]{ellenbergL:homological-stability-for-generalized-hurwitz-spaces}
(where $t = 1$).

\begin{notation}
	\label{notation:surface}
	Let $\mathbf R :=[0,t] \times [0,1]$ be a rectangle.
	Decompose the side $\{t\} \times [0,1]$ into $4g + 2f$ consecutive
	intervals $J_1, \ldots, J_{4g}, J'_1, \ldots, J'_{2f}$ of equal length, 
	ordered and oriented with increasing second coordinate, as in \autoref{figure:surface-picture}.
	Let $W$ be the set of the $f$ points consisting of the larger endpoint of $J'_{2i+1}$ for $0 \leq i \leq f - 1$.
	Let $\mathbf R - W$ denote the punctured rectangle where we remove
	$W$.
	Let $\mathcal M_{g,f,t}$ denote the quotient of $\mathbf R -W$ obtained by
	identifying $J_{4i+1}$ with $J_{4i+3}$, $J_{4i+2}$ with $J_{4i+4}$, and
	$J'_{2j+1}$ with $J'_{2j+2}$
	via their unique orientation reversing isometry
	for $0 \leq i \leq g-1$ and $0 \leq j \leq f - 1$ .
	Let $\mathfrak p : \mathbf R- W \to \mathcal M_{g,f,t}$ denote the quotient map.
	Then, $\mathcal M_{g,f,t}$ is homeomorphic to $\Sigma^1_{g,f}$.
\end{notation}
In the case $g = 1, f =2$, the quotient $\mathcal M_{g,f,t}$ of $\mathbf R - W$ is depicted in
\autoref{figure:surface-picture}.

\begin{figure}
\includegraphics[scale=.4]{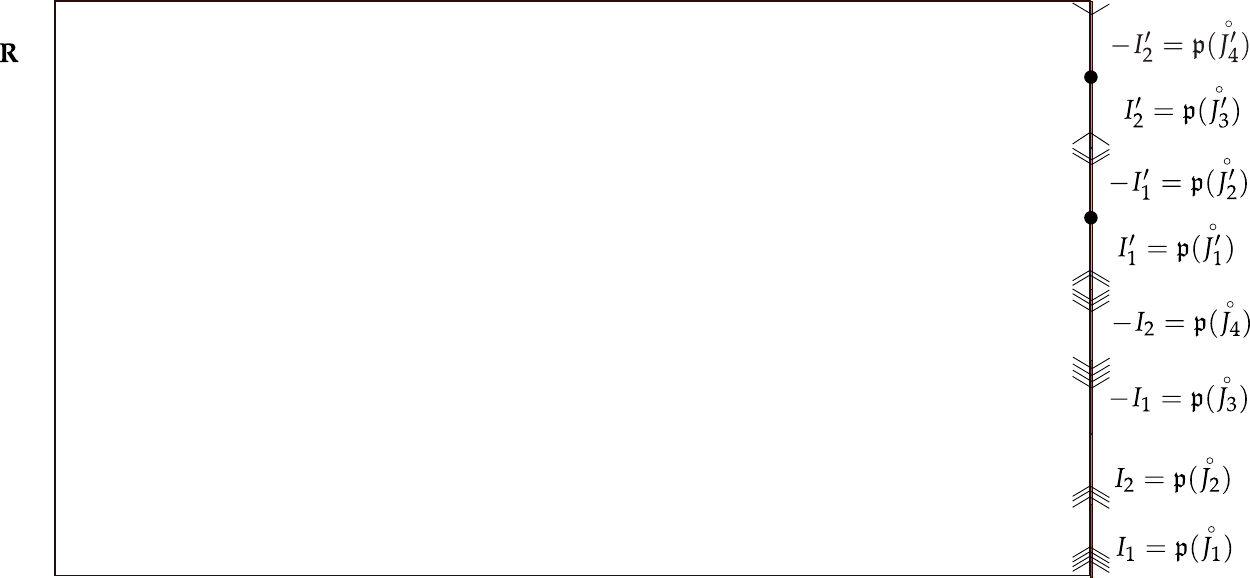}
\caption{
	This picture depicts the quotient $\mathcal M_{g,f,t}$ of the rectangle
	$\mathbf R - W$ in the case $g = 1, f = 2$. 
	The boundary component of $\mathcal M$
consists of the union of the upper, left, and lower edges.
The arrows indicate the orientations of the segments of the edges. 
The segments of the same color are glued to each other with the
orientations indicated.
The two black
dots indicate the two punctures comprising $W$.}
\label{figure:surface-picture}
\end{figure}

We next use the above to define a topological model for configuration space.
For the next notation, it will be useful to recall 
the topological space $\on{confbig}$
defined in
\cite[Notation A.2.1]{landesmanL:the-stable-homology-of-non-splitting},
whose points are given by pairs $t \in \mathbb R_{\geq 0}$ and configurations of
finitely many distinct unordered points in $(0,t) \times (0,1)$. 
As in \cite[Notation A.2.2]{landesmanL:the-stable-homology-of-non-splitting}, we
think of a standard generator of the braid group as rotating two adjacent points
clockwise in a half twist around each other.

\begin{notation}
	\label{notation:conf}
	Fix $g,f \in \mathbb Z_{\geq 0}$.
	Using notation from \autoref{notation:surface},
	define the topological space $\on{conf}^{\Sigma^1_{g,f}}$
as the set of pairs $(t,x)$ for $t \in \mathbb
R_{> 0}$ 
and $x$ a (possibly empty) configuration of finitely many distinct
unordered points in 
$\mathfrak p( [0,t] \times (0,1) -W) \subset \mathcal M_{g,f,t}$ that do not contain the image of the endpoints of $J_i$ or $J'_i$.
This topological space is a left module for the topological monoid
$\on{confbig}$,
as defined in \cite[Notation
A.2.1]{landesmanL:the-stable-homology-of-non-splitting}:
Let $(y,s) \in \on{confbig}$
so that $s \in \mathbb R_{\geq 0}$ and
$y \subset (0,s) \times (0,1) \subset [0,s] \times
[0,1]$ a configuration of points.
The left action is given by $(y,s) \cdot (x,t) = (y \cdot x, s+t)$, where $y
\cdot x$
denotes the concatenation of $y$ and $x$ viewed as a configuration in
$\mathcal M_{g,f,s+t}$.
We use $\on{conf}^{\Sigma^1_{g,f}}_n \subset \on{conf}^{\Sigma^1_{g,f}}$ to denote the
component parameterizing configurations $x$ consisting of $n$ points.
There is a map of topological spaces $t: \on{conf}^{\Sigma^1_{g,f}} \to \mathbb
R_{> 0}$ sending $(t,x) \mapsto t$.
There is a subset $\on{ord}^{\Sigma^1_{g,f}} \subset
\on{conf}^{\Sigma^1_{g,f}}$ consisting of
configurations $x = ( (a_1, b_1), \ldots, (a_n, b_n))$ where $b_i =1/2$
	for all $1 \leq i \leq n$.
For each $n$, the intersection $\on{ord}^{\Sigma^1_{g,f}} \cap \on{conf}^{\Sigma^1_{g,f}}_n$
is contractible, and we use this to view
$\on{ord}^{\Sigma^1_{g,f}} \cap \on{conf}^{\Sigma^1_{g,f}}_n \subset \on{conf}^{\Sigma^1_{g,f}}_n$
as a fixed contractible space, which we think of as a basepoint.

We also define $\on{conf}^{\circ, \Sigma^1_{g,f}} \subset\on{conf}^{\Sigma^1_{g,f}}$
	to be the subset of $(t,x)$ such that $x \subset 
	\mathfrak p( (0,t] \times ((0,1) -W)) \subset \mathcal M_{g,f,t}$, i.e. we
	prohibit any points of $x$ from lying on the left boundary of $\mathcal
	M_{g,f,t}$.
\end{notation}

For the next notation, it will be useful to recall 
the topological space $\on{hur}^c$ and $\on{hurbig}^c$ from 
\cite[Notation A.2.4]{landesmanL:the-stable-homology-of-non-splitting}.
Indeed, $\on{hurbig}^c$ has points given as $B_n$ equivalence classes
$( (x_1, \ldots, x_n),t,\gamma,(\alpha_1, \ldots, \alpha_n))$ where $t \in
\mathbb R_{\geq 0}$, $(\{x_1, \ldots, x_n\},t)$ is a point of $\on{confbig}$,
$\gamma$ is a path from $(\{x_1, \ldots, x_n\},t)$ to a point of
$\ord^{\Sigma^1_{g,f}}_n$ with second coordinate $t$, and $\alpha_1, \ldots,
\alpha_n \in c$.
The space $\on{hur}^c$ is defined similarly except we require that the
configuration $x = \{x_1, \ldots, x_n\}$ is contained in $[1/2,t-1/2] \times
(0,1)$.
Here is our promised model for Hurwitz modules.

\begin{notation}
	\label{notation:hur}
	Fix a rack $c$ and a Hurwitz module $S = ({\Sigma^1_{g,f}}, \{T_n\}_{n \in \mathbb
	Z_{\geq 0}}, B^{\Sigma^1_{g,f}}_n \times T_n \to T_n)$ over $c$, as in
	\autoref{definition:hurwitz-module}.
	Using the contractible set $\on{ord}^{\Sigma^1_{g,f}}$ constructed in 
	\autoref{notation:conf} as a basepoint, we can identify the fundamental group of
	$\on{conf}^{\Sigma^1_{g,f}}_n$ with the surface braid group
	$B_n^{\Sigma^1_{g,f}}
	\simeq \pi_1(\on{conf}^{\Sigma^1_{g,f}}_n,
	\on{ord}^{\Sigma^1_{g,f}}_n)$.
	We recall $T_n = c^n \times T_0$ as in
	\autoref{definition:hurwitz-module}.
Let $\widetilde{\on{conf}}^{\Sigma^1_{g,f}}_n$ denote the universal cover of 
	$\on{conf}^{\Sigma^1_{g,f}}_n$.
	We may then construct $\on{hur}^{c,S}$
	as a cover of $\on{conf}^{\Sigma^1_{g,f}}$ given by the quotient of $T_n \times
\widetilde{\on{conf}}^{\Sigma^1_{g,f}}_n$ by the action of
	$B^{\Sigma^1_{g,f}}_n$.
	Explicitly, we can represent such a point by a $B^{\Sigma^1_{g,f}}_n$ equivalence
	class of data of the form $(x,t,\gamma, \alpha = (\alpha_1, \ldots,
	\alpha_n, s) )$
	for $(x,t) \in \on{conf}^{\Sigma^1_{g,f}}_n$, $\gamma$ a homotopy class of
	paths from $(x,t)$ to $\on{ord}^{\Sigma^1_{g,f}}_n$, $s \in T_0$, and $\alpha_i \in c$ for
	$1 \leq i \leq n$, so that $\alpha \in T_n$.

	Then, $\on{hur}^{c,S}$ has a left action of $\on{hurbig}^c$
	given as follows:
	Let $(y,t',\eta,\beta = (\beta_1, \ldots, \beta_j)) \in
	\on{hurbig}^c_j$, with $y \in (0,t') \times (0,1)$ a
configuration of $j$ points, 
$\eta$ a homotopy class of paths from $y$ to $\on{ord}$ the set of configurations of
$j$ points with second coordinate $1/2$, and
$\beta_i \in c$ for $1 \leq i \leq j$.
The left action is given by $(y,t',\eta, \beta) \cdot (x,t, \gamma, \alpha) = (y
\cdot x, t'+t, \eta \cdot \gamma, \alpha \cdot \beta)$, where $y
\cdot x$
denotes the concatenation of $y$ and $x$ viewed as a configuration in
$\mathcal M_{g,f,t+t'}$, $\eta \cdot \gamma$ denotes the homotopy class of paths by
concatenating $\eta$ on $(0,t') \times [0,1]$ with $\gamma$ on $\mathcal
M^t_{g,f}$, and
$\alpha \cdot \beta \in T_{j+n}$ denotes the concatenation of $\alpha$ and
$\beta$.

We also let $\on{hur}^{\circ, c,S} := \on{hur}^{c,S}
\times_{\on{conf}^{\Sigma^1_{g,f}}} \on{conf}^{\circ, \Sigma^1_{g,f}}.$

	Fix an $S$-component $z \subset c$ as defined in
	\autoref{definition:s-component}.
	We view $\on{hurbig}^c$, $\on{hur}^{c,S}$, $\on{hur}^{\circ,c,S}$
	as $\mathbb N$-graded topological spaces, with the grading defined as
	follows: a point of such a space has a corresponding configuration $x =
	\{x_1, \ldots, x_n\}$ with labels $\alpha_1, \ldots,
	\alpha_n$; the point is in grading $j$ if precisely $j$ of the
	$\alpha_1, \ldots, \alpha_n$ lie in $z$.
\end{notation}

\subsection{A scanning model}
\label{subsection:scanning-model}
Having created a topological model for Hurwitz space in \autoref{notation:hur},
we next wish to relate this to a more convenient model for our proofs.
The first step of this is to relate it to what we call a scanning model.
In \cite[Notation A.3.1]{landesmanL:the-stable-homology-of-non-splitting},
given two sets $M$ and $N$, we 
defined a certain topological space $B[M,\Hur^c,N]$, which we are referring to
as the scanning model.
We now introduce notation closely related to
\cite[Notation A.3.1]{landesmanL:the-stable-homology-of-non-splitting},
where we replace $N$ with a Hurwitz 
module.

\begin{notation}
	\label{notation:b-model-definition}
	Let $c$ be a rack, let $M$ be a graded set with a right action of
	$\Hur^{c}$, and let $S$ be a Hurwitz module over $c$.
	Consider the graded topological
	space $B[M,\on{hur}^{\circ, c,S}]$
	consisting of points which are of the form
	\begin{align}
		\label{equation:b-point}
		(a,y)
	\end{align}
	where $a \in M$ and 
	$y = (x,t,\gamma,\alpha = (\alpha_1, \ldots, \alpha_n,s)) \in
	\operatorname{hur}^{\circ,c,S}$. The topology on
	$B[M,\on{hur}^{\circ, c,S}]$ has a basis given as follows. Consider the following data:
	\begin{enumerate}[(a)]
		\item A number $d \in (0,t)$.
		\item A finite collection of pairwise disjoint open balls
			$U_1,\ldots,U_n$ in contained $\mathfrak p(\mathbf R -
			W) \subset \mathcal M_{g,f,t}$ whose preimage in $\mathbf R - W$ (as
			in \autoref{notation:surface}) is
			contained in
		$[d,t]\times [0,1]$.
		\item A homotopy class of paths $\phi$ from the configuration of
			the centers of the
			balls $U_i$, viewed as an element of
			$\mathrm{conf}_n^{\Sigma^1_{g,f}}$, to the contractible
			set $\ord_n^{\Sigma^1_{g,f}}$.
		\item Elements $\alpha'_1,\ldots,\alpha'_n \in c$ and $s' \in
			T_0$.
		\item An element $m \in M$.
	\end{enumerate}
	The grading of the point $(a,y)$ is the sum of the grading for $a$ and the
	grading for $y$.

	We next define subsets $\mathfrak B(d,U_i,\phi,\alpha'_i,s',m)$, in
	terms of data as above,
	which form a basis of the topology on $B[M,\on{hur}^{\circ, c,S}]$. 
	A point of the form \eqref{equation:b-point} lies in
	$\mathfrak B(d,U_i,\phi,\alpha'_i,s',m)$ 
	if the following conditions hold.
	\begin{enumerate}
		\item None of the points in $x$ lie in $\mathfrak p([d,t]\times
			[0,1])- \cup_1^n U_i$, 
			and there is a unique point from $x$ in each $U_i$.
		\item Recall the notion of cutting, as defined in
			\cite[Construction
			A.2.5]{landesmanL:the-stable-homology-of-non-splitting}.
			Cutting the element of $\operatorname{hur}^{\circ,c,S}$ to
		restrict it to $\mathfrak p([d,t] \times [0,1] -W)$ yields a point
		$y' \in \operatorname{hur}^{\circ,c,S}$. Then, using the homotopy
	class of $\phi$, the corresponding element of $T_{n'} = c^{n'}
			\times T_0$ associated to $y'$ is
			$(\alpha'_1,\ldots,\alpha'_{n'},s')$.
		\item[(3)] Define $y_1 \in \operatorname{hurbig}^c$ to be the
		element of $\operatorname{hurbig}^{c}$ 
		(analogously to \cite[Notation
		A.2.4]{landesmanL:the-stable-homology-of-non-splitting})
		obtained by
		cutting $y$ and restricting to the interval $\mathfrak p([0,d]
		\times[0,1])$.
		We then require that $ay_1 = m$.
	\end{enumerate}
\end{notation}

We now want to relate the above scanning model to a certain bar construction.
For the next statement, recall that we use $\on{hur}^c$ for the topological
model of Hurwitz space constructed in 
\cite[Notation A.2.4]{landesmanL:the-stable-homology-of-non-splitting}.
The following lemma is very similar to
\cite[Lemma A.3.4]{landesmanL:the-stable-homology-of-non-splitting}, but where the
set $N$ is replaced with $\on{hur}^{c,S}$.
Since the proof is quite similar, we will be brief in describing it.
In the next lemma, if $H$ is a topological monoid, $M$ is a right module, and
$N$
is a left module, we use notation $M \otimes_H N$ for the two-sided bar
construction, see, for example, \cite[Notation
A.3.3]{landesmanL:the-stable-homology-of-non-splitting}.
We note that this bar construction obtains a grading when $M$,
$\on{hur}^c,$ and $\on{hur}^{c,S}$ are all graded.

\begin{lemma}
	\label{lemma:bar-to-b}
	Let $c$ be a rack and let $S$ be a Hurwitz module over $c$.
	Let $M$ be a set with a right action of $\on{hur}^c$.
	There is a weak homotopy equivalence of graded spaces
	$\sigma: M \otimes_{\on{hur}^c} \on{hur}^{c,S} \to
	B[M,\on{hur}^{\circ, c,S}]$, natural in $c$
	and $M$.
\end{lemma}
\begin{proof}
	We begin by defining $\sigma$.
	A point of $M \otimes_{\on{hur}^c} \on{hur}^{c,S}$ can be described
	as a tuple 
	\begin{align*}
		(m,z,(x_1, \ldots, x_n),(y_0, \ldots, y_n))
	\end{align*}
	 where $x_i \in
	\on{hur}^c, m \in M, z \in \on{hur}^{c,S}$ and $0 \leq y_i \leq 1$ with
	$\sum_{i=0}^n y_i = 1$.
	Let $x \in \on{hur}^c$ denote the product of $x_1 \cdots x_n$. Then $t := t(x) =
	\sum_{i=1}^n t(x_i)$. In this case, $x$ is a labeled configuration on
	$[0,t] \times [0,1]$. Extend this to $[-1/2,t] \times[0,1]$ to view
	$x$ as a labeled configuration in $[-1/2,t] \times[0,1]$ supported
	on $(0,t) \times (0,1)$ and let $t' := \sum_{i=1}^n y_i (\sum_{j=1}^i
	t_j)$. Choose $\epsilon > 0$ sufficiently small so that there are no
	points of the configuration associated to $x$ lie in $(t'-1/2, t' -
	1/2 + \epsilon] \times [0,1]$. We now use a cutting construction, as
	in \cite[Construction
	A.2.5]{landesmanL:the-stable-homology-of-non-splitting}.
	We cut $x$ at $t'-1/2 + \epsilon$ to obtain $w, x'' \in
	\on{hurbig}^c$, where $w$ is supported on $[-1/2,t'-1/2 + \epsilon]
	\times[0,1]$ and $x''$ is supported on $[t'-1/2 +
	\epsilon,t]\times[0,1]$. Extend $x''$ to lie in $[t'-1/2,t]\times[0,1]$
	by extending the length of the interval on the left by $\epsilon$ and
	let $x'$ denote the resulting element on $\on{hurbig}^{c}$.
	Observe that $x'$ does not depend on $\epsilon$ and the class of $w
	\in \pi_0\Hur^c$ is also independent of $\epsilon$.
	Now, define the map $\sigma$ to send the above point to the point $(m
	\cdot w, x' \cdot z) \in B[M,\on{hur}^{\circ, c,S}]$.

	In order to check the map $\sigma$ defined above is well defined, we
	need to check it glues along the identifications defining the two sided bar
	construction. We omit this verification, except to mention that
	verifying this glues along relation \cite[Notation
	A.3.3]{landesmanL:the-stable-homology-of-non-splitting}(1), related to
	the left action on $M$, involves using that $t' \geq t_1$ and the
	element $x_1$ lies in $\on{hur}^c$, and not $\on{hurbig}^c$.
	The verification that $\sigma$ is continuous is straightforward and
	similar to the verification carried out in
\cite[Lemma A.3.4]{landesmanL:the-stable-homology-of-non-splitting} so we omit it.
	One can also verify that $\sigma$ is surjective on path components in a
	fashion similar to the analogous step of the proof of \cite[Lemma
	A.3.4]{landesmanL:the-stable-homology-of-non-splitting}.
	The remainder of the verification that $\sigma$ is a homotopy
	equivalence is analogous to that carried out in the proof of 
	\cite[Lemma A.3.4]{landesmanL:the-stable-homology-of-non-splitting}, by
	demonstrating the analogs of conditions (i) and (ii) about lifting maps
	of pairs and nullhomotopies for maps of pairs
	occurring in the proof of \cite[Lemma A.3.4]{landesmanL:the-stable-homology-of-non-splitting}
	and we omit further details.
\end{proof}

\subsection{A quotient model}
\label{subsection:quotient-model}

We next re-express the scanning model
$B[M,\on{hur}^{\circ, c,S}]$ of
$M \otimes_{\on{hur}^c} \on{hur}^{c,S}$ as a quotient model.
We will ultimately identify it with the ind-homotopy type of a family of graded spaces
$\overline{Q}_\epsilon[M,\on{hur}^{c,S}]$ as $\epsilon$ approaches $0$
in \autoref{lemma:scan-to-epsilon-spacing}.

\begin{notation}
	\label{notation:q}
	Let $c$ be a rack and $S$ be a Hurwitz module over $c$.
	For $M$ a graded set with a right action of $\Hur^c$,
define	
$Q[M , \on{hur}^{c,S}]$
to be the graded topological space consisting of 
pairs $(a,b)$ with $a \in M$
and $b = (x = \{x_1, \ldots, x_n\},t,\gamma,(\alpha_1, \ldots, \alpha_n,s))\in
\on{hur}^{c,S}$.

Define 
$\overline{Q}[M , \on{hur}^{c,S}]$
as the quotient of 
$Q[M , \on{hur}^{c,S}]$
under the following equivalence relation:
Suppose we write the path $\gamma$ from $x = \{x_1, \ldots, x_n\}$ to
$\on{ord}_n^{\Sigma^1_{g,f}}$ as a tuple $\gamma = (\gamma_1, \ldots, \gamma_n)$ where each
$\gamma_i$ connects $x_i$ to one of the $n$ points in a particular element of
the contractible set $\on{ord}_n^{\Sigma^1_{g,f}}$.
Suppose further that
\begin{enumerate}
	\item the first coordinate of $x_1$ is $0$ and
	\item there is some $v$ so that $\gamma_1$ has image in $[0,v] \times[0,1]$ while
$\gamma_2, \ldots, \gamma_n$ have image in $(v,t] \times [0,1]$;
i.e. $\gamma_1$ is
left of $\gamma_2, \ldots, \gamma_n$.
\end{enumerate}
Then, we identify we identify the point
$(a,b)$ with the point 
\begin{align*}
	(a \cdot \alpha_1, ( \{x_2, \ldots, x_n\}, t,
(\gamma_2,\ldots, \gamma_n), (\alpha_2, \ldots, \alpha_n)),
\end{align*}
where
$a \cdot \alpha_1$ denotes the result of the right action of $\alpha_1 \in
\pi_0(\on{hurbig}^c)$ on
the element $a \in M$.
\end{notation}
\begin{remark}
	\label{remark:arrange-2}
	If we have a point of 
	$Q[M , \on{hur}^{c,S}]$
satisfying condition $(1)$, we can always arrange that condition $(2)$ is
satisfied by repeatedly using the action of $B_n^{\Sigma^1_{g,f}}$ and applying 
homotopies of $\gamma$ to move $\gamma_1$ to the left of $\gamma_2, \ldots,
\gamma_n$.
\end{remark}

We now relate the scanning model to the quotient model.
\begin{lemma}
	\label{lemma:b-to-q}
	For $c$ a rack, $S$ a Hurwitz module over $c$, and
	$M$ a right $\on{Hur}^c$ module,
	there is a weak equivalence of graded topological spaces
	$\overline{Q}[M , \on{hur}^{c,S}] \to B[M ,\on{hur}^{\circ, c,S}].$
\end{lemma}
\begin{proof}
	The map is given by the map $(a, y) \mapsto (a,y)$, and one can verify this is
	a weak equivalence by imitating the proofs of 
	\cite[Proposition A.4.4 and Lemma
	A.4.7]{landesmanL:the-stable-homology-of-non-splitting};
	in our setting the argument is slightly easier because one does not need
	to worry about the part of \cite[Lemma
	A.4.7]{landesmanL:the-stable-homology-of-non-splitting} relating to
	applying the flow as we do not arrange any condition relating to the
	vertical spacing between points.
\end{proof}

\subsection{Refinements of the quotient model}
\label{subsection:quotient-refinements}

Ultimately, we are aiming to relate the bar construction to a certain refinement
of the quotient model. We accomplish this in
\autoref{proposition:pointed-scanning} after introducing a sequence of
refinements of the quotient model, and relating the quotient model to those
refinements.
We start by introducing a refinement where the time parameter is $1$.

\begin{notation}
	\label{notation:t-1-bar}
	Let $c$ be a rack and $S$ be a Hurwitz module over $c$.
	For $M$ a graded set with a right action of $\Hur^c$,
define	
$\overline{Q}_{t = 1}[M , \on{hur}^{c,S}] \subset \overline{Q}[M , \on{hur}^{c,S}]$
to be the graded topological space consisting of points of the form
$(a,(x,1,\gamma,\alpha))$, i.e. points such that $t= 1$.
\end{notation}

\begin{lemma}
	\label{lemma:deform-to-1}
	Let $c$ be a rack and $S$ be a Hurwitz module over $c$.
	For $M$ a graded right $\on{Hur}^c$ module,
	there is a deformation retraction of 
	$\overline{Q}[M , \on{hur}^{c,S}]$ onto
	$\overline{Q}_{t = 1}[M , \on{hur}^{c,S}]$.
\end{lemma}
\begin{proof}
	Define the retraction
$h: \overline{Q}[M , \on{hur}^{c,S}]
	\times[0,1]\to \overline{Q}[M , \on{hur}^{c,S}]$
	sending $( (a,(x,t,\gamma,\alpha)),s) \mapsto 
	(a,(x^s,(1-s)t+s,\gamma^s,\alpha))$, where $x^s$ and $\gamma^s$ are the
	configuration and paths obtained by stretching $x$ and $\gamma$ linearly to be length $(1-s)t+s$; explicitly 
	if $x = \{((u_1, v_2), \ldots,(u_n,v_n)\}$ then $x^s =  \{(( \frac{(1-s)t+s}{t}
	u_1, v_2), \ldots,(\frac{(1-s)t+s}{t}u_n,v_n)\}$ and if $\gamma_i(z) =
	(l,m)$ then $\gamma_i^s(z) = (\frac{(1-s)t+s}{t}l,m)$.
	This defines the desired deformation retraction of 
	$\overline{Q}[M , \on{hur}^{c,S}]$ onto
	$\overline{Q}_{t = 1}[M , \on{hur}^{c,S}]$.
\end{proof}

We next introduce a refinement of the quotient model which has an $\epsilon$
spacing between the vertical coordinates of the points of the configuration and
vertical coordinates of endpoints of the glued intervals on $\mathcal
M_{g,f,1}$.

\begin{notation}
	\label{notation:point-pushed-quotient}
	Let $\Phi$ denote the set of $y$-coordinates of endpoints of the
	intervals $J_1, \ldots, J_{4g}, J_1', \ldots, J'_{2f}$ defining
	$\mathcal M_{g,f,1} \simeq \Sigma_{g,f}^1$ as in
	\autoref{notation:surface}.

	Let $\delta := \min_{x,y \in \Phi} |x-y|$ denote the minimum
	difference between two elements of $\Phi$.
	Fix some $0< \epsilon < \delta$
	and
	let $(\mathbf R - W)^\epsilon$ denote the set of points whose $y$
	coordinates have distance $\geq \epsilon$ from $\Phi$ 
	and let $\mathcal M_{g,f,1}^\epsilon$ denote the 
	denote image of 
	$(\mathbf R - W)^\epsilon$ in $\mathcal M_{g,f,1}$.

	Let
	$\overline{Q}^\epsilon_{t=1}[M ,\on{hur}^{c,S}] \subset
	\overline{Q}_{t = 1}[M ,\on{hur}^{c,S}]$
	denote the closed subset  
	of points $(a,(x,1,\gamma,\alpha)) \in \overline{Q}_{t = 1}[M
	,\on{hur}^{c,S}]$ so that each point $x_i \in x$ lies in 
	$\mathcal M_{g,f,1}^\epsilon \subset \mathcal M_{g,f,1}$.
\end{notation}
\begin{remark}
	\label{remark:rectangles-description}
	The topological space $\mathcal M_{g,f,1}^\epsilon$ from 
	\autoref{notation:point-pushed-quotient} can be viewed as a disjoint union of
	$2g + f$ rectangles. The bottom $2g$ of these rectangles are obtained by
	gluing the rectangle in with right boundary $J_{i}^\epsilon$ (with $i \equiv 1$
	or $2 \bmod 4$) to the rectangle with right boundary
	$J_{i+2}^\epsilon$,
	where, if $J_j = 1 \times [a_j, b_j]$, we use $J_j^\epsilon := 1
	\times [a_j+\epsilon, b_j-\epsilon]$ .
	The top $f$ of these rectangles are obtained by gluing the rectangle
	with right boundary $(J'_{i})^\epsilon$ for $i \equiv 1 \bmod 2$
	to the rectangle with right boundary $(J'_{i+1})^\epsilon$, where if
	$J'_j = 1 \times [a'_j, b'_j]$, $(J'_j)^\epsilon = 1 \times
	[a'_j+\epsilon, b'_j-\epsilon]$.
	See \autoref{figure:rectangle-division} for a visual depiction.
\end{remark}

\begin{figure}
\includegraphics[scale=.4]{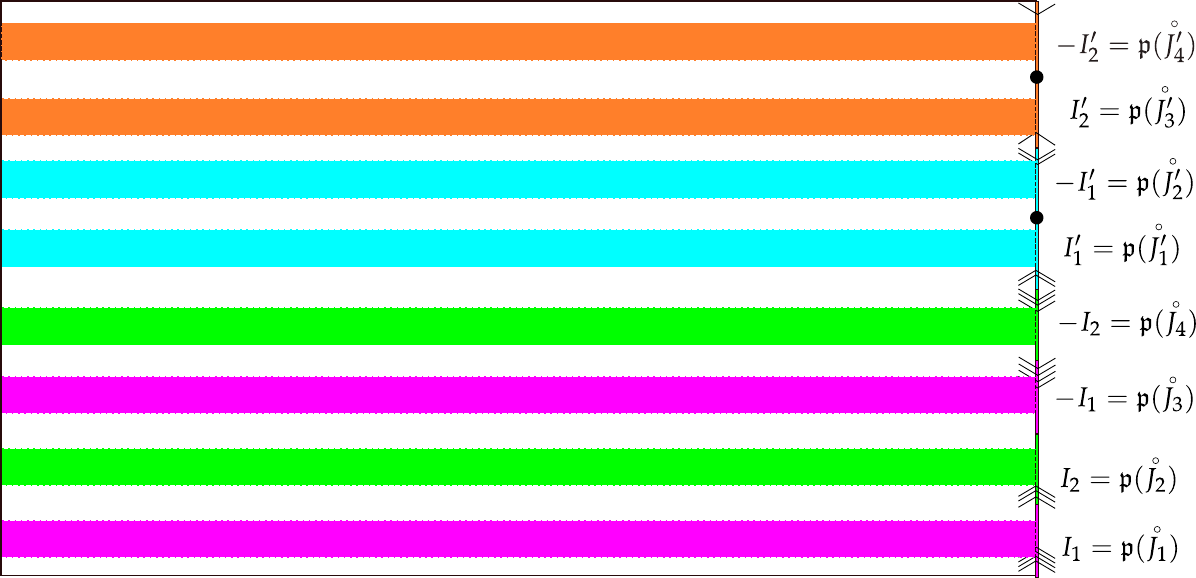}
\caption{
	This depicts $\mathcal M_{g,f,1}^\epsilon$ in the case $g = 1, f = 2$.
	The surface $\mathcal M_{1,2,1}^\epsilon$ is a union of $2g + f = 4$
	rectangles. There are eight rectangles pictured in four colors. Each pair
	of rectangles of the same color are glued along their right edge so that
	$\mathcal M_{1,2,1}^\epsilon$
	consists of four rectangles.}
\label{figure:rectangle-division}
\end{figure}

For the next lemma, we will need the notion of an ind-weak equivalence.
\begin{definition}
	\label{definition:}
	If $L: \on{Top}^{\mathbb N} \to \on{Spc}^{\mathbb N}$ is
the functor of infinity categories sending a pointed graded topological space to
its weak homotopy type. Then a pointed map $f: X \to X'$ of graded spaces in
$\on{Ind}(\on{Top}^{\mathbb N})$ is an
{\em ind-weak equivalence} if $\colim Lf$ is an equivalence.
\end{definition}

\begin{lemma}
	\label{lemma:deform-epsilon-distance}
	For $\delta$ as in \autoref{notation:point-pushed-quotient}, let $\epsilon < \delta$.
	For $c$ a rack, $S$ a Hurwitz module over $c$,
	and
	$M$ a graded right $\on{Hur}^c$ module,
	the inclusions 
	$\overline{Q}^\epsilon_{t=1}[M,\on{hur}^{c,S}]
	\to
	\overline{Q}_{t =1}[M , \on{hur}^{c,S}]$
	define an ind-weak homotopy equivalence of graded topological spaces over the poset of real numbers
	$0 < \epsilon < \delta$.
\end{lemma}
\begin{proof}
	To prove the result, we use \cite[Lemma
	A.4.6]{landesmanL:the-stable-homology-of-non-splitting}, and verify the
	two conditions there.
	Choose a continuous vector field on the interior of $\mathcal M_{g,f,1}$ whose preimage in $\mathbf
	R - W$ has the following properties:
	\begin{enumerate}
		\item On each horizontal line in $\mathbf R$ with coordinate in
			$\Phi$, choose the
	vector field so that it points directly left (with vanishing $y$
	coordinate and with negative $x$ coordinate) such that the induced flow on this line reaches $t=0$ in finite time for all $x$ coordinates smaller than $1$.
\item On each horizontal line
	in $\mathbf R$ between two points $(w_1,1), (w_2,1)$ for $w_i \in \Phi$, choose the
	vector field at $(u,v)$ so that the flow has non-positive $x$
	coordinate,
	the $y$-coordinate is positive if $u < \frac{w_1+w_2}{2}$, and the
	$y$-coordinate is negative if $u > \frac{w_1+w_2}{2}$. 
	\end{enumerate}
	See \autoref{figure:vector-field} for a picture of a possible such
	vector field in the case $g = 1, f = 2$.
\begin{figure}
\includegraphics[scale=.4]{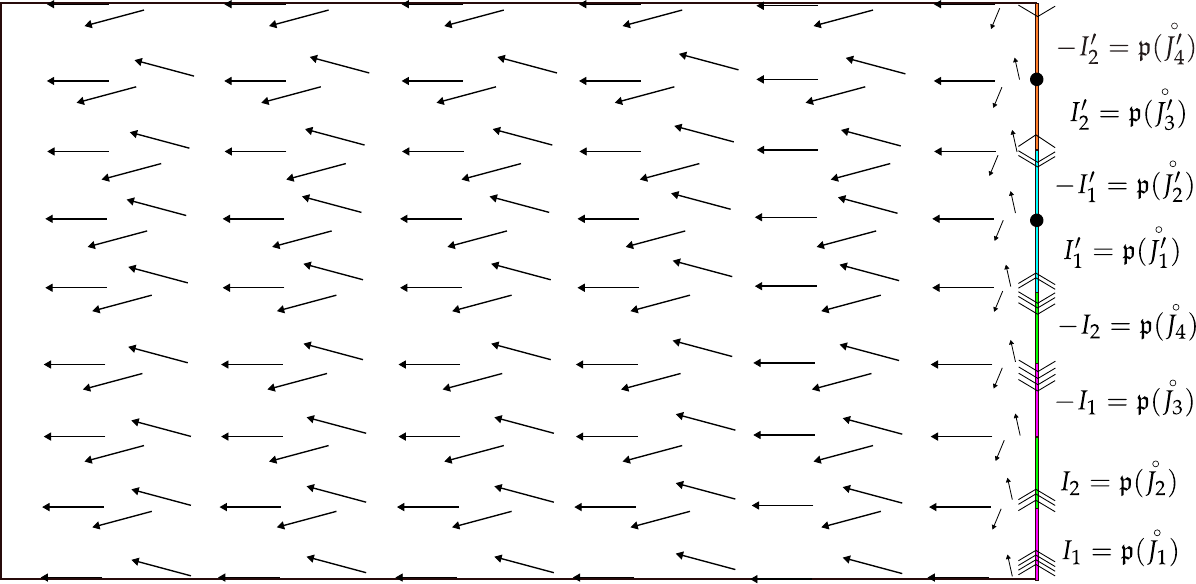}
\caption{
	This picture depicts a vector field on $\mathcal M_{g,f,1}$ as described
	in the proof of \autoref{lemma:deform-epsilon-distance} in the case $g=
1, f = 2$.}
\label{figure:vector-field}
\end{figure}
Then, flowing
	along a vector field as described above defines a function
	$\Psi: \overline{Q}_{t =1}[M , \on{hur}^{c,S}] \times [0,1] \to
	\overline{Q}_{t =1}[M , \on{hur}^{c,S}]$.

	One can choose the vector field as above so that
	the map $\Psi$ moreover has the following properties: 
	First, for any $x \in
	\overline{Q}^\epsilon_{t=1}[M ,\on{hur}^{c,S}]$, and any $s \in
	[0,1]$, we have $\Psi(x,s) \in \overline{Q}^\epsilon_{t=1}[M ,\on{hur}^{c,S}]$.
	Second, for any $x \in \overline{Q}_{t=1}[M ,\on{hur}^{c,S}]$,
	$\Psi(x,1) \in \overline{Q}^\epsilon_{t=1}[M ,\on{hur}^{c,S}]$.
		These two properties induce a pointed homotopy as in
		\cite[Lemma
		A.4.6]{landesmanL:the-stable-homology-of-non-splitting}
		satisfying the analogous two properties there,
		and so 
	$\overline{Q}^\epsilon_{t=1}[M ,\on{hur}^{c,S}] \to
	\overline{Q}_{t =1}[M , \on{hur}^{c,S}]$
	defines an ind-weak homotopy equivalence.
\end{proof}

We next also include an $\epsilon$ spacing between any two points of the
configuration.

\begin{notation}
	\label{notation:scanned-quotient}
	Let $c$ be a rack and $S$ be a Hurwitz module over $c$.
	Let $M$ be a graded $\on{hurbig}^c$ module.
	Let 
	$\overline{Q}_{\epsilon,\epsilon'}[M ,\on{hur}^{c,S}] \subset
	\overline{Q}^\epsilon_{t=1}[M ,\on{hur}^{c,S}]$
	denote the subset of
	$(a,(x,1,\gamma,\alpha)) \in \overline{Q}^\epsilon_{t=1}[M
	,\on{hur}^{c,S}]$ satisfying the following conditions:
	\begin{enumerate}
		\item For any two points $x_1, x_2$ in the
	configuration corresponding to $x$, with $y$ coordinates $y(x_1),
	y(x_2)$ we have $|y(x_1) - y(x_2)| \geq \epsilon'$.
		\item If $(1,y(x_1))$ is identified with some point $(1,y') \in
	\mathbf{R} = [0,1] \times [0,1]$ with $t = 1$ from \autoref{notation:surface}
	then we also have $|y' - y(x_2)| \geq \epsilon'$.
	\end{enumerate}
	We use $\overline{Q}_{\epsilon}[M ,\on{hur}^{c,S}]$ to denote $\overline{Q}_{\epsilon,\epsilon}[M ,\on{hur}^{c,S}]$.
		Also, define $Q_\epsilon[M ,\on{hur}^{c,S}] \subset
	Q[M ,\on{hur}^{c,S}]$ to be the preimage of 
	$\overline{Q}_\epsilon[M ,\on{hur}^{c,S}] \subset\overline{Q}[M
	,\on{hur}^{c,S}]$ under the quotient map
	$Q[M ,\on{hur}^{c,S}] \to \overline{Q}[M ,\on{hur}^{c,S}]$.

	Moreover, if $M$ is a pointed graded right $\on{Hur}^c$ module,
	define $Z_{M,S} \subset Q_{\epsilon}[M ,
		\on{hur}^{c,S}]$ to be the subspace consisting of all
		points whose projection to $M$ is the base point.
	Let $Q^*_\epsilon[M ,
		\on{hur}^{c,S}_+] :=
		{Q}_{\epsilon}[M ,
		\on{hur}^{c,S}]/Z_{M,S}$
		and define $\overline{Q}^*_\epsilon[M ,
		\on{hur}^{c,S}_+]$ to be the quotient of 
		$\overline{Q}_\epsilon[M,\on{hur}^{c,S}]$ by the image of
		$Z_{M,S}$.
\end{notation}

\begin{lemma}
	\label{lemma:scan-to-epsilon-spacing}
	For $c$ a rack, $S$ a Hurwitz module over $c$,
	and
	$M$ a graded right $\on{Hur}^c$ module,
	the inclusions
	$\overline{Q}_{\epsilon,\epsilon'}[M ,\on{hur}^{c,S}] \to
	\overline{Q}^\epsilon_{t=1}[M ,\on{hur}^{c,S}]$ over the poset of real
	numbers $\delta > \epsilon> 0, \delta > \epsilon' > 0$ 
	form an ind-weak equivalence. 
\end{lemma}
\begin{proof}
	Note that 
	any point of $\overline{Q}^\epsilon_{t=1}[M ,\on{hur}^{c,S}]$
	corresponds to a configuration of points whose vertical coordinate lies
	at least distance $\epsilon$ from any element of $\Phi$, as in
	\autoref{notation:point-pushed-quotient}.
	As described in \autoref{remark:rectangles-description},
	such a configuration space can be identified with a disjoint union of
	configuration spaces in rectangles, and the result can be proven via an
	argument analogous to that in the proof of \cite[Lemma
	A.4.7]{landesmanL:the-stable-homology-of-non-splitting}, using the flow
	from \cite[Construction
	A.4.3]{landesmanL:the-stable-homology-of-non-splitting} to push points
	toward the two vertical boundaries of each rectangle. Specifically, if
	such a rectangle is obtained by gluing two rectangles with right sides
$J_i^\epsilon$ with $J_{i+2}^\epsilon$ or $(J'_i)^\epsilon$ 
with $(J'_{i+1})^\epsilon$, as in
\autoref{remark:rectangles-description}, then the flow is obtained from
	pushing the point away from this glued side and toward the left boundary
	of $\mathcal M_{g,f,1}^\epsilon$.
\end{proof}

Finally, we prove a version of the above lemma where we also include base points.

\begin{proposition}
	\label{proposition:pointed-scanning}
	Let $c$ be a rack, $S$ a Hurwitz module over $c$,
	and let $M$ be a graded pointed set with a right $\Hur^c$ action.
	With notation as in \autoref{notation:scanned-quotient}, 
	$M \otimes_{\Hur^c} \Hur^{c,S}$ is identified with the
	ind-weak homotopy type of 
	$\overline{Q}_\epsilon[M ,\on{hur}^{c,S}]$ and
	$M \otimes_{\Hur^c_+} \Hur^{c,S}_+$ is identified with the
	ind-weak homotopy type of 
	$\overline{Q}^*_\epsilon[M ,\on{hur}^{c,S}_+]$
\end{proposition}
\begin{proof}
	By combining 
	\autoref{lemma:bar-to-b},
	\autoref{lemma:b-to-q},
	\autoref{lemma:deform-to-1},
	\autoref{lemma:deform-epsilon-distance},
	\autoref{lemma:scan-to-epsilon-spacing}
	we obtain an ind-weak homotopy equivalence between 
	$\overline{Q}_{\epsilon}[M ,\on{hur}^{c,S}]$
	and
	$M \otimes_{\on{hur}^c} \on{hur}^{c,S}$. Here we use that the diagonal
	is cofinal in the product of two copies of the poset of real numbers
	between $0$ and $\delta$.
	One can then use an argument analogous to that in the proof of
	\cite[Theorem A.4.9]{landesmanL:the-stable-homology-of-non-splitting}
	to include base points and obtain an identification
	between
	$M \otimes_{\Hur^c_+} \Hur^{c,S}_+$ and
	the ind-weak homotopy type of 
	$\overline{Q}^*_\epsilon[M ,\on{hur}^{c,S}_+]$.
\end{proof}

\section{Stability of a quotient}
\label{section:quotient-stability}

Recall that our general strategy from
\cite{landesmanL:homological-stability-for-hurwitz} to prove homological
stability of Hurwitz spaces
was to first
prove that a suitable quotient satisfies homological stability, and then to
remove elements in the quotient one at a time, and show that even without
quotienting, these Hurwitz spaces still satisfy homological stability.
We will apply a similar approach to bijective Hurwitz modules.
In \autoref{theorem:all-quotient-bounded}, at the end of this section, we will complete the first step, where we show the quotient
satisfies homological stability.
This stabilization for the quotient
will essentially follow from a general theorem we proved in a previous paper
\cite[Theorem 3.1.4]{landesmanL:homological-stability-for-hurwitz},
and the main difficulty will be in verifying condition (b) of that result, which
is a statement about the cohomology of a certain bar construction,
$*_+\otimes_{\Hur^c_+} \Hur^{c,S}_+$.

The key input from our prior work we will need is that a suitable quotient of
Hurwitz space itself has homology which stabilizes.
%

We next recall the notion of being bounded in a linear range, 
following \cite[Definition 3.1.1]{landesmanL:homological-stability-for-hurwitz}
which captures the
idea that the homology groups of some sequence of spaces stabilize to $0$ in a
linear range.
\begin{definition}
	\label{definition:bounded-linear}
	Suppose $k$ is a commutative ring and $X$ is a $\ZZ$-graded $k$-module spectrum,
with $X_j$ the $j$th graded part.
For a positive real number $r_1$ and a real number $r_2$, we say $X$ is $f_{r_1, r_2}$-bounded if
$\pi_i(X_j) = 0$ whenever $j > r_1 i + r_2$.
We then say $X$ is {\em bounded in a linear range} there exist real numbers
$r_1$ and $r_2$ with $r_1 \geq 0$ so that $X$ is $f_{r_1, r_2}$ bounded.
\end{definition}

We are aiming to prove a certain quotient of Hurwitz space stabilizes,
which will essentially follow from a general theorem we proved in a previous paper
\cite[Theorem 3.1.4]{landesmanL:homological-stability-for-hurwitz}.
The main difficulty will be in verifying condition (b) of that result, which
is a statement about the homology of a certain bar construction,
$*_+\otimes_{\Hur^c_+} \Hur^{c,S}_+$, which we verify next.
This measures the generators (or cells) of $\Hur^{c,S}$ over $\Hur^c$.
For the next proposition, recall that we have defined a grading on 
$\widetilde{C}_*(*_+\otimes_{\Hur^c_+}
	\Hur^{c,S}_+ )$ 
	coming from the grading defined in \autoref{notation:hur}
	keeping track of the number of points which lie in a chosen
	$S$-component $z$ of $c$.
\begin{proposition}
	\label{proposition:module-cells-bounded}
	Let $c$ be a rack and $S$ be a Hurwitz module.
	We have that $\widetilde{C}_*(*_+\otimes_{\Hur^c_+}
	\Hur^{c,S}_+ )$ is $f_{1,0}$-bounded.
\end{proposition}
\begin{proof}
	Given \autoref{proposition:pointed-scanning}, the argument is now very
	similar to that presented in the proof of \cite[Lemma
	3.2.8]{landesmanL:homological-stability-for-hurwitz}, as we now
	explain. 
	
	Namely, \autoref{proposition:pointed-scanning} implies that $*_+\otimes_{\Hur^c_+} \Hur^{c,S}_+$ is identified with the
	ind-weak homotopy type of 
	$\overline{Q}^*_\epsilon[*_+ ,\on{hur}^{c,S}_+]$.

	Recall \autoref{remark:rectangles-description}, which implies that the configuration of points 
	associated to points in $\overline{Q}^*_\epsilon[*_+ ,\on{hur}^{c,S}_+]$  can be viewed as a disjoint union of
	$2g + f$ rectangles. 	
	
	We can consider the subspace $\overline{L}^*_\epsilon[*_+
	,\on{hur}^{c,S}_+]$ of $\overline{Q}^*_\epsilon[*_+ ,\on{hur}^{c,S}_+]$
	where, in each rectangle as above, the configuration of points are
	evenly spaced in the vertical direction, including spacing between the
	top point and the top of the rectangle as well as between the bottom
	point and the bottom of the rectangle. There is an evident deformation retraction of $\overline{Q}^*_\epsilon[*_+ ,\on{hur}^{c,S}_+]$ onto this subspace given by linearly moving points vertically until they are evenly spaced.
	
	In grading $n$,	we claim $\overline{L}^*_\epsilon[*_+
	,\on{hur}^{c,S}_+]$ is a wedge of $n$-spheres: This is because it is the
	$1$-point compactification of a disjoint union of copies of $(0,1)^n$, where there are $n$ points evenly spaced in the vertical direction of the rectangles.
	
	The result follows since $n$-spheres are $n$-connective, so $*_+\otimes_{\Hur^c_+} \Hur^{c,S}_+$ is $n$-connective in grading $n$.
	\end{proof}

\begin{notation}
	\label{notation:component-notation}
	Let $c$ be a finite rack. 
	Following notation in \cite[Notation
4.4.1]{landesmanL:the-stable-homology-of-non-splitting}, for $x \in c$, we use
$\alpha_x$ to denote the corresponding component of $\pi_0 \Hur_1^c$.
For $X \subset c$, we also write $\alpha_X := \{\alpha_x, x \in X\},$
we use $\alpha_X^i := \{\alpha_x^i, x \in X\}$ for $i$ an integer. Fix an
$S$-component $z \subset c$, and choose an ordering $x_1,\dots,x_{|c|}$ on $c$
where the elements of $z$ come first. For any subset $X \subset c$, we use
$\Hur^{c,S}_+/(\alpha_X^{\ord(X)})$ to denote the tensor product
$\Hur^{c}_+/(\alpha_{x_{i_1}}^{\ord(x_{i_1})})\otimes_{\Hur^{c}_+}\dots
\otimes_{\Hur^{c}_+}\Hur^{c}_+/(\alpha_{x_{i_{|X|}}}^{\ord(x_{i_{|X|}})})
\otimes_{\Hur^{c}_+}\Hur^{c,S}_+$ where $i_1,\dots,i_{|X|}$ are the indices of the elements of $X$ in order of the ordering on $c$. We use the same notation to denote iterated quotients after taking chains.
\end{notation}

The following lemma can be proven via a straightforward generalization of the proof of
\cite[Lemma 3.2.7]{landesmanL:homological-stability-for-hurwitz}.

\begin{lemma}
	\label{lemma:nilpotent-augmentation}
	Let $c$ be a finite rack and $S$ be a Hurwitz module over $c$.
	Let $I$ be the augmentation ideal of $\pi_0 {C}_*(\Hur^{c})\simeq
	\pi_0 \widetilde{C}_*(\Hur_+^{c})$, 
	and let $I_{>0} \subset I$ be the subset of $I$ with non-negative grading.
	Let $z = \{y_1, \ldots, y_{|z|}\}$.
	Then left multiplication
	by $I_{>0}^{1+\sum_{i=1}^{|z|} 2^i \ord(y_i)-1}$ acts by $0$ on 
	$\widetilde{C}_*(\Hur^{c,S}_+/(\alpha_c^{\ord(c)}))$.
\end{lemma}

We can now deduce our main result on the stability of a quotient of Hurwitz
modules. We note that this works for general Hurwitz modules, and not
just bijective Hurwitz modules.
For the next statement, we fix a rack $c$ a Hurwitz module $S$ over $c$, $z
\subset c$ an $S$-component, and 
give 
$\widetilde{C}_*(\Hur^{c,S}_+/(\alpha_c^{\ord(c)}))$ the grading
induced by the grading on 
$\Hur^{c,S}$ 
described in \autoref{notation:hur}.
We use 
$\ord_c(z)$ to denote the maximal order of the action of an element $y_i \in z$.
\begin{theorem}
	\label{theorem:all-quotient-bounded}
	Let $c$ be a finite rack, and $S$ a Hurwitz module over $c$ with $0$ set
	$T_0$, and $z \subset c$ an $S$-component.
	With notation as above,
	$\widetilde{C}_*(\Hur^{c,S}_+/(\alpha_c^{\ord(c)}))$
	is $f_{r_1, r_2}$ bounded, where the values of $r_1$ and $r_2$ depend
	only on
	$|z|$ and $\ord_c(z)$.
\end{theorem}
\begin{proof}
	This follows from the final statement of
	\cite[Theorem
	3.1.4]{landesmanL:homological-stability-for-hurwitz}
	once we verify the three conditions (a), (b), and (c) stated there, and
	show that the constants $v,w,d,t,\mu,b$ defined there only depend on
	$|z|$ and $\ord_c(z)$.
	We take $R = \widetilde{C}_*(\Hur^{c}_+)$.
	We can take the constant $d$ to be $1$ because $R$ is generated in
	degree $1$ by the elements of $c$.
	Indeed, condition (a) was shown in 
	\cite[Lemma 3.2.8]{landesmanL:homological-stability-for-hurwitz},
	where it was shown we can take $v = 1, w = 0$.
	Using \autoref{proposition:module-cells-bounded}, we see 
	$\widetilde{C}_*(*_+ \otimes_{\Hur^c_+} \Hur^{c,S}_+/(\alpha_c^{\ord(c)}))$
	is $f_{1,0}$ bounded.
	For $x \in c$, each $\alpha_x^{\on{ord}(x)}$ either acts trivially on this and either has degree $0$ if
	$x \in c - z$ or has degree at most $\ord_c(z)$ if $x \in z$,
	we find that the quotient by the actions of $\alpha_x^{\on{ord}(x)}, x \in c$ is $f_{|\ord_c(z)|,0}$ bounded.
	Hence, condition (b) follows with $\mu = |\ord_c(z)|, b = 0$.
	Finally condition (c) was shown in
	\autoref{lemma:nilpotent-augmentation}, which also shows that we can take $t$ to only depend on $|z|$ and $\ord_c(z)$.
\end{proof}

\section{An equivalence of bar constructions}
\label{section:equivalence}

We have shown in \autoref{theorem:all-quotient-bounded} that a certain quotient
of a Hurwitz module has vanishing stable homology.
We next aim to show that the Hurwitz space itself has homology which stabilizes,
which we demonstrate by the technique of ``unquotienting'' via computing the stable homology of each
quotient in a fashion similar to that carried out in
\cite{landesmanL:homological-stability-for-hurwitz}.
The main result of this section is \autoref{proposition:subrack-comparison},
stating that a certain comparison of bar
constructions is a homology equivalence.
This will be used in \autoref{section:homological-stability}
as the key input for this unquotienting procedure.

In order to compare these bar constructions, we will construct a certain
nullhomotopy. This nullhomotopy involves the notion of {\em allowable moves}
which we define next.
Allowable moves describe how we are allowed to move the labeled points
around Hurwitz modules in 
$\overline{Q}_\epsilon[M,\on{hur}^{N_c(c'),S'}]$ 
introduced in
\autoref{notation:scanned-quotient}. 

\begin{definition}
	\label{definition:allowable-and-output}
	Using notation from \autoref{notation:point-pushed-quotient} and
	$c,c',S,S'$ as in 
	\autoref{lemma:normalizer-is-hurwitz-module},
	fix a point of
	$\overline{Q}_\epsilon[ \pi_0(\Hur^{c'})[\alpha_{c'}^{-1}]_+,
	\on{hur}^{N_c(c'),S'}]$
	which we may think of as a tuple $(m, (x, 1, \gamma, \alpha))$
	satisfying the constraints of \autoref{notation:scanned-quotient}.
	Say $x = \{x_1, \ldots, x_n\} \subset \mathcal
	M^\epsilon_{g,f,1}$.
	Choose a collection of horizontal paths $\eta_1, \ldots, \eta_{n+2g+f}$
	lying in $\mathcal M^\epsilon_{g,f,1}$
	which we describe next.
	First, identify $\mathcal M^\epsilon_{g,f,1}$
	with a collection of $2g + f$ rectangles as in
	\autoref{remark:rectangles-description} in a way so that the $\rho$th such
	rectangle has vertical coordinate ranging from $a_\rho$ to $b_\rho$,
	and the $\rho$th such rectangle, counting from the bottom, has $n_\rho$ points from $\{x_1, \ldots,
	x_n\}$ with vertical coordinates $v^\rho_0 := a_\rho < v^\rho_1 < \cdots <
	v^\rho_{n_i} < v^\rho_{n_\rho+1} := b_\rho$.
	Then there are $n_\rho+1$ such paths contained in the $i$th rectangle
	which are given by straight lines across the rectangle with vertical
	coordinates $\frac{v^\rho_j + v^\rho_{j+1}}{2}$, for $0 \leq j \leq
	n_\rho$.
	We orient $\eta_i$ so that the starting endpoint, viewed as a point in
	$\mathbf R$, always has higher
	second coordinate than the ending endpoint. 
	In particular, the starting point of the allowable path $\eta_{i+1}$ is
	higher than the starting point of the allowable path $\eta_i$.
	We call the $\eta_1, \ldots, \eta_{n+2g+f}$ the set of {\em allowable
	paths} of the point of $(m, (x, 1, \gamma, \alpha))$ and
	an {\em allowable move} consists of moving a point with label $\beta
	\in c'$ from the left boundary across one of the allowable paths $\eta_i$.
	See \autoref{figure:allowable-paths} for a pictorial depiction of the
	allowable paths associated to a particular configuration.

\begin{figure}
\includegraphics[scale=.6]{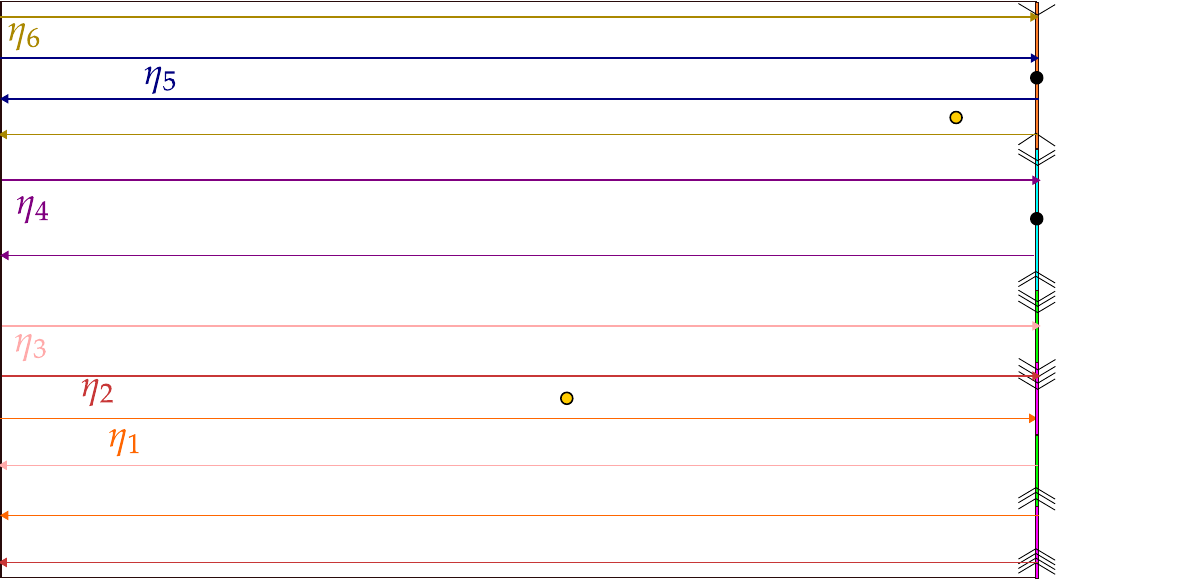}
\caption{
	This picture the allowable paths $\eta_1, \ldots, \eta_6$ in a
particular configuration in $\mathcal M_{1,2,1}$ with $2$ points.}
\label{figure:allowable-paths}
\end{figure}

	After moving a point with label $\beta$ through path $\eta_i$, we may
	consider $\eta_i$ as a path in $B_{n+1}^{\Sigma^1_{g,f}}$ and if
	$\eta_i (\beta, \alpha_1, \ldots, \alpha_n, s) = (\beta',
	\alpha_1',\ldots, \alpha_n', s')$ and $\alpha' := (\alpha_1', \ldots,
	\alpha_n',s')$, we denote by 
	$(\beta'; (x, 1, \gamma, \alpha')) \in c \times \Hur^{c,S}$
	the {\em output} of the allowable move $(\beta, \eta_i) \in c \times
	\pi_1(\Sigma^1_{g,f})$
	and we call $\beta' \in c$ the {\em left output}
	and we call $\alpha'$ the
	{\em right output}. 
		If we have a sequence of allowable moves $(\beta_1, \eta_{i_1}), \ldots,
	(\beta_j, \eta_{i_j})$, then inductively define the {\em output} of this sequence
	as follows: if $(\beta''; (x, 1, \gamma, \alpha''))$ is the output of
	applying the first $j-1$ allowable moves in the sequence, then the
	output of the sequence is the output $(\beta'; (x,1,\gamma, \alpha'))$ of applying the allowable move
		$(\beta_j,\eta_{i_j})$ to $(x, 1, \gamma, \alpha'')$, $\beta'$
		is the {\em left output}
	and $\alpha'$ is the {\em right output}.
\end{definition}

\begin{notation}
	\label{notation:allowable-no-points}
	We work in the setting of \autoref{definition:allowable-and-output}.
	In the special case 
	that $n = 0$, 
	so 
	$x = \emptyset \subset \mathcal	M^\epsilon_{g,f,1}$,
we use 
$\xi_1, \ldots, \xi_{2g+f}$ as alternate notation for the allowable paths
$\eta_1, \ldots, \eta_{2g+f}$.
See \autoref{figure:empty-allowable-paths} for a visualization in the case $g
=2, f = 1$.
We will also use $\overline{\xi}_i := \xi_{2g+f+1-i}$.
\end{notation}

\begin{figure}
\includegraphics[scale=.4]{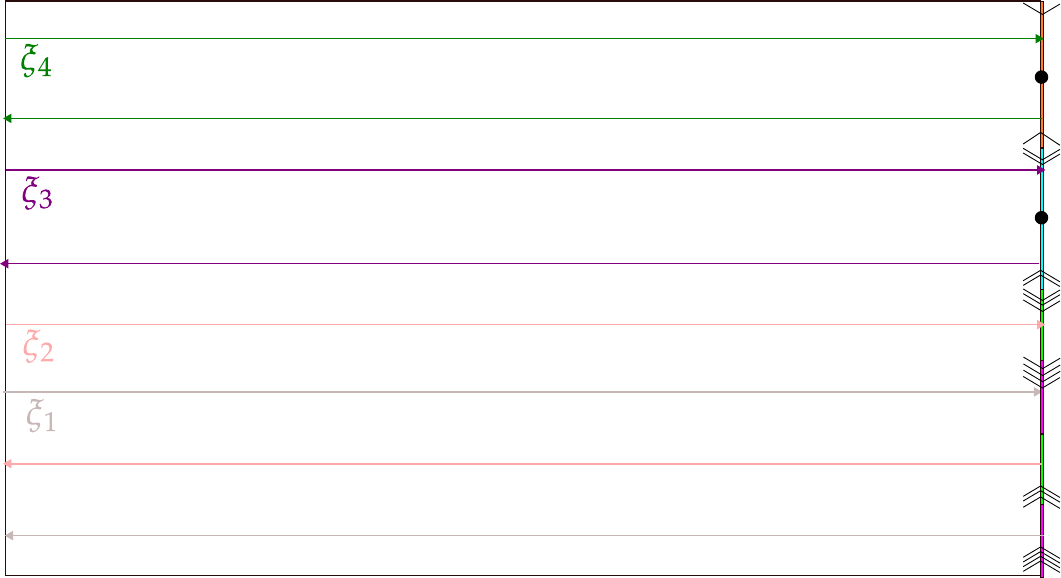}
\caption{
	This picture shows the paths $\xi_1, \ldots, \xi_4$ 
	$\mathcal M_{1,2,1}$
	which are the names
	we are using for the
	allowable paths $\eta_1, \ldots, \eta_4$
	associated to the empty configuration.}
\label{figure:empty-allowable-paths}
\end{figure}

The following lemma is fairly straightforward to see using that 
$\Sigma^1_{g,f} - \{x_1, \ldots,
x_n\}$ has $2g+n+f$ generators and there are also $2g + n + f$ allowable paths.

\begin{lemma}
	\label{lemma:generating-allowable}
	Fix a point of 
	$(m, (x = \{x_1, \ldots, x_n\}, 1, \gamma, \alpha)) \in
	\overline{Q}_\epsilon^*[
	\pi_0(\Hur^{N(c)})[\alpha_{N(c)}^{-1}]_+, \on{hur}^{c,S}_+]$.
The allowable paths $\xi_1, \ldots, \xi_{2g+f}$
defined in \autoref{notation:allowable-no-points}
generate the fundamental group $\pi_1(\Sigma^1_{g,f}, \star)$; here, we use
$\star$ to denote a contractible subset of the boundary of $\Sigma^1_{g,f}$ containing all the
endpoints of the $\xi_i$, such as the left boundary of the rectangle in
\autoref{figure:empty-allowable-paths}.
\end{lemma}

In order to describe the desired equivalence of bar constructions, we next
describe the relation between allowable moves and certain Hurwitz spaces.

\begin{lemma}
	\label{lemma:output-leaving-c'}
	With notation for $c,c',S,S'$ as in \autoref{lemma:normalizer-is-hurwitz-module},
	let 
	$(m, (x = \{x_1, \ldots, x_n\}, 1, \gamma, \alpha = (\alpha_1,\ldots,
	\alpha_n, s))) \in \overline{Q}_\epsilon[
	\pi_0(\Hur^{c'})[\alpha_{c'}^{-1}]_+, \on{hur}^{N_c(c'),S'}]$.
	Viewing $\alpha$ as an element of $T_n$, suppose $\alpha_1, \ldots,
	\alpha_n \in N_c(c')$ but $\alpha$ is not in
	$N_c(c')^n \times T'_0 \subset T_n$.
Then there is some sequence of allowable moves 
whose left output does not lie in $c'$.
\end{lemma}
\begin{proof}
	By definition of $S_{c'}$, we must have 
	$s \notin T'_0$ and there must be some sequence $(\beta_1, \ldots,
	\beta_r, s)$ with each $\beta_i \in c'$ which is equivalent under the
	action of $B_n^{\Sigma^1_{g,f}}$ to a
	sequence $(\beta'_1, \ldots, \beta'_r, s')$ with some $\beta'_i \notin
	c'$.
	We may assume $r$ is minimal so there is a unique such $\beta'_i$.
	In particular, for any $\alpha_1, \ldots, \alpha_n \in N_c(c')$, 
	$(\alpha_1, \ldots, \alpha_n,\beta_1, \ldots, \beta_r,
	s)$ is also equivalent under the 
$B_r^{\Sigma^1_{g,f}} \subset B_{n+r}^{\Sigma^1_{g,f}}$ (coming from the last
$r$ points)
	to 
$(\alpha_1, \ldots, \alpha_n,\beta'_1, \ldots, \beta'_r,
s')$.
Now, we apply the automorphism of the braid group moving the first $n$ points past
the last $r$ points. 
This gives an identification
of a sequence of the form
$(\delta_1, \ldots, \delta_r, \alpha_1, \ldots, \alpha_n, s)$ with
$\delta_1,\ldots, \delta_r \in c'$
with a sequence of the form 
$(\delta'_1, \ldots, \delta'_r, \alpha''_1, \ldots, \alpha''_n, s')$
where $\alpha''_i \in N_c(c')$ and a unique $\delta_i' \notin c'$, but where we only apply an element of 
$\sigma \in B_r^{\Sigma^1_{g,f}} \subset B_{n+r}^{\Sigma^1_{g,f}}$, this time acting on the
coming from the first $r$ points.

Let us now explain why the above observation implies the lemma.
Define $m' = m \delta_r^{-1} \cdots
\delta_1^{-1}$
so that $m = m' \cdot \delta_1 \cdots \delta_r$.
Now, as explained above, there is an element 
$\sigma \in B_r^{\Sigma^1_{g,f}}\subset B_{n+r}^{\Sigma^1_{g,f}}$ which sends
$(\delta_1, \ldots, \delta_r, \alpha_1, \ldots, \alpha_n, s)$
to
$(\delta'_1, \ldots, \delta'_r, \alpha''_1, \ldots, \alpha''_n, s')$.
Recall that
$B_r^{\Sigma^1_{g,f}}$
is generated by the joint actions of 
$B_r^{\Sigma^1_{0,0}}$, acting on the first $r$ points, together with the
allowable paths
$\xi_1, \ldots, \xi_{2g+f}$ associated to the
empty configuration) using
\autoref{lemma:generating-allowable},
where we view $\xi_i$ as elements of $\pi_1(\Sigma^1_{g,f},\star)$ for $\star$
a contractible subspace of the boundary of $\Sigma^1_{g,f}$ containing the
endpoints of all $\xi_i$.
Hence,
we can write $\sigma = \sigma_1 \cdots \sigma_j$ where each $\sigma_i$ either
lies in
$B_r^{\Sigma^1_{0,0}}$ or is one of the $\xi_i$.
The element $\xi_i$ acts on $(m' \zeta_1 \cdots \zeta_r, \theta_1, \ldots,
\theta_n, t)$
by sending it to $(m' \zeta_1 \cdots \zeta'_r, \theta_1', \ldots,
\theta_n',t')$ where $\zeta'_r$ is the left output of the allowable move
$(\zeta_r,\eta_i)$ associated to $\xi_i$ on $(x,1,\zeta,\theta)$ and 
the $i$th generator of the braid group sends the element 
$(m' \zeta_1 \cdots \zeta_r, \theta_1, \ldots,
\theta_n, t)$
to
$(m' \zeta_1 \cdots \zeta_i (\zeta_i^{-1} \zeta_{i+1}\zeta_i) \cdots \zeta_r, \theta_1, \ldots,
\theta_n, t)$.
However, we may observe that any such element of the braid group acts trivially
by definition of $\pi_0 \Hur^{c'}$, which implies that $\sigma$ can be expressed
as a sequence of allowable moves applied to 
$(m, (x,1,\gamma,\alpha))$.
Finally, since the product $\delta'_1 \cdots \delta'_r$ contains a unique
element not in $c'$ by assumption, the product does not lie in $\pi_0\Hur^{c'}$,
and hence must be identified with the basepoint.
It follows that the result of one of the allowable moves corresponding
to $\sigma$ must have some left output not in $c'$, as claimed.
\end{proof}

In the above lemma, we only showed one could use a sequence of allowable moves to escape
$c'$, but it turns out one can already escape $c'$ via a single allowable move,
as we next deduce.
\begin{lemma}
	\label{lemma:single-move-output-leaving-c'}
	With notation for $c,c',S,S'$ as in \autoref{lemma:normalizer-is-hurwitz-module},
	let 
	$(m, (x = \{x_1, \ldots, x_n\}, 1, \gamma, \alpha = (\alpha_1,\ldots,
	\alpha_n, s))) \in \overline{Q}_\epsilon[
	\pi_0(\Hur^{c})[\alpha_{c}^{-1}]_+, \on{hur}^{N_c(c'),S'}]$.
	Viewing $\alpha$ as an element of $T_n$, suppose $\alpha_1, \ldots,
	\alpha_n \in N_c(c')$ but $\alpha$ is not in
	$N_c(c')^n \times T'_0 \subset T_n$.
	Then there is a single allowable move whose left output does not lie in
	$c'$.
\end{lemma}
\begin{figure}
\includegraphics[scale=.4]{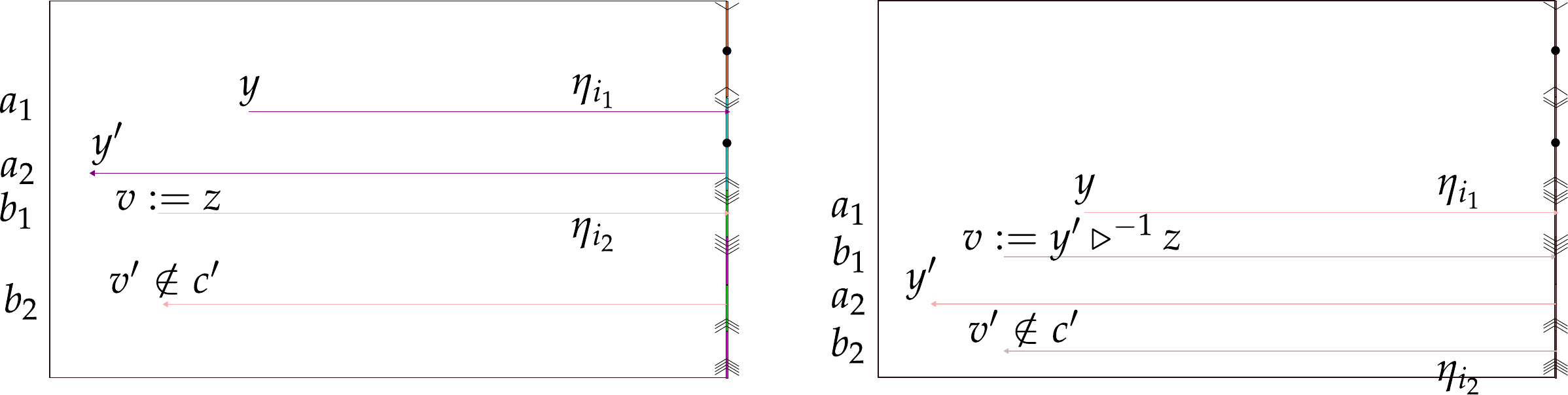}
\caption{
	This is a depiction of the proof of
	\autoref{lemma:single-move-output-leaving-c'}. The left hand side
	depicts the case that the first allowable move is above the second,
	while the right hand side depicts the case in which they overlap. In the
	first case $(z,\eta_{i_2})$ has left output not in $c'$
	while in the second case $(y' \triangleright^{-1} z,\eta_{i_2})$ has
	left output not in $c'$.
}
\label{figure:two-moves-in-one}
\end{figure}
\begin{proof}
	Using \autoref{lemma:output-leaving-c'},
	there is some sequence of allowable moves whose left output does not lie
	in $c'$.
	To conclude the proof, it suffices to show that if there is a sequence
	of two allowable moves $(y, \eta_{i_1}), (z, \eta_{i_2})$
	whose left output leaves $c'$ then the single move of the form $(v, \eta_{i_2})$
	already has left output not in $c'$.
	We will only analyze the case $i_2 < i_1$ (meaning that $\eta_{i_2}$ is
		below $\eta_{i_1}$, since the case $i_2 > i_1$ is
	similar. (We note here that it is also possible $i_2 = i_1$ in the way
		we have numbered things, but then it is also possible to
		slightly perturb the vertical coordinate of $\eta_{i_2}$ so
	that the paths $\eta_{i_1}$ and $\eta_{i_2}$ are disjoint. Hence we may assume $i_1 \neq i_2$.)

	Let the starting vertical coordinate of $\eta_{i_1}$ be $a_1$ and the
	ending vertical coordinate be $a_2 < a_1$. Similarly, let the starting
	vertical coordinate of $\eta_{i_2}$ be $b_1$ and the ending vertical
	coordinate be $b_2 < b_1$.
	In the case that $b_2 < b_1 < a_2 < a_1$ we take $v := z$ and otherwise
	(in which case $b_1 > a_2$)
	we take $v := y' \triangleright^{-1} z$, where $y'$ is the left output
	of the first move $(y,\eta_{i_1})$.
	See \autoref{figure:two-moves-in-one} for a visualization of these
	two cases.

	To prove our claim above, 
	we will construct two paths in the configuration space of $n+2$ points
	in $\Sigma^1_{g,f}$ which are homotopic.
	The initial points of these paths are obtained by first pulling
	$y$ a small distance $3\mu$ along $\eta_{i_1}$ from the boundary and then pulling
	$z$ a smaller distance $2\mu$ along $\eta_{i_2}$ from the boundary.
	The terminal points of these paths are obtained by passing $y$
	along $\eta_{i_1}$ until it reaches a distance $\mu$ from the boundary
	and then moving the second point initially labeled $z$ along $\eta_{i_2}$ until it reaches a distance
	$2\mu$ from the boundary.
	The first path $\gamma_1 = \epsilon_2 \circ \delta_1$ is given by
	applying $\delta_1$ which moves the first point along
	$\eta_{i_1}$ and then applying $\epsilon_2$ which moving the second point
	along $\eta_{i_2}$.
	The second path $\gamma_2 = \delta_2 \circ \epsilon_1$ is given by first
	applying $\epsilon_1$ which moves the second point
	along $\eta_{i_2}$ and then applying $\delta_2$ which moves the first point along $\eta_{i_1}$.
	Since $\gamma_1$ and $\gamma_2$ do not intersect, there is a homotopy
	between these two paths given by linearly changing
	the start time at which one moves the first point along $\eta_{i_1}$ and
	the second point along $\eta_{i_2}$, while maintaining their speeds.

	Now, we wish to show that in the above cases, the move $(v,
	\eta_{i_2})$ always has left output outside of $c'$.
	By construction of our path $\gamma_2$ above, we can identify the left output of this move
	with the label of the second point after applying $\epsilon_1$.
	First, suppose $b_2 < b_1 < a_2 < a_1$ so $v = z$.
	After applying $\delta_1$, the first point becomes $z \triangleright^{-1} y'$ while the second point remains $z$,
	and then after applying $\epsilon_2$ the second point becomes labeled
	$z' \notin c'$.
	On the other hand, if we first apply $\epsilon_1$, the label of the
	second point changes to some $v'$ and the label of the first point
	changes to some $y''$. After applying $\delta_2$, the label of the
	second point remains $v'$ and so we conclude $v' = z' \notin c'$ as
	desired.

	Next, we consider the case that $b_1 > a_2$.
	Recall in this case that we set $v = y' \triangleright^{-1} z$.
	First let us consider what happens after applying $\gamma_1 =
	\epsilon_2 \circ \delta_1$.
	In this case, after applying $\delta_1$, the first point changes to
	$y'$ which then passes below the second point and so changes the
	second point to $y' \triangleright v = z$. Applying $\epsilon_2$ then
	sends $z$ to $z' \notin c'$.
	On the other hand, let us examine what happens after applying 
	$\gamma_2 = \delta_2 \circ \epsilon_1$.
	After applying $\epsilon_1$, the first point becomes some $y''$ and
	the second point becomes some $v'$.
	We want to show $v'\notin c'$. However, after then applying $\delta_2$,
	the second point is unchanged, and also becomes $z'$ because $\gamma_1$
	is homotopic to $\gamma_2$. This implies $v' = z' \notin c'$, as
	desired.
\end{proof}

With the above set up, we can now prove our main technical result relating two
bar constructions, needed for proving homological stability of Hurwitz modules.
Recall the definition of normalizer of a rack from
\autoref{definition:normalizer}.
We note that although the statement and proof of the next result is very similar to that of \cite[Proposition
4.5.11]{landesmanL:the-stable-homology-of-non-splitting},
there was substantial subtlety in generalizing it to the setting of Hurwitz
modules,
which was primarily showed up in the earlier results of this section and
previous 
ones.

Recall as in \cite[Notation
4.5.8]{landesmanL:the-stable-homology-of-non-splitting} that given a subrack $c'
\subset c$ there is a map of $\EE_1$-algebras in pointed spaces
$\tilde{r}^{c'}_c: \Hur^c_+ \to \Hur^{c'}_+$ sending components not in
$\Hur^{c'}$ to the base point. We observe that if $(c',S') \subset (c,S)$ are
subsets in the sense of \autoref{definition:subset}, then there is a compatible
restriction map of modules $\Hur^{c,S}_+ \to \Hur^{c',S'}_+$.

\begin{proposition}
	\label{proposition:subrack-comparison}
	Retain notation for $c,c',S,S'$ as in
	\autoref{lemma:normalizer-is-hurwitz-module}.
	Then the natural restriction map
	\begin{align}
	\label{equation:s-to-s-prime-equivalence}
	\left( \pi_0 \Hur^{c'} \right)[\alpha_{c'}^{-1}]_+\otimes_{\Hur^c_+} \Hur^{c,S}_+ 
	\to \pi_0 \Hur^{c'}[\alpha_{c'}^{-1}]_+
\otimes_{\Hur^{N_c(c')}_+} \Hur^{N_c(c'),S'}_+ \end{align}
is a homology equivalence.
\end{proposition}
\begin{proof}
	The map \eqref{equation:s-to-s-prime-equivalence} has a section induced
	by the inclusions of racks $c' \subset N_c(c') \subset c$.
	It suffices to show this section induces a homology equivalence.
Let $S'$ be as in 
	\autoref{lemma:normalizer-is-hurwitz-module}.
	By \autoref{proposition:pointed-scanning}, and using notation from
	there,
	we can identify the map \eqref{equation:s-to-s-prime-equivalence} with a
	collection of maps indexed by $\epsilon$
\begin{align*}
	\overline{Q}^*_\epsilon [ 
			\pi_0 \Hur^{c'}
		[\alpha_{c'}^{-1}]_+,\on{hur}_+^{c,S}]
		\to
		\overline{Q}^*_\epsilon [ 
		\pi_0 \Hur^{c'}[\alpha_{c'}^{-1}]_+,
		\on{hur}_+^{N_c(c'),S'}].
	\end{align*}
	We now use the notation $\delta \in \mathbb R$ for the number defined in
	\autoref{notation:point-pushed-quotient}.
		In order to prove the section above is an equivalence, it suffices to
	show the inclusion
	\begin{align*}
		\iota_\epsilon : 
		\overline{Q}^*_\epsilon [ 
			\pi_0 \Hur^{c'}[\alpha_{c'}^{-1}]_+,
		\on{hur}_+^{N_c(c'),S'}]
		\to 	\overline{Q}^*_\epsilon [ 
			\pi_0 \Hur^{c'}[\alpha_{c'}^{-1}]_+,\on{hur}_+^{c,S}]
	\end{align*}
	is an ind-weak homology equivalence (as defined in 
		\cite[Definition
	A.4.5]{landesmanL:the-stable-homology-of-non-splitting})
	as $\epsilon$ approaches $0$ with $0 < \epsilon < \delta$.
	Let $M^{c,c'}_\epsilon$ denote the quotient of the inclusion
	$\iota_\epsilon$.
	By an argument similar to the proof of \cite[Lemma
	A.4.8]{landesmanL:the-stable-homology-of-non-splitting},
	$\iota_\epsilon$ has the homotopy extension property.
	In order to show $\iota_\epsilon$ is an ind-weak homology
	equivalence, it suffices to prove 
	\begin{align}
		\label{equation:contractible-condition}
	\text{ $M^{c,c'}_\epsilon$ is ind-weakly homology
equivalent to a point.}
	\end{align}
	Any point of $M^{c,c'}_\epsilon$ apart from the basepoint can be
	represented by a point of 
$\overline{Q}^*_\epsilon [ \pi_0 \Hur^{c'}[\alpha_{c'}^{-1}]_+,\on{hur}_+^{c,S}]$
of the form $(m, (x,t=1,\gamma,\alpha=(\alpha_1, \ldots,
	\alpha_n,s))$ for $m \in \pi_0 \Hur^{c'}[\alpha_{c'}^{-1}]$ and
	either 
some $\alpha_i \in c- N_c(c')$
or $s \notin T'_0$, for $T'_0$ the $0$-set of $S'$.

	We next define a filtration and show
	\eqref{equation:contractible-condition} by demonstrating it for each
	associated graded part of the filtration.
	More precisely, define the filtration $F_\bullet M_\epsilon^{c,c'}$
	on $M_\epsilon^{c,c'}$
	where for $j \geq 0$, 
	$F_j M_\epsilon^{c,c'}$ is the subset of $M_\epsilon^{c,c'}$
consisting of the base point together with the image of those points whose associated values of $j_1$ and
$j_2$ satisfy $j_1 + j_2
\leq j$, with $j_1, j_2$ defined as follows:
\begin{enumerate}
	\item Define $j_1$ to be the minimum value of $\mu$ so that, for $(\alpha_{1}, \ldots, \alpha_n,s) \in
T_{n}$, the $n$-set of $S$,
there is some allowable move of the form $(\beta, \eta_{\mu})$
with left output not in $c'$.
\item Let $y \in \{x_1, \ldots, x_n\}$.
	When $y$ is moved horizontally, suppose it hits the left boundary of $\mathcal
	M_{g,f,1}$ at $(0,u_y)$ and $(0,v_y)$ with $u_y > v_y$. Suppose that
	when $y$ is
	moved to hit the point $(0,v_y)$, it acts on the label of the left
	boundary by some 
	$w_y \in c$. Then
	$j_2$ is the number of $y$ so that $w_y \in c'$.
\end{enumerate}
We will explain later in the proof why this filtration is a filtration by cofibrations.
Let $G_j M^{c,c'}_\epsilon := F_j M^{c,c'}_\epsilon/ F_{j-1} M^{c,c'}_\epsilon$
denote the associated graded of the filtration. If the filtration is by cofibrations, it implies that the chains of the associated graded space is the associated graded of the chains. Since the filtration is finite in each degree, it suffices to additionally prove the following:
	\begin{align}
		\label{equation:graded-contractible-condition}
	\text{	For each $j \geq 0$, 
		$G_jM^{c,c'}_\epsilon$ is ind-weakly homology
equivalent to a point.}
	\end{align}
	For fixed $j \geq 0$ and $\epsilon > 0$, it then suffices to find
	some smaller $\epsilon'$ so that
	$G_j M_\epsilon^{c,c'} \to G_j M_{\epsilon'}^{c,c'}$ is
	nullhomotopic.

	Define $Q^{c,c'}_\epsilon$ as shorthand notation for $Q^*_\epsilon [\pi_0 \Hur^{c'}[\alpha_{c'}^{-1}]_+,\on{hur}_+^{c,S}]$, as defined in
	\autoref{notation:scanned-quotient}.
Let $\theta : Q^{c,c'}_\epsilon \to 
\overline{Q}^*_\epsilon [ \pi_0 \Hur^{c'}[\alpha_{c'}^{-1}]_+,\on{hur}_+^{c,S}]
\to M^{c,c'}_\epsilon$ denote the
composite projection.
Define a filtration $F_\bullet Q^{c,c'}_\epsilon := \theta^{-1} (F_\bullet
M_\epsilon^{c,c'})$ and define $G_\bullet Q^{c,c'}_\epsilon$ as the
associated graded.

Let $\epsilon' := \epsilon/2$.
(This choice of $\epsilon'$ is coming from the fact that allowable paths pass
halfway between the vertical coordinates of any two points in the union of the relevant
configuration with $W$.)
Choose some $j \geq 0$. We next construct a continuous homotopy $H: F_j
Q^{c,c'}_\epsilon \times I \to G_j M^{c,c'}_{\epsilon'}$.
In order to define this homotopy, we begin by choosing a fixed ordering of the
elements of $c'$.
For the subset of $(m,y) \in Q^{c,c'}_\epsilon$ where either $m$ is the base
point or $y \in F_{j-1}Q^{c,c'}_\epsilon$, 
the image $\theta(m,y)$ is the base point and we choose the constant homotopy at
the base point. That is, for such $(m,y)$ we take $H ( (m,y), t) := H( (m,y),
0) = \theta(m,y)$.
It remains to define this homotopy for points of the form 
$( (m,y), t)$ where $m$ is not the base point and $y \in F_j
Q^{c,c'}_\epsilon - F_{j-1} Q^{c,c'}_\epsilon$.
For such a point $(m,y)$, we define 
we define the homotopy as follows:
by definition of the filtration and
\autoref{lemma:single-move-output-leaving-c'},
there is some allowable move of the form $(\beta,
\eta_{j_1})$ with left output in $c - c'$.
We choose the allowable move as above where $\beta$ appears earliest with
respect to the ordering on $c'$ we chose above.
We take the homotopy that performs this allowable move at constant speed.
At time $t = 0$, note that $H$ is given by the composite
$F_j Q^{c,c'}_\epsilon \to G_j M^{c,c'}_{\epsilon} \to G_j M^{c,c'}_{\epsilon'}$.
It therefore suffices to show that $H$ descends to a continuous map
$\overline H : G_j M^{c,c'}_\epsilon \times I \to G_j M^{c,c'}_{\epsilon'}$ which is the
constant map to the base point when $t = 1$,
as this will then imply $G_j M^{c,c'}_{\epsilon} \to G_j M^{c,c'}_{\epsilon'}$
is nullhomotopic.
The latter condition that 
$\overline{H}$ is the constant map to the basepoint when $t = 1$
holds because 
the definition of the filtration guarantees that the left output of the
allowable move $(\beta, \eta_{j_1})$ is in $c - c'$.
This means
that at the end of the homotopy, it is identified with the base
point in
$\overline{Q}^*_\epsilon [\pi_0 \Hur^{c'}[\alpha_{c'}^{-1}]_+,\on{hur}_+^{c,S}]$,
hence in 
$G_j M^{c,c'}_{\epsilon'}$.

Hence, it remains to show that $H$ descends to a continuous map 
$\overline{H}: G_j M^{c,c'}_\epsilon \times I \to G_j M^{c,c'}_{\epsilon'}$
and that
$F_\bullet M_\epsilon^{c,c'}$ is a filtration by cofibrations.
Note that $H$ is indeed compatible with the relation sending 
$F_{j-1} Q^{c,c'}_\epsilon$ to the base
point by construction. 
So, to check 
the map $H$ descends, we only need to verify it is compatible with the
relation from \autoref{notation:scanned-quotient} (see also
\autoref{notation:q})
defining 
$\overline{Q}^*_\epsilon [ \pi_0 \Hur^{c'}[\alpha_{c'}^{-1}]_+,\on{hur}_+^{c,S}]$
	as a quotient of
	$Q^{c,c'}_{\epsilon}$.

Next, we verify that the filtration 
$F_{j} M_{\epsilon}^{c,c'} \subset M_{\epsilon}^{c,c'}$ is by cofibrations. If
it is a closed subset, then it is easy to see that it is a cofibration as it is
a filtration by sub-CW-complexes.
To see
this filtration is closed, it suffices to check its preimage in 
$Q_{\epsilon}^{c,c'}$ 
is closed.
Equivalently, we wish to check that when we apply the equivalence relation used
to define 
$\overline{Q}^*_\epsilon [ \pi_0 \Hur^{c'}[\alpha_{c'}^{-1}]_+,\on{hur}_+^{c,S}]$
	as a quotient of
$Q_{\epsilon}^{c,c'}$,
a point in $F_jQ_{\epsilon}^{c,c'}$ is sent to another point in $F_j
Q_{\epsilon}^{c,c'}$.
We now suppose some point $x_1$ 
in the configuration
$(m,(x,t=1,\gamma,\alpha))$
is on a path so that if it moves horizontally it hits the left
boundary at $(0,u)$ and $(0,v)$ with $u > v$. Applying the equivalence relation
\autoref{notation:q}, $x_1$ is absorbed into the
boundary, and the resulting point is either
of the form 
$(m',(x',t',\gamma',\alpha'))$ or the basepoint.
We check that values of $j_1$ and $j_2$ associated to this new configuration are
at most their values associated to the previous configuration.
This will show the filtration is closed.
First, if $x_1$ hits the boundary and acts by some element of $c - c'$, the
new configuration is the base point, which lies in every step of the filtration by
assumption.
Hence, we may assume that $x_1$ acts on the boundary by some element of $c'$.

First, we argue that the value of $j_2$ decreases when $x_1$ hits the boundary, and it strictly decreases if
$x_1$ hits the boundary at $(0,v)$.
Assume that $x_1$ acts on the left boundary by an element $w \in c'$.
In this case, suppose $y \in \{x_2, \ldots, x_n\}$ is some other point that 
acts on the left boundary by $w_y$ as in the definition of the value of
$j_2$. Then, after $x_1$ hits the boundary at some point $(0,h)$ with $h$ either
$u$ or $v$, the value of $w_y$ associated to
$y$ in $(m',(x',t',\gamma',\alpha'))$
will still be
$w_y$
if $v_y < h$ and it will become $w \triangleright w_y$ if $v_y > h$.
Since $w \in c'$, $w \triangleright w_y \in c'$ if and only if $w_y \in c'$.
Hence, the value of $j_2$ associated to
$(m',(x',t',\gamma',\alpha'))$ 
is bounded above by the value associated to 
$(m,(x,t=1,\gamma,\alpha))$, and it is strictly smaller $h = v$.

Let $(a,0)$ denote the starting point of $\eta_{j_1}$ and $(b,0)$ denote its ending
point, so $a > b$.
Next, we claim that the value of $j_1$ decreases when $x_1$ hits the boundary,
and it strictly decreases if
$u < a.$
Again, we may assume $x_1$ acts by an element $w \in c'$, as if it acts by an
element in $c - c'$, the configuration is sent to the base point.
Let $j'_1$ denote $j_1$ if $u < a$ and let $j_1 - 1$ if $u < a$.
To demonstrate the above claim, it suffices to show that after $x_1$ collides
with the boundary, there is some $\beta' \in
c'$ so that the allowable move $(\beta',\eta_{j'_1})$ has left output in $c - c'$.
Up to homotopy, $\eta_{j'_1}$ starts at $a$ and ends at $b$, so we may assume it has the same
starting and ending points as $\eta_{j_1}$.
Suppose $x_1$ collides with the left boundary at some point $(0,h)$, with $h$
either $u$ or $v$.
Let $z \in c-c'$ denote  the left output of 
the allowable move
$(\beta, \eta_{j_1})$ for the original element
$(m,(x,t=1,\gamma,\alpha))$. 
If $h > a$, then the left output of $(\beta, \eta_{j'_1})$ for 
$(m',(x',t',\gamma',\alpha'))$ is also $z \in c-c'$.
To conclude, it remains to deal with the case $a > h$.
In this case, we claim that we can take $\beta' := w \triangleright\beta$.
When
$h > b$, we see the left output for
$(\beta',\eta_{j'_1})$ in 
$(m',(x',t',\gamma',\alpha'))$ is also $z$, using that
$\beta \cdot w = w \cdot \beta'$. 
Finally, if $h < b$, the left output for 
$(\beta',\eta_{j'_1})$ in 
$(m',(x',t',\gamma',\alpha'))$ 
is $w \triangleright^{-1} \beta$ since this satisfies
$(w \triangleright^{-1} \beta) w = \beta w.$
Note that $w \triangleright^{-1} \beta \in c - c'$ since $\beta \in c-c'$ and $w
\in c'$.
This shows that 
the filtration
$F_{j} M_{\epsilon}^{c,c'} \subset M_{\epsilon}^{c,c'}$ is by cofibrations. 

To conclude, it remains to show our map $H$ descends to $\overline{H}$, by
showing it is compatible with the equivalence relation from 
\autoref{notation:q}.
We consider the three cases that we apply the equivalence relation from
\autoref{notation:q},
where the point $x_1$ hits the left boundary.
To set up notation, we continue to assume $\eta_{j_1}$ meets the left boundary at the
points $(0,a)$ and
$(0,b)$ with $a > b$ and $x_1$ hits the boundary at $(0,u)$ and $(0,v)$ with $u
> v$.
We may assume that the action of $x_1$ on the left boundary is via an element of
$c'$, as if $x_1$ acts by some element of $c - c'$, the configuration will be
sent to the base point and the homotopy $H$ will be compatible with such
equivalences.


\begin{figure}
\includegraphics[scale=.4]{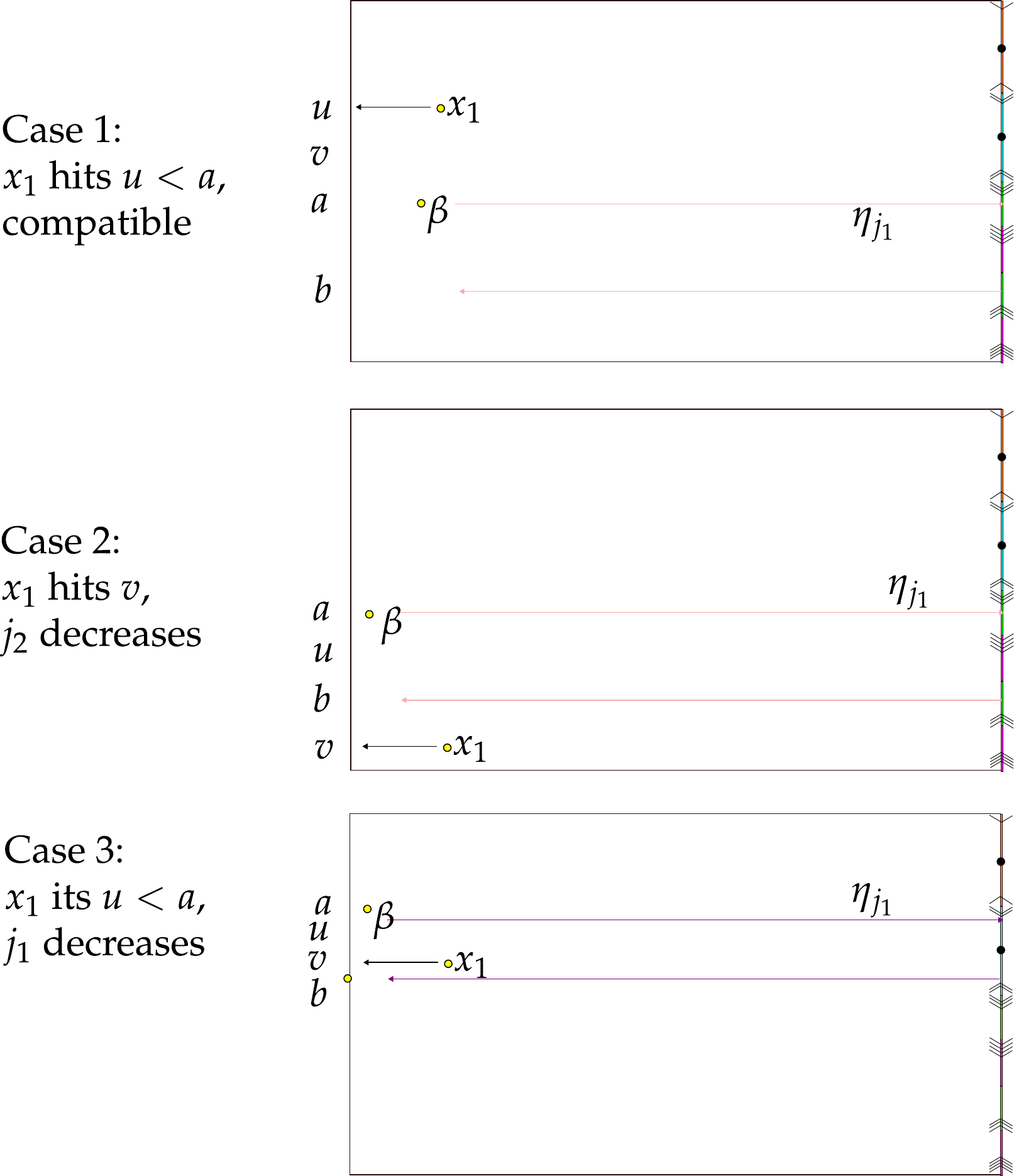}
\caption{
	This is a picture of the nullhomotopy $H$ in each of the three cases
	that $x_1$ hits the left boundary at $u < a$, $x_1$ hits the left
	boundary at $v$, and $x_1$ hits the left boundary at $u > a$.
	In the first case, the homotopy is compatible with $x_1$ hitting the
	boundary, while in the latter two cases, the filtration decreases.
}
\label{figure:stability-nullhomotopy}
\end{figure}

The remainder of the proof is divided into three cases which are visualized in
\autoref{figure:stability-nullhomotopy}.
First, we consider the case that $x_1$ hits the boundary at $(0,u)$ with $u >
a$.
In this case,
because $(0,u)$ lies completely above the path $\eta_{j_1}$,
the left output of the allowable move $(\beta,\eta_{j_1})$
will be the same before and after applying the equivalence relation from
\autoref{notation:q}, associated to $x_1$ hitting the left boundary at $(0,u)$.
Hence, the homotopy $H$ will be compatible with such
an equivalence.

Second, we consider the case that $x_1$ collides with the boundary at $(0,v)$.
As mentioned above, we may assume that $x_1$ acts on the left boundary by an element $w \in c'$.
In this case, we saw that the value of $j_2$ decreases by $1$ when we were
checking above that the filtration is closed. Since we also saw the value of
$j_1$ does not increase, the 
resulting point lies in 
$F_{j-1} Q^{c,c'}_\epsilon$
and hence is identified with the base point in the quotient
$F_j Q^{c,c'}_\epsilon/F_{j-1} Q^{c,c'}_\epsilon$, implying $H$ is compatible
with this equivalence.

To finish showing $H$ descends to $\overline{H}$, we only need to deal with the case that 
$x_1$ hits the boundary at $(0,u)$ with $a > u$ and acts via some element of
$c'$.
As we demonstrated when we were showing above that the filtration was closed,
since $a > u$, the value of $j_1$ strictly decreases.
This again implies that, when $x_1$ hits the boundary, the point is sent to
$F_{j-1} Q^{c,c'}_\epsilon$
and hence is identified with the base point in the quotient
$F_j Q^{c,c'}_\epsilon/F_{j-1} Q^{c,c'}_\epsilon$.
Therefore, the homotopy $H$ is again compatible with this equivalence,
completing the proof.
\end{proof}

For proving the homology of Hurwitz modules stabilizes, it will also be useful to
have the following $n$-fold tensor product version of the result of
\autoref{proposition:subrack-comparison}, which was the $2$-fold version.
\begin{lemma}	\label{lemma:n-fold-comparison}
	Retain notation for $c,c',S,S'$ as in
	\autoref{lemma:normalizer-is-hurwitz-module}.
	For every $n \geq 1$, there is a homology equivalence
	\begin{equation}
		\label{equation:n-fold-equivalence}
		((\pi_0\Hur^{c'})[\alpha_{c'}^{-1}]_+)^{\otimes_{\Hur^c_+} n}
	\otimes_{\Hur^c_+} \Hur^{c,S}_+ 
	\xrightarrow{\simeq}  ((\pi_0\Hur^{c'})[\alpha_{c'}^{-1}]_+)^{\otimes_{\Hur^{N(c')}_+} n}\otimes_{\Hur^{N_c(c')}} \Hur^{N_c(c'),S'}_+.
\end{equation}
\end{lemma}
\begin{proof}
	The case $n = 1$ is the content of
	\autoref{proposition:subrack-comparison}.
	To prove the case that $n \geq 1$, note that there is a homology equivalence
	\begin{equation}
		\label{equation:n-fold-ring-equivalence}
((\pi_0\Hur^{c'})[\alpha_{c'}^{-1}]_+)^{\otimes_{\Hur^c_+} n} \xrightarrow{\simeq}
((\pi_0\Hur^{c'})[\alpha_{c'}^{-1}]_+)^{\otimes_{\Hur^{N(c')}_+} n}.
\end{equation}
This homology equivalence for $n = 2$ 
was shown in \cite[Proposition
4.5.11]{landesmanL:the-stable-homology-of-non-splitting}.
To prove
\eqref{equation:n-fold-ring-equivalence} in general, by induction, we may assume
it holds for $n - 1$, so we obtain the homology equivalences
\begin{align*}
&((\pi_0\Hur^{c'})[\alpha_{c'}^{-1}]_+)^{\otimes_{\Hur^c_+} n-1} 
\otimes_{\Hur^c_+}
((\pi_0\Hur^{c'})[\alpha_{c'}^{-1}]_+) \\
&\to
((\pi_0\Hur^{c'})[\alpha_{c'}^{-1}]_+)^{\otimes_{\Hur^{N(c')}_+} n-1} 
\otimes_{\Hur^c_+}
((\pi_0\Hur^{c'})[\alpha_{c'}^{-1}]_+)
\\
&\to
((\pi_0\Hur^{c'})[\alpha_{c'}^{-1}]_+)^{\otimes_{\Hur^{N(c')}_+} n-1} 
\otimes_{\Hur^{N(c')}_+}
((\pi_0\Hur^{c'})[\alpha_{c'}^{-1}]_+)
\end{align*}
where the last homology equivalence uses 
\cite[Proposition
4.5.11]{landesmanL:the-stable-homology-of-non-splitting}
again.
Tensoring the homology equivalence \eqref{equation:n-fold-ring-equivalence} over
$\pi_0 \Hur^{c'}[\alpha_{c'}^{-1}]_+$
with the homology equivalence 
\eqref{equation:s-to-s-prime-equivalence}
yields the desired homology equivalence
\eqref{equation:n-fold-equivalence}.
\end{proof}

\section{Proving homological stability}
\label{section:homological-stability}

In this section we prove that the homology of Hurwitz modules stabilize in a
linear range.
The main result of this section is \autoref{theorem:stable-homology}, which immediately implies
\autoref{theorem:homology-stabilizes-intro} from the introduction.
The first step to proving our homological stability result is the relate the
chains on a quotient of $\Hur^{c,S}$ to the chains on a quotient of
$\Hur^{N_c(c'),S'}$, which uses our identification of bar constructions from
\autoref{lemma:n-fold-comparison}, the output of the previous section, as input
for a descent argument.


\begin{lemma}
	\label{lemma:comparison-to-normalizer}
	Let $S$ be a bijective Hurwitz module over a finite rack $c$, $c' \subset c$ be a subrack.
	Let $c' \subset c$ be a subrack, let
	$S_{c'}$ be as in \autoref{definition:subsystem-for-subrack} and assume
	it has $0$ set $T'_0$. Take $(S',N_c(c')) \subset (S,c')$ to be the
	subset with $n$-set $N_c(c')^n \times T'_0$, which is a bijective Hurwitz
	module by 
\autoref{lemma:normalizer-is-hurwitz-module}.
	Using the notation $A_{c,S} := C_*(\Hur^{c,S}; \mathbb Z)$,
	the restriction map induces an equivalence
\begin{align}
	\label{equation:normalizer-comparison}
	f_{S,S'} : 
\left(A_{c,S}/\alpha_x^{\ord(x)}, x \in
	c-c')\right)[\alpha_{c'}^{-1}]
	\simeq
	\left(A_{N_c(c'),S'}/(\alpha_x^{\ord(x)}, x \in
	c-c')\right)[\alpha_{c'}^{-1}]
\end{align}
\end{lemma}
\begin{proof}
	Consider $c(1) := c, c(2) := N_c(c')$ and for $i \in \{1,2\}$ 
let
$R_i := C_*(\Hur^{c(i)})[\alpha_x^{-1},x \in c']$.
Let 
$R' := H_0(\Hur^{c'})[\alpha_x^{-1},x \in c']$.
Let $f_i: R_i \to R'$ be the map induced by the restriction map.

Let $I_i \subset R_i$ denote the ideal
generated by $\alpha_x$ for $x \in c-c'$. (In the case $i = 2$, so $c(i) = N_c(c')$, the
elements in $c - N_c(c')$ act by $0$.) 
Let $S_1 := S$ and $S_2 := S'$.
For $i \in \{1,2\}$, define the left $R$ module $M_i := C_* \left(\Hur^{c(i),S_i}/(\alpha_x^{\ord(x)}, x \in
c-c')\right)[\alpha_{c'}^{-1}]$.

We claim now that for a fixed $i \in \{1,2\}$, $I_i$ acts nilpotently on $\pi_j(M_i)$ for each $j$.
To see this, first note that it follows from \cite[Lemma 3.5.1 and Lemma
3.5.2]{landesmanL:the-stable-homology-of-non-splitting}
that each $\alpha_x$ for $x \in c - c'$ acts nilpotently on $\pi_j(M_i)$ for each
$j$.
A general element of $I_i$ can be written as $w = \sum_{x \in c - c'} y_x \alpha_x$
for some $y_x \in R_i$. We wish to show a product $w_1*\dots * w_N$ with $w_j
\in I_i, 1 \leq j \leq n$ acts by $0$ for $N\gg0$.
Note that for any
$y \in R_i$, we have $y \alpha_x =
\alpha_x\phi_x(y)$, where $\phi_x$ is induced by the automorphism $c(i) \to c(i), u
\mapsto x \triangleright u$.
Using the above and the pigeonhole principle we find that for any $t > 0$ there
is some $N$ so that $w_1*\dots*w_N$ is in the left ideal generated by $\{\alpha_x^t, x \in
c - c'\}$, proving the desired claim because each $\alpha_x$ acts nilpotently.

Since $I_i$ acts nilpotently on each $\pi_j(M_i)$, it follows from \cite[Lemma
4.0.4]{landesmanL:homological-stability-for-hurwitz}
that $M_i$ is $I_i$-nilpotent complete in the sense of \cite[Definition
4.0.1]{landesmanL:homological-stability-for-hurwitz}.
To prove the desired equivalence
\eqref{equation:normalizer-comparison}, as $M_i$ is $I_i$-nilpotent complete,
it suffices to prove compatible equivalences $R'^{\otimes_{R_1} n} \otimes_{R_1}
M_1 \simeq R'^{\otimes_{R_2} n} \otimes_{R_2} M_2$ for each $n \geq 1$.
This follows from \autoref{lemma:n-fold-comparison} upon applying
reduced chains to
\eqref{equation:n-fold-equivalence}
and quotienting by $(\alpha_x^{\ord(x)}, x \in c - c')$.
\end{proof}

We will now next put a filtration on $A_{c,S}$ so as to isolate the ``connected
part'' which is the analog of $C_*(\CHur^c)$ of chains on connected covers,
where the labels of the points generate $c$.
Let us explain the idea for where we are going next.
Once we define the filtration, \autoref{lemma:comparison-to-normalizer} will enable us to show
that the connected part associated to $A_{c,S}$ in
\eqref{equation:normalizer-comparison} is identified with the top graded part
for $A_{N_c(c'),S'}$. Since the latter vanishes, the former does as well, which
enables us to show this connected part vanishes, which means each $\alpha_{c'}$
acts invertibly and so we can remove it and still obtain something that
stabilizes.

\begin{construction}
	\label{construction:S-filtration}
	Given a finite rack $c$ and a finite bijective Hurwitz module $S$ over $c$, we put a doubly filtered
	structure on 
	$\Hur^{c,S}$. 
	We define $F_{*,*} \Hur^{c,S} : \mathbb N^2 \to \Mod_{\Hur^c}(\Spc^{\NN})$
	as follows.

	Suppose $c'' \subset c$ and $S''$ is a bijective Hurwitz module over $c''$
	which is a subset of the bijective Hurwitz module $S$
	in the sense of \autoref{definition:subset}.

	We then define 
	the $(i,j)$th part of the bifiltration
	$F_{i,j} \Hur^{c,S}$ to be the union
	of all components contained in some 
	$\Hur^{c'',S''}$ for $(c'',S'') \subset (c,S)$ with
	$|c''| \geq i$ and $|T''_0| \geq j$ for $T''_0$ the $0$-set of
	$S''$.

	We use $A_{c,S} := C_*(\Hur^{c,S}; \mathbb Z)$.
	We use $F_{*,*} A_{c,S}$ to denote the associated functor $\mathbb N^2 \to
	\Mod_{A_c}(\Mod(\ZZ)^{\NN})$ 
	obtained from $F_{*,*} \Hur^{c,S}$ by taking chains.
	We will also view $F_{*,*} A_{c,S}$ as giving a bifiltration on $A_{c,S}$ as an $A_c$
	module.
	If $T_0$ is the $0$-set of $S$,
	define $C A_{c,S} := F_{|c|,|T_0|} A_{c,S}$.
\end{construction}

The following lemma is immediate from \autoref{construction:S-filtration}.
\begin{lemma}
	\label{lemma:graded-filtration}
	Let $c$ be a finite rack. There is a natural isomorphism of bigraded left $A_c$ modules
	\begin{align*}
		\gr_{i,j} A_{c,S} \simeq
			\oplus_{(c'',S'') \subset (c,S), |c'| = i, |T''_0| = j} C
	A_{c'', S''},
	\end{align*}
	where above $T''_0$ is the $0$-set of $S''$,
	the sum is taken over all subsets $(S'', c'') \subset (S,c)$ in the sense of \autoref{definition:subset}, and 
	each $C A_{c'',S''}$, as defined in \autoref{construction:S-filtration}
	is given the structure of an $A_c$ module by letting elements of $c -
	c'$ act by $0$.
\end{lemma}

We next aim to show the stable homology of $CA_{c,S}/(\alpha_x^{\on{ord}(x)}, x \in
	c-c')[\alpha_x^{-1}, x \in c']$ vanishes if $c'$ is not a union of
	$S$-components of $c$.
	We will need the following two elementary lemmas.
This first lemma was proven in the final paragraph of \cite[Theorem
5.0.6]{landesmanL:homological-stability-for-hurwitz}.
\begin{lemma}
	\label{lemma:normalizer-containment}
	Suppose $c' \subset c$ is a subrack which is not a union of components
	of $c$. Then $N_c(c') \neq c$.
\end{lemma}
\begin{proof}
	By assumption, there is some component $c'' \subset c$ not contained in $c'$ but which
	meets $c'$. Hence there is some $x \in c'' \cap c'$ and some $y$ with $y
	\triangleright x \notin c'$. Therefore, $y \notin N_c(c')$.
\end{proof}

\begin{lemma}
	\label{lemma:s-prime-containment}
	Suppose $c' \subset c$ is a subrack which is not a union of
	$S$-components of $c$ and let $(N_c(c'),S') \subset(c,S)$ be the associated subset as in
	\autoref{lemma:normalizer-is-hurwitz-module}.
	Then we cannot have equality $N_c(c') = c$ and $S' = S$ as bijective
	Hurwitz	modules.
\end{lemma}
\begin{proof}
	By \autoref{lemma:normalizer-containment}, we must have that $c' \subset
	c$ is a union of components of $c$. 
	Suppose $T'_0$ is the $0$ set of $S'$ and $T_0$ is the $0$ set of $S$.
	By definition of the $S$-components of $c$,
	there must be some $t \in T_0$, $x \in c'$, and $\gamma \in B_1^{\Sigma^1_{g,f}}$ so
	that $\sigma^\gamma_t(x) \notin c'$.
	Therefore, $t \notin T'_0$ and so $T'_0 \neq T_0$ and hence
	$(S',c') \subsetneq (S,c)$.
\end{proof}

We are now prepared to show the stable value of 
$CA_{c,S}/(\alpha_x^{\on{ord}(x)}, x \in
c-c')[\alpha_x^{-1}, x \in c']$ vanishes.
This will enable us to remove one of the $\alpha_x^{\on{ord}(x)}$ in the quotient and
proceed inductively.

\begin{lemma}
	\label{lemma:quotient-localization-vanishes}
	Suppose $c$ is a finite rack,
	$S$ is a finite
	bijective Hurwitz module over $c$,
	and $c' \subset c$ is a subrack that is not a union of
$S$-components of $c$.
Then
	\begin{align}
		\label{equation:connected-vanishing}
		CA_{c,S}/(\alpha_x^{\on{ord(x)}}, x \in c-c')[\alpha_x^{-1}, x \in
		c'] = 0.
	\end{align}
\end{lemma}
\begin{proof}
	We prove our result by induction on $|c|$ and $|S|$.
	The map of \autoref{lemma:comparison-to-normalizer} is an equivalence, and its top associated bigraded piece is the map $CA_{c,S}/(\alpha_x^{\on{ord(x)}}, x \in c-c')[\alpha_x^{-1}, x \in
	c'] \to 0$. It thus suffices to show that all of the associated graded pieces $\gr_{i,j}f_{S,S'}$ with either $i<|c|$ or $j<|S|$ is an equivalence.
	
	Note that all summands in these associated graded terms match up on the
	source and target except for those where either $c''$ strictly contains
	$N_c(c')$ or $S''$ strictly contains $S'$. In this case, the
	contrapositive of
\autoref{lemma:s-prime-containment} implies that $c'$ is not a union of $S''$-components in $c''$. 
Therefore, applying the induction hypothesis to $c'\subset c''$, we find that $CA_{c'',S}/(\alpha_x^{\on{ord(x)}}, x \in c''-c')[\alpha_x^{-1}, x \in
	c']=0$. 
	Thus $\gr_{i,j}f_{S,S'}$ with either $i<|c|$ or $j<|S|$ are equivalences,
	and so 
	$\gr_{|c|,|S|}f_{S,S'}$ is as well, implying 
	\eqref{equation:connected-vanishing} holds.
\end{proof}

Using the vanishing established in
\autoref{lemma:quotient-localization-vanishes}, we can now inductively remove
elements from the quotient, to show the homology of Hurwitz modules stabilize.
The input for the base case comes from the stability of the quotient from
\autoref{theorem:all-quotient-bounded}.

\begin{lemma}
	\label{lemma:full-quotient-bounded}
	Let $c$ be a finite rack and $S$ be a finite bijective Hurwitz module over $c$.
	For any subset $V \subset c$ which contains some element of each
	$S$-component of $c$, $CA_{c,S}/(\alpha^{\ord(x)}_x, x \in V)$ is
	$f_{\mu(|z|,\ord_c(z)),b(|z|,\ord_c(z))}$ bounded with respect to the
	grading induced by $z \subset c$, where
	the functions 
	$\mu(|z|,\ord_c(z))$, and $b(|z|,\ord_c(z))$
	depend only on $|z|$ and $\ord_c(z)$.
\end{lemma}
\begin{proof}
	The proof will be by descending induction on $|V|$.
	First, recall that from \autoref{construction:S-filtration} and
	\autoref{lemma:graded-filtration}, $A_c$ has a finite 
	bifiltration with
	$\gr_{i,j} A_{c,S} \simeq
	\oplus_{(c'',S'') \subset (c,S), |c''| = i,|T''_0|=j} C
	A_{c'', S''},$
	where the sum is taken over all subracks $c'' \subset c$ and bijective
	Hurwitz modules $S''$ over $c''$ which are subsets of the bijective Hurwitz module $S$
	over $c$ so that $|c''| = i$ and the $0$ set $T''_0$ of $S''$ has
	$|T''_0| = j$.
	In the case $|V| = |c|$, we must have $V = c$, in which case
	\autoref{theorem:all-quotient-bounded} implies 
	$A_{c,S}/(\alpha^{\ord(x)}_x, x \in V)$ is 
	$f_{\mu^0(|z|,\ord_c(z)),b^0(|z|,\ord_c(z))}$ bounded,
	for $\mu^0,b^0 : \mathbb N^2 \to \mathbb N$ two functions. 
	Inducting on the size of $c$ and $T_0$, 
	we claim
	the associated graded pieces
	$\gr_{i,j} A_{c,S}/(\alpha^{\ord(x)}_x, x \in c)$
	are then
$f_{\mu^1(|z|,\ord_c(z)),b^1(|z|,\ord_c(z))}$ bounded
for $i < |c|$ or $j < |T_0|$, where
\begin{align*}
	\mu^1(s,t) &:= \max(t, \max_{s'\leq s, t'\leq t} \mu^0(s',t')) \\
b^1(s,t) &:= \max_{s'\leq s, t' \leq t} (b^0(s',t') + st+ \mu^1(s,t)).
\end{align*}
Indeed, by 
\autoref{lemma:graded-filtration},
the associated graded pieces are of the form
$CA_{c'',S''}/(\alpha_x^{\on{ord}(x)}, x \in c)$.
By induction, we may assume 
$CA_{c'',S''}/(\alpha_x^{\on{ord(x)}},x \in c'')$ are 
$f_{\mu^0(|z|,\ord_c(z)),b^0(|z|,\ord_c(z))}$ bounded.
If we assume $|c''| = s$ and $\ord_{c''}(z) = t$
we see that 
$CA_{c'',S''}/(\alpha_x^{\on{ord(x)}},x \in c)$
is a quotient of
$CA_{c'',S''}/(\alpha_x^{\on{ord(x)}},x \in c'')$
by elements of $c - z$ which act by $0$ and at most $|z|$ additional elements
$y_i \in z$, living in bidegree
$(\ord(y_i),1)$ with $\ord(y_i) \leq \ord_c(z) = t$.
It follows that
$CA_{c'',S''}/(\alpha_x^{\on{ord(x)}},x \in c)$
is 
$f_{\max(\ord_c(z), \mu^0(|c''|,\ord_c(c''))),b^0(|z|,\ord_c(z))+ |z| \ord_c(z)}$
bounded.
This implies the claim that 
$\gr_{i,j} A_{c,S}/(\alpha^{\ord(x)}_x, x \in c)$
is
$f_{\mu^1(|z|,\ord_c(z)),b^1(|z|,\ord_c(z))-\mu^1(|z|,\ord_c(z))}$ bounded.
Now, the cofiber $Q$ of the map 
\begin{align}
	\label{equation:connected-to-all-cofiber}
CA_{c,S}/(\alpha^{\ord(x)}_x, x \in c) \to A_{c,S}/(\alpha^{\ord(x)}_x, x \in c)
\end{align}
is filtered by 
the associated graded pieces of the bifiltration $F_{i,j}$, except $CA_{c,S}$, which are 
$f_{\mu^1(|z|,\ord_c(z)),b^1(|z|,\ord_c(z)) -  \mu^1(|z|,\ord_c(z))}$ bounded.
Therefore, the $-1$ suspension, $\Sigma^{-1} Q$,
is the fiber of of \eqref{equation:connected-to-all-cofiber}.
Since, $Q$ is 
$f_{\mu^1(|z|,\ord_c(z)),b^1(|z|,\ord_c(z)) -  \mu^1(|z|,\ord_c(z))}$ bounded,
we find
$\Sigma^{-1} Q$ is 
$f_{\mu^1(|z|,\ord_c(z)),b^1(|z|,\ord_c(z))}$ bounded.
As $A_{c,S}/(\alpha^{\ord(x)}_x, x \in c)$ is also
$f_{\mu^1(|z|,\ord_c(z)),b^1(|z|,\ord_c(z))}$ 
bounded we obtain that 
$CA_{c,S}/(\alpha^{\ord(x)}_x, x \in c)$,
is
$f_{\mu^1(|z|,\ord_c(z)),b^1(|z|,\ord_c(z))}$ 
bounded
as well.

	Having established the base case that $V = c$, we next suppose that
	$CA_{c,S}/(\alpha^{\ord(x)}_x, x \in V')$ is 
$f_{\mu^1(|z|,\ord_c(z)),b^1(|z|,\ord_c(z))+ |(c - V')\cap z| \cdot \mu^1(|z|,\ord_c(z))}$ bounded
	for all $V'$ with $|V'| > |V|$ and verify that 
$CA_{c,S}/(\alpha^{\ord(x)}_x, x \in V)$ is 
$f_{\mu^1(|z|,\ord_c(z)),b^1(|z|,\ord_c(z))+|(c - V')\cap z| \cdot \mu^1(|z|,\ord_c(z))}$ bounded.
	By \cite[Lemma 5.0.1]{landesmanL:homological-stability-for-hurwitz}
	(which we use to remove elements in $z$ from the quotient), 
	and \cite[Lemma 5.0.2]{landesmanL:homological-stability-for-hurwitz}
	(which we use to remove elements in $c - z$ from the quotient), 
	it
	suffices to show 
	$CA_{c,S}/(\alpha^{\ord(x)}_x, x \in V)[\alpha_y^{-1}] = 0$ for each $y
	\in c- V$.
	Once we establish this, we will conclude by taking
	\begin{align*}
		\mu(|z|, \ord_c(z)) &:= \mu^1(|z|,\ord_c(z)) \\
		b(|z|,\ord_c(z) &:=
	b^1(|z|,\ord_c(z))+|z| \cdot \mu^1(|z|,\ord_c(z)).
	\end{align*}

	By induction on $|V|$ and on $|T_0|$, we claim $CA_{c,S}/(\alpha^{\ord(x)}_x, x \in
	V)[\alpha_y^{-1}]/(\alpha_w^{\ord(w)}) = 0$ for each $w \in c - V - y$.
	We know $CA_{c,S}/(\alpha^{\ord(x)}_x, x \in
V \cup \{w\})[\alpha_y^{-1}] = 0$ by induction so now explain why 
$CA_{c,S}/(\alpha^{\ord(x)}_x, x \in
	V)[\alpha_y^{-1}]/(\alpha_w^{\ord(w)}) = CA_{c,S}/(\alpha^{\ord(x)}_x, x \in
V \cup \{w\})[\alpha_y^{-1}]$
	This holds because inverting $\alpha_y$ commutes with tensoring and quotients by
	$\alpha_x^{\on{ord}(x)}$ by 
\cite[Lemma 3.4.4]{landesmanL:the-stable-homology-of-non-splitting}, which
applies as $\alpha_x^{\on{ord}(x)}$ is $\mathbb E_2$ central; 
here, \cite[Lemma 3.4.4]{landesmanL:the-stable-homology-of-non-splitting}
applies because $\alpha_x^{\ord(x)}$ is $\EE_2$-central (\cite[Lemma
	3.2.3]{landesmanL:homological-stability-for-hurwitz}), and inverting a central element is base changing along a homological epimorphism (by \cite[Remark 3.3.2]{landesmanL:the-stable-homology-of-non-splitting}, the localized ring, which is always homological epimorphism by \cite[Example 3.3.1]{landesmanL:the-stable-homology-of-non-splitting}, is computed as the colimit along multiplication by $r$).
This establishes the above claim.

Therefore, applying 
\cite[Lemma 5.0.1]{landesmanL:homological-stability-for-hurwitz}
and iteratively applying 
\cite[Lemma 3.3.4]{landesmanL:the-stable-homology-of-non-splitting},
it suffices to show
$CA_{c,S}/(\alpha^{\ord(x)}_x, x \in V)[\alpha_x^{-1},x \in c - V]=0$.
In case $c - V$ is not a subrack of $c$, we find that there is some $x,y \in c -
V$ with $x \triangleright y \in V$.
As $\alpha_y \alpha_x  = \alpha_x \alpha_{x \triangleright y} \in \pi_0 \Hur^c$,
we find $\alpha_{x \triangleright y}$ acts both nilpotently and invertibly on 
$CA_{c,S}/(\alpha^{\ord(x)}_x, x \in V)[\alpha_x^{-1},x \in c - V]$, implying it
is $0$.
Hence, we may assume $c-V$ is a nonempty subrack of $c$.
In this case, 
\autoref{lemma:quotient-localization-vanishes}.
implies
$CA_{c,S}/(\alpha^{\ord(x)}_x, x \in V)[\alpha_x^{-1},x \in c - V]=0$
holds.
\end{proof}

Finally, we conclude by giving a straightforward rephrasing of
\autoref{lemma:full-quotient-bounded}
so that this rephrasing is equivalent to the version stated in the introduction,
\autoref{theorem:homology-stabilizes-intro}.
\begin{theorem}
	\label{theorem:stable-homology}
	Let $c$ be a finite rack and $S$ be a finite bijective Hurwitz module over $c$ and let
	$CA_{c,S} := C_*(\Hur^{c,S})$.
	Let $z \subset c$ denote and $S$-component of $c$ and suppose $y \in z$.
	Then,
	$z$ induces a grading on $\Hur^{c,S}$ where a component of $\Hur_n^{c,S}$
	lies in grading $j$ if $j$ of the $n$ labeled points lie in $z$.
	Then, $CA_{c,S}/\alpha_y$ is 
	$f_{\mu(|z|,\ord_c(z)),b(|z|,\ord_c(z))}$ bounded with respect to the
	grading induced by an $S$-component $z \subset c$, where
	$\mu(|z|,\ord_c(z)),b(|z|,\ord_c(z))$ are functions
	depending only on $|z|$ and $\ord_c(z)$.
\end{theorem}
\begin{proof}
	By \autoref{lemma:full-quotient-bounded},
	there is a subset $V \subset c$ so that $y$ is the only element of $V$
	lying in the $S$-component $z$ and
$CA_{c,S}/(\alpha^{\ord(x)}_x, x \in V)$ is 
$f_{\mu(|z|,\ord_c(z)),b(|z|,\ord_c(z))}$ bounded.
Moreover, we will assume 
$\mu(|z|,\ord_c(z)) \geq 1$ (and in fact this is satisfied by the specific
function constructed in \autoref{lemma:full-quotient-bounded}).
	Define a bigrading on $CA_{c,S}$ so that the first grading is induced by
	the component of $z$ and the second grading is induced by all other
	components of $c$.
	Repeatedly applying \cite[Lemma
	5.0.2]{landesmanL:homological-stability-for-hurwitz} to each element of
	$V-y$ for this bigrading,
	we find $CA_{c,S}/\alpha_y^{\ord(y)}$ is also
$f_{\mu(|z|,\ord_c(z)),b(|z|,\ord_c(z))}$ bounded. 
If $\ord(y) = 1$, we are done, so we may assume $\ord(y) \geq 2$.
The above implies that $(CA_{c,S}/\alpha_y^{\ord(y)})/\alpha_y$ is $\max(f_{\mu(|z|,\ord_c(z)),b(|z|,\ord_c(z))},f_{\mu(|z|,\ord_c(z)),b(|z|,\ord_c(z))-\mu((|z|,\ord_c(z))) + 1})$ bounded. 
Since we assumed $\mu \geq 1$, this maximum is equal to
$f_{\mu(|z|,\ord_c(z)),b(|z|,\ord_c(z))}$.
Note next that we have an equivalence of $\mathbb Z$ modules
$(CA_{c,S}/\alpha_y^{\ord(y)})/\alpha_y \simeq
(CA_{c,S}/\alpha_y)/\alpha_y^{\ord(y)}$, though this is not necessarily an
equivalence of $A_c$ modules.
Since we are assuming $\ord(y) \geq 2$, it follows from
\cite[Lemma 3.5.2]{landesmanL:the-stable-homology-of-non-splitting} that
$\alpha_y^{\ord(y)}$ acts by $0$ on 
$CA_{c,S}/\alpha_y$, hence also by $0$ on 
$(CA_{c,S}/\alpha_y)/\alpha_y^{\ord(y)} \simeq
(CA_{c,S}/\alpha_y^{\ord(y)})/\alpha_y,$ which means this has
$CA_{c,S}/\alpha_y$ as a retract.
Therefore, 
$CA_{c,S}/\alpha_y$
is also $f_{\mu(|z|,\ord_c(z)),b(|z|,\ord_c(z))}$ bounded, as desired.
\end{proof}

\section{Chain homotopies}
\label{section:chain-homotopies}

Having shown the homology of Hurwitz modules stabilize, we next wish to
compute their stable homology. That is, we wish to prove
\autoref{theorem:one-large-stable-homology}.
The general approach will be somewhat similar in nature to showing the homology
stabilizes. However, in showing the homology stabilizes, we needed to show a
certain complex was integrally nullhomotopic, and so we could realize the
nullhomotopy of chain complexes as coming from a nullhomotopy of spaces. 
However, when we compute the stable homology, we will invert the size of the
structure group, so the result will not be integral, and it seems unlikely it
will be induced by a nullhomotopy of spaces. Instead, we will construct a
nullhomotopy of chain complexes in this section, which we use to compute the
stable homology in the next section.
After defining the relevant chain complexes in
\autoref{subsection:defining-chain-complexes},
the main results of this section are \autoref{proposition:ring-chain-homotopy},
which computes the relevant chain homotopy for Hurwitz spaces, and
\autoref{proposition:module-chain-homotopy}, which computes the relevant chain
homotopy for Hurwitz modules.

\subsection{Defining the chain complexes}
\label{subsection:defining-chain-complexes}

Fix a rack $c$, a bijective Hurwitz module $S$ over $c$ and an $S$-component $c'
\subset c$.
Let $k$ be a ring.
We will define two related chain complexes.
The first, defined in \autoref{notation:space-chain-complex} gives a chain
complex whose homology agrees with that of a certain bar construction related to 
Hurwitz space and the second one introduced in \autoref{notation:module-chain-complex}
computes the homology of a certain bar construction related to Hurwitz 
modules. We prove this relation in \autoref{lemma:chains-identification}.
We now introduce some notation for various generalizations of the
$\triangleright$ action.

\begin{notation}
	\label{notation:iterated-action}
	If $w = w_1 \cdots w_k \in \pi_0 \Hur^c$ and $z \in c$, we use the notation $w
\triangleright z := w_k \triangleright \left( w_{k-1} \triangleright \cdots
\triangleright (w_1 \triangleright z) \right)$
and
$w
\triangleright^{-1} z :=w_1 \triangleright^{-1} \left( w_2
	\triangleright^{-1} \cdots
\triangleright^{-1} (w_k \triangleright^{-1} z) \right)$.
We omit the verification that the above definition is independent of the choice of
representative $w = w_1 \cdots w_k$ for $w$.
\end{notation}

We next introduce notation which extends linearly the $\triangleright$ action from an
action of $c$ on itself to an action of $k\{c\}$ on itself.
\begin{notation}
	\label{notation:}
Fix a ring $k$.
We will extend the action $\triangleright$ linearly to define an action of
$k\{c\}$ on $k\{c\}$. This means that if $x = \sum_i \alpha_i x_i$ and $y =
\sum_j \beta_j y_j$ for $x_i, y_i \in c$ and $\alpha_i, \beta_j \in k$, then $x \triangleright y := \sum_{i,j}
\alpha_i \beta_j x_i \triangleright y_j$.
Similarly,
$x \triangleright^{-1} y := \sum_{i,j}
\alpha_i \beta_j x_i \triangleright^{-1} y_j$.
Generalizing \autoref{notation:iterated-action}, for $v, v_1, \ldots, v_j \in
k\{c\}$,
we use $(v_1  \cdots v_j)
\triangleright v := v_j \triangleright \left( v_{j-1} \triangleright \cdots
\triangleright (v_1 \triangleright v) \right)$
and
$(v_1 \cdots v_j)
\triangleright^{-1} v :=v_1 \triangleright^{-1} \left( v_2
	\triangleright^{-1} \cdots
\triangleright^{-1} (v_k \triangleright^{-1} v) \right)$.
\end{notation}

With the above notation in place, we next define a chain complex that computes
the homology of bar constructions related to Hurwitz spaces.
\begin{notation}
	\label{notation:space-chain-complex}
	Suppose $c$ is a rack, $M_+$ is a discrete right $\Hur^c_+$ module and
	$P_+$ a
discrete left $\Hur^c_+$ module, so that the actions of $\Hur^c_+$ on $M_+$ and
$P_+$
factor through $\pi_0(\Hur^c)_+$.
Define the free $k$ module $V^{c,M,P;k}_n := k\{M\} \otimes k\{c^n\} \otimes
k\{P\}$.
We next define differentials to make a chain complex $V^{c,M,P;k}$ whose $n$th
graded part is $V^{c,M,P;k}_n$.
We can represent a basis element of $V^{c,M,P;k}_n$ as a tuple $(m,x_1, \ldots,
x_n,p)$ with $m \in M, p \in P, x_i \in c$.
For $1 \leq j \leq n$, using notation from \autoref{notation:iterated-action},
define
\begin{align*}
	\delta^l_{n,j}(m, x_1, \ldots, x_n,p) &:= \left( mx_j, x_j \triangleright
	x_1, \ldots, x_j \triangleright x_{j-1}, x_{j+1}, \ldots, x_n,p\right),
	\\
	\delta^r_{n,j}(m, x_1, \ldots, x_n,p) &:= \left( m,x_1, \ldots, x_{j-1},
	x_{j+1}, \ldots, x_n, ((x_n \cdots x_{j+1}) \triangleright x_j)p\right).
\end{align*}
Then, define the differential $\delta_n: V^{c,M,P;k}_n \to V^{c,M,P;k}_{n-1}$ by
\begin{align*}
	\delta_n \left( m, x_1, \ldots, x_n,p \right) := \sum_{j=1}^n (-1)^{j-1}
	\delta^l_{n,j}\left( m, x_1, \ldots, x_n, p \right) + \sum_{j=1}^n
	(-1)^j \delta^r_{n,j}(m, x_1, \ldots, x_n,p).
\end{align*}
\end{notation}

\begin{remark}
	\label{remark:}
	The complex in \autoref{notation:space-chain-complex} is nearly the same
	as the two-sided $\mathcal K$-complex we introduced in 
	\cite[Definition 3.2.1]{landesmanL:an-alternate-computation}, except
	that the complex there is bigraded, whereas here we only keep track of a
	single grading,
	and the sign convention for the differentials there is slightly
	different than the one here.
\end{remark}

Finally, we define a chain complex that computes the homology of a bar
construction related Hurwitz 
modules.

\begin{notation}
	\label{notation:module-chain-complex}
	Let $k$ be a ring, let $c$ be a rack, let $S$ be a
	bijective Hurwitz module over $c$ and let $c' \subset c$ be an $S$-component
	of $c$.
	Let $M$ be a set so that $M_+$ is a discrete right pointed $\Hur^{c}_+$ module.
	Define the free $k$-module
	$W^{c,S,M;k}_n := M\otimes k\{T_n\},$
where $k\{T_n\}$ denotes the free module over $k$ generated by the elements of $T_n$.
The homological degree refers to the value of $n$ while the  grading
of a term $(x_1, \ldots, x_n, s) \in T_n$ is the number of elements among $x_1,
\ldots, x_n$ lying in $z$, and corresponds to the grading on $\on{hur}^{c,S}$
obtained from
\autoref{notation:hur}.
We next define the differentials to make a chain complex which we call
$W^{c,S,M;k}$,
whose term in the $n$th homological degree is
$W^{c,S,M;k}_n$.
A general element
of
$W^{c,S,M;k}_n$
can be represented as a linear combination of elements of the
form
\begin{align}
	\label{equation:y-data}
	\left (m, y_1^1, \ldots, y_{i_1}^1, \ldots, y_1^{2g+f}, \ldots,
	y_{i_{2g+f}}^{2g+f} ,t\right)
\end{align}
where
$n = i_1 + \cdots + i_{2g+f},m \in k\{M\}, t \in
k\{T_0\},$ and $y_i^j \in k\{c\}$.
At this point, we suggest glancing at \autoref{figure:chain-connection} for a
visualization of the geometric meaning of these indices.
In order to define the differentials, it will be convenient to give
additional names to the elements as above.
Namely, we write an element as above in the form
\begin{align}
	\label{equation:x-data}
	\left( m, x_1, \ldots, x_n, t \right)
\end{align}
where $n = i_1 + \cdots + i_{2g+f}$ and $x_j$ is equal to the $j$th element to
the right of $m$, i.e. if $j = i_1 + \cdots + i_{q-1} + u$ then $x_j = y_u^q$.
In the above setting, if $x_j = y_u^q$, we say $q_{( m, x_1, \ldots, x_n, t)}(j) := q$
$u_{( m, x_1, \ldots, x_n, t)}(j) := u$
and define 
\begin{align*}
	b_{( m, x_1, \ldots, x_n, t)}(j) :=
\begin{cases}
	i_1 + \cdots +
	i_{q} & \text{ if } q \leq f \text{ or }q \equiv f \bmod 2  \\
	i_1 + \cdots + i_{q+1} & \text{ if } q >f \text{ and }q \equiv f+1
	\bmod 2. \\
\end{cases}
\end{align*}
We note that the data in \eqref{equation:y-data} is equivalent to the data in
\eqref{equation:x-data} together with the function
$q_{( m, x_1, \ldots, x_n, t)}$, which then uniquely determines the functions
$u_{( m, x_1, \ldots, x_n, t)}$ and $b_{( m, x_1, \ldots, x_n, t)}$.
Then, the differential is given as follows.
Using notation from \autoref{notation:iterated-action},
define
\begin{align}
	\label{equation:chi-right-differential}
	\chi_j^{(m, x_1, \ldots, x_n, t)}
	:= (x_{j+1} \cdots x_n) \triangleright x_j.
\end{align}
Also, use notation $\overline{\xi}_1, \ldots, \overline{\xi}_{2g+f} \in
\pi_1(\Sigma^1_{g,f})$ from \autoref{notation:allowable-no-points}, (recalling
$\overline{\xi}_i = \xi_{2g+f+1-i}$,)
and define
\begin{align}
	\label{equation:zeta-right-differential}
	\zeta_j^{(m, x_1, \ldots, x_n, t)} :=  (x_{b_{( m, x_1, \ldots, x_n, t)}(j)+1} \cdots x_{j-1}
	x_{j+1} \cdots x_n) \triangleright^{-1}
	\left( \sigma_{t}^{\overline{\xi}_{q_{( m, x_1, \ldots, x_n, t)}(j)}} \left( \chi_j^{(m, x_1, \ldots, x_n, t)}
	\right) \right).
\end{align}
For $1 \leq j \leq n$
let
\begin{align*}
	&d^l_{n,j} \left( m, x_1, \ldots, x_n, t \right) :=
(m x_j ,x_j\triangleright x_1, \ldots, x_j \triangleright x_{j-1}
, x_{j+1}, \ldots, x_n,t) \\
&d^r_{n,j} \left( m, x_1, \ldots, x_n, t \right) := \left( m \zeta_j^{(m, x_1, \ldots, x_n, t)}, \zeta_j^{(m, x_1, \ldots, x_n, t)} \triangleright x_1, \ldots, \zeta_j^{(m, x_1, \ldots, x_n, t)} \triangleright
	x_{b_{( m, x_1, \ldots, x_n, t)}(j)},
\right.
\\
	& \qquad \qquad \qquad \qquad \qquad \qquad
	\left. 
x_{b_{( m, x_1, \ldots, x_n,
t)}(j)+1}, 
\ldots, x_{j-1}, x_{j+1}, \ldots, x_n,
\tau_{\chi_j^{(m, x_1, \ldots, x_n, t)}}^{\overline{\xi}_{q_{( m, x_1, \ldots, x_n, t)}(j)}}(t)
\right).
\end{align*}
Define the differential by
\begin{equation}
\begin{aligned}
	\label{equation:differential}
	d_n 
	\left( m, x_1, \ldots, x_n, t \right) :=
	\sum_{j=1}^n (-1)^{j-1} d^l_{n,j}\left( m, x_1, \ldots, x_n, t \right)
	+
	\sum_{j = 1}^n
	(-1)^{j} d^r_{n,j}\left( m, x_1, \ldots, x_n, t \right).
\end{aligned}
\end{equation}

\end{notation}
\begin{remark}
	\label{remark:}
	The main cases of
the construction in \autoref{notation:module-chain-complex}
to keep in mind are the cases
	$M =\pi_0\Hur^{c}[\alpha_{c'}^{-1}]$ and
	$M =\pi_0\Hur^{c/c'}[\alpha_{c'/c'}^{-1}]$, for $c'$ a subrack of $c$.
\end{remark}

We now show that the above chain complexes compute the homology of certain bar
constructions involving Hurwitz spaces and Hurwitz modules.

\begin{figure}
\includegraphics[scale=.4]{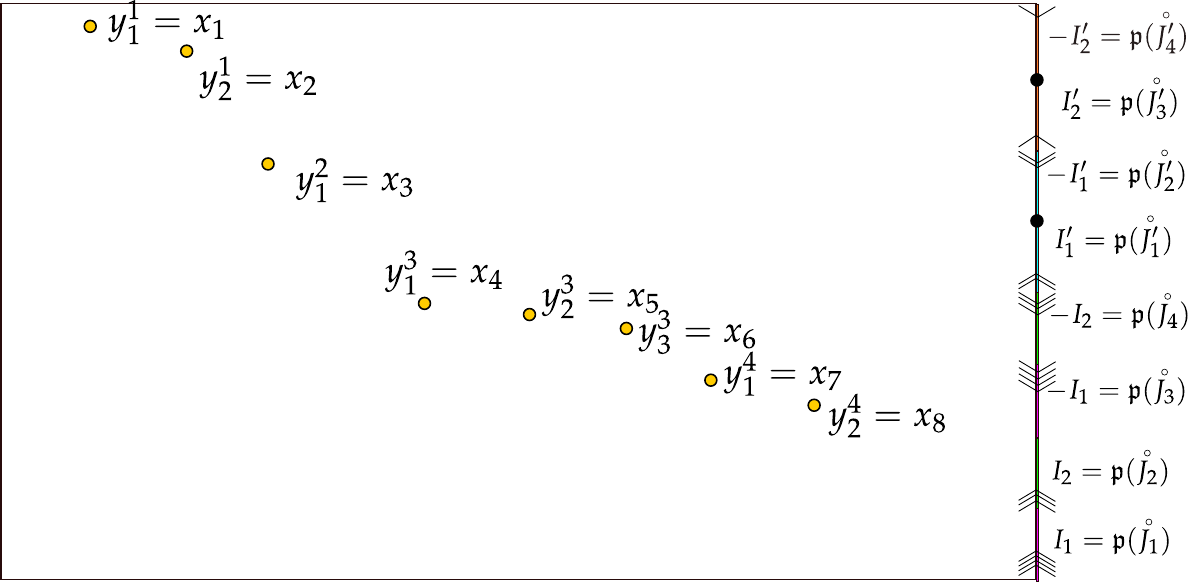}
\caption{
	This is a representation of the arrangement of the points $y_v^r$ in the
	configuration, giving a correspondence between cells of 
$\overline{Q}^*_\epsilon(M_+, \on{hur}^{c,S}_+)$
and cells of
$W^{c,S,M;k}$.
In the picture, $g = 1, f = 2, i_1 = 2, i_2 = 1, i_3 = 3, i_4 = 2$.
}
\label{figure:chain-connection}
\end{figure}

\begin{lemma}
	\label{lemma:chains-identification}
	Let ${C}_*$ denote the chains functor and $\widetilde{C}_*$ to denote
	the reduced chains functor.
	We use notation from \autoref{notation:space-chain-complex} and
	\autoref{notation:module-chain-complex}.
	There is an equivalence 
	$W^{c,S,M;k} \simeq \widetilde{C}_* \left( M_+
	\otimes_{\Hur^{c}_+}\Hur^{c,S}_+;k\right)$
	sending the grading defined on $W^{c,S,M;k}$
	to the grading on the right hand side induced by the trivial grading on
	$M$ and the gradings on
	$\Hur^{c} \simeq \on{hurbig}^c$ and on $\Hur^{c,S} \simeq \on{hur}^{c,S}$
	defined in 
	\autoref{notation:hur}.

	Additionally, there is an equivalence $V^{c,M,N;k} \simeq
	\widetilde{C}_* \left(M_+ \otimes_{\Hur^{c}_+}N_+;k\right)$
	which identifies the grading on $V^{c,M,N;k}$ with the grading
		induced by the grading on $\Hur^c\simeq \on{hurbig}^c$ defined in
		\autoref{notation:hur} and the trivial gradings on $M$ and $N$.
	\end{lemma}
\begin{proof}
	We first explain the equivalence
$W^{c,S,M;k} \simeq \widetilde{C}_* \left( M_+
	\otimes_{\Hur^{c}_+}\Hur^{c,S}_+;k\right)$.
	The key is to use the description of 
$M_+ \otimes_{\Hur^{c}_+}\Hur^{c,S}_+$
as a colimit over $\epsilon$ of the homotopy type of $\overline{Q}^*_\epsilon(M_+,
\on{hur}^{c,S}_+)$, as shown in 
\autoref{proposition:pointed-scanning}.
However, we will see that the $i$th homology of these spaces are independent of
$\epsilon$ once $\epsilon$ is sufficiently small, so that we can just work with a
fixed, sufficiently small, $\epsilon$ to compute the $i$th homology.
We describe a bijection between the cells of
$W^{c,S,M;k}$
and the components of
${Q}^*_\epsilon(M_+, \on{hur}^{c,S})$ where the left label isn't $+$.
Then, it only remains to identify the differentials in 
$W^{c,S,M;k}$
with the attaching maps for the components of 
${Q}^*_\epsilon(M_+, \on{hur}^{c,S}_+)$
by realizing
$\overline{Q}^*_\epsilon(M_+, \on{hur}^{c,S}_+)$
as a quotient of
${Q}^*_\epsilon(M_+, \on{hur}^{c,S}_+)$.
To obtain the bijection,
consider a component of
${Q}^*_\epsilon(M_+, \on{hur}^{c,S}_+)$. 
Staying within the component, arrange the points in the corresponding configuration so that they are of the form $y_1^r, \ldots, y_{i_r}^r$, and have
have preimage in $\mathbf{R}-W$ whose vertical coordinate lies in either $J'_i$
for $i$ even or
$J_j$ where $j$ is $3$ or $4$ modulo $4$; i.e. for each pair of glued edges
among the $J'_i$ and $J_j$ the corresponding $y_v^r$ lie to the left of the
higher of the two, and choose the path $\gamma$ to be the path that linearly moves the second coordinate towards $\frac 1 2$ for all points. This gives a well defined label to each point, which by abuse of notation we also denote $y_{i_j}$. When it is convenient, we also rename these labels as $x_1, \ldots, x_n$, so that $x_i$ is positioned below and to the
right of $x_{i-1}$ in $\mathcal M_{g,f,1}$.
See \autoref{figure:chain-connection} for a figure depicting a typical situation
as above.

The gluing maps come from moving each of the $n$ points $x_1, \ldots, x_n$
either to the left until they hit the boundary or to the right (when we say we
	move them right, we mean that we move $x_i$ until it hits the right side of $\mathcal
	M_{g,f,t}$, in which case it is identified with a lower vertical
	coordinate, and then we move it left at that lower coordinate
until it hits the boundary).
Said briefly,
we claim that
if one keeps track of the relabelings coming from the surface braid group action
described in \autoref{notation:hur}, the gluing for the points moving left come
from the first sum in \eqref{equation:differential}
and the gluing maps from the points moving right come from the second sum in 
\eqref{equation:differential}.

We now explain the above claims.
First, consider the relabelings obtained from moving the point $x_i$ to the
left. We claim the result is $d_{n,i}^l(m, x_1, \ldots, x_n,t)$.
To see this, the corresponding element of the braid group associated to moving
$x_i$ below $x_{i-1},x_{i-2}, \ldots, x_1$ is $\sigma_1\cdots
\sigma_{i-1}$. 
Applying this transformation sends the point labeled $x_i$ to the left unchanged
until it hits the left boundary which becomes $m \cdot x_i$, and each of the
points labeled $x_j$ for $j < i$ become $x_j \triangleright x_i$, which is
precisely 
$d_{n,i}^l(m, x_1, \ldots, x_n,t)$.
Similarly, one can see that the result of moving $x_i$ to the right is precisely
$d_{n,i}^r(m, x_1, \ldots, x_n,t)$.

To conclude 
the proof that
$W^{c,S,M;k} \simeq {C}_* \left( M_+
\otimes_{\Hur^{c}}\Hur^{c,S};k\right)$, 
it remains to explain how we chose orientations of the
cells to explain the signs appearing in the boundary maps in
\eqref{equation:differential}.
We can view our complex as a cubical complex with the cell parameterizing
locations of $n$ points as being an $n$-dimensional cube, and the boundaries of
the cube are $n-1$ dimensional cubes where one point moves to the boundary
on each codimension $1$ face. From this perspective, we have described a cubical
complex, and so the signs on the differentials are the usual convention for
cubical complexes, as described, for example, in \cite[Proposition
2.36]{kaczynskiMM:computational-homology}.

The proof of the second claimed equivalence $V^{c,M,N;k} \simeq
\widetilde{C}_* \left(M_+ \otimes_{\Hur^{c}_+}N_+;k\right)$
is obtained similarly, where one uses
the description 
$M_+ \otimes_{\Hur_+^c} N_+ \simeq \overline{Q}^*_\epsilon[M_+,\Hur_+^c, N_+]$ from
\cite[Theorem A.4.9]{landesmanL:the-stable-homology-of-non-splitting} (with 
$\overline{Q}^*_\epsilon[M_+,\Hur_+^c, N_+]$ defined in the statement of 
\cite[Theorem A.4.9]{landesmanL:the-stable-homology-of-non-splitting})
in place of
\autoref{proposition:pointed-scanning}.
The remainder of the proof is similar to the above argument and we omit
further details.
\end{proof}

\subsection{Chain homotopies for Hurwitz space bar constructions}
\label{subsection:chain-homotopy-space}

In this subsection, we verify a certain equivalence of chain complexes related
to bar constructions of Hurwitz spaces in
\autoref{proposition:ring-chain-homotopy}.

The following notation will be crucially used in the ensuing nullhomotopies.
\begin{notation}
	\label{notation:group-elements}
	Let $c$ be a rack and $c' \subset c$ a normal subrack. Recall we use
	$G^{c'}_{c}$ to denote the relative structure group as in
\autoref{example:normal-relative-subrack}.
For each $g \in G^{c'}_{c}$ choose an expression $g = w^g_1 \cdots
	w^g_{i_g}$ with each $w^g_i \in c'$.
	Let $E_{c,c'}$ denote the set of pairs of the form 
	$\{ (x,g) : x \in
	{G}^{c'}_{c'}, g \in {G}^{c'}_{c} \}$.
	In particular, $E_{c,c'}$ has $|G^{c'}_{c'}| \cdot |G^{c'}_{c}|$ elements.
	Associated to each pair $(x,g) \in E_{c,c'}$ we define the operation $x
	\succ g := (x \triangleright w^g_1) \cdots (x \triangleright
	w^g_{i_g})$, which we view as a product of $i_g$ elements
	$\pi_0(\Hur^c_1)$.

	Let $G^{c'}_{\pi_0}$ denote the 
	kernel
	of the map $\pi_0 \Hur^{c'}[\alpha_{c'}^{-1}] \to
	\Hur^{c'/c'}[\alpha_{c'/c'}^{-1}]$.
	We can write any $g \in G^{c'}_{\pi_0}$ as a sequence of elements of the
	form ${y^g_1} (z^g_1)^{-1} \cdots {y^g_{i_g}}
({z^g_{i_g}})^{-1}$ for $y^g_i, z^g_i \in c'$.
	Let $E^{\pi_0}_{c'}$ denote the of tuples of the form
	$\{ (x; g) 
		 : x \in
	 G^{c'}_{c'}, g \in G^{c'}_{\pi_0} \}$.
	 In particular, $E^{\pi_0}_{c'}$ has $|G^{c'}_{\pi_0}| \cdot |G^{c'}_{c'}|$
	 many elements.
	 Associated to each pair $(x,g) \in E^{\pi_0}_{c'}$, we define the
	 operation 
	 $ x \succ g := (x \triangleright{y^g_1}) (x \triangleright
	 (z^g_1)^{-1}) \cdots
	 (x\triangleright {y^g_{i_g}})
	 (x \triangleright {z^g_{i_g}}^{-1})$.
\end{notation}

We next record a simple lemma in the structure theory of racks, which describes the fibers of
$\pi_0 \Hur^c[\alpha_{c'}^{-1}] \to \pi_0
\Hur^{c/c'}[\alpha_{c'/c'}^{-1}]$.

\begin{lemma}
	\label{lemma:same-image-relation}
	Suppose $c$ is a rack and ${c'} \subset c$ is a normal subrack. Suppose
	$u,v \in \pi_0 \Hur^c[\alpha_{c'}^{-1}]$ have the same image in $\pi_0
	\Hur^{c/{c'}}[(\alpha_{c'/c'})^{-1}]$.
	Then there is some $w \in \pi_0 \Hur^{c'}[\alpha_{c'}^{-1}]$ so that $uw= v$.
\end{lemma}
\begin{proof}
	After multiplying by a suitable power of elements in ${c'}$, we can assume $u,v \in
	\pi_0 \Hur^c$, with ${c'}$ not inverted, and we can write $u = u_1 \cdots u_n$ and $v = v_1 \cdots v_n$, with $u_i, v_i\in
	c$ so that $u_i$ has the same image as $v_i$ in $\Hur^{c/{c'}}$. By induction on
	$n$, it suffices to show we can find some $w \in \pi_0
	\Hur^{c'}[\alpha_{c'}^{-1}]$ so that $u w$ is
	equivalent under the braid group action to an element of the form $v'_1
	v'_2\cdots v'_n w'$ with $w' \in \pi_0 \Hur^{c'}[\alpha_{c'}^{-1}]$ and $v'_1 = v'$.
	By assumption, $u_1$ and $v_1$ have the same image in $c/{c'}$, which means
	that by definition there is some $x = x_1 \cdots x_j$ 
	(using notation from \autoref{notation:iterated-action})
	with $x_1, \ldots, x_j \in {c'}$ so
	that $x \triangleright u_1 = v_1$.
Then, $u = u x x^{-1} = x (x \triangleright u_1) \cdots (x \triangleright
u_n) x^{-1} = 
(x \triangleright u_1) \cdots (x \triangleright
u_n) (((x \triangleright u_1) \cdots (x \triangleright u_n))
\triangleright x)x^{-1}$, which indeed starts with $v_1 = x
\triangleright u_1$.
We can use this construction to produce our desired
element $v_1' \cdots v_n' w' \in
\Hur^c[\alpha_{c'}^{-1}]$ whose first
term is $v'_1 =v_1$, completing the proof.
\end{proof}

The next lemma is an important step in proving the upcoming
\autoref{proposition:ring-chain-homotopy}. It shows that if the module on the
right side of the bar construction is averaged we can also arrange that the
module on the left side of the bar construction is averaged.
\begin{figure}
	\includegraphics[scale=.4]{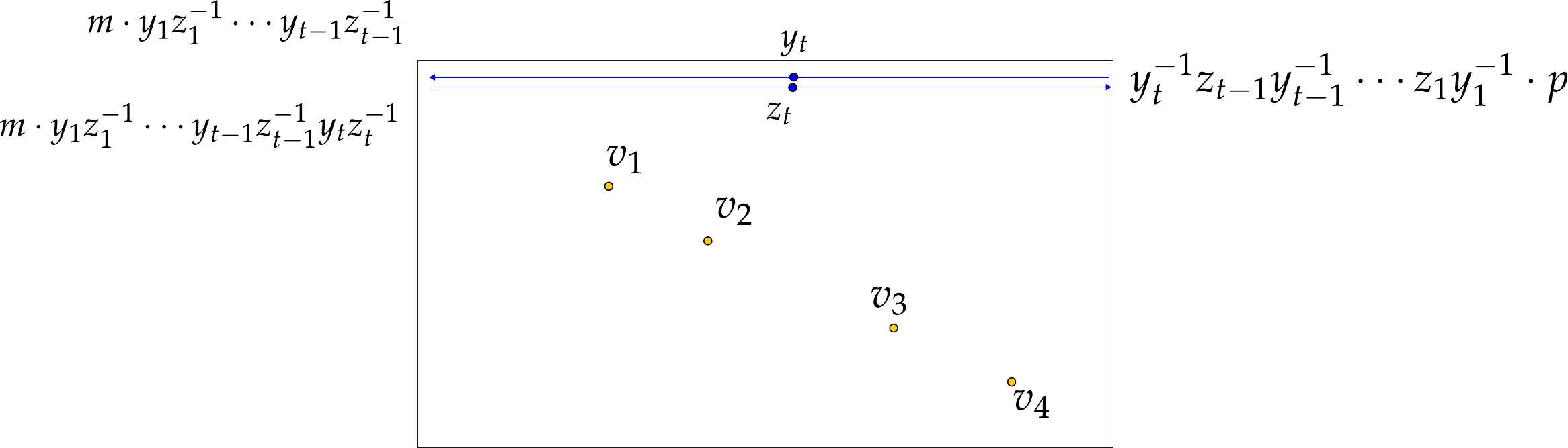}
\caption{
	This is a visualization of part of the nullhomotopy $K$ in the proof of
\autoref{lemma:ring-reduction}.
We pull $y_t$ to the left across the top and then move $z_t$ to the right back
across the top. After pulling $y_t$ out, the right label becomes $y_t^{-1}
z_{t-1} y_{t-1}^{-1} \cdots z_1 y_1^{-1} p$ and the left label starts as $m
\cdot y_1 z_1^{-1} \cdots y_{t-1} z_{t-1}^{-1}$. When the $y_1$ hits the left
label, and $z_1$ comes out, the left label becomes
$m \cdot y_1 z_1^{-1} \cdots y_{t} z_{t}^{-1}$ and then $z_1$ traverses back to
the right.
We then perform this up through $t = \ell$ and average over $E^{\pi_0}_{c'}$.
}
\label{figure:left-ring-nullhomotopy}
\end{figure}
\begin{remark}
	\label{remark:}
In \autoref{lemma:same-image-relation}, we pursue an algebraic approach to verify the 
	nullhomotopy depicted in
	\autoref{figure:left-ring-nullhomotopy}.
	because it seemed technically trickier to make the idea from
	\autoref{figure:left-ring-nullhomotopy}
	rigorous. 
	Nevertheless, this picture served as the inspiration for our algebraic
	nullhomotopy.
	A similar comment applies to 
	\autoref{figure:ring-vrack-nullhomotopy},
	\autoref{figure:module-rectangle-action}, and
	\autoref{figure:module-vrack-nullhomotopy}.
\end{remark}
\begin{lemma}
	\label{lemma:ring-reduction}
	Let $c$ be a finite rack and ${c'} \subset c$ a union of components of $c$.
	The natural maps induce an equivalence
	\begin{align*}
	&H_0(A_{c})[\alpha_{{c'}}^{-1}] \otimes_{A_c[\alpha_{{c'}^{-1}}]}
H_0(A_{c/{c'}}[\alpha_{c'/c'}^{-1}])[|G^{c'}_{c'}|^{-1}] \\
	&\simeq
	\left( H_0(A_{c/{c'}})[\alpha_{c'}^{-1}] \otimes_{A_{c}[\alpha_{c'}^{-1}] }
	H_0(A_{c/c'}[\alpha_{(c'/c')}^{-1}]) \right)[|G^{c'}_{c'}|^{-1}].
	\end{align*}
\end{lemma}
\begin{proof}
	Let $k := \mathbb Z[|G^{c'}_{c'}|^{-1}]$.
	Let 
	$P := \pi_0(\Hur^{c/c'})[\alpha_{{c'/c'}}^{-1}]$
and let
	$M := \pi_0(\Hur^{c})[\alpha_{{c'}}^{-1}]$.
	Define $\on{Avg}_{c'} : M \to M$ given by $m \mapsto 
\frac{1}{|\Pi^{-1}(\Pi(m))|}
\sum_{\substack{m' \in
\pi_0 \Hur^c[\alpha_{c'}^{-1}] \\
\Pi(m') = \Pi(m)}} m'$, for 
$\Pi: \pi_0 \Hur^c[\alpha_{c'}^{-1}] \to \pi_0
\Hur^{c/c'}[\alpha_{c'/c'}^{-1}]$ the projection map.
We now explain why 
$|\Pi^{-1}(\Pi(m))|$ is invertible in $k$ so that 
$\on{Avg}_{c'}$ makes sense with $k$ coefficients.
Any fiber of $\Pi$ has a transitive action of $\ker(\pi_0
\Hur^{c'}[\alpha_{c'}^{-1}] \to \pi_0 \Hur^{c'/c'}[\alpha_{c'}^{-1}])$ 
by \autoref{lemma:same-image-relation}
and so by \cite[Lemma 6.0.4]{landesmanL:homological-stability-for-hurwitz},
any prime dividing $|\Pi^{-1}(\Pi(m))|$
also divides $|G^{c'}_{c'}|$, which we have inverted.
Hence, $\on{Avg}_{c'}$ makes sense with $k$ coefficients.

The projection map
$V^{c,M,P;k} \to V^{c,P,P;k}$ has a section
$V^{c,P,P;k} \to V^{c,M,P;k}$ so that the composite map
$V^{c,M,P;k} \to V^{c,P,P;k} \to V^{c,M,P;k}$
sends $(m,v_1, \ldots, v_n, p) \mapsto (\on{Avg}_{c'}(m), v_1,
\ldots, v_n, p)$.
Moreover, \autoref{lemma:same-image-relation} implies that any two elements of
$c$ in
the same $c'$ orbit act the same way on $P$
and so
differential on $V^{c,M,P;k}$ restricts to the differential on $V^{c,P,P;k}$.
Hence,
the above section defines a subcomplex.
Then, by \autoref{lemma:chains-identification}, it suffices to show the above section
induces an
equivalence on homology.

Now, define a filtration 
$F^\bullet$ on $V^{c,M,P;k}/V^{c,P,P;k}$ so that $F^e$ consists of those $(m,v_1, \ldots,
v_n, p)$ with $m \in M, v_1, \ldots, v_n \in
c, p \in P$ with at most $e$ elements among
$v_1, \ldots, v_n \in c-c'$.
To accomplish our goal, we produce a suitable homotopy 
$K_n : V^{c,M,P;k}_n/V^{c,P,P;k}_n \to V^{c,M,P;k}_{n+1}/V^{c,P,P;k}_{n+1}$
with the property that $K_n$ preserves the filtration $F^\bullet$ and
$\delta_{n+1} K_n + K_{n-1} \delta_n - \id|_{F^e} \subset F^{e-1}$.
Once we show this, it will follow that each associated graded piece of the
filtration is nullhomotopic, and hence 
it will follows that
$V^{c,M,P;k}/V^{c,P,P;k}$ is nullhomotopic.
Note that basis elements for the quotient
$V^{c,M,P;k}_n/V^{c,P,P;k}_n$ can be written in the form $(m, v_1, \ldots,
v_n,p)$ with $\on{Avg}_{c'}(m) = 0$.
Here is the claimed homotopy, which is visually depicted in
\autoref{figure:left-ring-nullhomotopy}:
\begin{align*}
	&K_n(m, v_1, \ldots, v_n, p) \\
	&:=
	\frac{1}{|E^{\pi_0}_{c'}|} \sum_{\substack{(x,g)\in E^{\pi_0}_{c'} \\
	x \succ g = 
y_1z_1^{-1} \cdots y_\ell z_\ell^{-1}}}
\sum_{t =
	1}^\ell \left(-\left( m \cdot y_1 \cdot z_1^{-1} \cdots y_{t-1} \cdot
		z_{t-1}^{-1}, y_t, v_1, \ldots, v_n,y_t^{-1} z_{t-1}
	y_{t-1}^{-1} \cdots z_1 y_1^{-1} p\right)
	\right.
		\\
		&\qquad \qquad\qquad\qquad+ \left. \left( m \cdot y_1 \cdot z_1^{-1} \cdots y_{t} \cdot
			z_{t}^{-1}, z_t, v_1, \ldots, v_n,y_t^{-1} z_{t-1}
		y_{t-1}^{-1} \cdots z_1 y_1^{-1} p\right)
			\right),
\end{align*}
We next verify
\begin{align}
\label{equation:ring-average-homotopy-telescope}
-\delta_{n+1,1}^r K_n(m, v_1,\ldots, v_n,p) =(m, v_1,\ldots, v_n,p)
\end{align}
using the assumption that $\on{Avg}_{c'}(m) = 0$.
We next perform this calculation, whose steps we will explain following it.
\begin{equation}
	\label{equation:telescope-left-change}
\begin{aligned}
	&-\delta_{n+1,1}^r K_n(m, v_1,\ldots, v_n,p)
	\\
	&= 
	\delta_{n+1,1}^r
\frac{1}{|E^{\pi_0}_{c'}|} \sum_{\substack{(x,g)\in E^{\pi_0}_{c'} \\
	x \succ g = 
y_1z_1^{-1} \cdots y_\ell z_\ell^{-1}}}
 \sum_{t =
	1}^\ell \left(\left( m \cdot y_1 \cdot z_1^{-1} \cdots y_{t-1} \cdot
		z_{t-1}^{-1}, y_t, v_1, \ldots, v_n,y_t^{-1} z_{t-1}
	y_{t-1}^{-1} \cdots z_1 y_1^{-1} p\right)
	\right.
		\\
		&\qquad \qquad\qquad\qquad - \left. \left( m \cdot y_1 \cdot z_1^{-1} \cdots y_{t} \cdot
			z_{t}^{-1}, z_t, v_1, \ldots, v_n,y_t^{-1} z_{t-1}
		y_{t-1}^{-1} \cdots z_1 y_1^{-1} p\right)
			\right)\\
&= 
\frac{1}{|E^{\pi_0}_{c'}|} \sum_{\substack{(x,g)\in E^{\pi_0}_{c'} \\
	x \succ g = 
y_1z_1^{-1} \cdots y_\ell z_\ell^{-1}}}
 \sum_{t =
	1}^\ell \left(\left( m \cdot y_1 \cdot z_1^{-1} \cdots y_{t-1} \cdot
		z_{t-1}^{-1}, v_1, \ldots, v_n, z_{t-1}
	y_{t-1}^{-1} \cdots z_1 y_1^{-1} p\right)
	\right.
		\\
		&\qquad \qquad\qquad\qquad - \left. \left( m \cdot y_1 \cdot z_1^{-1} \cdots y_{t} \cdot
			z_{t}^{-1}, v_1, \ldots, v_n,z_t y_t^{-1} z_{t-1}
		y_{t-1}^{-1} \cdots z_1 y_1^{-1} p\right)
			\right)\\
&=\frac{1}{|E^{\pi_0}_{c'}|} \sum_{\substack{(x,g)\in E^{\pi_0}_{c'} \\
	x \succ g = 
y_1z_1^{-1} \cdots y_\ell z_\ell^{-1}}}
 \left(\left( m, v_1, \ldots, v_n, p\right)- \left( m \cdot y_1 \cdot
				z_1^{-1} \cdots y_{\ell} \cdot
			z_{\ell}^{-1}, v_1, \ldots, v_n,z_\ell y_\ell^{-1}\cdots z_1 y_1^{-1} p\right)
			\right)\\
			&=
\left( m, v_1, \ldots, v_n, p\right)
- \frac{1}{|E^{\pi_0}_{c'}|} \sum_{\substack{(x,g)\in E^{\pi_0}_{c'} \\
	x \succ g = 
y_1z_1^{-1} \cdots y_\ell z_\ell^{-1}}}
\left( m \cdot y_1 \cdot
				z_1^{-1} \cdots y_{\ell} \cdot
			z_{\ell}^{-1}, v_1, \ldots, v_n,z_\ell y_\ell^{-1}\cdots
		z_1 y_1^{-1} p\right) \\
			&=
\left( m, v_1, \ldots, v_n, p\right)
- \frac{1}{|E^{\pi_0}_{c'}|} \sum_{\substack{(x,g)\in E^{\pi_0}_{c'} \\
	x \succ g = 
y_1z_1^{-1} \cdots y_\ell z_\ell^{-1}}}
 \left( m \cdot y_1 \cdot
				z_1^{-1} \cdots y_{\ell} \cdot
			z_{\ell}^{-1}, v_1, \ldots, v_n,p\right)
			\\
			&=
\left( m, v_1, \ldots, v_n, p\right)
-	\left( \on{Avg}_{c'}(m), v_1, \ldots, v_n,p\right)
	\\
	&=
	\left( m, v_1, \ldots, v_n, p\right).
		\end{aligned}
\end{equation}
The second equality in 
\eqref{equation:telescope-left-change} 
uses the condition that $p \in P$ and so for any
$y, y' \in c$ with the same image in $c/c'$, $y \cdot p = y' \cdot p$.
More precisely, we use
\begin{align*}
	((v_1 \cdots v_n) \triangleright y_t) \cdot 
y_t^{-1} z_{t-1} y_{t-1}^{-1} \cdots z_1 y_1^{-1} \cdot p
= y_t \cdot 
y_t^{-1} z_{t-1} y_{t-1}^{-1} \cdots z_1 y_1^{-1}\cdot p
=
z_{t-1} y_{t-1}^{-1} \cdots z_1 y_1^{-1} \cdot p.
\end{align*}
The fifth equality uses that $z_\ell y_\ell^{-1} \cdots z_1 y_1^{-1}$ maps to
the trivial element in $\pi_0 \Hur^{c/c'}$ by construction of $E^{\pi_0}_{c'}$.
The sixth equality uses that $\on{Avg}_{c'}(m) = \frac{1}{|G^{c'}_{\pi_0}|}
\sum_{g \in G^{c'}_{\pi_0}} g \triangleright m$, which follows from
\autoref{lemma:same-image-relation} because it implies the fibers of
$\pi_0\Hur^c[\alpha_{c'}^{-1}] \to \pi_0 \Hur^{c/c'}[\alpha_{c'/c'}^{-1}]$
have a transitive action of $G^{c'}_{\pi_0}$.

We next claim
\begin{align}
\label{equation:ring-average-homotopy-cancel}
	\delta^l_{n+1,j+1} K_n &= K_{n-1} \delta^l_{n,j} \\
	\delta^r_{n+1,j+1} K_n &= K_{n-1} \delta^r_{n,j}
\end{align}
for $j \geq 1$,
modulo 
$F^{e-1}$. 
The second relation is fairly immediate upon writing out the definitions.
The first relation can be seen to hold
by working modulo $F^{e-1}$ we can ignore any differentials
removing $v_j \in c - c'$, and so may assume that the $v_j$, which the relevant
differential removes, lies
in $c'$. The above relations can then be deduced from the assumption
that $E^{\pi_0}_{c'}$ is closed under the bijective operation 
$(x, g) \mapsto (v_j \cdot x, g)$,
where $v_j \cdot x$ is multiplication in $G^{c'}_{c'}$,
and this sends
\begin{align*}
	y_1 z_1^{-1} \cdots y_\ell
z_\ell^{-1} = x \succ g \mapsto (v_j \cdot x) \succ g = (v_j \triangleright y_1) (v_j \triangleright z_1^{-1}) \cdots (v_j
\triangleright y_\ell)
(v_j \triangleright z_\ell^{-1}).
\end{align*}
The reader may consult \eqref{equation:h-then-delta-ring-left-big}
and \eqref{equation:delta-then-h-ring-left-big}
for a similar computation, spelled out in more detail.
Finally, 
\begin{align}
	\label{equation:ring-average-homotopy-last}
	\delta^l_{n+1,n} K_n = 0
\end{align}
because
\begin{align*}
(m \cdot y_1 \cdot z_1^{-1} \cdots y_{t-1} \cdot z_{t-1}^{-1}) \cdot y_t =
(m \cdot y_1 \cdot z_1^{-1} \cdots y_{t} \cdot
z_{t}^{-1}) \cdot z_t.
\end{align*}
agree as elements in $\pi_0 \Hur^c[\alpha_{c'}^{-1}]$.
Summing \eqref{equation:ring-average-homotopy-telescope},
\eqref{equation:ring-average-homotopy-cancel},
and
\eqref{equation:ring-average-homotopy-last},
we obtain the claim that 
$\delta_{n+1} K_n + K_{n-1} \delta_n - \id|_{F^e} \in
F^{e-1}$.
This implies the identity acts nilpotently on
$V^{c,M,P;k}/V^{c,P,P;k}$, and therefore this quotient vanishes, concluding the
proof.
\end{proof}

In order to set up notation for our ensuing equivalence of bar constructions, we
introducing an averaging operator that will be used to relate a bar construction
associated to $c$ to one associated to $c/c'$.

\begin{notation}
	\label{notation:u-operator}
		Fix a finite rack $c$ and a subrack $c' \subset c$ which is a union of
	components of $c$. Let $k$ be a ring on which the order of the relative structure
	group $G^{c'}_c$, as in
	\autoref{example:relative-subrack}, is invertible.
		Let 
	$U_{c'}: k\{c\}\to k\{c\}$ be the operator
	$U_{c'} := \frac{1}{|G^{c'}_{c}|} \sum_{g \in G^{c'}_{c}} g \triangleright$
	which sends $x \mapsto \frac{1}{|G^{c'}_{c}|} \sum_{g \in G^{c'}_{c}} g
	\triangleright x$.
\end{notation}

\begin{definition}
	\label{definition:averaged-basis-rack}
	Let $r_1, \ldots, r_{|c/c'|}$ denote a collection of representatives
	of the $G^{c'}_{c}$ orbits of $c'$, 
	The image of $U_{c'} : k\{c\} \to k\{c\}$ is the free $k$-module
	generated by the basis $\{U_{c'}(r_i)\}_{1\leq i \leq |c/c'|}$. We refer
	to such elements $U_{c'}(r_i)$ as \textit{averaged basis elements} 
	Because the map $U_{c'}$ is base changed from the PID $\ZZ[\frac 1
	{|G^{c'}_{c}|}]$ to $k$, the kernel of $U_{c'}$ is free, so we may extend
	the averaged basis elements to a basis of $k\{c\}$ by including elements
	of the kernel of $U_{c'}$ which additionally are supported in a single
	$c'$ orbit (so they are of the form $\sum_{y \in G^{c'}_{c} \cdot z}
	\alpha_y y$ for some $z \in c$). We refer to the additional elements as
	\textit{antiaveraged basis elements}. 
	We refer to elements in the image
	of $U_{c'}$ as \textit{averaged} elements and elements in the kernel of
	$U_{c'}$ as \textit{antiaveraged}. 
\end{definition}

\begin{remark}
	\label{remark:}
	Equivalently to the above definition, averaged elements are linear combinations of averaged
	basis elements and antiaveraged elements are linear combinations
	of antiaveraged basis elements.
\end{remark}

We now record the main equivalence relating to bar constructions of Hurwitz
spaces which will be crucial for our results on computing the stable homology of
Hurwitz space in all directions.

\begin{proposition}
	\label{proposition:ring-chain-homotopy}
	Let $c$ be a finite rack and ${c'} \subset c$ a normal subrack.
	There is an equivalence
	\begin{align*}
	&H_0(A_{c})[\alpha_{{c'}}^{-1}] \otimes_{A_c[\alpha_{{c'}^{-1}}]}
	H_0(A_{c/{c'}}[\alpha_{{c'}/{c'}}^{-1}])[
	|G^{c'}_{c}|^{-1}] \\
	&\simeq
	\left( H_0(A_{c/{c'}})[\alpha_{c'}^{-1}] \otimes_{A_{c/{c'}}[\alpha_{{c'}/{c'}}^{-1}] }
	H_0(A_{c/{c'}}[\alpha_{({c'}/{c'})}^{-1}])
\right)[|G^{c'}_{c}|^{-1}].
	\end{align*}
\end{proposition}
\begin{proof}
	Let $k := \mathbb Z[|{G}^{c'}_{c}|^{-1}]$.
	Let 
$P := \pi_0(\Hur^{c/c'})[\alpha_{{c'/c'}}^{-1}]$.
By \autoref{lemma:ring-reduction} and \autoref{lemma:chains-identification}
we only need show the projection map
$V^{c,P,P;k} \to V^{c/c',P,P;k}$
is an equivalence on homology.
We let $v_i \in c$ and use $\overline{v}_i$ as notation for the image of $v_i$
in $c/c'$.
Note that the above map has a section
$V^{c/c',P,P;k} \to V^{c,P,P;k}$ given by
$(m,\overline{v}_1, \ldots, \overline{v}_n, p) \mapsto (m, U_{c'}(v_1), \ldots, U_{c'}(v_n), p)$, with
$U_{c'}$ as defined in \autoref{notation:u-operator}.
It suffices to show
this section induces an equivalence on homology.

Equivalently, it suffices to produce a nullhomotopy of the quotient
$V^{c,P,P;k}/V^{c/c',P,P;k}$, which we do next.
Any element of this quotient can be presented as a linear combinations of
tuples
$(m, v_1, \ldots, v_n,p)$ with $m,p \in P, v_1, \ldots, v_n \in k\{c\}$ where there is some $i$ so that $v_1,
\ldots, v_{i-1}$ are averaged basis elements and $v_i$ is an antiaveraged basis
element and $v_{i+1}, \ldots, v_n \in c$ (meaning they are elements of $k\{c\}$
of the form $1 \cdot x$ for $x \in c$).
Now, define a filtration 
$F^\bullet$ on $V^{c,P,P;k}/V^{c/c',P,P;k}$ so that $F^e$ is spanned by those $(m,v_1, \ldots,
v_n, p)$ with $m,p \in P, v_1, \ldots, v_n \in
k\{c\}$ so that $v_1, \ldots, v_n$ either lie in $k\{c'\}$ or $k\{c-c'\}$ and
there are at most $e$ elements among $v_1, \ldots, v_n \in k\{c-c'\}$.

\begin{figure}
	\includegraphics[scale=.4]{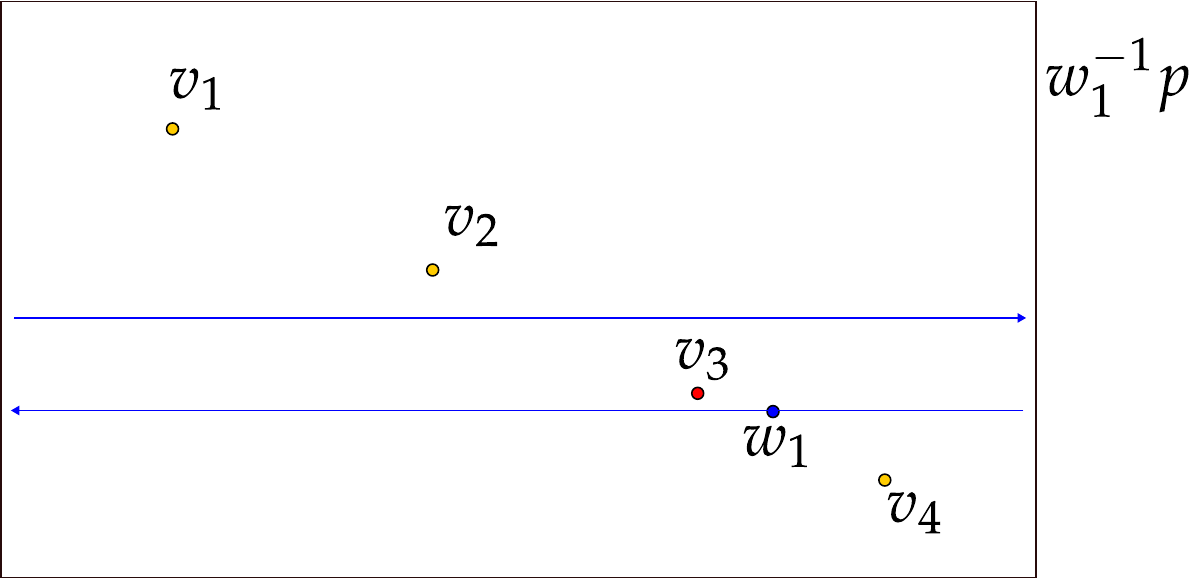}
\caption{
	This is a visualization of part of the nullhomotopy $H$ in the proof of
	\autoref{proposition:ring-chain-homotopy}.
	The $v_i$ are written in the averaged basis, and the yellow
	$v_1,v_2,v_4$
	are averaged while the red dot $v_3$ is antiaveraged. When we move $w_1$ to the left below $v_3$ and then back to the right above
	$v_3$, we cause $v_3$ to be changed to $w_1 \triangleright v_3$.
	In the homotopy, we then repeat this for 
	$w_2, \ldots, w_t$ so that $v_3$ is changed to $(w_1 \cdots w_t) \triangleright v_3$.
	Now, $w_1 \cdots w_t$ was made to realize one of the group elements in
	$G^{c'}_{c'}$, and averaging over all such elements
	modifies $v_3$ to $U_{c'}(v_3)$, which vanishes because $v_3$
	is antiaveraged.
	This operation may not be compatible with other $v_j$ hit the boundary,
	but by summing over all of $E_{c, c'}$, it becomes compatible.
}
\label{figure:ring-vrack-nullhomotopy}
\end{figure}

Define
	$H_n :
	V^{c,P,P;k}_n \to
	V^{c,P,P;k}_{n+1}$ as follows:
Suppose $(m,v_1, \ldots, v_n,p)$ as above with $m,p \in P$, $v_1, \ldots, v_{i-1}$
averaged and $v_i$ antiaveraged, and $v_{i+1}, \ldots, v_n \in c$.
Recall the notation $E_{c,c'}$ from \autoref{notation:group-elements}.
	Define
	\begin{equation}
	\begin{aligned}
	\label{equation:homotopy-ring}
	H_n(m, v_1, \ldots, v_n, p) &:= \frac{1}{|E_{c,c'}|} \cdot \left(
		H^{u}_n(m, v_1, \ldots, v_n, p) +
	H^{d}_n(m, v_1, \ldots, v_n, p) \right) \\
	H^{u}_n(m, v_1, \ldots, v_n, p) &:= \sum_{\substack{ (x,g) \in 
	E_{c,c'} \\  x \succ g = w_1 \cdots w_t}}
		\sum_{s = 1}^t (-1)^{i+1}(m , v_1, \ldots,
			v_{i-1}, w_s, (w_1 \cdots w_s) \triangleright
			v_i,v_{i+1}
		\ldots, v_n, w_s^{-1} p) \\
	H^{d}_n(m, v_1, \ldots, v_n, p) &:= \sum_{\substack{ (x,g) \in 
	E_{c,c'} \\  x \succ g = w_1 \cdots w_t}}
\sum_{s = 1}^t(-1)^{i+1} (m , v_1, \ldots,
			v_{i-1}, (w_1 \cdots w_{s-1}) \triangleright v_i, w_s ,
	v_{i+1}, \ldots, v_n, w_s^{-1} p).
	\end{aligned}
\end{equation}
	and extend $H_n$ to all of 	
	$V^{c,P,P;k}_n$ by linearity.

	To show $H_n$ forms a nullhomotopy we concretely wish to show
$\delta_{n+1} H_n + H_{n-1}
\delta_n -\id|_{F^e} \subset F^{e-1}.$	
The reader may consult \autoref{figure:ring-vrack-nullhomotopy} for a
visualization of this chain homotopy.

	We will check this by writing the above as a sum of terms.
	The main point is that, for $(m, v_1, \ldots, v_n, p)$ in $F^e$
	with $v_1, \ldots, v_{i-1}$ are averaged and $v_i$ is antiaveraged,
	we have
	\begin{align}
		\label{equation:ring-telescope}
	\frac{1}{|E_{c,c'}|} \cdot
	\left((-1)^{i}\delta^r_{n+1,i}H^{u}_n +
	(-1)^{i+1}\delta^r_{n+1,i+1}H^{d}_n \right)
	(m, v_1, \ldots, v_n,p)
	=
	(m, v_1, \ldots, v_n,p),
	\end{align}
	and the remaining terms in the expression for 
$\delta_{n+1} H_n + H_{n-1}
\delta_n$ sum an element of $F^{e-1}$.
We next verify \eqref{equation:ring-telescope}.
One key fact we will use is that for $y,y' \in c$ with the same image in $c/c'$,
and $p \in M$, we have $y \cdot p = y' \cdot p$ by definition of $M$.
In particular, $y' \cdot y^{-1} \cdot p = p$.
Using this and expanding the above, and simplifying the telescoping sum gives
	\begin{equation}
		\label{equation:ring-telescope-computation}
	\begin{aligned}
&\frac{1}{|E_{c,c'}|} \cdot
\left((-1)^{i}\delta^r_{n+1,i}H^{u}_n +
(-1)^{i+1}\delta^r_{n+1,i+1}H^{d}_n -\id \right) (m, v_1, \ldots, v_n, p) \\
	&=
-(m, v_1, \ldots, v_n, p) +
\frac{1}{|E_{c,c'}|}(
\\
& -\sum_{\substack{ (x,g) \in 
	E_{c,c'} \\  x \succ g = w_1 \cdots w_t}}
 \sum_{s = 1}^t (m, v_1, \ldots,
			v_{i-1}, (w_1 \cdots w_s) \triangleright v_i,
	\ldots, v_n, 
	(( (w_1 \cdots w_s) \triangleright v_i) v_{i+1} \cdots v_n) \triangleright w_s
	\cdot w_s^{-1} \cdot p)\\
	&+ \sum_{\substack{ (x,g) \in 
	E_{c,c'} \\  x \succ g = w_1 \cdots w_t}}
 \sum_{s = 1}^t (m, v_1, \ldots,
			v_{i-1}, (w_1 \cdots w_{s-1}) \triangleright v_i, v_{i+1},
\ldots, v_n, 
((v_{i+1} \cdots v_n) \triangleright w_s) \cdot w_s^{-1} \cdot
p))\\
	&=
-(m, v_1, \ldots, v_n, p)+
	\frac{1}{|E_{c,c'}|}\left(	-\sum_{\substack{ (x,g) \in 
	E_{c,c'} \\  x \succ g = w_1 \cdots w_t}}
\sum_{s = 1}^t (m, v_1, \ldots,
			v_{i-1}, (w_1 \cdots w_s) \triangleright v_i,
	\ldots, v_n, p) \right.\\
	&\qquad \left.+ \sum_{\substack{ (x,g) \in 
	E_{c,c'} \\  x \succ g = w_1 \cdots w_t}}
 \sum_{s = 1}^t (m, v_1, \ldots,
			v_{i-1}, (w_1 \cdots w_{s-1}) \triangleright v_i, v_{i+1},
\ldots, v_n, p)\right) \\
&= -(m, v_1, \ldots, v_n, p) + (m, v_1, \ldots, v_n, p) \\
&\qquad \qquad- \frac{1}{|E_{c,c'}|} \sum_{\substack{ (x,g) \in 
	E_{c,c'} \\  x \succ g = w_1 \cdots w_t}}
(m, v_1, \ldots, v_{i-1}, (w_{1} \cdots w_t) \triangleright v_i, v_{i+1},
	\ldots, v_n, p) \\
	&= - \frac{1}{|G^{c'}_{c'}| \cdot |{G}^{c'}_{c}|}\sum_{x \in
	G^{c'}_{c'}, g \in {G}^{c'}_{c}} 
	(m, v_1, \ldots, v_{i-1}, (x \triangleright g) \triangleright v_i, v_{i+1},
	\ldots, v_n, p)\\
	&=- \frac{1}{|G^{c'}_{c'}|} \sum_{x \in
	G^{c'}_{c'}} \left(
		(m, v_1, \ldots,
			v_{i-1}, U_{c'} (v_i), v_{i+1},
	\ldots, v_n, p) \right) = 0,
	\end{aligned}
\end{equation}
	where the final expression vanishes since we are assuming $v_i$ is
	antiaveraged so $U_{c'}(v_i) = 0$.

	So, it is enough to show the remaining terms in the expression for
$\delta_{n+1} H_n + H_{n-1}
\delta_n$, other than those in \eqref{equation:ring-telescope},
cancel when evaluated on $(m,v_1,\ldots, v_n,p)$ with $v_1, \ldots, v_{i-1}$
in the averaged basis and $v_i$ in the antiaveraged basis.
	Indeed, expanding term by term, we next claim
	\begin{align}
		\label{equation:left-small}
		\delta^l_{n+1,j} H_n(m, v_1, \ldots, v_n, p) &=
		-H_{n-1}\delta^l_{n,j}(m, v_1, \ldots, v_n, p) \text{ for } 1
		\leq j < i \\
		\label{equation:right-small}
		\delta^r_{n+1,j} H_n(m, v_1, \ldots, v_n, p) &=  -H_{n-1}\delta^r_{n,j}(m, v_1, \ldots, v_n, p) \text{ for } 1
		\leq j < i\\
		\label{equation:left-big}
		\delta^l_{n+1,j+1} H_n(m, v_1, \ldots, v_n, p) &=
		H_{n-1}\delta^l_{n,j}(m, v_1, \ldots, v_n, p) \text{ for }i+1 \leq j \leq n\\
		\label{equation:right-big}
		\delta^r_{n+1,j+1} H_n(m, v_1, \ldots, v_n, p) &=  H_{n-1}\delta^r_{n,j}(m, v_1, \ldots, v_n, p)
		\text{ for }i+1 \leq j \leq n
	\end{align}
	on $F^e$, modulo $F^{e-1}$.
	Let us start by explaining the proof of \eqref{equation:left-big}. The other
	three relations are similar, but easier to verify.
	We note that because we are working modulo $F^{e-1}$, we are free to
	assume that $v_j \in c'$, as otherwise the terms above will lie in
	$F^{e-1}$.
	We let $i+ 1 \leq j \leq n$ and hence $v_j \in c$.
	We can separately show 
\begin{align}
	\label{equation:delta-and-h-down}
	\delta^l_{n+1,j+1} H_n^{d}(m, v_1, \ldots, v_n, p) &=
	H_{n-1}^{d}\delta^l_{n,j}(m, v_1, \ldots, v_n, p) \\
	\label{equation:delta-and-h-up}
	\delta^l_{n+1,j+1} H_n^{u}(m, v_1, \ldots, v_n, p) &=
	H_{n-1}^{u}\delta^l_{n,j}(m, v_1, \ldots, v_n, p).
\end{align}

Let us just explain \eqref{equation:delta-and-h-down}, as
\eqref{equation:delta-and-h-up} is
		similar.
		Expanding the two sides, we obtain
		\begin{equation}
	\begin{aligned}
		\label{equation:h-then-delta-ring-left-big}
		&\delta^l_{n+1,j+1} H_n^{d}(m, v_1, \ldots, v_n, p)  \\
		&= \sum_{\substack{ (x,g) \in 
	E_{c,c'} \\  x \succ g = w_1 \cdots w_t}}
 \sum_{s = 1}^t(-1)^{i+1}
		(m v_j , v_j
			\triangleright  v_1, \ldots,
		v_j \triangleright v_{i-1}, v_j \triangleright ((w_1 \cdots
			w_{s-1})
	\triangleright v_i), v_j \triangleright  w_s , \\
	& \qquad \qquad \qquad \qquad v_j \triangleright
v_{i+1}, \ldots, v_j \triangleright v_{j-1}, v_{j+1}, \ldots, v_n, w_s^{-1} p)
\\
&= \sum_{\substack{ (x,g) \in 
	E_{c,c'} \\  x \succ g = w_1 \cdots w_t}}
 \sum_{s = 1}^t(-1)^{i+1} (m v_j , v_j
			\triangleright  v_1, \ldots,
			v_j \triangleright v_{i-1},((v_j \triangleright w_1)
			\cdots (v_j \triangleright w_{s-1}))
	\triangleright (v_j \triangleright v_i), v_j \triangleright w_s ,
	\\
	& \qquad \qquad \qquad \qquad v_j \triangleright
v_{i+1}, \ldots, v_j \triangleright v_{j-1}, v_{j+1}, \ldots, v_n, w_s^{-1} p)
	\end{aligned}
\end{equation}
	and 
	\begin{equation}
	\begin{aligned}
		\label{equation:delta-then-h-ring-left-big}
		&H_{n-1}^{d}\delta^l_{n,j}(m, v_1, \ldots, v_n, p) \\
		&= H_{n-1}^{d} (m v_j, v_j \triangleright v_1, \ldots, v_j \triangleright
		v_{j-1}, v_{j+1}, \ldots, v_n, p) \\
		&= 
		\sum_{\substack{ (x,g) \in 
	E_{c,c'} \\  x \succ g = w_1 \cdots w_t}}
 \sum_{s = 1}^t(-1)^{i+1} (mv_j, v_j
			\triangleright  v_1, \ldots,
			v_j \triangleright v_{i-1}, (w_1 \cdots w_{s-1})
		\triangleright (v_j \triangleright v_i), w_s, \\
		& \qquad \qquad \qquad \qquad v_j \triangleright
	v_{i+1}, \ldots, v_j \triangleright v_{j-1}, v_{j+1}, \ldots, v_n, w_s^{-1} p) \\
&=\sum_{\substack{ (x,g) \in 
	E_{c,c'} \\  x \succ g = w_1 \cdots w_t}}
 \sum_{s = 1}^t(-1)^{i+1} (mv_j, v_j
			\triangleright  v_1, \ldots,
		v_j \triangleright v_{i-1}, ((v_j \triangleright w_1) \cdots (v_j
		\triangleright w_{s-1}))
	\triangleright (v_j \triangleright v_i), v_j \triangleright w_s, \\
 & \qquad \qquad \qquad \qquad 
	v_j \triangleright
v_{i+1}, \ldots, v_j \triangleright v_{j-1}, v_{j+1}, \ldots, v_n, w_s^{-1} p).
	\end{aligned}
\end{equation}
	The last equation used that the $E_{c,c'}$ is closed under the
	bijective operation $(x,g) \mapsto (v_j \cdot x,g)$, where $v_j \cdot x$
	denotes multiplication in $G^{c'}_{c'}$, which sends
	\begin{align*}
	w_1 \cdots w_t = x \succ g
\mapsto (v_j \cdot x) \succ g = (v_j \triangleright w_1) \cdots (v_j
\triangleright w_t).
	\end{align*}
Since the final lines in 
	\eqref{equation:h-then-delta-ring-left-big}
	and \eqref{equation:delta-then-h-ring-left-big} agree, we obtain
	\eqref{equation:delta-and-h-down}.
	As mentioned above, the verification of
	\eqref{equation:delta-and-h-up} is similar to that of
	\eqref{equation:delta-and-h-down}, and hence summing these two
	establishes
	\eqref{equation:left-big}.
	The verifications of \eqref{equation:right-big} and
	\eqref{equation:right-small} are relatively easier, and do not involve
	any reordering of the summations, but follow from the fact that $w_s$
	and $w'_s$ act the same way on $P$ for $w'_s$ in the same $c'$ orbit as
	$w_s$.
	The verification of \eqref{equation:left-small}
	is also straightforward.
	One point that is important to note in the verification of
	\eqref{equation:left-small} and \eqref{equation:right-small}, is that
	$\delta_{i+1,j}^l(m,v_1, \ldots, v_n,m)$ and 
	$\delta_{i+1,j}^r(m,v_1, \ldots, v_n,m)$ are elements such that the first $i-2$ coordinates in
	$k\{c\}$ are averaged and the $i-1$th entry is antiaveraged. Hence, when
	we apply $H_{n-1}$ to these elements, 
	the homotopy will insert $w_e$ at the $i-1$th and $i$th slots.
	This is in contrast to $H_n$, which inserts $w_e$ at the $i$th and 
	$i+1$th slots.
	
Next, we observe
	\begin{align}
		\label{equation:ring-cancel-setup}
		\delta^l_{n+1,i+1}H^{d}_n(m, v_1, \ldots, v_n, p) =
		\delta^l_{n+1,i}H^{u}_n(m, v_1, \ldots, v_n, p),
	\end{align}
	as both are equal to 
$\sum_{\substack{ (x,g) \in 
	E_{c,c'} \\  x \succ g = w_1 \cdots w_t}}
 \sum_{s = 1}^t (-1)^{i+1}(mw_s, v_1, \ldots,
	v_{i-1}, (w_1 \cdots w_{s}) \triangleright v_i,
\ldots, v_n, w_s^{-1} p).$ 
Finally, one can also verify
		\begin{equation}
	\begin{aligned}
		\label{equation:anti-averaged-vanishing}
		\delta^l_{n,i}(m, v_1, \ldots, v_n, p) &= \delta^r_{n,i}(m, v_1,
		\ldots, v_n, p) = 0 \\
\delta^l_{n+1,i+1} H^{u}_n(m, v_1, \ldots, v_n, p) &=
\delta^r_{n+1,i+1} H^{u}_n(m, v_1, \ldots, v_n, p) = 0 \\
\delta^l_{n+1,i}H^{d}_n(m, v_1, \ldots, v_n, p) &=
\delta^r_{n+1,i}H^{d}_n(m, v_1, \ldots, v_n, p) =0
	\end{aligned}
\end{equation}
using that $v_1, \ldots, v_{i-1}$ are averaged $v_i$ is antiaveraged, and the
actions of elements of $c$ on $P$ only depends on their $c'$ orbit.
For example, if $v_i = \sum_y \alpha_y y$ with $y \in c'$ all in the
same orbit as some fixed $z \in c$ (using the assumption that $v_i$ was an
antiaveraged basis element),
	we have
	\begin{align*}
	\delta^l_{n,i}(m, v_1, \ldots, v_n, p) &= \sum_y \alpha_y(m
			\cdot y,
		y \triangleright v_1, \ldots, y \triangleright v_{i-1}, v_{i+1},
	\ldots, v_n, p)\\
	&= \sum_y \alpha_y(m \cdot y,
		v_1, \ldots, v_{i-1}, v_{i+1},
	\ldots, v_n, p) \\
	&= ((\sum_y \alpha_y) m \cdot z, v_1, \ldots, v_{i-1}, v_{i+1}, \ldots,
	v_n,p) \\
	&= 0
	\end{align*}
	since  $\sum_y \alpha_y = 0$.
	The verifications of the other statements in
	\eqref{equation:anti-averaged-vanishing} have similar proofs.
	Finally, summing
	\eqref{equation:ring-telescope},
	\eqref{equation:left-small},
	\eqref{equation:right-small},
	\eqref{equation:left-big},
	\eqref{equation:right-big},
	\eqref{equation:ring-cancel-setup}, and
	\eqref{equation:anti-averaged-vanishing},		
and keeping track of signs yields the desired statement 
that $\delta_{n+1} H_n + H_{n-1}\delta_n =\id$.
\end{proof}

\subsection{Chain homotopies for Hurwitz module bar constructions}
\label{subsection:module-chain-homotopies}

Having verified an equivalence relevant for Hurwitz spaces in
\autoref{proposition:ring-chain-homotopy}, we next compute an
equivalence relevant for bijective Hurwitz modules in
\autoref{proposition:module-chain-homotopy}.
For the main result of this section relating two bar constructions, we will have
to invert the order of a group $G^{c'}_{S}$ coming from the action of a subrack on a Hurwitz module,
which plays an analogous role to that played by the group $G^{c'}_c$ in the previous
subsection.
It will take a bit of notation to define this; the definition is given in
\autoref{definition:module-structure-group}.

\begin{notation}
	\label{notation:looping-operator}
	Let $c$ be a rack,
$S = ({\Sigma^1_{g,f}}, \{T_n\}_{n \in \mathbb
Z_{\geq 0}}, \{\psi_n\}_{n \in \mathbb
Z_{\geq 0}})$ a bijective Hurwitz module over $c$ and $c' \subset c$ an
$S$-component.
Let $k$ be an arbitrary ring and let $M :=
\pi_0(\Hur^{c/c'})[\alpha_{c'/c'}^{-1}]$.
	With notation as in \autoref{notation:module-homotopy-notation},
	fix $1 \leq \rho \leq 2g + f$.
	Given $(m,v_1, \ldots, v_n,s) \in W^{c,S,M;k}$, with $m \in M, s \in
	T_0, v_1, \ldots, v_n \in c$, suppose $i$ is the minimal index such that
	$q_{(m,v_1, \ldots, v_n,s)}(i) = \rho$. 
	Define
	$\iota^\rho_x(m, v_1, \ldots, v_n,s) := (mx^{-1}, x \triangleright^{-1}v_1, \ldots, x \triangleright^{-1}v_{i-1}, x, v_i, \ldots, v_n,s)$
where 
\begin{align*}
	q_{(m, x \triangleright^{-1} v_1, \ldots, x \triangleright^{-1} v_{i-1}, x, v_i, \ldots, v_n,s)}(i') :=
q_{(m, v_1, \ldots, v_n,s)}(i'-\epsilon)
\end{align*}
where $\epsilon = 0$ if $i' \leq i$ and
$\epsilon = 1$ if $i' >i$.
\end{notation}
\begin{remark}
	\label{remark:}
	We note that 
$\iota^\rho_x(m, v_1, \ldots, v_n, s)$ can be characterized as the unique tuple
with $x$ in the $i$th position and
$q_{\iota^\rho_x(m, v_1, \ldots, v_n, s)}(i) = \rho$
such that
$d^l_{n,i} \iota^\rho_x(m, v_1, \ldots, v_n, s) = (m, v_1, \ldots, v_n,s)$.
\end{remark}

\begin{definition}
	\label{definition:module-structure-group}
	With notation as in 
\autoref{notation:looping-operator},
	for each $x \in c', 1 \leq \rho \leq 2g + f$, the operation
\begin{align*}
	w \cdot^\rho (m, v_1, \ldots, v_n,s) := d^r_{n,i} \iota_w^\rho(m,v_1, \ldots,
	v_n,s)
\end{align*}
defines
automorphism 
$w \cdot^\rho: W^{c,S,M;k}\to W^{c,S,M;k}$.
We suggest the reader consult \autoref{figure:module-rectangle-action} for a
visual depiction of what this action means.

Consider 
the subgroup of automorphisms 
$G^{c'}_S \subset \on{Aut}(W^{c,S,M;\mathbb Z})$ ranging over all actions of the form $w_1 \cdot^{\rho_1} \cdots
w_k \cdot^{\rho_k}$ so that the induced map on $M$ is the identity. (This is
	equivalent to the condition that the tuple of elements of $c'/c'$ associated to
$w_1 \cdots w_k$ is the same as the corresponding tuple after ``looping $w_i$
around boundary of the $\rho_i$th rectangle,'' see \autoref{remark:looping}.)
We define $G^{c'}_S$ to be the
{\em module structure group} associated to the bijective Hurwitz module $S$.
For any ring $k$, any element of $G^{c'}_S$ also determines an element of
$\on{Aut}(W^{c,S,M;k})$ via base change along $k \to \mathbb Z$.
For $(m, v_1, \ldots, v_n, s) \in W^{c,S,M;k}$ and $h \in G^{c'}_S$ we use 
$(m, v_1, \ldots, v_n, s)^h$ to denote the result of acting on 
$(m, v_1, \ldots, v_n, s)$ by $h$, thought of as an element of
$\on{Aut}(W^{c,S,M;k})$.
\end{definition}
\begin{remark}
	\label{remark:looping}
	Loosely speaking, the operation $x \cdot^\rho$ for $x,y \in c'$ corresponds to looping $x$ around
	the $\rho$th rectangle.
\end{remark}

Soon, we will want to invert the order of
$G^{c'}_S$. In order to make sense of this, we will need to know
it is a finite group, which we now verify.
\begin{lemma}
	\label{lemma:module-structure-group-finite}
	For $c$ a finite rack, $S = ({\Sigma^1_{g,f}}, \{T_n\}_{n \in \mathbb
Z_{\geq 0}}, \{\psi_n\}_{n \in \mathbb
Z_{\geq 0}})$ a finite bijective Hurwitz module over $c$,
and $c' \subset c$ an $S$ component.
Then, the group $G^{c'}_S$ is a finite group.
\end{lemma}
\begin{proof}	
	Each element of $G^{c'}_S$ acts 
	$W^{c,S,M;\mathbb Z}$ in a specific way.
	Namely, for a fixed value of $m \in M$, there is a basis of
	the subset of
$W^{c,S,M;\mathbb Z}_n$
spanned by elements of the form $(m, x_1, \ldots, x_n,s)$ with $s \in T_0, x_1, \ldots, x_n
\in c$, with $m \in M$ a fixed value. 
There are $|c|^n \cdot |T_0|$ such elements 
as $x_i \in c$ and $s \in T_0$ vary.
By construction, the action of $G^{c'}_S$ is trivially on $M$.
Therefore, the action of $G^{c'}_S$ on 
$W^{c,S,M;\mathbb Z}$ factors through
a subgroup of
$\prod_{n \geq 0} \aut( c^n \times T_0)$.

To conclude, it suffices to show the action of $G^{c'}_S$ on 
$W^{c,S,M;\mathbb Z}$ is determined by the action on
$W^{c,S,M;\mathbb Z}_n$ for a fixed finite set of values of $n$.
That is, we wish to show
there is some constant $N_0$ so that for
$n > N_0$, the
action of $G^{c'}_S$ on $W^{c,S,M;\mathbb Z}_n$
is determined by its action on
$W^{c,S,M;\mathbb Z}_m$ for $m \leq N_0$.
Suppose that every element of $G^c_{c'}$ can be written as a product of $K$
elements.
Then we claim we may take $N_0 = (2g+f)(K+1)$.
By choosing $N_0$ this way, we claim can find an element of
$W^{c,S,M;\mathbb Z}_{N_0}$
so that the product of the elements in the $\rho$th each scanned rectangle is
$g_\rho \in G^c_{c'}$ and each scanned rectangle contains an element $x_\rho \in
c$:
Said more precisely, for any sequence $g_1, \ldots, g_{2g+f} \in G^{c}_{c'},
x_1, \ldots, x_{2g+f} \in c,$
we can choose
$(m, v_1, \ldots, v_{N_0},s)$ with 
$q_{(m,\ldots, s)}(\rho(K+1)+j) = \rho$ for $1 \leq j \leq K+1$,
$v_{\rho(K+1)+1} = x_\rho$, and $v_{\rho(K+1)+1} \cdots v_{\rho(K+1) + K+1} =
g_\rho \in G^{c}_{c'}$ for each $1 \leq \rho \leq 2g+f$.
Indeed, the above is possible because we can choose $v_{\rho(K+1)+2} \cdots v_{\rho(K+1) + K+1}$ to have product
$\alpha_{x_\rho}^{-1} g_\rho$ by definition of the constant $K$.

Now, we wish to show that knowing the action of a given element of $G^{c'}_S$ on
all such elements $(m, v_1, \ldots, v_{N_0},s)$ as above determines the action on all elements of
$W^{c,S,M;\mathbb Z}_n$ for arbitrary $n$.
Note that if two elements $y_i$ and $y_j$ satisfy
$q_{(m,y_1, \ldots, y_n,s)}(i)= q_{(m,y_1, \ldots, y_n,s)}(j)$,
(meaning that $y_i$ and $y_j$ lie in the same rectangle after scanning,)
then the action of $G^{c'}_S$ on $y_i$ and $y_j$ acts through the same element
of $G^{c'}_c$, and
this action only depends on 
the value of $s$ and the product of the elements in each of the $\rho$ rectangles $1
\leq \rho \leq 2g + f$
(those elements $y_j$ with $q_{(m,y_1, \ldots, y_n,s)}(j)  = \rho$), as follows
from the formula for the action given in
\autoref{definition:module-structure-group}.
Using the collection of elements $(m, v_1, \ldots, v_{N_0},s)$ described above, if we fix the product of the elements in the
$\rho$th rectangle to be $g_\rho$, the action of an element of $G^{c'}_S$ acts
on the $\rho$th rectangle by an element of $G^{c'}_c$ whose value on any
$x_\rho \in c$ is determined by our assumption.
Therefore, the action on 
$W^{c,S,M;\mathbb Z}_n$ is determined by its actions on those tuples $(m, v_1,
\ldots, v_{N_0},s)$ described above, as we wished to show.
\end{proof}

We next state our main equivalence relating to bar constructions of bijective Hurwitz
modules.
\begin{proposition}
	\label{proposition:module-chain-homotopy}
	Let $c$ be a finite rack, 
	$S$ a finite bijective Hurwitz module over $c$,
	and
	${c'} \subset c$ be an $S$-component of $c$.
	There is an equivalence
	\begin{align*}
		& \left( H_0(A_{c/c'})[\alpha_{c'/c'}^{-1}] \otimes_{A_c[\alpha_{c'}^{-1}]}
		A_{c,S}
	[\alpha_{c'}^{-1}]\right)[|G^{c}_{c'}|^{-1},|G^{c'}_{c}|^{-1},|G^{c'}_{S}|^{-1}]
		\\
	&\simeq
	\left( H_0(A_{c/c'})[\alpha_{{c'/c'}}^{-1}] \otimes_{A_{c/{c'}}[\alpha_{{c'}/{c'}}^{-1}] }
	A_{c/{c'},S/{c'}}[\alpha_{{c'}/{c'}}^{-1}]
\right)[|G^{c}_{c'}|^{-1},|G^{c'}_{c}|^{-1},|G^{c'}_{S}|^{-1}].
	\end{align*}
\end{proposition}
We give the proof after introducing some notation.
	\begin{notation}
		\label{notation:module-homotopy-notation}
		Let $c$ be a finite rack, let $S$ be a finite bijective Hurwitz module over
		$c$ and let $c' \subset c$ be an $S$ component.
		Let $k := \mathbb
		Z[|G^{c}_{c'}|^{-1},|G^{c'}_{c}|^{-1},|G^{c'}_{S}|^{-1}]$.
	Let 
	$M :=\pi_0(\Hur^{c/c'})[\alpha_{c'/c'}^{-1}]$.

	There is a projection
	$W^{c,S,M;k} \to W^{c/c',S/c',M;k}$. This has a section given by a
	map
	$W^{c/c',S/c',M;k} \to W^{c,S,M;k}$ defined as follows.
	The source is spanned by elements of the form $(m, \overline{v}_1,
	\cdots, \overline{v}_n,p)$
	where
	$m \in M$,
	$v_i \in c$ with image $\overline{v} \in c/c'$,
	and $p \in S/c',$ which we can think of as a $c'
	\times \pi_1(\Sigma_{g,f}^1)$ orbit of $T_0$.
	The section is given by 
	$(m, \overline{v}_1, \ldots, \overline{v}_n,p) \mapsto (m,U_{c'}(v_1),
		\ldots, U_{c'}(v_n),\frac{1}{|p|}
		\sum_{t \in p} t)$.
	\end{notation}
\begin{proof}[Proof of \autoref{proposition:module-chain-homotopy} assuming \autoref{lemma:module-s-homotopy} and
	\autoref{lemma:module-c-homotopy}]
	First, by 
\autoref{lemma:chains-identification} we can identify the two sides of the
statement with
$W^{c,S,M;k}$ and $W^{c/c',S/c',M;k}$,
so we only need show these two complexes are homotopic.
This follows from composing the homotopies defined below in 
\autoref{lemma:module-s-homotopy}
and
\autoref{lemma:module-c-homotopy}.
\end{proof}

To conclude the proof of \autoref{proposition:module-chain-homotopy} it remains
to prove
\autoref{lemma:module-s-homotopy} and
\autoref{lemma:module-c-homotopy}.
This will occupy the remainder of the section.

We next define $\overline{W}^{c,S,M;k}$ to
	as the subcomplex invariant under the action of 
	$G^{c'}_{S}$ in \autoref{notation:loop-averaged}.
	Then, \autoref{lemma:module-s-homotopy}
	shows 
${W}^{c,S,M;k}$ is homotopic to 
$\overline{W}^{c,S,M;k}$ and then 
\autoref{lemma:module-c-homotopy} shows
$\overline{W}^{c,S,M;k}$ is homotopic to $W^{c/c',S/c',M;k}$.

\begin{notation}
	\label{notation:loop-averaged}
	With notation as in \autoref{notation:module-homotopy-notation},
	there is an averaging operator
	$U^{c'}_S: W^{c,S,M;k} \to W^{c,S,M;k}$ which sends 
	$(m,v_1, \ldots, v_n, s) \mapsto \frac{1}{|G^{c'}_{S}|} \sum_{h \in
G^{c'}_{S}} (m, v_1, \ldots, v_n,s)^h$, where the notation 
$(m, v_1, \ldots, v_n,s)^h$ denotes the action defined in
\autoref{definition:module-structure-group}.
Let $\overline{W}^{c,S,M;k}$ denote the image of $U^{c'}_S$.
\end{notation}

\begin{notation}
	\label{notation:e-set-for-S}
	With notation as in 
	\autoref{definition:module-structure-group},
	for each element $h \in G^{c'}_S$ choose a representative way to write $h$ in the form
	$w_{i_h}^h \cdot^{\rho_{i_h}^h} \cdots w_2^h \cdot^{\rho_2^h} w_1^h
	\cdot^{\rho_1^h}$
	with each $w_i^h \in c'$ and $1 \leq \rho_i^h \leq 2g+f$.
	Define the set
	\begin{align*}
		E_{c',S} := \{ (z,h)	: h \in G^{c'}_S, 
		z = (z_1, \ldots, z_{2g+f}) \in (G^{c'}_{c'})^{2g+f}\}
	\end{align*}
	and use the notation
	$z \succ h$ to denote the tuple
	$(\rho^h_{i_h}, \ldots, \rho^h_1; z_{\rho_{i_h}^h} \triangleright
	w_{i_h}^h, \ldots, z_{\rho^h_1} \triangleright w_1^h)$.
\end{notation}

\begin{remark}
	\label{remark:size-of-E}
	By \autoref{lemma:module-structure-group-finite}, 
	$|E_{c',S}|$ is a finite set and 
	any prime dividing its order divides either $|G^{c'}_S|$ or
	$|G^{c'}_{c'}|$.
Note that there is a surjective map $G^{c'}_{c} \subset G^{c'}_{c'}$ coming from
restricting the automorphism of $c$ to one of $c'$, so any prime dividing
$|G^{c'}_{c'}|$
also divides $|G^{c'}_c|$.
\end{remark}

We now verify that each element of $E_{c',S}$ corresponds to an element of
$G^{c'}_S$.
\begin{lemma}
	\label{lemma:z-action-well-defined}
	For
	$h \in G^{c'}_{S}$ in the form
	$w := w^h_{i_h} \cdot^{\rho^h_{i_h}} \cdots w^h_1 \cdot^{\rho^h_1}$
	and any $z := (z_1, \ldots, z_{2g+f}) \in (G^{c'}_{c'})^{2g+f}$,
	we also have that	
	$w^z := (z_{\rho^h_{i_h}} \triangleright w^h_{i_h}) \cdot^{\rho^h_{i_h}} \cdots
		(z_{\rho^h_1} \triangleright w^h_1) \cdot^{\rho^h_1}$
acts by an element of $G^{c'}_S$.
\end{lemma}
\begin{proof}
Suppose $w$ above acts by an element
$h \in G^{c'}_S$ and $w^z$ acts by an element $h^z$. We claim $h^z \in G^{c'}_{S}$.
Indeed, using 
\autoref{lemma:dependence-on-image}, the action of
$w \cdot^\rho$ on $M$ agrees with the action of
$(x \triangleright w) \cdot^\rho$ on $M$ for any $x \in c'$.
From this it follows that 
$h^z$ acts the same way on $M$ that $h$ acts.
Since $h$ acts trivially on $M$, $h^z$ acts trivially on $M$ as well, implying $h^z \in
G^{c'}_S$.
\end{proof}

\begin{figure}
	\includegraphics[scale=.4]{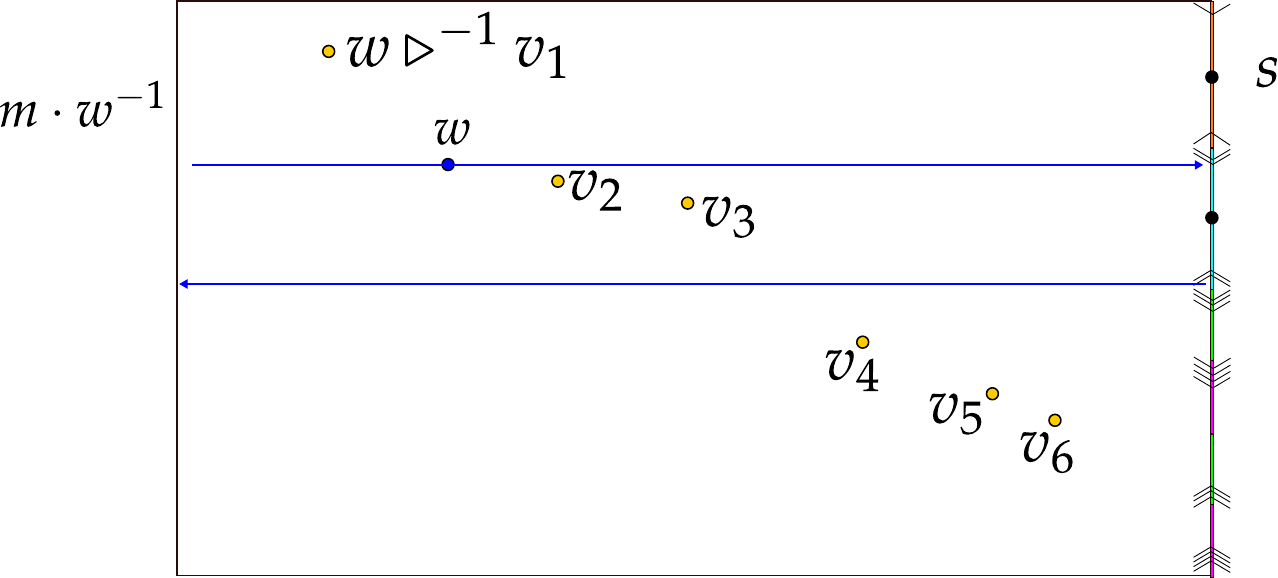}
\caption{
	This is a visualization of action defined in
\autoref{definition:module-structure-group}. Specifically, it depicts the action
	of $w \cdot^\rho$ with $w \in c$ and $\rho = 2$, corresponding to the
	second rectangle from the top.
It is used in the proof of \autoref{lemma:module-s-homotopy}.
	The homotopy $K$ there can be thought of as applying a sequence of
	such homotopies, corresponding to elements of
$E_{c',S}$,
and then averaging over these $|E_{c',S}|$ operations.
}
\label{figure:module-rectangle-action}
\end{figure}

With all the above notation set up,
we verify the first of two homotopies needed for
\autoref{proposition:module-chain-homotopy}.

\begin{lemma}
	\label{lemma:module-s-homotopy}
	With notation as in \autoref{notation:module-homotopy-notation},
	the inclusion 
	$\overline{W}^{c,S,M;k} \to W^{c,S,M;k}$ induces a homology equivalence.
\end{lemma}
\begin{proof}
	We prove this by exhibiting a suitable chain homotopy.
Any element of 
$W^{c,S,M;k}$ can be written as a linear combination of elements of the form
$(m, v_1, \ldots, v_n,s)$ with $m \in M, s \in T_0, v_i \in k\{c\}$.
We will produce a nullhomotopy of
$W^{c,S,M;k}/\overline{W}^{c,S,M;k}$.
Using notation from \autoref{notation:module-chain-complex}, 
\autoref{definition:module-structure-group}, and 
\autoref{notation:e-set-for-S},
we define
$K_n: W^{c,S,M;k}_n/\overline{W}_n^{c,S,M;k} \to
W^{c,S,M;k}_{n+1}/\overline{W}_{n+1}^{c,S,M;k}$
by
\begin{equation}
\begin{aligned}
	\label{equation:homotopy-module-averaged}
	&K_n(m, v_1, \ldots, v_n, s) := 
	\frac{1}{|E_{c',S}|} \cdot 
	\sum_{\substack{(z,h) \in
	E_{c',S} \\ z \succ h = 
(\rho_\ell^{(z,h)}, \ldots, \rho_1^{(z,h)}; x_\ell^{(z,h)}, \ldots, x_1^{(z,h)})}} 
			\sum_{e = 1}^{\ell} 
K^{(z,h)}_{e,n}(m, v_1, \ldots, v_n, s) \\
&K^{(z,h)}_{e,n}(m, v_1, \ldots, v_n, s) :=\\
	&
	(-1)^{i_e^{(z,h)} -1}
	\iota^{\rho_e^{(z,h)}}_{x_e^{(z,h)}}
	(x_{e-1}^{(z,h)} \cdot^{\rho_{e-1}^{(z,h)}}
		(x_{e-2}^{(z,h)} \cdot^{\rho_{e-2}^{(z,h)}} \cdots (x_1^{(z,h)}
				\cdot^{\rho_1^{(z,h)}}(m,
v_1, \ldots, v_n, s)) \cdots )),
	\end{aligned}
\end{equation}
where above $i_e^{(z,h)}$ is the minimal index such that $q_{(m, v_1, \ldots,
v_n,s)}(i_e^{(z,h)}) = \rho_e^{(z,h)}$.
Observe that $|E_{c',S}|$ is invertible in $k$ via the definition of $k$ and the computation of the size of
$E_{c',S}$ in \autoref{remark:size-of-E}.
We use the filtration $F^\bullet$ defined so that
$F^\omega \subset W^{c,S,M;k}$ is the subcomplex spanned by those tuples $(m, v_1,
\ldots, v_n,s)$ so that at most $e$ elements among $v_1, \ldots, v_n$
lie in
in $k\{c- c'\}$.
With this definition in hand, 
we claim 
\begin{align}
	\label{equation:chain-homotopy-expression-s}
	(d_{n+1} K_n + K_{n-1} d_n - \id)(m, v_1, \ldots, v_n,s) = -U^{c'}_S(m, v_1,
	\ldots, v_n, s)
\end{align}
on the associated graded of the filtration $F^\omega$ (meaning that we assume
the input lies in $F^\omega$ and ignore terms in $F^{\omega-1}$).
The claim produces a nullhomotopy of the complex
$W^{c,S,M;k}/\overline{W}^{c,S,M;k}$ on the associated graded of $F^\bullet$, so
implies that the complex is nullhomotopic, which will conclude the proof.

Now, the verification of 
\eqref{equation:chain-homotopy-expression-s} proceeds in a similar fashion to
the homotopies we saw earlier in \autoref{subsection:chain-homotopy-space}.
Namely, one can verify via a telescoping argument similar to
\eqref{equation:ring-telescope-computation} that
\begin{equation}
\begin{aligned}
	\label{equation:module-average-telescope}
	&\left(
	\sum_{\substack{(z,h) \in
	E_{c',S} \\ z \succ h = 
(\rho_\ell^{(z,h)}, \ldots, \rho_1^{(z,h)}; x_\ell^{(z,h)}, \ldots, x_1^{(z,h)})}}
 \sum_{e = 1}^{\ell}
\left( (-1)^{i_e^{(z,h)}-1}d^l_{i_e^{(z,h)},n+1} K^{(z,h)}_{e,n}
\right.\right.
\\
&\left.\left.+
		(-1)^{i_e^{(z,h)}}d^r_{i_e^{(z,h)},n+1}
	K^{(z,h)}_{e,n}\right)
-\id \right)
(m, v_1, \ldots, v_n,s)
= -U^{c'}_S(m, v_1,
\ldots, v_n, s),
\end{aligned}
\end{equation}
using the fact that $d^l_{i_e^{(z,h)},n+1} \iota^{\rho(\gamma_e)}_{i_e^{(z,h)}}(m, v_1, \ldots,
v_n,s) = (m, v_1, \ldots, v_n,s)$.
(One way to verify this is to expand each $v_i$ as a linear combination
of elements of $c$, and then to verify the above equality for each term in the
linear combination.)
Next, we 
use a similar computation to that carried out in \eqref{equation:h-then-delta-ring-left-big}
and \eqref{equation:delta-then-h-ring-left-big}.
We claim that one can similarly verify that, on $F^{\omega}$,
\begin{align}
	\label{equation:module-average-cancel}
	\sum_{\substack{(z,h) \in
	E_{c',S} \\ z \succ h = 
(\rho_\ell^{(z,h)}, \ldots, \rho_1^{(z,h)}; x_\ell^{(z,h)}, \ldots, x_1^{(z,h)})}}
 \sum_{e = 1}^{\ell}
(-1)^j d^\nu_{j,n+1} K^{(z,h)}_{e,n}
	+
	(-1)^{j'} K^{(z,h)}_{e,n-1}
d^\nu_{j',n} = 0,
\end{align}
modulo $F^{\omega-1}$,
for $\nu \in \{l,r\}$
and $j' = j$ if $j < i_e^{(z,h)}$ while $j' = j-1$ if $j \geq i_e^{(z,h)}$.
The above verification
relies on 
\autoref{lemma:z-action-well-defined}
and the fact that the map $( (z_1, \ldots, z_{2g+f}) ,h) \mapsto ( ( v_j \cdot z_1,
\ldots, v_j \cdot z_m, z_{m+1}, \ldots, z_{2g+f}) ,h)$ is a bijection for any $1
\leq m \leq 2g+f$,
where $v_j \cdot z_t$ denotes multiplication in $G^{c'}_{c'}$,

Summing \eqref{equation:module-average-telescope} and
\eqref{equation:module-average-telescope}
and keeping track of signs verifies
\eqref{equation:chain-homotopy-expression-s}, completing the proof.
\end{proof}

\begin{figure}
	\includegraphics[scale=.4]{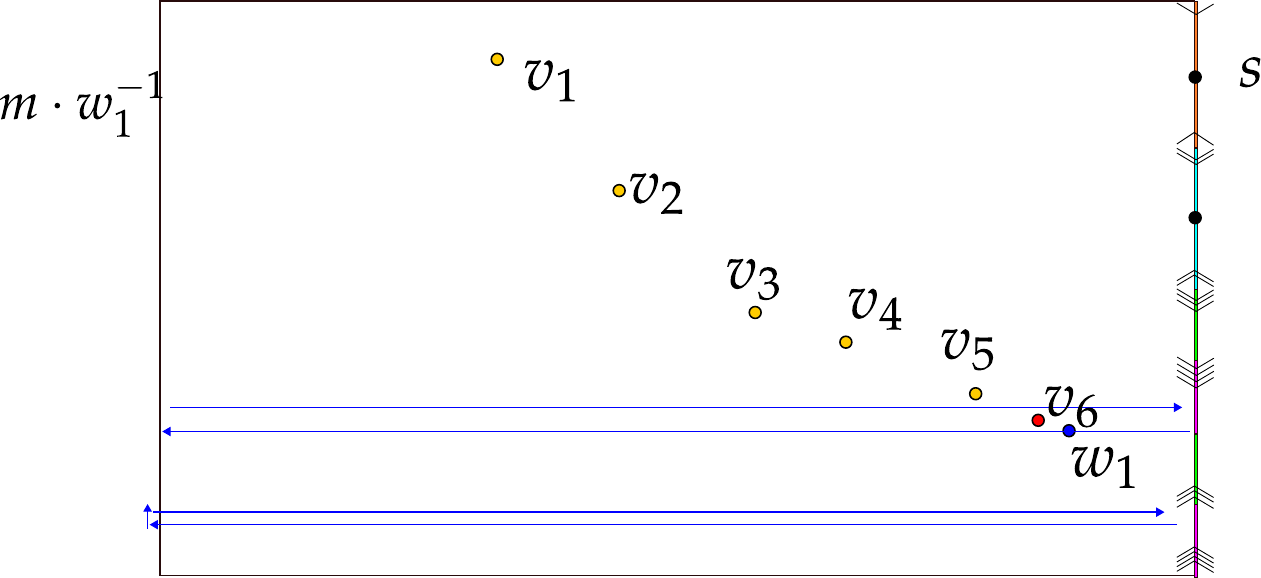}
\caption{
	This is a visualization of part of the nullhomotopy $H$ in the proof of
\autoref{lemma:module-c-homotopy}.
	The $v_i$ are written in the averaged basis, and the yellow
	$v_1,v_2,v_4$
	are averaged while the red dot $v_6$ is antiaveraged. We perform an
	allowable move directly above $v_6$ and then directly below $v_6$ so
	that the result sends $v_6$ to $w_1 \triangleright v_6$.
	In the homotopy, we then repeat this for 
	$w_2, \ldots, w_t$ so that $v_6$ is changed to $(w_1 \cdots w_t) \triangleright v_6$.
	Now, $w_1 \cdots w_t$ was made to realize one of the group elements in
	$G^{c}_{c'}$, and averaging over all such elements
	modifies $v_6$ to $U_{c'}(v_6)$, which vanishes because $v_6$
	is antiaveraged.
	This operation may not be compatible with other $v_j$ hit the boundary,
	but by summing over all of $E_{c, c'}$, it becomes compatible.
}
\label{figure:module-vrack-nullhomotopy}
\end{figure}

Combined with \autoref{lemma:module-s-homotopy},
the next lemma completes the proof of
\autoref{proposition:module-chain-homotopy}.
\begin{lemma}
	\label{lemma:module-c-homotopy}
	With notation as in \autoref{notation:module-homotopy-notation} and
	\autoref{notation:loop-averaged},
	the map 
	$\overline{W}^{c,S,M;k} \to W^{c/c',S/c',M;k}$ is an equivalence.
\end{lemma}
\begin{proof}
	There is a section $W^{c/c',S/c',M;k} \to W^{c,S,M;k} \to
	\overline{W}^{c,S,M;k}$ obtained from the section defined in
	\autoref{notation:module-homotopy-notation}.
	We will equivalently show
	$\overline{W}^{c,S,M;k}/W^{c/c',S/c',M;k}$ is nullhomotopic.

Define a filtration $F^\bullet$ on
$\overline{W}^{c,S,M;k}/W^{c/c',S/c',M;k}$
where an element lies in $F^\omega$ if there are at most $\omega$ elements among $v_1,
\ldots, v_n$ lying in $k\{c - c'\}$.

We can represent any element of 
$W^{c,S,M;k}$ in the form 
$(m, v_1, \ldots, v_n, s)$ where 
$v_1, \ldots, v_{i-1}$ are averaged basis elements and $v_i$ is an antiaveraged
basis element,
and we can represent any element of 
$\overline{W}^{c,S,M;k}$ as a linear combination of elements of the form $U^{c'}_S(m,
v_1,\ldots, v_n,s)$ for $(m, v_1, \ldots, v_n,s)$ in the above form and $U^{c'}_S$ as
defined in \autoref{notation:loop-averaged}.
Recall also the set $E_{c,c'}$ from \autoref{notation:group-elements}.
We define a linear map
$H_n : W_n^{c,S,M;k}/W_n^{c/c',S,M;k} \to
W_{n+1}^{c,S,M;k}/W_{n+1}^{c/c',S,M;k}$
as follows
\begin{equation}
\begin{aligned}
	\label{equation:homotopy-module-unaveraged}
	H_n(m, v_1, \ldots, v_n, s) &:= \frac{1}{|E_{c,c'}|} \cdot \left(
		H^{u}_n(m, v_1, \ldots, v_n, s) +
	H^{d}_n(m, v_1, \ldots, v_n, s) \right), \\
	H^{u}_n(m, v_1, \ldots, v_n, s) &:= \sum_{\substack{ (x,g) \in 
	E_{c,c'} \\  x \succ g = w_1 \cdots w_t}}
 \sum_{e = 1}^{\ell} (-1)^{i-1}(m w_e^{-1} , v_1, \ldots,
v_{i-1}, \\
&\left( (w_{e-1} \cdots w_1) \triangleright v_i\right) \triangleright^{-1} w_e,(w_{e-1} \cdots w_1) \triangleright v_i,
		\ldots, v_n, s), \\
		H^{d}_n(m, v_1, \ldots, v_n, s) &:= \sum_{\substack{ (x,g) \in 
	E_{c,c'} \\  x \succ g = w_1 \cdots w_t}}
 \sum_{e = 1}^{\ell}(-1)^{i-1} (m w_e^{-1} , v_1, \ldots,
	v_{i-1},  (w_{e-1} \cdots w_1) \triangleright v_i, w_e , v_{i+1}, \ldots, v_n, s).
	\end{aligned}
\end{equation}
where, following \autoref{notation:module-chain-complex},
\begin{align*}
	q_{(m w_e^{-1} , v_1, \ldots,
v_{i-1},( (w_{e-1} \cdots w_1) \triangleright v_i) \triangleright^{-1} w_e,(w_{e-1} \cdots w_1) \triangleright v_i,
		\ldots, v_n, s)}(j) 
&= 
q_{(m w_e^{-1} , v_1, \ldots,
	v_{i-1},  (w_{e-1} \cdots w_1) \triangleright v_i, w_e , v_{i+1}, \ldots, v_n, s)}(j) \\
&:=
\begin{cases}
	q_{(m, v_1, \ldots, v_n,s)}(j)  & \text{ if } j \leq i \\
	q_{(m, v_1, \ldots, v_n,s)}(j-1)  & \text{ if } j > i.
\end{cases}
\end{align*}
Colloquially, $H_n$ is defined by inserting the new coordinate involving 
$w_e$ or $( (w_{e-1} \cdots w_1) \triangleright v_i) \triangleright^{-1} w_e$
in each summand in the same rectangle
that the element $v_i$ lies in, and all other rectangle labelings remain the same.

We next check that for $(m, v_1, \ldots, v_n,s) \in F^\omega$ with $v_1,
\ldots,v_{i-1}$ averaged and $v_i$ antiaveraged,
we have
\begin{align}
	\label{equation:almost-h-homotopy}
	(d_{n+1} H_n + H_{n-1} d_n)(m, v_1, \ldots, v_n,s) = (m, v_1, \ldots,
	v_n,s)+
	H_{n-1} (-1)^i d^r_{n,i}(m,v_1,\ldots, v_n,s),
\end{align}
modulo $F^{\omega-1}$.
(We note that we are not making any claim that $H$ is a homotopy, it is just a
	linear map, and we are only claiming an equality of elements in
\eqref{equation:almost-h-homotopy}.)

The key point in verifying \eqref{equation:almost-h-homotopy} is that
\begin{align}
	\label{equation:unaveraged-module-leftward}
	(d_{n+1,i}^l H^{u}_n + d_{n+1,i+1}^l H^{d}_n)(m,v_1,\ldots, v_n,s)  =
	(m,v_1,\ldots, v_n,s),
\end{align}
when $v_1, \ldots, v_{i-1}$ are averaged basis elements and $v_i$ is an
antiaveraged basis element.
One can verify 
\eqref{equation:unaveraged-module-leftward}
via a similar telescoping sum argument to that given in
\eqref{equation:ring-telescope-computation}, using crucially that $v_i$ is antiaveraged.
We next verify the remaining terms in the sum all cancel.
We have
\begin{align}
	\label{equation:unaveraged-module-cancel-up}
	(-1)^{j'-1} d_{n+1,j'}^l H^{u}_n (m,v_1,\ldots, v_n,s) &+H^{u}_{n-1}
	(-1)^{j-1} d_{n,j}^l(m,v_1,\ldots, v_n,s) = 0
	\\
	\label{equation:unaveraged-module-cancel-down}
	(-1)^{j'} d_{n+1,j'}^r H^{d}_n (m,v_1,\ldots, v_n,s) &+H^{d}_{n-1}
	(-1)^j d_{n,j}^r(m,v_1,\ldots, v_n,s) =0,
\end{align}
modulo $F^{\omega-1}$,
where 
\begin{align*}
j' :=	
\begin{cases}
	j & \text{ if } j < i \\
	j+1 & \text{ if } j > i.
\end{cases}
\end{align*}
The equalities in
\eqref{equation:unaveraged-module-cancel-up}
and
\eqref{equation:unaveraged-module-cancel-down}
follow from similar computations to that carried out in 
\eqref{equation:h-then-delta-ring-left-big}
and
\eqref{equation:delta-then-h-ring-left-big} to verify
\eqref{equation:left-big}.
We note that since we are working modulo $F^{\omega-1}$, we can ignore all terms
where the corresponding differentials remove some $v_j \in k\{c -
c'\}$, and the remaining $v_j$ then act via an element of $k\{c'\}$.
Next, observe that that the operation $(x,g) \mapsto (v_j \cdot x, g)$ induces a
bijection on $E_{c,c'}$, 
where $v_j \cdot x$ denotes multiplication in $G^{c'}_{c'}$.
We claim that the set of $w_e$ and $( (w_{e-1} \cdots w_1) \triangleright v_i) \triangleright^{-1} w_e$
will be closed under the action of such $v_j$,
using that if $x \succ g = w_1 \cdots w_\ell$ then 
$(v_j \cdot x) \succ g = (v_j \triangleright w_1) \cdots (v_j \triangleright
w_\ell)$.
Indeed, this is immediate for
$w_e$ while for 
$( (w_{e-1} \cdots w_1) \triangleright v_i) \triangleright^{-1} w_e$
this follows from the calculation
\begin{align*}
	v_j \triangleright (( (w_{e-1} \cdots w_1) \triangleright v_i)
\triangleright^{-1} w_e) = 
( ((v_j \triangleright w_{e-1}) \cdots (v_j \triangleright w_1)) \triangleright
(v_j \triangleright v_i))
\triangleright^{-1} (v_j \triangleright w_e).
\end{align*}

Next, we observe,
\begin{equation}
\begin{aligned}
	\label{equation:unaveraged-module-rightward}
	d_{n+1,i}^r H^{u}_n(m, v_1, \ldots, v_n, s)
	&= d_{n+1,i+1}^r H^{d}_{n}(m, v_1, \ldots, v_n, s) \\
	d_{n+1,i+1}^r H^{u}_n(m, v_1, \ldots, v_n, s)
	&= d_{n+1,i}^r H^{d}_{n}(m, v_1, \ldots, v_n, s)
\end{aligned}
\end{equation}
by construction of $H_n$.

So far, we have accounted for nearly all the terms of the summation, and we
claim that the remaining terms also cancel. Namely, one can directly verify
\begin{equation}
\begin{aligned}
	\label{equation:module-unaveraged-vanishing}
	d^l_{n,i} (m, v_1, \ldots, v_n, s) &= 0 \\
	d^l_{n+1,i+1} H^u_n(m, v_1, \ldots, v_n,s) &= 
	d^l_{n+1,i} H^d_n(m, v_1, \ldots, v_n,s) = 0.
\end{aligned}
\end{equation}
Summing the above expressions from
\eqref{equation:unaveraged-module-leftward},
\eqref{equation:unaveraged-module-cancel-up},
\eqref{equation:unaveraged-module-cancel-down}, 
\eqref{equation:unaveraged-module-rightward}, and
\eqref{equation:module-unaveraged-vanishing}
yield
\eqref{equation:almost-h-homotopy}.

Now, we claim that the linear map $H_n$ restricts to a nullhomotopy of
$\overline{W}^{c,S,M;k}/W^{c/c',S/c',M;k}$.
Observe first that if $(m,v_1, \ldots, v_n,s) \in {W}^{c,S,M;k}$ with $v_1,
\ldots, v_{i-1}$ averaged and $v_i$ antiaveraged then 
$(m,v_1, \ldots, v_n,s)^h$, for $h \in G^{c'}_S$, has the same property that its first $i-1$ elements
are averaged and the $i$th element is an antiaveraged basis element
antiaveraged.
Let $c_0 \subset c' \subset c$ be a component of $c$ contained in $c'$ so that
$v_i \in k\{c_0\}$.
Thus, $U^{c'}_S(m,v_1, \ldots, v_n,s)$ is a linear combination of such elements implying
that on $F^\omega$ we have
\begin{align*}
	(d_{n+1} H_n + H_{n-1} d_n)(U^{c'}_S(m, v_1, \ldots, v_n,s)) =
	U^{c'}_S(m, v_1, \ldots, v_n,s)+
	H_{n-1} (-1)^i d^r_{n,i} (U^{c'}_S(m,v_1,\ldots, v_n,s))
\end{align*}
modulo $F^{\omega-1}$, by \eqref{equation:almost-h-homotopy}.

Next, we claim that 
$d^r_{n,i} (U^{c'}_S(m,v_1,\ldots, v_n,s)) = 0$.
Suppose $v_i$ satisfies $q_{(m, v_1, \ldots, v_n, s)} = \rho$. We may moreover
assume $v_i \in
k\{c_0\}$, as otherwise 
$H_{n-1} d^r_{n,i} (U^{c'}_S(m,v_1,\ldots, v_n,s))$ lies in $F^{\omega-1}$ and we may
ignore it.
Since the averaging operator $U^{c'}_S$ commutes with $d^r{n,i}$, it suffices
to show
$U^{c'}_S(d^r_{n,i} (m,v_1,\ldots, v_n,s)) = 0$.
Say $v_i = \sum_{y \in c_0} \alpha_y y$ with $\sum_{y \in c_0} \alpha_y = 0$.
Then, we can write
\begin{align*}
	d^r_{n,i} (m, v_1, \ldots, v_{i-1}, \sum_{y \in c_0} \alpha_y y, \ldots, v_n,s)
&= 
\sum_{y \in c_0} \alpha_y w_y,\\
w_y &:=d^r_{n,i} (m, v_1, \ldots, v_{i-1}, y, \ldots, v_n,s).
\end{align*}
To show that the application of $U^{c'}_S$ to the above expression vanishes, it is
enough to show that each of the elements
$w_y$
for varying $y \in c_0$ map to the same element under the operator $U^{c'}_S$. Indeed, once we
show these lie in the same orbit, the condition
that $\sum_{y \in c_0} \alpha_y = 0$ will imply 
\begin{align*}
	U^{c'}_S(d^r_{n,i} (m,v_1,\ldots, v_n,s)) = U^{c'}_S(\sum_{y \in c_0} \alpha_y w_y) =
	(\sum_{y \in c_0} \alpha_y) \cdot (U^{c'}_S(w_{y_0})) = 0 \cdot U^{c'}_S(w_{y_0}) = 0,
\end{align*}
where $y_0 \in c_0$ is some representative choice of element.
To check that the $w_y$ map to the same element under $U^{c'}_S$,
since $v_1, \ldots, v_{i-1}$ are averaged,
\begin{align*}
	y \cdot^\rho (m, v_1, \ldots,v_n,s) =w_y.
\end{align*}
So $w_y$ and $w_{y'}$ are related by applying the inverse
of the $y \cdot^\rho$ to $w_y$ followed by the $y' \cdot^\rho$ action.
Since $y, y' \in c_0$ both have the same image in $c'/c'$, the composite of the
inverse of $y \cdot^\rho$ followed by $y' \cdot^\rho$ will lie in
$G^{c'}_S$, as desired.

Altogether, the above implies that $H$ defines a nullhomotopy of the subcomplex
$Z^{c,S,M;k} \subset \overline{W}^{c,S,M;k}$
spanned by elements of the form $U^{c'}_S(m,v_1, \ldots, v_n,s)$ where some $v_i$ is
antiaveraged and $v_1, \ldots, v_{i-1}$ are averaged.
We claim that in fact this subcomplex is a complement to the section
$\overline{W}^{c/c',S/c',M;k} \to \overline{W}^{c,S,M;k}$, which will complete
the proof.
Indeed, any element of ${W}^{c,S,M;k}$ can be written as a linear combination of
elements $(m, v_1, \ldots, v_n,s)$ with $v_1, \ldots, v_{i-1}$ averaged and
$v_i$ antiaveraged, together with elements where all $v_1, \ldots, v_n$ are
averaged.
The key point in this case is that the value of the final coordinate
will be invariant under the action of $G^{c'}_S$, which in this case factors
through $\aut(T_0)$, and so the final coordinate consists of a $k$ multiple of
an orbit of $S/c$ under the action of $G^{c'}_S$, and we can think of it as lying
in the $0$-set of $S/c'$.
Indeed, a complement to 
$Z^{c,S,M;k}$ is given by the span of $U^{c'}_S(m,v_1, \ldots, v_n,s)$ where $v_1,
\ldots, v_n$ are all averaged.
However, the action of $U^{c'}_S$ on such tuples factors through the action of
$G^{c'}_{S}$ and sends such a tuple $(m,v_1, \ldots, v_n,s)$ to its image under
the composite $W^{c,S,M;k} \to W^{c/c',S/c',M;k} \to W^{c,S,M;k}$. That is, 
$W^{c/c',S/c',M;k}$ defines a complement to 
$Z^{c,S,M;k}$. Since we have shown
$Z^{c,S,M;k}$ is nullhomotopic, 
we obtain
$W^{c/c',S/c',M;k} \to \overline{W}^{c,S,M;k}$
is an equivalence, as desired.
\end{proof}

\section{Computing the stable homology}
\label{section:stable-homology}

In this section, we compute the stable homology of Hurwitz modules.
In order to verify some technical conditions that allow us to commute pullbacks with tensor products, we prove that certain maps of simplicial sets are Kan fibrations
in \autoref{subsection:kan}.
We then compute the stable homology of Hurwitz spaces in
\autoref{subsection:stable-homology-space} and compute the stable homology of
Hurwitz modules in \autoref{subsection:stable-homology-module}.

\subsection{Verifying certain maps are Kan fibrations}
\label{subsection:kan}

The main result of this subsection is \autoref{proposition:kan-fibration}, which
verifies a technical condition that certain maps of simplicial sets are Kan
fibrations.
The reader 
interested in the main ideas of the proofs and not the technical details
will likely wish to skip this subsection.

In what follows, for $Y$ a monoid in sets with a left action on a set $X$ and a
right action on a set $Z$, we use $\Barc(X,Y,Z)$ to denote the simplicial
set coming from the bar construction: i.e whose $p$-simplices are given by $X \times Y^p \times Z$ and the face maps
are induced by the above described actions.
To identify the stable homology of Hurwitz spaces, we will need to check several maps of simplicial
sets are Kan fibrations, and the following definition is relevant for all of
these maps.

\begin{definition}
	\label{definition:}
	Let $c$ be a rack and $c' \subset c$ a normal (possibly empty) subrack.
	If $M$ is a discrete left (respectively, right) module for $\pi_0 \Hur^c[\alpha_{c'}^{-1}]$ and $N$ is a
	discrete left (respectively, right) module 
	for $\pi_0 \Hur^{c/c'}[\alpha_{c'/c'}^{-1}]$ then we say a $\pi_0
	\Hur^c[\alpha_{c'}^{-1}]$-module map $\phi:M \to
N$ is {\em module surjective} if $M \to N$ is surjective and for any
$m \in M$ with $\phi(m) = xn$ for $x \in \pi_0 \Hur^{c/c'}[\alpha_{c'/c'}^{-1}]$
and $n \in N$, there is some $\widetilde{x} \in  \pi_0 \Hur^c[\alpha_{c'}^{-1}]$
and $\widetilde{n} \in M$ so that $\widetilde{x}$ projects to $x$,
$\phi(\widetilde{n}) = n$ and $m = \widetilde{x}\widetilde{n}$.     
\end{definition}

For several examples of module surjective maps, see
\autoref{lemma:module-surjective}. Here is an example of a surjective map of
modules that is not module surjective.
\begin{example}
	If we take $\phi:M\to N$ to be $\pi_0
	\Hur^{c/c'}[\alpha_{c'/c'}^{-1}]\{a\} \coprod \pi_0
	\Hur^{c/c'}[\alpha_{c'/c'}^{-1}]\{b\} \to \pi_0
	\Hur^{c/c'}[\alpha_{c'/c'}^{-1}]$ via the map that sends the generators
	$a,b$ to $1,x$ respectively, where $x$ is not an invertible element of
	$\Hur^{c/c'}[\alpha_{c'/c'}^{-1}]$, then $\phi$ is surjective but not
	module surjective, because we can take $m = b$, $x=x,n=1$, so that
	$\phi(m) = xn$. However, the desired $\widetilde{n} \in M, \widetilde{x} \in
	\pi_0 \Hur^c[\alpha_{c'}^{-1}]$ doesn't exist because any lift
	$\widetilde{n}$ of $n$ would necessarily lie in 
	$\Hur^{c/c'}[\alpha_{c'/c'}^{-1}]\{a\}$ and hence 
$\widetilde{x}\widetilde{n} \in \Hur^{c/c'}[\alpha_{c'/c'}^{-1}]\{a\}$ so 
	$\widetilde{x}\widetilde{n}\neq b$.
\end{example}

We can now prove the main result of this subsection, which will be used to
verify the conditions of
\cite[Theorem B.4]{bousfield2006homotopy} to commute $\times$ and $\otimes$.

\begin{proposition}
	\label{proposition:kan-fibration}
	Let $c$ be a rack and $c' \subset c$ a normal subrack. Suppose $M$ is a
	right module for
	$\pi_0 \Hur^c[\alpha_{c'}^{-1}]$, $P$ is a left module for
	$\pi_0 \Hur^c[\alpha_{c'}^{-1}]$,
	$N$ is a right module for 
	$\pi_0\Hur^{c/c'}[\alpha_{c'/c'}^{-1}]$ and $Q$ is a right module for
$\pi_0\Hur^{c/c'}[\alpha_{c'/c'}^{-1}]$.
Suppose we are given maps $M \to N$ and $P \to Q$ which are module surjective.

	Then the map of
	simplicial sets 
	\begin{align}
		\label{equation:kan-fibration}
		\Barc(M,\pi_0 \Hur^c[\alpha_{c'}^{-1}], P) \to 
		\Barc(N,{\pi_0\Hur^{c/{c'}}[\alpha_{c'/c'}^{-1}]},Q )
	\end{align}
is a Kan fibration.
\end{proposition}

\begin{proof}
	We need to show that given a	
	diagram
	\begin{equation}
		\label{equation:horn-filling}
		\begin{tikzcd} 
			\Lambda^n_i \ar {r} \ar {d} & 	\Barc(M,\pi_0 \Hur^c[\alpha_{c'}^{-1}], P)\ar {d} \\
	\Delta^n \ar {r}\ar[ur,dashed] &	\Barc(N,{\pi_0\Hur^{c/{c'}}[\alpha_{c'/c'}^{-1}]},Q )	\end{tikzcd}\end{equation}
	where $\Lambda^n_i$ denotes the $i$-horn of the $n$-simplex $\Delta^n$,
	there is a unique dashed map making the diagram commute.
	First, note that for $Z$ a monoid, $X$ a right $Z$ module and $Y$ a left
	$Z$ module, we can realize the simplicial set $\Barc(X,Z,Y)$
	as the simplicial set associated to the nerve of the $1$-category whose
	objects are pairs $(x,y) \in X \times Y$ and whose morphisms are
	triples $(x,s,y) \in X \times Z \times Y$ with source $(xs,y)$ and
	target $(x,sy)$. The composition in this category sends the pair $(xs_1,s_2,y),(x,s_1,s_2y)$ to $(x,s_1s_2,y)$.
	In particular, the $n$-simplices are given by tuples $(x,s_1, \ldots,
	s_n, y) \in X \times Z^{n} \times Y$ with the $i$th vertex of this simplex given by
$(xs_1 \cdots s_i, s_{i+1} \cdots s_n y)$.
	Because $1$-categories are $2$-coskeletal when viewed as simplicial sets, it follows that we may restrict ourselves to considering fillers of horns $\Lambda^n_i$ for $n\leq 2$, since otherwise there is a unique solution of the lifting problem on the source and target.
	Similarly, since there is a unique filler of the inner horn $\Lambda^2_1$, we may restrict ourselves to outer horns.

It remains to verify the unique filling of outer horns in the cases that $n = 1$ and $n = 2$.
First, we check the case $n = 1$. Let us just check the filling of the horn
$\Lambda^0_1$ as the horn $\Lambda^1_1$ is analogous.
In this case, the diagram
\eqref{equation:horn-filling} unwinds to the following data:
we are given the data of some $x \in M ,y \in
P$
together with a morphism 
$(\overline{u},\overline{s}, \overline y)$ with source $(\overline u \overline
s,\overline y) = (\overline x,\overline y)$.
Hence, to produce the desired commutative diagram \eqref{equation:horn-filling}
we only need to produce some $u \in M, s \in \pi_0 \Hur^c[\alpha_{c'}^{-1}]$ mapping to
$\overline u \in N$ and $\overline{s} \in \pi_0\Hur^{c/{c'}}[\alpha_{c'/c'}^{-1}]$
so that $us = x$,
as then we will obtain the morphism $(u,s,y)$ lifting $(\overline u, \overline s,
\overline y)$.
The existence of such $u$ and $s$ follows from the assumption that $M \to N$ is
module surjective.

To conclude, we only need verify that we can fill the outer horns in the case $n =
2$. Again, the cases $\Lambda^0_2$ and $\Lambda^2_2$ are analogous so we only
verify $\Lambda^0_2$.
Again, let us unwind what data of producing the dashed arrow in
\eqref{equation:horn-filling} amounts to.
We are given the data of morphisms $(x_1, s, y_0)$ and $(x_2, r,y_0)$ as well as
a $2$-simplex $(\overline{x}_2, \overline s, \overline t, \overline{y}_0)$
in $\Barc (N,\pi_0\Hur^{c/{c'}}[\alpha_{c'/c'}^{-1}],Q)$
and we
need to produce a simplex $(x_2, s, t, y_0)$ in
$\Barc(M,{\pi_0 \Hur^c[\alpha_{c'}^{-1}]}, P)$
mapping to the above specified
$2$-simplex in $\Barc(N ,{\pi_0\Hur^{c/{c'}}[\alpha_{c'/c'}^{-1}]},
Q)$.
Concretely, this just unwinds to finding some $t$ so that $st = r$ and
$t$ has image $\overline t$.
As usual, by multiplying all the above data by suitable elements in ${c'}$, we can
arrange that $s,r$ both lie in $\pi_0\Hur^c$ and $\overline{s},
\overline{t},\overline{r}$ lie in $\pi_0\Hur^{c/{c'}}$.
Using the same argument as in the case of filling outer horns when $n = 1$, we
can produce some $s', t' \in \Hur^c$ whose images are $\overline s$ and $\overline t$ and
whose product agrees with $r$ in $\pi_0\Hur^c$. 
Namely, suppose $r \in \Hur^c_n$ and $s \in \Hur^c_j$.
There is some element $\gamma$ so that we can identify $\overline s
\overline t$ with $\overline r$ in $\Hur^{c/{c'}}$ and then 
if we write $\gamma^{-1}(r) = x_1 \cdots x_n$ and take $s' := x_1 \cdots x_j,
t':= x_{j+1}\cdots x_n$, we will have that the image of $s'$ is $\overline s$
and the image of $t'$ is $\overline t$.
It remains to show that if we are given some $s',t'$ with images $\overline s,
\overline t$ and $s$ also has image $\overline s$, we can find some $t$ so that
$st = s't'$.
By \autoref{lemma:same-image-relation} and the assumption that $s$ and $s'$ have the same image,
there is some $w \in \pi_0 \Hur^{c'}[\alpha_{c'}^{-1}]$ so that $s'w = s$.
Then, if we take $t := w^{-1} t'$, we get $st = (s'w)(w^{-1}t') = s't' = r$ and
hence $t = w^{-1}t'$ has the same image as $t'$ in
$\pi_0\Hur^{c/c'}[\alpha_{c'/c'}^{-1}]$,
completing the proof.
\end{proof}

In order to apply \autoref{proposition:kan-fibration}, we will need to show that
the relevant maps are module surjective.
The next lemma provides several examples of such module surjective maps.

\begin{lemma}
	\label{lemma:module-surjective}
	Let $c$ be a rack and $c' \subset c$ a normal subrack.
Let $S = ({\Sigma^1_{g,f}}, \{T_n\}_{n \in \mathbb
Z_{\geq 0}}, \{\psi_n: B^{\Sigma^1_{g,f}}_n \times T_n \to T_n\}_{n \in \mathbb
Z_{\geq 0}})$
be a bijective Hurwitz module over $c$.
	The following maps are module surjective:
	\begin{enumerate}
		\item The projection map $\pi_0 \Hur^c[\alpha_{c'}^{-1}] \to
			\pi_0 \Hur^{c/c'}[\alpha_{c'/c'}^{-1}]$, with the source
			viewed as an 
$\pi_0 \Hur^c[\alpha_{c'}^{-1}]$ module and the target as a 
$\pi_0 \Hur^{c/c'}[\alpha_{c'/c'}^{-1}]$ module.
\item The projection map 
$\pi_0 \Hur^c[\alpha_{c'}^{-1}] \to
			\pi_0 \Hur^{c/c'}[\alpha_{c'/c'}^{-1}]$ where both the
			source and target are viewed as 
			$\pi_0 \Hur^c[\alpha_{c'}^{-1}]$ modules.
		\item The identity map from a module to itself.
\item 
	The projection map 
	$\pi_0 \Hur^{c,S}[\alpha_{c'}^{-1}] \to
	\pi_0 \Hur^{c/c',S/c'}[\alpha_{c'/c'}^{-1}]$ where the
			source is a 
$\pi_0 \Hur^c[\alpha_{c'}^{-1}]$
module
			and target the target is a
			$\pi_0 \Hur^{c/c'}[\alpha_{c/c'}^{-1}]$ module.
	\end{enumerate}
\end{lemma}
\begin{proof}
	In all parts, the map of modules is clearly surjective, using that $c
	\to c/c'$ is surjective, and if $T_0$ is the $0$-set of $S$ and
	$\overline{T}_0$
	is the $0$ set of $S/c'$, then $T_0
	\to \overline{T}_0$ is surjective.
	Hence, we only need to verify the second condition in the definition of
	module surjective.
	The second condition from the definition of module surjective
	in cases (2) and (3) is easily seen to hold upon taking
	$\widetilde{x}=x$ and $\widetilde{n} = x^{-1} m$. 

	It remains only to verify the second condition of module surjective in
	cases $(1)$ and $(4)$.
	Moreover, case $(1)$ is actually a special case of $(4)$ where we take
	$g = f = 0$ and $S$ to have $0$ set a singleton so that $T_0 \times
	\pi_1(\Sigma^1_{0,0})$ acts trivially on $c'$.
	Hence, we now verify the second condition in case $(4)$.
	Suppose we are given $x \in \pi_0\Hur^{c,S}_n[\alpha_{c'}^{-1}]$,
	$\overline u \in \pi_0\Hur^{c/c'}_j[\alpha_{c'/c'}^{-1}]$ and $\overline{s} \in
	\pi_0 \Hur^{c/c',S/c'}_{n-j}[\alpha_{c'}^{-1}]$ such that $\overline x = \overline u \overline s$ where $\overline x$ is the image of $x$ in $\pi_0\Hur^{c/c',S/c'}$. We want to find lifts of $\overline u$ and $\overline
	s$ to $\pi_0\Hur^c_j[\alpha_{c'}^{-1}]$ and $\pi_0\Hur^{c,S}_{n-j}[\alpha_{c'}^{-1}]$ with $x = us$.
Multiplying by a suitable power of elements of ${c'}$, we may assume that the above
elements all lie in $\pi_0 \Hur^{c}, \pi_0 \Hur^{c/c'}, \pi_0 \Hur^{c,S}, \pi_0
\Hur^{c/c',S/c'}$ and $\pi_0\Hur^{c/{c'}}$,
with no localization
at $\alpha_{c'}^{-1}$ or $\alpha_{c'/c'}^{-1}$.
To produce our desired $u$ and $s$, note that after possibly replacing the
elements above with $c'$ multiples, we have an equality
$\overline x = \overline u \overline s$ in $\pi_0 \Hur^{c/c', S/c'}$. 
Let $\overline{T}_n$ be the $n$ set of $S/c'$.
Let $x' \in T_n$ denote some representative of $x$ and $\overline{x}'$ denote
its image in $\overline{T}_n$.
Rephrasing the above, if we view $\overline{u}' \in (c/c')^j$ corresponding to
$\overline{u} \in \pi_0 \Hur^{c/c'}_j$ and 
$\overline{s}' \in \overline{T}_{n-j}$  as corresponding to $\overline{s} \in
\pi_0 \Hur^{c/c',S/c'}_{n-j}$, 
we can view the concatenation $\overline{u}' \overline{s}' \in \overline{T}_n$,
and by the assumption
that $\overline{x} = \overline{u}\overline{s}$,
there is some element
$\gamma \in B_n^{\Sigma^1_{g,f}}$ with $\gamma(\overline{u}' \overline{s}') =
\overline{x}'$.
Write $\gamma^{-1}(x') = (y'_1, \ldots, y'_j, y'_{j+1}, y'_n,t')$ and define
$u' := (y_1', \ldots, y_j') \in c^j$ and $s' := (y_{j+1}', \ldots, y_n', t') \in
T_{n-j}$.
Then, the image of $u'$ in $(c/c')^j$ is $\overline{u}'$ and the image of $s'
\in \overline{T}_{n-j}$ is $\overline{s}'$.
Taking $u$ to be the image of $u' \in \pi_0 \Hur^c$ and $s$ to be the image of
$s' \in \pi_0 \Hur^{c,S}$, we find $x$ agrees with the concatenation of $u$ and
$s$, so $u$ and $s$ are the desired lifts of $\overline{u}$ and $\overline{s}$.
\end{proof}

\subsection{The stable homology of Hurwitz spaces in all directions}
\label{subsection:stable-homology-space}

We are now able to compute the stable homology of Hurwitz spaces in all
directions.
To do this, we will use descent, and to check the required isomorphisms between
fiber products of the relevant covers,
we use the nullhomotopy from \autoref{proposition:ring-chain-homotopy}
as well as a result of Bousfield-Friedlander to commute $\times$ and
$\otimes$, whose hypotheses we verify using \autoref{proposition:kan-fibration}.

Recall that we use the notation $A_c := C_*(\Hur^c) = C_*(\Hur^c;\mathbb Z)$.
\begin{theorem}
	\label{theorem:stable-homology-ring}
	Let $c$ be a finite rack and $c' \subset c$ be a union of connected components
	of $c$.
	There is an equivalence
	\begin{align}
		\label{equation:ring-localization-identification}
		C_*(\Hur^c)[\alpha_{c'}^{-1},|G^{c'}_{c}|^{-1}] \simeq
		C_*(\Hur^{c/c'} \times_{\pi_0
		\Hur^{c/c'}} \pi_0 \Hur^c)[\alpha_{c'}^{-1},|G^{c'}_{c}|^{-1}].	
	\end{align}
\end{theorem}
\begin{proof}
	To simplify notation, we let $G :=
	|G^{c'}_{c}|^{-1}$.
	Note that both the source and target of
	\eqref{equation:ring-localization-identification} are $0$-nilpotent
	complete with respect to $H_0(\Hur^c)[\alpha_{c'}^{-1}][G^{-1}]$  
(in the sense of \cite[Definition
4.0.1]{landesmanL:homological-stability-for-hurwitz})
	by \cite[Lemma
	4.0.4]{landesmanL:homological-stability-for-hurwitz}. Therefore, to
	prove \eqref{equation:ring-localization-identification}, it suffices to
	prove \autoref{lemma:tensor-ring-equivalence}
for every $n \geq 0$.
\end{proof}
\begin{lemma}
	\label{lemma:tensor-ring-equivalence}
Let $c$ be a finite rack and $c' \subset c$ be a union of connected components
	of $c$.
Let $G :=|G^{c'}_{c}|$.
For every $n \geq 0$, there is an equivalence
\begin{equation}
		\label{equation:iterated-stable-homology-cover}
		\begin{aligned}
		&C_*(\pi_0\Hur^c)[\alpha_{c'}^{-1},G^{-1}]^{\left(\otimes_{\left(C_*(\Hur^c)[\alpha_{c'}^{-1},G^{-1}]\right)}
	n+1 \right)} 
	\\
	&\simeq
		C_*(\pi_0 \Hur^c)[\alpha_{c'}^{-1},G^{-1}]^{\left(\otimes_{
					\left( C_*\left(\Hur^{c/c'}
							\times_{\pi_0\Hur^{c/c'}}
				\pi_0\Hur^c \right) \right)[\alpha_{c'}^{-1},G^{-1}]}
	n+1\right)}.
	\end{aligned}
\end{equation}
\end{lemma}
\begin{proof}
		The case $n = 0$ of
		\eqref{equation:iterated-stable-homology-cover}
		is trivial as both sides are 	
	$H_0(A_c)[\alpha_{c'}^{-1},G^{-1}]$.
	The case $n \geq 2$ follows from the case $n = 1$ by iteratively
	applying then $n  =1$ case of 
	\eqref{equation:iterated-stable-homology-cover}.
	We conclude by proving the $n =1$ case.
We can identify
\begin{equation}
	\begin{aligned}
	\label{equation:c-iterated-switch}
	&\pi_0\Hur^c \otimes_{\Hur^c} \pi_0\Hur^c[\alpha_{c'}^{-1}] \\
	&\simeq \pi_0\Hur^c[\alpha_{c'}^{-1}] \otimes_{\Hur^c[\alpha_{c'}^{-1}]} \pi_0\Hur^c[\alpha_{c'}^{-1}] \\
&\simeq \left( \pi_0\Hur^c[\alpha_{c'}^{-1}] \times_{\pi_0 \Hur^c[\alpha_{c'}^{-1}]} \pi_0 \Hur^c[\alpha_{c'}^{-1}] \right)
\otimes_{\Hur^c[\alpha_{c'}^{-1}] \times_{\pi_0 \Hur^c[\alpha_{c'}^{-1}]} \pi_0
\Hur^c[\alpha_{c'}^{-1}]} \\
&\qquad \left( \pi_0\Hur^{c/c'}[\alpha_{c'}^{-1}]
\times_{\pi_0\Hur^{c/c'}[\alpha_{c'}^{-1}]} \pi_0 \Hur^c[\alpha_{c'}^{-1}] \right) \\
&\simeq 
\left(\pi_0\Hur^c[\alpha_{c'}^{-1}] \otimes_{\Hur^c[\alpha_{c'}^{-1}]} \pi_0 \Hur^{c/c'}[\alpha_{c'}^{-1}] \right)
\times_{\left(\pi_0 \Hur^c[\alpha_{c'}^{-1}] \otimes_{\pi_0\Hur^c[\alpha_{c'}^{-1}]} \pi_0
\Hur^{c/c'}[\alpha_{c'}^{-1}]\right)} \\
&\qquad\left( \pi_0 \Hur^c[\alpha_{c'}^{-1}]
\otimes_{\pi_0\Hur^c[\alpha_{c'}^{-1}]} \pi_0 \Hur^c[\alpha_{c'}^{-1}] \right)
\\
&\simeq 
\left(\pi_0\Hur^c \otimes_{\Hur^c} \pi_0 \Hur^{c/c'}[\alpha_{c'}^{-1}]
 \right)\times_{\left(\pi_0 \Hur^c \otimes_{\pi_0\Hur^c} \pi_0
\Hur^{c/c'}\right)[\alpha_{c'}^{-1}]} \left( \pi_0 \Hur^c
\otimes_{\pi_0\Hur^c} \pi_0 \Hur^c \right)[\alpha_{c'}^{-1}]
\\
& \simeq 
\left(\pi_0\Hur^c \otimes_{\Hur^c} \pi_0 \Hur^{c/c'}[\alpha_{c'}^{-1}] \right) \times_{\pi_0
\Hur^{c/c'}[\alpha_{c'}^{-1}]} \pi_0 \Hur^c [\alpha_{c'}^{-1}]
	\end{aligned}
\end{equation}
where the third isomorphism uses
\autoref{proposition:kan-fibration} via
\autoref{lemma:module-surjective}(2) and (3)
and
\cite[Theorem B.4]{bousfield2006homotopy}.
By a similar computation, 
using 
\autoref{proposition:kan-fibration}
via
\autoref{lemma:module-surjective}(1)
and
\cite[Theorem B.4]{bousfield2006homotopy}
we can also identify
\begin{equation}
\begin{aligned}
	\label{equation:c-prime-iterated-switch}
	&\pi_0\Hur^c \otimes_{\left(\Hur^{c/c'} \times_{\pi_0 \Hur^{c/c'}} \pi_0
	\Hur^c \right)} \pi_0\Hur^c [\alpha_{c'}^{-1}]
	\\
	&\simeq \pi_0\Hur^c[\alpha_{c'}^{-1}] \otimes_{\left(\Hur^{c/c'}[\alpha_{c'}^{-1}] \times_{\pi_0 \Hur^{c/c'}[\alpha_{c'}^{-1}]} \pi_0
	\Hur^c[\alpha_{c'}^{-1}] \right)} \pi_0\Hur^c [\alpha_{c'}^{-1}]
	\\
	&\simeq 
	\left( \pi_0\Hur^{c/c'}[\alpha_{c'}^{-1}] \times_{\pi_0 \Hur^{c/c'}[\alpha_{c'}^{-1}]} \pi_0\Hur^c[\alpha_{c'}^{-1}] \right) \otimes_{\left(\Hur^{c/c'}[\alpha_{c'}^{-1}] \times_{\pi_0 \Hur^{c/c'}[\alpha_{c'}^{-1}]} \pi_0
	\Hur^c[\alpha_{c'}^{-1}] \right)} \\
	&\qquad \left(\pi_0\Hur^{c/c'}[\alpha_{c'}^{-1}] \times_{\pi_0 \Hur^{c/c'}[\alpha_{c'}^{-1}]}
	\pi_0\Hur^c [\alpha_{c'}^{-1}] \right)
	\\
	&\simeq \left(\pi_0 \Hur^{c/c'}[\alpha_{c'}^{-1}] \otimes_{\Hur^{c/c'}[\alpha_{c'}^{-1}]} \pi_0 \Hur^{c/c'}[\alpha_{c'}^{-1}]
	\right) \times_{\left(\pi_0 \Hur^{c/c'}[\alpha_{c'}^{-1}] \otimes_{\pi_0 \Hur^{c/c'}[\alpha_{c'}^{-1}]} \pi_0
		\Hur^{c/c'}[\alpha_{c'}^{-1}]\right)} \\
		&\qquad \left(\pi_0 \Hur^c[\alpha_{c'}^{-1}] \otimes_{\pi_0\Hur^c[\alpha_{c'}^{-1}]}
	\pi_0\Hur^c[\alpha_{c'}^{-1}] \right) \\
&\simeq \left(\pi_0 \Hur^{c/c'} \otimes_{\Hur^{c/c'}} \pi_0 \Hur^{c/c'}[\alpha_{c'}^{-1}]
\right) \times_{\pi_0 \Hur^{c/c'}[\alpha_{c'}^{-1}]} \pi_0 \Hur^c [\alpha_{c'}^{-1}]
\end{aligned}
\end{equation}

Finally, applying the functors given by taking chains and inverting $G$, to the final lines of
\eqref{equation:c-iterated-switch} and
\eqref{equation:c-prime-iterated-switch}
we obtain an equivalence
\begin{align*}
&C_* \left(\left(\pi_0\Hur^c \otimes_{\Hur^c} \pi_0 \Hur^{c/c'}[\alpha_{c'}^{-1}] \right) \times_{\pi_0
\Hur^{c/c'}[\alpha_{c'}^{-1}]} \pi_0 \Hur^c [\alpha_{c'}^{-1}] \right)[G^{-1}] \\
&\simeq
	C_*\left(\left(\pi_0 \Hur^{c/c'} \otimes_{\Hur^{c/c'}} \pi_0 \Hur^{c/c'}[\alpha_{c'}^{-1}]
	\right) \times_{\pi_0 \Hur^{c/c'}[\alpha_{c'}^{-1}] } \pi_0 \Hur^c
[\alpha_{c'}^{-1}]\right)[G^{-1}],
\end{align*}
where
\autoref{proposition:ring-chain-homotopy} identifies the $\mathbb Z[G^{-1}]$
homology of the left hand factors, and we can identify the homology of the
pullbacks
because the base of the pullback is discrete.
Therefore, the result of applying chains and inverting $G$ to the first lines of 
\eqref{equation:c-iterated-switch} and
\eqref{equation:c-prime-iterated-switch}
are also equivalent, which is identified with the equivalence
\eqref{equation:iterated-stable-homology-cover} when $n = 1$.
\end{proof}

We now deduce one of our main results from the introduction, which is
essentially a
rephrasing of \autoref{theorem:stable-homology-ring}.
\subsubsection{Proof of \autoref{theorem:one-large-stable-homology-hurwitz-space}}
\label{subsubsection:stable-ring-proof}

We consider
$C_*(\CHur^c), C_*(\Hur^c), C_*(\CHur^{c/c_1}), C_*(\Hur^{c/c_1})$ as graded rings 
with respect to the number of elements in the component $c_1 \subset c$.

Using \cite[Theorem 1.4.1]{landesmanL:homological-stability-for-hurwitz},
for $n > Ii + J$
every element $\alpha_x$ for $x\in c_1$ induces an isomorphism from the $n$th
graded part of $H_i(\CHur^c)$ to the $n+1$st graded part of $H_i(\CHur^c)$.
Therefore, the $n$th graded part of
$H_i (\CHur^c)$ agrees with
the $n$th graded part of 
$H_i(\CHur^c)[\alpha_{c_1}^{-1}]$.
Similarly, 
the $n$th graded part of 
$H_i (\CHur^{c/c_1})$ agrees with
the $n$th graded part of
$H_i(\CHur^{c/c_1})[\alpha_{c_1/c_1}^{-1}]$.

To conclude the proof, it suffices to show
\begin{align*}
	H_i(\CHur^{c})[\alpha_{c_1}^{-1}, |G^{c'}_c|^{-1}]\simeq H_i(\CHur^{c/c_1} \times_{\pi_0 \Hur^{c/c_1}} \pi_0
\Hur^{c})[\alpha_{c_1/c_1}^{-1}, |G^{c'}_c|^{-1}].
\end{align*}
This identification holds
by \autoref{theorem:stable-homology-ring},
since the equivalence there sends components of 
$\Hur^{c}$ contained in 
$\CHur^{c}$
to components of
$\Hur^{c/c_1} \times_{\pi_0 \Hur^{c/c_1}} \pi_0 \Hur^{c}$
contained in 
$\CHur^{c_1} \times_{\pi_0 \Hur^{c/c_1}} \pi_0 \Hur^{c}$.
\qed

\subsection{The stable homology of bijective Hurwitz modules}
\label{subsection:stable-homology-module}

We conclude this section by computing the stable homology of Hurwitz
modules. We essentially compute their stable homology in
\autoref{theorem:identify-stable-localization} and then explain how this is
equivalent to \autoref{theorem:one-large-stable-homology} in
\autoref{subsubsection:stable-value-proof}.

The idea for proving \autoref{theorem:identify-stable-localization} is very
similar to the idea we used to prove \autoref{theorem:stable-homology-ring}.
We will argue via descent. To identify the relevant fiber products are
equivalent, we will massage these fiber products
using several applications of a result of Bousfield-Friedlander
to commute pullbacks and tensor products, whose hypotheses we verify using \autoref{proposition:kan-fibration}.
We can then identify the resulting fiber products using
\autoref{proposition:module-chain-homotopy} and
\autoref{lemma:tensor-ring-equivalence}.

The following proposition is a consequence of \Cref{theorem:identify-stable-localization}, but we in fact use it as an important ingredient in proving the theorem:

\begin{proposition}\label{proposition:trivialize-action}
	Let $c$ be a finite rack and let $S$ be a bijective
	Hurwitz module over $c$.
	Suppose $c' \subset c$ is a subrack which is an $S$-component of $c$.
	Suppose that $x$ is an invertible element in $\pi_0\Hur^{c'}[\alpha_{c'}^{-1}]$ and $y$ is a component of $\pi_0\Hur^{c,S}[\alpha_{c'}^{-1}]$ such that $xy=y$. Then multiplication by $x$ on the component of $\pi_0\Hur^{c}[\alpha_{c'}^{-1}]\otimes_{\Hur^{c}}\Hur^{c,S}$ corresponding to $y$ induces the identity map on homology after inverting $|G^{c}_{c'}|$.
\end{proposition}

\begin{proof}
	Recall from \autoref{proposition:pointed-scanning} that the space $\pi_0\Hur^{c}[\alpha_{c'}^{-1}]\otimes_{\Hur^{c}}\Hur^{c,S}$
	is the ind-weak homotopy type of
	$\overline{Q}_\epsilon[\pi_0\Hur^{c}[\alpha_{c'}^{-1}]
	,\on{hur}^{c,S}]$. Let us use $Y_{\epsilon}$ to denote the component
	of this family of spaces corresponding to the element $y$. We will show
	that the left multiplication by $x$ map $\mu_x:Y_{\epsilon} \to Y_{\epsilon'}$ induces the same map on $\ZZ[\frac 1 {|G^{c}_{c'}|}]$-homology as the inclusion $i:Y_{\epsilon} \to Y_{\epsilon'}$, where $\epsilon'$ satisfies $0<\epsilon'<\frac \epsilon N $ for some $N\gg0$, which will prove the claim.
	
	A point of $Y_{\epsilon}$ can be
	represented in the form $(m, (x,1,\gamma,\alpha=(\alpha_1, \ldots, \alpha_n,s))$ for $m \in \left( \pi_0 \Hur^{c} \right)[\alpha_{{c'}}^{-1}]$. 
	We can consider a point where $n=0$, so that its data is determined by
	$(m,1,\id,\alpha=(s))$. Because the points $(m,1,\id,\alpha=(s))$ and
	$(xm,1,\id,\alpha=(s))$ are in the same component, there is a path
	$\gamma''$ from the former to the latter. Because the $1$-skeleton of
	$Y_{\epsilon'}$ consists of paths moving points across the middle of the
	rectangles $J_i^{\epsilon}$, as defined in
	\autoref{remark:rectangles-description}, we may assume that $\gamma''$ is the concatenation of finitely many paths of this form. For any $\epsilon>0$, after passing to $Y_{\epsilon'}$ for $\epsilon'< \epsilon/N$ for $N\gg0$, via a homotopy that is an affine transformation in the vertical coordinate, we may choose a path $\gamma'_{\epsilon}$ homotopic to $\gamma''$ such that at each point, every point in $\mathcal M_{g,f,1}^{\epsilon'}$ has vertical coordinate $\epsilon'$ away from the boundary of each of the rectangles in  $\cup_{i}\overline{(J_i^{\epsilon'}-J_{i}^{\epsilon})}$.
	
	We now claim that there is a finite connected cover $\pi:Y'_{\epsilon}\to Y_{\epsilon}$ with a point $\sigma$ lying over $(m,1,\id,\alpha=(s))$, 
	and a continuous homotopy $\tilde{H}_{\epsilon} : Y'_{\epsilon} \times I
	\to Y_{\epsilon'}$ such that
	\begin{enumerate}
		\item $\tilde{H}_{\epsilon}$ is a homotopy from the map $i\circ \pi$ to $\mu_x \circ \pi$.
		\item The restriction $I \xrightarrow{(\sigma,\id)} 
			Y'_{\epsilon} \times I\xrightarrow{\tilde{H}_{\epsilon}}
			Y_{\epsilon'}$ of $\tilde{H}_{\epsilon}$ to the point
			$\sigma$ agrees with the path $\gamma'_{\epsilon}$.
		\item For any point $p$ in $Y'_{\epsilon}$, the underlying configuration of points in $\mathcal M_{g,f,1}^{\epsilon'}$ of the path $(\tilde{H}_{\epsilon})(p,t)$ as $t \in [0,1]$ varies 
			is the disjoint union of the configuration of points associated to $(\tilde{H}_{\epsilon})(p,0)$ and $\gamma'_{\epsilon}(t)$. 
	\end{enumerate}
	First, we explain why $\tilde{H}_{\epsilon}$ as above exists.
	Consider the sheaf on $Y'_{\epsilon}$ sending an open $U \to Y'_{\epsilon}$ to the collection of homotopies starting from $U \to Y'_{\epsilon} \xrightarrow{i \circ \pi} Y_{\epsilon'}$ satisfying the condition that the underlying configuration of points for any $p \in U$ at time $t$ of the homotopy is the disjoint union of the configuration of points associated to $i \circ \pi(p)$ and $\gamma'_{\epsilon}(t)$. 
	This is a finite locally constant \'etale sheaf, since
	a homotopy is locally determined by the elements of $c$ used in the labels on the
	elements in the path $\gamma'_\epsilon$, since we are fixing the configuration of points to contain those of $\gamma'_{\epsilon}$, but not fixing the labels.
	We claim that the homotopies additionally satisfying (1) are a locally constant subsheaf. 
	This claim can be rephrased as saying that it is an open and closed
	condition for a choice of labels for the
	points appearing in the path $\gamma'_\epsilon$ to result in a
homotopy to multiplication by $x$.
	
	We now verify the above claim.
	Recall that by definition of
	$\overline{Q}_\epsilon[\pi_0\Hur^{c}[\alpha_{c'}^{-1}]
	,\on{hur}^{c,S}]$, $Y_{\epsilon}$ is a quotient of the components of
	$Q_{\epsilon}[\pi_0\Hur^{c}[\alpha_{c'}^{-1}], \on{hur}^{c,S}]$ (see
	\autoref{notation:scanned-quotient}) with image in $Y_{\epsilon}$. 
	We first observe that it is an open and closed
	condition for a choice of labels for the
	points appearing in the path $\gamma'_\epsilon$ to result in a
homotopy to multiplication by $x$ on each 
	such component of $Q_{\epsilon}[\pi_0\Hur^{c}[\alpha_{c'}^{-1}],
	\on{hur}^{c,S}]$ where no points in the configuration hit the left
	boundary. Moreover, observe that the equivalence relations defining
	$Y_{\epsilon}$ involve right multiplication on the label on the left,
	which commute with left multiplication by $x$. 
	Therefore, the condition that the homotopy is between the identity and multiplication by $x$ on a boundary point is equivalent to the condition on a nearby point in the interior. 
	This implies that the claim. Hence, the homotopies additionally
	satisfying $(1)$ form a locally constant subsheaf.
	
	By the sheaf--\'etale space correspondence, the component of this
	locally constant sheaf corresponding to $\sigma$ determines, via (2), the cover $Y'_{\epsilon}$ along with the homotopy $\tilde{H}_{\epsilon}$ satisfying the desired properties.
	
	So far, we have produced a homotopy between $i \circ \pi$ and
	$\mu_x \circ \pi$ and hence both induce the same map on homology from
	$Y'_\epsilon$ to $Y_{\epsilon'}$.
	Next, we claim that the degree of the cover $Y'_{\epsilon} \to Y_{\epsilon}$ is a unit in $\ZZ[\frac 1 {|G^{c}_{c'}|}]$. 
	Once we establish this, it follows via transfer that the homology of
	$Y_{\epsilon}$ with $\ZZ[\frac 1 {|G^{c}_{c'}|}]$ coefficients is a
	summand of the homology of $Y'_\epsilon$ with $\ZZ[\frac 1
	{|G^{c}_{c'}|}]$ coefficients, and therefore left multiplication by $x$
	induces the identity map on homology of $Y_\epsilon$ after inverting
	$G^c_{c'}$.

	We now conclude the proof by showing the degree of $Y'_{\epsilon} \to
	Y_{\epsilon}$ is a unit in
$\ZZ[\frac 1 {|G^{c}_{c'}|}]$.
	To see this, we recall as above that there is a map from the fiber over
	$(m,1,\id,\alpha=(s))$ to a product $c'^{\mu}$, where $\mu$ is the
	number of points appearing in the interior of the rectangle during the
	path $\gamma'_{\epsilon}$.
		It is enough to prove that the action of the fundamental group of $
	Y_{\epsilon}$ on this subset of $c'^{\mu}$ factors through $(G^{c}_{c'})^{\mu}$. 
	
	In other words, we need to show that given a path $\beta:I \to Y_{\epsilon}$ from
	$(m,1,\id,\alpha=(s))$ to itself, if we lift $\beta(r)$ to a path
	$\tilde{\beta}:I \to Y'_{\epsilon}$ starting at a fiber over
	$(m,1,\id,\alpha=(s))$, then the path $\tilde{H}_{\epsilon}(\tilde{\beta}(1),-)$
 is obtained from $\tilde{H}_{\epsilon}(\tilde{\beta}(0),-)$ by the action of
some element of $(G^{c}_{c'})^\mu$. Since we are free to change $\beta$ up to
homotopy, 
we can assume that $\beta$ is a finite concatenation of paths $s_k, 1
\leq k \leq l$ with each $s_k$ a path moving along across the middle of one of the rectangles
$J_i^{\epsilon}$. One can see that each of these moves acts on the element of
$(c')^{\mu}$ by the rack action of elements of $c$, which in particular act through the group $(G^{c}_{c'})^{\mu}$. 
This map from the fundamental group to $(G^{c}_{c'})^{\mu}$ is a homomorphism, concluding the proof.
\end{proof}
\begin{remark}
	We note that it often seems unnecessary to invert
	$|G^{c}_{c'}|$ to make \autoref{proposition:trivialize-action} true. For
	example in the case that $S$ comes from a group action, as in
	\autoref{example:group-hurwitz-modules}, it seems to hold integrally. However we don't know any argument integrally proving the proposition for an arbitrary bijective Hurwitz module.
\end{remark}

\begin{theorem}
	\label{theorem:identify-stable-localization}
	Let $c$ be a finite rack and let $S$ be a finite bijective
	Hurwitz module over $c$.
	Suppose $c' \subset c$ is a subrack which is an $S$-component of $c$.
	Let $H$ be the product $|G^{c'}_{c}||G^{c'}_{S}||G^{c}_{c'}|$.
	The natural map induces an equivalence
	\begin{equation}
	\begin{aligned}
		\label{equation:module-localization-equivalence}
		&C_*(\Hur^{c,S})[\alpha_{c'}^{-1},
		H^{-1}] \\
		&\simeq
C_*(\Hur^{c/c',S/c'}
\times_{\pi_0
\Hur^{c/c',S/c'}}\pi_0\Hur^{c,S})[\alpha_{c'}^{-1},H^{-1}].
	\end{aligned}
\end{equation}
\end{theorem}
\begin{proof}
	Note that both the source and target of
	\eqref{equation:module-localization-equivalence} are $0$-nilpotent
	complete with respect to $C_*\left(\pi_0 \Hur^c \right)[\alpha_{c'}^{-1}]$ as 
	$C_*(\Hur^c)[\alpha_{c'}^{-1}]$ modules,
(in the sense of \cite[Definition
4.0.1]{landesmanL:homological-stability-for-hurwitz})
	by \cite[Lemma
	4.0.4]{landesmanL:homological-stability-for-hurwitz}. Therefore, to
	prove \eqref{equation:module-localization-equivalence}, for every $n
	\geq 0$, it suffices to
	identify
	\begin{align*}
	&\left(C_*\left(\pi_0 \Hur^c \right)[\alpha_{c'}^{-1}]
	\right)^{\otimes_{\left(C_*(\Hur^c)[\alpha_{c'}^{-1}]\right)} n+1}
	\otimes_{C_*(\Hur^c)[\alpha_{c'}^{-1}] }
	C_*(\Hur^{c,S})[\alpha_{c'}^{-1},H^{-1}] \\
	&\simeq  \left(C_*(\pi_0
	\Hur^c)[\alpha_{c'}^{-1}]\right)^{\otimes_{\left(C_*(\Hur^c)[\alpha_{c'}^{-1}] \right)} n+1} \otimes_{
C_*(\Hur^c)[\alpha_{c'}^{-1}]} \\
&\qquad   C_*\left(\Hur^{c/c',S/c'} \times_{\pi_0
\Hur^{c/c',S/c'}} \pi_0 \Hur^{c,S} \right)[\alpha_{c'}^{-1}, H^{-1}].
	\end{align*}
	The case $n > 0$ follows
	from the case $n = 0$ by applying $n$ times the functor
	\begin{align*}
	C_*\left(\pi_0 \Hur^c \right)[\alpha_{c'}^{-1}]
	\otimes_{C_*(\Hur^c)[\alpha_{c'}^{-1}]} (-).
	\end{align*}
	Hence, it suffices to prove the case $n = 0$, which we can rewrite as
	\begin{equation}
	\begin{aligned}
		\label{equation:module-descent-0}
		&C_*\left(\pi_0 \Hur^c \right)[\alpha_{c'}^{-1}]
		\otimes_{C_*(\Hur^c)[\alpha_{c'}^{-1}]}
		C_*(\Hur^{c,S})[\alpha_{c'}^{-1}, H^{-1}] \\
	&\simeq  C_*(\pi_0 \Hur^c)[\alpha_{c'}^{-1}]
	\otimes_{C_*(\Hur^c)[\alpha_{c'}^{-1}]} C_*\left(\Hur^{c/c',S/c'} \times_{\pi_0
\Hur^{c/c',S/c'}} \pi_0 \Hur^{c,S} \right)[\alpha_{c'}^{-1}, H^{-1}].
	\end{aligned}
\end{equation}
	We first claim it is enough to check that these are equivalent after applying 
	\begin{align*}
		C_*\left(\pi_0 \Hur^{c/c'} \right)[\alpha_{c'/c'}^{-1}]
	\otimes_{C_*\left(\pi_0 \Hur^c \right)[\alpha_{c'}^{-1}]}(-)	
	\end{align*}
	to both the source and target of \eqref{equation:module-descent-0}.
	To see this, first observe that 
	the map $\omega: \pi_0 \Hur^{c/c'}[\alpha_{c'/c'}^{-1}] \to
	\pi_0 \Hur^c [\alpha_{c'}^{-1}]$ is surjective, as is easy to see
	directly and is stated in
	\autoref{lemma:module-surjective}(1).
	Hence, 
	the quotient of the
	source of $\omega$ by $\ker \omega$ is the target of $\omega.$
	We can also identify $\ker \omega$ as the kernel of the map
	$\pi_0 \Hur^{c'}
	[\alpha_{c'}^{-1}] \to \pi_0 \Hur^{c'/c'} [\alpha_{c'/c'}^{-1}]$, whose
	order we have inverted by \cite[Lemma
	6.0.4]{landesmanL:homological-stability-for-hurwitz}. 
	It is clear that
	the action of the finite group $\ker \omega$ on the right hand side of
	\eqref{equation:module-descent-0} acts just on components, and it
	follows from \autoref{proposition:trivialize-action} that the same is true of the left hand side. Thus it suffices to prove the equivalence after taking the orbits by the action. 
	That is, it suffices to prove
		\begin{equation}
		\begin{aligned}
			\label{equation:module-descent-1}
			&C_*\left(\pi_0 \Hur^{c/c'} \right)[\alpha_{c'/c'}^{-1}]
			\otimes_{C_*(\Hur^c)[\alpha_{c'}^{-1}]}
			C_*(\Hur^{c,S})[\alpha_{c'}^{-1}, H^{-1}] \\
			&\simeq  C_*(\pi_0 \Hur^{c/c'})[\alpha_{c'/c'}^{-1}]
			\otimes_{C_*(\Hur^c)[\alpha_{c'}^{-1}]} C_*\left(\Hur^{c/c',S/c'} \times_{\pi_0
				\Hur^{c/c',S/c'}} \pi_0 \Hur^{c,S} \right)[\alpha_{c'}^{-1}, H^{-1}].
		\end{aligned}
	\end{equation}
	
	To this end,
	we observe
	\begin{equation}
\begin{aligned}
	\label{equation:module-right-simplicifcation}
	&C_*(\pi_0 \Hur^{c/c'})[\alpha_{c'/c'}^{-1}]
	\otimes_{C_*(\Hur^c)[\alpha_{c'}^{-1}]} C_*\left(\Hur^{c/c',S/c'} \times_{\pi_0
\Hur^{c/c',S/c'}} \pi_0 \Hur^{c,S} \right)[\alpha_{c'}^{-1}, H^{-1}] \\
	&\simeq C_*(\pi_0 \Hur^{c/c'})[\alpha_{c'/c'}^{-1}]
	\otimes_{C_*\left(\Hur^{c/c'} \times_{\pi_0 \Hur^{c/c'}} \pi_0
	\Hur^c\right)[\alpha_{c'}^{-1}]} \\
	&\qquad C_*\left(\Hur^{c/c',S/c'} \times_{\pi_0
\Hur^{c/c',S/c'}} \pi_0 \Hur^{c,S} \right)[\alpha_{c'}^{-1}, H^{-1}] \\
&\simeq C_*(\pi_0 \Hur^{c/c'}[\alpha_{c'/c'}^{-1}])
	\otimes_{C_*\left(\Hur^{c/c'} \times_{\pi_0 \Hur^{c/c'}} \pi_0
	\Hur^c\right)} C_*\left(\Hur^{c/c',S/c'} \times_{\pi_0
	\Hur^{c/c',S/c'}} \pi_0 \Hur^{c,S} \right)[H^{-1}] \\
	&\simeq C_*\left( \pi_0 \Hur^{c/c'}[\alpha^{-1}_{c'/c'}] \otimes_{\left(\Hur^{c/c'} \times_{\pi_0 \Hur^{c/c'}} \pi_0
		\Hur^c\right)} \left( \Hur^{c/c',S/c'} \times_{\pi_0
	\Hur^{c/c',S/c'}} \pi_0\Hur^{c,S} \right)\right)[H^{-1}]
		\\
		&\simeq C_*\left( \left( \pi_0 \Hur^{c/c'}[\alpha^{-1}_{c'/c'}] \times_{\pi_0 \Hur^{c/c'}[\alpha^{-1}_{c'}]}
		\pi_0 \Hur^{c/c'}[\alpha^{-1}_{c'/c'}]  \right)
		\otimes_{\left(\Hur^{c/c'}[\alpha^{-1}_{c'}] \times_{\pi_0 \Hur^{c/c'}[\alpha^{-1}_{c'}]} \pi_0
	\Hur^c[\alpha^{-1}_{c'}]\right)} \right.  \\
	& \qquad \left. \left( \Hur^{c/c',S/c'}[\alpha^{-1}_{c'}] \times_{\pi_0
	\Hur^{c/c',S/c'}[\alpha^{-1}_{c'}]} \pi_0\Hur^{c,S}[\alpha^{-1}_{c'}]
	\right)\right)[H^{-1}]
 \\
&\simeq C_* \left( \left( \pi_0 \Hur^{c/c'}[\alpha_{c'}^{-1}] \otimes_{\Hur^{c/c'}[\alpha_{c'}^{-1}]}
		\Hur^{c/c',S/c'}[\alpha_{c'}^{-1}] \right) \times_{\left(\pi_0 \Hur^{c/c'}[\alpha_{c'}^{-1}]
	\otimes_{\pi_0 \Hur^{c/c'}[\alpha_{c'}^{-1}]} \pi_0 \Hur^{c/c',S/c'}[\alpha_{c'}^{-1}] \right)} \right.
	\\
	&\qquad
	\left. \left( \pi_0 \Hur^{c/c'}[\alpha_{c'/c'}^{-1}] \otimes_{\pi_0 \Hur^c[\alpha_{c'}^{-1}]} \pi_0 \Hur^{c,S}[\alpha_{c'}^{-1}]
	\right)\right)[H^{-1}]
	\\
	&\simeq C_*\left(\left( \pi_0 \Hur^{c/c'}[\alpha_{c'}^{-1}] \otimes_{\Hur^{c/c'}[\alpha_{c'}^{-1}]}
	\Hur^{c/c',S/c'}[\alpha_{c'}^{-1}] \right) \times_{\pi_0 \Hur^{c/c',S/c'}[\alpha_{c'}^{-1}]}
\right.
	\\
	&\qquad
	\left.\left( \pi_0 \Hur^{c/c'}[\alpha_{c'/c'}^{-1}] \otimes_{\pi_0 \Hur^c[\alpha_{c'}^{-1}]} \pi_0 \Hur^{c,S}[\alpha_{c'}^{-1}]
	\right) \right)[H^{-1}]
\\
	&\simeq C_*\left(\left( \pi_0 \Hur^{c/c'}[\alpha_{c'}^{-1}] \otimes_{\Hur^{c/c'}[\alpha_{c'}^{-1}]}
	\Hur^{c/c',S/c'}[\alpha_{c'}^{-1}] \right) \times_{\pi_0 \Hur^{c/c',S/c'}[\alpha_{c'}^{-1}]}
\pi_0 \Hur^{c/c',S/c'}[\alpha_{c'}^{-1}] \right) \\
&\simeq C_*\left( \pi_0 \Hur^{c/c'}[\alpha_{c'}^{-1}] \otimes_{\Hur^{c/c'}[\alpha_{c'}^{-1}]}
	\Hur^{c/c',S/c'}[\alpha_{c'}^{-1}] \right)[H^{-1}]
\end{aligned}
\end{equation}
Indeed, the first equivalence in \eqref{equation:module-right-simplicifcation} above uses
\autoref{theorem:stable-homology-ring}.
The 
fifth equivalence 
in
\eqref{equation:module-right-simplicifcation}
uses 
\autoref{proposition:kan-fibration} via
\autoref{lemma:module-surjective}(1) and (4)
and
\cite[Theorem B.4]{bousfield2006homotopy}.
The seventh uses that the base of the fiber product is discrete so that it
suffices to identify the $\mathbb Z[H^{-1}]$ homology of the right hand terms.
That is, it suffices to show
the map $\Omega: \pi_0
\Hur^{c/c'}[\alpha_{c'/c'}^{-1}] \otimes_{\pi_0 \Hur^c[\alpha_{c'}^{-1}]} \pi_0
\Hur^{c,S}[\alpha_{c'}^{-1}]\to \pi_0\Hur^{c/c',S/c'}[\alpha_{c'}^{-1}]$ is a
$\ZZ[H^{-1}]$-homology equivalence, which we next explain. Indeed, the source
and target of $\Omega$ have no higher homology groups since we have inverted the order of the kernel of the map $\pi_0 \Hur^c[\alpha_{c'}^{-1}] \to \pi_0 \Hur^{c/c'}[\alpha_{c'/c'}^{-1}]$. 
Moreover $\pi_0$ of the source and target of $\Omega$ can be seen to agree as
they are identified with $H_0$ of the source and target of the map in \autoref{proposition:module-chain-homotopy}.
We also have equivalences
\begin{equation}
\begin{aligned}
	\label{equation:module-left-simplification}
	&C_*\left(\pi_0 \Hur^{c/c'} \right)[\alpha_{c'/c'}^{-1}]
	\otimes_{C_*(\Hur^c)[\alpha_{c'}^{-1}]}
	C_*(\Hur^{c,S})[\alpha_{c'}^{-1},H^{-1}]
		\\
&\simeq C_*\left(\pi_0 \Hur^{c/c'} \right)[\alpha_{c'/c'}^{-1}]
		\otimes_{C_*(\Hur^c)}
		C_*(\Hur^{c,S})[H^{-1}]
		\\
	&\simeq C_* \left( \pi_0 \Hur^{c/c'}[\alpha_{c'/c'}^{-1}] \otimes_{\Hur^c} \Hur^{c,S}
	\right)[H^{-1}]
	\\
		&\simeq C_* \left( \pi_0 \Hur^{c/c'}[\alpha_{c'/c'}^{-1}] \otimes_{\Hur^{c/c'}} \Hur^{c/c',S/c'}
		\right)[H^{-1}]
	\\
\end{aligned}
\end{equation}

The third equivalence 
of 
\eqref{equation:module-left-simplification}
uses
\autoref{proposition:module-chain-homotopy}.

Finally, the final line of \eqref{equation:module-left-simplification}
agrees with the final line of \eqref{equation:module-right-simplicifcation}
while the first lines of these respective equations agree with the two sides of
\eqref{equation:module-descent-1}, and hence 
\eqref{equation:module-descent-1} holds (and the equivalences identify with the natural comparison map).
\end{proof}

We now easily deduce \autoref{theorem:one-large-stable-homology}
from \autoref{theorem:identify-stable-localization} and
\autoref{theorem:homology-stabilizes-intro}.
The proof is similar to that given in
\cite[\S4.2.3]{landesmanL:the-stable-homology-of-non-splitting}.

\subsubsection{Proof of \autoref{theorem:one-large-stable-homology}}
\label{subsubsection:stable-value-proof}

Using notation from \autoref{construction:S-filtration},
we will consider 
$CA_{c,S}, A_{c,S}, CA_{c/c_1,S/c_1}, A_{c/c_1, S/c_1}$ as graded rings 
with respect to the number of elements in the $S$-component $c_1 \subset c$ or
the $S/c_1$-component $c_1/c_1 \subset c/c_1$, where the relevant grading was
defined precisely in
\autoref{notation:hur}.

Using \autoref{theorem:homology-stabilizes-intro},
for $n > Ii + J$
every element $\alpha_x$ for $x\in c_1$ induces an isomorphism from the $n$th
graded part of $H_i(CA_{c,S})$ to the $n+1$st graded part of $H_i(CA_{c,S})$.
Therefore, the $n$th graded part of
$H_i (CA_{c,S})$ agrees with
the $n$th graded part of 
$H_i(CA_{c,S})[\alpha_{c_1}^{-1}]$.
Similarly, 
the $n$th graded part of 
$H_i (CA_{c/c_1,S/c_1})$ agrees with
the $n$th graded part of
$H_i(CA_{c/c_1,S/c_1})[\alpha_{c_1/c_1}^{-1}]$.

Therefore, it suffices to show
\begin{align*}
	H_i(\CHur^{c,S})[\alpha_{c_1}^{-1}, H^{-1}]
\simeq
H_i(\CHur^{c/c',S/c'} \times_{\pi_0 \Hur^{c/c',S/c'}} \pi_0
\Hur^{c,S})[\alpha_{c_1/c_1}^{-1}, H^{-1}]
\end{align*}
for $H := |G^{c'}_c|\cdot |G^{c}_{c'}| \cdot|G^{c'}_{S}|^{-1}$.
This follows from
\autoref{theorem:identify-stable-localization},
since the equivalence there sends the components of 
$\Hur^{c,S}$ contained in 
$\CHur^{c,S}$
to the components of
\begin{align*}
\Hur^{c/c',S/c'} \times_{\pi_0 \Hur^{c/c',S/c'}} \pi_0 \Hur^{c,S}
\end{align*}
contained in 
$\CHur^{c/c',S/c'} \times_{\pi_0 \Hur^{c/c',S/c'}} \pi_0 \Hur^{c,S}$.
\qed

\section{Application to the BKLPR conjectures}
\label{section:poonen-rains}

In the special case that $c$ is a rack corresponding to a single conjugacy class
in a group which satisfies an additional {\em non-splitting} property, 
a version of \autoref{theorem:stable-homology}, showing that the homology of
bijective Hurwitz modules stabilize, was already proven in
\cite[Theorem
4.2.6]{ellenbergL:homological-stability-for-generalized-hurwitz-spaces}.
However, the stable value of this stable homology was not determined there.
Using our determination of the value of this stable homology in
\autoref{theorem:one-large-stable-homology},
we are able to upgrade \cite[Theorem
1.1.2]{ellenbergL:homological-stability-for-generalized-hurwitz-spaces}
from a statement with a large $q$ limit to a statement which holds for a fixed
sufficiently large $q$. 

\begin{remark}
	\label{remark:}
	In what follows, we spell out the details of the proof that the
	BKLPR moments are as predicted in suitable quadratic twist
	families.
	Our main result here is \autoref{theorem:poonen-rains-moments}
	where we verify the moments of Selmer groups in quadratic twist families
	are as predicted by the BKLPR conjectures, at least for fixed sufficiently large finite
	extensions of the ground field, depending on the moment.
	This improves on a previous result of the first author with Jordan
	Ellenberg \cite[Theorem
	1.1.6]{ellenbergL:homological-stability-for-generalized-hurwitz-spaces}
	where we only computed the $H$ moments in this context in the large $q$
	limit,
	while here we compute these moments for fixed $q$ as above, without
	needing to take a limit.
	However, the proof of \autoref{theorem:poonen-rains-moments}
	is extremely similar to that of 
	\cite[Theorem
	1.1.6]{ellenbergL:homological-stability-for-generalized-hurwitz-spaces}
	where the new ingredient we now have is the computation of the stable
	cohomology of the relevant spaces coming from
	\autoref{theorem:one-large-stable-homology}. We conclude 
	\autoref{theorem:poonen-rains-moments} by plugging
	the result of 
\autoref{theorem:one-large-stable-homology}
in the the rather general \cite[Lemma
5.2.2]{landesmanL:the-stable-homology-of-non-splitting}.

	Because this is a rather formal verification, and relatively
	straightforward proof depends rather heavily
	on the notation introduced in the long paper
	\cite{ellenbergL:homological-stability-for-generalized-hurwitz-spaces}
	we have opted to avoid reintroducing notation already defined at length
	(which would take many pages)
	in
	\cite{ellenbergL:homological-stability-for-generalized-hurwitz-spaces}
	and instead content ourselves with referencing the definitions made in
	that paper.
	For the reader unacquainted with 
	\cite{ellenbergL:homological-stability-for-generalized-hurwitz-spaces},
	the summary of the notation in \cite[Figure
	2]{ellenbergL:homological-stability-for-generalized-hurwitz-spaces} may
	be helpful.
\end{remark}

For the statement of the next theorem, we use the notation
$\bklpr \nu$ and $\paritybklpr \nu i$ for $i\in \{0,1\}$
as random variables modeling distributions of Selmer groups, and distributions
of Selmer groups conditioned on the parity of the rank being $i$, as defined in 
\cite[Definition
2.2.3]{ellenbergL:homological-stability-for-generalized-hurwitz-spaces}.
We use $\mathbb E|\hom(R,H)|$ to denote the expected number
	of homomorphisms from a random variable $R$ as above to the finite group
	$H$.
We also use the notation $\qtwist n U B$ for the stack of double covers of $U$,
branched over a degree $n$ divisor,
see \cite[Notation
5.1.4]{ellenbergL:homological-stability-for-generalized-hurwitz-spaces} for a
precise definition;
this can be thought of as a moduli space of quadratic twists.
We use the notation 
$\selspacemoments {\mathscr F^n_B} H$
and
$\rankselspacemoments {\mathscr F^n_B}H$
for certain twists of Hurwitz stacks parameterizing pairs of an elliptic curve and 
a suitable collection of Selmer elements,
as defined in \cite[Notation
8.2.1]{ellenbergL:homological-stability-for-generalized-hurwitz-spaces}.
The next result is stronger than, but similar to, \cite[Theorem
9.2.4]{ellenbergL:homological-stability-for-generalized-hurwitz-spaces} and the
proof is quite similar.

\begin{theorem}
        \label{theorem:symplectic-sheaf-point-counts}
        Suppose $B = \spec R$ for $R$ a DVR of generic characteristic $0$
                with
        closed point $b$ with residue field $\mathbb F_{q_0}$ and geometric
        point $\overline b$ over $b$.
Suppose $\nu$ is an odd integer and $r \in \mathbb Z_{>0}$ so that every prime $\ell \mid \nu$
        satisfies $\ell > 2r + 1$.
        Let $B$ be an integral affine base scheme,
        $C$ a smooth proper curve with geometrically connected fibers of genus
	$g$ over $B$,
        $Z \subset C$ finite \'etale nonempty over $B$ of degree $f+1$, and $U := C - Z$, with
        $j: U \to C$ the inclusion. Suppose
        $2\nu$ is invertible on $B$.
        Let $\mathscr F$ be a rank $2r$, tame, locally constant constructible, symplectically self-dual sheaf
        of free $\mathbb Z/\nu \mathbb Z$ modules over $U$ (see \cite[Definition
		5.1.1]{ellenbergL:homological-stability-for-generalized-hurwitz-spaces}.
        We assume there is some point $x \in C_{\overline b}$ at which
        $\drop_x(\mathscr F_{\overline b}[\ell]) = 1$ for every prime $\ell \mid
	\nu$ (see \cite[Definition
	5.2.4]{ellenbergL:homological-stability-for-generalized-hurwitz-spaces}).
        Also suppose
        $\mathscr F_{\overline b}[\ell]$ is irreducible
        for each $\ell \mid \nu$, and that 
        the map $j_* \mathscr F_{\overline b}[\ell^w] \to j_* \mathscr F_{\overline b}[\ell^{w-t}]$ is
        surjective for each prime $\ell \mid \nu$ such that $\ell^w \mid \nu$, and
        $w \geq t$.
        Fix $A \to U_b$ a polarized abelian scheme with polarization degree
        prime to $\nu$.
Suppose 
$\mathscr F$
satisfies
$\mathscr F_b \simeq A[\nu]$.
        For any finite $\mathbb Z/\nu \mathbb Z$ module $H$, 
        and any finite field extension $\mathbb F_{q_0} \subset \mathbb F_q$,
        there are constants $I(H)$,
        $J(H)$,
        $C_{H}$, 
	depending on $H$,
	as well as 
	$C_{H,g,f}$, depending on $H,g,$ and $f$, so that for $\sqrt q > 2C_{H}$ and $n > C_{H}$ even,
                \begin{align}
                \label{equation:full-moments}
        \left|\frac{ |\selspacemoments {\mathscr F^n_B} H (\mathbb
        F_q)|}{q^{n}} - \mathbb E|\hom(\bklpr
        \nu ,H)| \cdot \frac{ |\qtwist n U B(\mathbb F_q)|}{ q^n} \right| &\leq
	\frac{4C_{H,g,f}}{1- \frac{C_{H}}{\sqrt{q}}}
        \left(\frac{C_{H}}{\sqrt{q}}\right)^{\frac{n-J( H)}{I(H)}} 
 \\
        \label{equation:parity-moments}
                \left |\frac{\# \rankselspacemoments {\mathscr F^n_B}H (\mathbb F_q)}{
        q^{n}} - \mathbb E|\hom(\paritybklpr
        \nu {\rk V_{\mathscr F^n_B} \bmod 2},H)|\cdot \frac{ |\qtwist n U B(\mathbb F_q)|}{ q^n}\right| &\leq 
	\frac{4C_{H,g,f}}{1- \frac{C_{H}}{\sqrt{q}}}\left(\frac{C_{H}}{\sqrt{q}}\right)^{\frac{n-J(H)}{I(H)}}.
        \end{align}
\end{theorem}
\begin{proof}
        We aim to prove this by applying \cite[Lemma
        5.2.2]{landesmanL:the-stable-homology-of-non-splitting}
        to the two sequence of stacks
$\selspacemoments {\mathscr F^n_B} H$
and
$\rankselspacemoments {\mathscr F^n_B}H$ for $B = \spec \mathbb F_q$.
To apply this, we need to verify the two conditions 
of
\cite[Lemma 5.2.2]{landesmanL:the-stable-homology-of-non-splitting}.
For the reader's convenience, we note that
\cite[Lemma 5.2.2]{landesmanL:the-stable-homology-of-non-splitting}
is a lemma that provides a bound on the limiting number of $\mathbb F_q$ points
of a sequence of varieties granting two conditions: first that the trace of
Frobenius on their cohomologies stabilize and second the their cohomology is
exponentially bounded.

To verify the first condition
\cite[Lemma 5.2.2]{landesmanL:the-stable-homology-of-non-splitting}(1), we first
claim that the composite map
$\psi: \rankselspacemoments {\mathscr F^n_B} H \xrightarrow{\phi}
\selspacemoments {\mathscr F^n_B} H \to \qtwist n U B \to \conf n U B$,
induces an isomorphism on stable
cohomology on each component; 
this means concretely that there are constants $I$ and $J$,
depending only on $H$ (and not on $\mathscr F$), so that for $n > Ii + J$, and $Z \subset 
\rankselspacemoments {\mathscr F^n_B} H$ any component, the map $H^i(\conf n U B,\mathbb Q_\ell)
\to H^i(Z,\mathbb Q_\ell)$ is an isomorphism.
Observe also since the map $\phi$ above is a finite \'etale cover, this also implies that the
stable cohomology of
$\selspacemoments {\mathscr F^n_B} H$ is identified with that of 
$\conf n U B$.

We next set out to show the composite map $\psi: 
\rankselspacemoments {\mathscr F^n_B} H \to \conf n U B$
induces an isomorphism on stable
cohomology on each component.
Since these stacks are smooth and are gerbes over their coarse spaces, they have
cohomology groups isomorphic to that of their coarse spaces, via the coarse
space map, by
\cite[Proposition 2.2.8]{behrend:thesis}.
Hence, it suffices to verify the claim regarding the stable cohomology 
when $B = \spec \mathbb C$ using 
the isomorphism between their cohomology over $\mathbb C$ and over $\mathbb F_q$
coming from \cite[Proposition 7.7]{EllenbergVW:cohenLenstra},
which in turn uses the normal crossings compactification of $\conf n U B$ coming from
\cite[Corollary
B.1.4]{ellenbergL:homological-stability-for-generalized-hurwitz-spaces}.
We next relate the cohomology of 
$\rankselspacemoments {\mathscr F^n_B} H$ to that of a certain Hurwitz space
$\on{Hur}_{\mathscr F^n_B}^{H,\on{rk}}$
(which is defined in 
\cite[Notation
8.2.1]{ellenbergL:homological-stability-for-generalized-hurwitz-spaces}
as a double cover of the Hurwitz stack 
$\on{Hur}_{\mathscr F^n_{\mathbb C}}^H$ 
described in
\cite[Notation
6.2.1]{ellenbergL:homological-stability-for-generalized-hurwitz-spaces})
and
$\on{Hur}^{\rk}_{S^n_{\mathscr F,H,g,f}}$ 
(defined in \cite[Example
8.1.11]{ellenbergL:homological-stability-for-generalized-hurwitz-spaces}
as a double cover of the Hurwitz scheme 
$\on{Hur}_{S^n_{\mathscr F,H,g,f}}$ 
defined in 
\cite[Example
8.1.3]{ellenbergL:homological-stability-for-generalized-hurwitz-spaces}).
In the case $B = \spec \mathbb C$,
we can use the isomorphism from
\cite[Corollary 6.4.7]{ellenbergL:homological-stability-for-generalized-hurwitz-spaces}
which identifies 
$\rankselspacemoments {\mathscr F^n_B} H$ with
$\on{Hur}_{\mathscr F^n_B}^{H,\on{rk}}$
to reduce to identifying the stable cohomology of each component of
$\on{Hur}_{\mathscr F^n_B}^{H,\on{rk}}$
with that of $\conf n U B$.
Moreover, the Hurwitz space
$\on{Hur}^{\rk}_{S^n_{\mathscr F,H,g,f}}$ 
(which is roughly a version of 
$\on{Hur}_{\mathscr F^n_B}^{H,\on{rk}}$
where one marks a point of the cover over infinity)
is a finite unramified covering space
of
$\on{Hur}_{\mathscr F^n_{\mathbb C}}^{H,\on{rk}}$.
Hence, it suffices to show the stable cohomology of each component of
$\on{Hur}^{\rk}_{S^n_{\mathscr F,H,g,f}}$ 
agrees with
that of $\conf n U B$.
Let 
$c$ denote the conjugacy class of order $2$ elements in $\mathbb Z/2 \mathbb Z \ltimes H$,
where $\mathbb Z/2 \mathbb Z$ acts on $H$ by negation.
Then
$\on{Hur}^{\rk}_{S^n_{\mathscr F,H,g,f}}$
can be identified with $\Hur^{c,S}$, where $c$ is the rack described above and
$S = (\Sigma^1_{g,f}, \{T_n\}_{n \in \mathbb Z_{\geq 0}},
        \{\psi_n\}_{n \in \mathbb Z_{\geq 0}})$
	is a bijective Hurwitz module described in \cite[Lemma 8.1.8]{ellenbergL:homological-stability-for-generalized-hurwitz-spaces}.
	(Technically, 
\cite[Lemma
8.1.8]{ellenbergL:homological-stability-for-generalized-hurwitz-spaces}
describes a coefficient system, which is like a bijective Hurwitz module valued
in vector spaces instead of sets, but 
\cite[Remark 8.1.9]{ellenbergL:homological-stability-for-generalized-hurwitz-spaces}
explains that the relevant vector space is actually the free vector space on a
set, so this coefficient system actually comes from a bijective Hurwitz module.)
By \autoref{lemma:quotient-components-conf}, each component of
$\Hur_n^{c/c,S/c}$ is identified with $\conf n U B$.

The claim regarding the existence of $I$ and $J$ depending only on $H$ at the
beginning of this proof then follows from
\autoref{theorem:one-large-stable-homology}.
Moreover, for $B = \spec \mathbb C$,
by
\autoref{theorem:one-large-stable-homology},
the stable homology of each component of $\Hur^{c,S}$ is identified with the stable homology
of $\Hur^{c/c,S/c}$ and hence with that of $\conf n U B$.

In order to complete the verification of 
\cite[Lemmma 5.2.2]{landesmanL:the-stable-homology-of-non-splitting}(1),
when $B = \spec \overline{\mathbb F}_q$
we need to show the trace of $\on{Frob}_q^{-1}$, for $\on{Frob}_q$ geometric
Frobenius, on the stable cohomology 
of $\conf n U B$ stabilizes.
Indeed, this follows from 
\cite[Theorem 1.2]{petersen:a-spectral-sequence}(2).
Furthermore, we need to determine the number of components of the above Selmer
spaces.
Indeed, the number of geometric components is given by
\cite[Proposition
9.2.1]{ellenbergL:homological-stability-for-generalized-hurwitz-spaces}, 
which shows that
every
component of both 
$\selspacemoments {\mathscr F^n_B} H$
and
$\rankselspacemoments {\mathscr F^n_B}H$
is geometrically connected, and the number of such components is also computed to
be 
$\mathbb E|\hom(\bklpr \nu ,H)| $
and
$\mathbb E|\hom(\paritybklpr \nu {\rk V_{\mathscr F^n_B} \bmod 2},H)|$
in the two cases.

To verify the second condition,
\cite[Lemma 5.2.2]{landesmanL:the-stable-homology-of-non-splitting}(2)
for $S$ as earlier in this proof,
we wish to show there are constants $C_{H, g,f}$ and $C_H$ so that $\dim
H_i(\Hur^{c,S}_n) \leq C_{H, g,f} C_H^i$.
Indeed, this was essentially shown in 
\cite[Corollary 4.3.4 and Proposition
4.3.3]{ellenbergL:homological-stability-for-generalized-hurwitz-spaces}, except
the bound was written there in the form $K^{i+1}$ for a slightly different value
of $K$. However, examining the proof of 
\cite[Corollary 4.3.4 and Proposition
4.3.3]{ellenbergL:homological-stability-for-generalized-hurwitz-spaces},
specifically the fourth to last line, we see that we can take
$C_{H, g,f} := 2^{2g+f+J+2} |c|^{J +2}$
and $C_H := (2|H|)^I$
(upon noting that $|H| = |c|$ and $\mathbb U$ in 
\cite[Proposition 4.3.3]{ellenbergL:homological-stability-for-generalized-hurwitz-spaces}
can be taken to have degree $2$ using
\cite[Proposition
A.3.1]{ellenbergL:homological-stability-for-generalized-hurwitz-spaces}).

Combining the above, if we let $V_i$ denote the vector space with Frobenius
action equal to the $i$th 
cohomology of $\conf n U B$ for $n$ sufficiently large relative to $i$.
Then, the above application of \cite[Lemma
5.2.2]{landesmanL:the-stable-homology-of-non-splitting}
yields
\begin{align}
\label{equation:full-moments-stable-trace}
        \left|\frac{ |\selspacemoments {\mathscr F^n_B} H (\mathbb
        F_q)|}{q^{n}} - \mathbb E|\hom(\bklpr
        \nu ,H)| \cdot \sum_{i=0}^\infty (-1)^i \tr(\frob_q^{-1} |V_i)  \right| &\leq
	\frac{2C_{H,g,f}}{1- \frac{C_{H}}{\sqrt{q}}}
        \left(\frac{C_{H}}{\sqrt{q}}\right)^{\frac{n-J(H)}{I(H)}} 
 \\
        \label{equation:parity-moments-stable-trace}
                \left |\frac{\# \rankselspacemoments {\mathscr F^n_B}H (\mathbb F_q)}{
        q^{n}} - \mathbb E|\hom(\paritybklpr
        \nu {\rk V_{\mathscr F^n_B} \bmod 2},H)|\cdot \sum_{i=0}^\infty (-1)^i \tr(\frob_q^{-1} |V_i) \right| &\leq 
	\frac{2C_{H,g,f}}{1- \frac{C_{H}}{\sqrt{q}}}\left(\frac{C_{H}}{\sqrt{q}}\right)^{\frac{n-J(H)}{I(H)}} 
        \end{align}
To conclude, it remains to relate 
\eqref{equation:full-moments-stable-trace}
to \eqref{equation:full-moments}
and \eqref{equation:parity-moments-stable-trace}
to \eqref{equation:parity-moments}.
We next explain how to deduce
\eqref{equation:full-moments}
from \eqref{equation:full-moments-stable-trace}.
Note that 
$\qtwist n U B = \selspacemoments {\mathscr F^n_B} {\id}$.
Applying 
\eqref{equation:full-moments-stable-trace}
for both $H$ and $\id$ and adding the results, we find
\begin{align*}
        &\left|\frac{ |\selspacemoments {\mathscr F^n_B} H (\mathbb
        F_q)|}{q^{n}} - \mathbb E|\hom(\bklpr
        \nu ,H)| \cdot \frac{ |\qtwist n U B(\mathbb F_q)|}{ q^n} \right| \\
        &\leq
        \left|\frac{ |\selspacemoments {\mathscr F^n_B} H (\mathbb
        F_q)|}{q^{n}} - \mathbb E|\hom(\bklpr
        \nu ,H)| \cdot \sum_{i=0}^\infty (-1)^i \tr(\frob_q^{-1} |V_i)  \right|
        \\
        &+
\left|
\mathbb E |\hom(\bklpr \nu ,H)|
\frac{ |\qtwist n U B|}{q^{n}} - \mathbb E|\hom(\bklpr
        \nu ,H)| \cdot \sum_{i=0}^\infty (-1)^i \tr(\frob_q^{-1} |V_i)  \right| 
        \\
        &\leq
	\frac{2C_{\id,g,f}}{1- \frac{C_{\id}}{\sqrt{q}}}
        \left(\frac{C_{\id}}{\sqrt{q}}\right)^{\frac{n-J(\id)}{I(H)}} +
	\frac{2C_{H,g,f}}{1- \frac{C_{H}}{\sqrt{q}}}
        \left(\frac{C_{H}}{\sqrt{q}}\right)^{\frac{n-J(
        H)}{I(H)}} \\
	&\leq \frac{4 \max(C_{H,g,f}, C_{\id,g,f})}{1- \frac{ \max(C_{H}, C_{\id})}{\sqrt{q}}}
        \left(\frac{\max(C_{H},C_{\id})}{\sqrt{q}}\right)^{\frac{n-\max(J(H), J(\id))}{I(H)}}.
\end{align*}
So, by replacing $C_{H, g,f}$ with $\max(C_{H, g,f}, C_{\id,
g,f})$, replacing $C_{H}$ with $\max(C_{H}, C_{\id})$, 
and replacing $J(H)$ with $\max(J(H), J(\id))$,
we obtain
\eqref{equation:full-moments}.
Similarly, we can deduce 
\eqref{equation:parity-moments}
from
\eqref{equation:parity-moments-stable-trace}.
\end{proof}

\subsection{Proof of \autoref{theorem:poonen-rains-moments}}
\label{subsection:proof-poonen-rains}

\autoref{theorem:poonen-rains-moments}
	follows from \autoref{theorem:symplectic-sheaf-point-counts}
	in the same way that \cite[Theorem
	1.1.6]{ellenbergL:homological-stability-for-generalized-hurwitz-spaces}
	follows from \cite[Theorem
	9.2.4]{ellenbergL:homological-stability-for-generalized-hurwitz-spaces}.
	We note that the constant $C_H$ in 
\autoref{theorem:poonen-rains-moments}
is the square of the constant also called $C_H$ in 
\autoref{theorem:symplectic-sheaf-point-counts}.

	In a bit more detail, let $b = \spec \mathbb F_q$. We may view $(C,U,Z,A[\nu])$ as {\em symplectic
	sheaf data over $b$} in the sense of \cite[Definition
	10.2.2]{ellenbergL:homological-stability-for-generalized-hurwitz-spaces}.
	Let $B$ be a complete dvr with closed point $b$ and generic
	characteristic $0$. By \cite[Lemma
	10.2.3]{ellenbergL:homological-stability-for-generalized-hurwitz-spaces},
	we can realize
	$(C,U,Z,A[\nu])$ as the restriction along $b \to B$ of some symplectic
	sheaf data $(C_B, U_B, Z_B, \mathscr F_B)$ on $B$.

	Since $\sym^2 H$ is the $H$-surjection moment of the BKLPR distribution
	as explained in \cite[Proposition
	2.3.1]{ellenbergL:homological-stability-for-generalized-hurwitz-spaces},
	the result then follows from 
	\autoref{theorem:symplectic-sheaf-point-counts}
	and an inclusion-exclusion to show certain components of
	$\on{Sel}^H_{\mathscr F^n_b}$ 
	(defined in \cite[Notation
	8.2.1]{ellenbergL:homological-stability-for-generalized-hurwitz-spaces})
	correspond to surjections onto $H$, in
	place of all homomorphisms onto $H$.
	\qed

\subsection{Proof of \autoref{theorem:poonen-rains-elliptic-curves}}
\label{subsection:proof-poonen-rains-elliptic}

\autoref{theorem:poonen-rains-elliptic-curves} is a 
special case of the substantially more general 
\autoref{theorem:poonen-rains-moments}, as we now explain.
If we take the group $H$ appearing in 
\autoref{theorem:poonen-rains-moments}
to be $\mathbb Z/d \mathbb Z$,
we find $\# \sym^2 H = \# (\mathbb Z/d \mathbb Z) = d$. The order of 
$\# \sel_\nu(A_x)$ is then the sum of the number of surjections onto
$\mathbb Z/d \mathbb Z$ for each $d \mid \nu$.
It only remains to verify the hypotheses in 
\autoref{theorem:poonen-rains-elliptic-curves}
hold.
Note that $A[\nu] \to U$ is tame because we assume $q$ is prime to $6$.
The irreducibility assumption in 
\autoref{theorem:poonen-rains-moments}
holds in 
\autoref{theorem:poonen-rains-elliptic-curves}
by \cite[Proposition 2.7]{zywina:inverse-orthogonal}.
Note that a
nonconstant elliptic curve with squarefree discriminant is necessarily
nonisotrivial, and has a place of multiplicative reduction.
The remaining conditions in 
\autoref{theorem:poonen-rains-moments}
therefore hold for nonconstant elliptic curves of squarefree discriminant since
the geometric component group of the N\'eron model of an elliptic curve with
squarefree discriminant is trivial.
\qed

\begin{remark}
	\label{remark:explicit-constant}
The constants $C_\nu$ and $C_H$ appearing in
\autoref{theorem:poonen-rains-elliptic-curves} and
\autoref{theorem:poonen-rains-moments} are completely explicit, though large, and can be
computed by tracing through the proof.
The proof shows that when $H = \mathbb Z/\nu \mathbb Z$ we have $C_{\nu} = C_H$,
so we will just explain how to compute the constants $C_H$ as in
\autoref{theorem:poonen-rains-moments}.
Tracing through the proof gives that $C_H = (2|H|)^{2I}$, for $I$ the
slope coming from an application of \autoref{theorem:homology-stabilizes-intro}
associated to $c$ the set of order two elements in $\mathbb Z/2 \mathbb Z
\ltimes H$.
The value of this $I$ can be computed to be $(N_0+2)\cdot 2$
\cite[Proposition 4.4, Proposition 8.1, Theorem 7.1, Corollary
7.4]{randal-williams:homology-of-hurwitz-spaces},
for $N_0$ as in
\cite[Proposition 4.4]{randal-williams:homology-of-hurwitz-spaces}.
For example, if $H = \mathbb Z/5 \mathbb Z$, one can compute $N_0 = 5$ so one can take
$C_H = 10^{28}$.
We note that this is smaller 
than the constant appearing in 
\cite[Remark 5.3.2]{landesmanL:the-stable-homology-of-non-splitting}
since we slightly improved the constant $C_H$ in the proof of
\autoref{theorem:symplectic-sheaf-point-counts}
compared to the constant described in
\cite[Remark 5.3.2]{landesmanL:the-stable-homology-of-non-splitting}, resting on
\cite[Proposition
4.3.3]{ellenbergL:homological-stability-for-generalized-hurwitz-spaces}.
\end{remark}

\section{Bhargava's conjecture}
\label{section:bhargava-conjecture}

In this section, we prove \autoref{theorem:bhargava-analog}, which implies
\autoref{theorem:bhargava-intro} from the introduction.
This can be rephrased as a question about counting $\mathbb F_q$ points on
certain Hurwitz schemes of $S_d$ covers, and so in order to apply
\autoref{theorem:one-large-stable-homology-hurwitz-space}, we will want to
determine the number of components of the relevant Hurwitz schemes, which is
essentially the content of \autoref{lemma:partial-components}, though we
rephrase this over finite fields in \autoref{lemma:unique-component}.
We now build up to computing the components of these Hurwitz spaces.

\begin{example}
	\label{example:symmetric-group-components}
	Let $G$ be the symmetric group $S_d$ and $c \subset G$ the conjugacy class of
	transpositions.
	We now explain why $H_2(G,c) = 0$.
	It is shown in \cite[Theorem 2.5 and Theorem 3.1]{wood:an-algebraic-lifting-invariant}
	that $H_2(G,c)$ is identified with the number of components of 
	$\cphur G n c {\mathbb C}$ with trivial boundary monodromy for sufficiently large
	even $n$.
	The result then follows from the fact that Hurwitz spaces
	simply branched overs of
	$\mathbb P^1$ with sufficiently many branch points have a unique connected component 
	see \cite[p. 224-225]{clebsch:zur-theorie-der-riemannschen-flache}
	and \cite{hurwitz:uber-riemannshce-flaschen}
	for classical references, and 
	\cite[\S1]{EisenbudEHS:onTheHurwitz}
	for a more modern reference.
	In particular, it follows that for any $\widetilde{c}$ containing the conjugacy
	class of transpositions, we also have $H_2(G,\widetilde{c}) = 0$, because that is a
	quotient of $H_2(G,c)= 0$.
\end{example}
\begin{remark}
	\label{remark:alternate-symmetric-group-components}
	One can alternatively compute $H_2(G,c)$ from its definition as a
		quotient of $H_2(G; \mathbb Z)$. This is trivial for $d \leq 3$
		and $\mathbb Z/2 \mathbb Z$ for $d > 3$. One can verify that if
		one takes two distinct commuting transpositions $x, y \in S_d$ for $d > 3$, the corresponding
		element of $H_2(G; \mathbb Z)$ under the map $H_2(\mathbb Z^2;
		\mathbb Z) \to H_2(G; \mathbb Z), (i,j) \mapsto x^i y^j$ is
		nontrivial. Hence $H_2(G, c)$ is trivial.
\end{remark}

Before continuing, we pause to give a couple interesting examples of
computations of the stable components of Hurwitz spaces. The next two examples
will not be needed elsewhere in this paper.
In the next example, we show that there is a unique stable component of Hurwitz
spaces for $A_4$ when one has many elements of each conjugacy class, but when
one only has $3$-cycles, there are multiple stable components.
\begin{example}
	\label{example:alternating-components}
	Let $c' \subset A_4$ denote the set of $3$-cycles, which is a union of
	two conjugacy classes.
	Let $c := A_4 - \id$.
	We will show $H_2(A_4, c) = 0$ but $H_2(A_4, c') \neq 0$.
	Therefore, even though $c'$ generates $A_4$, the Hurwitz space for $c'$
	may have more dominant stable components than the Hurwitz space for $c$.
	Let $K_4 := \mathbb Z/2 \mathbb Z \times \mathbb Z/2 \mathbb Z$.
	To show the above claims, we use the exact sequence
	\begin{equation}
		\label{equation:}
		\begin{tikzcd}
			0 \ar {r} & K_4 \ar {r} & A_4 \ar {r} & \mathbb Z/3
			\mathbb Z \ar {r}
			& 0. 
	\end{tikzcd}\end{equation}
	This gives a spectral sequence which allows us to compute $H_2(A_4;
	\mathbb Z)$.
	The spectral sequence includes terms 
	\begin{align*}
		H_0(\mathbb Z/3\mathbb Z;
	H_2(K_4; \mathbb Z)) &= H_0( \mathbb Z/3 \mathbb Z; \mathbb Z/2 \mathbb
	Z) = \mathbb Z/2 \mathbb Z, \\
	H_1( \mathbb Z/3 \mathbb Z; H_1(K_4; \mathbb Z)) &= H_1(
	\mathbb Z/3 \mathbb Z; K_4) = 0, \\
	H_2( \mathbb Z/3 \mathbb Z; H_0(K_4; \mathbb Z)) &= H_2( \mathbb Z/3
	\mathbb Z; \mathbb Z) = 0.
	\end{align*}
	Using that $H_i(\mathbb Z/3\mathbb Z; H_j(K_4; \mathbb Z))$ is
	$3$-torsion for $i > 0$, the 
$H_0(\mathbb Z/3\mathbb Z;
	H_2(K_4; \mathbb Z))$ term must survive the spectral sequence and 
	we obtain an isomorphism 
$\mathbb Z/2 \mathbb Z \simeq H_0(\mathbb Z/3\mathbb Z;
H_2(K_4; \mathbb Z)) \simeq H_2(A_4; \mathbb Z)$.
The generator of this cohomology group corresponds to the generator 
$H_2(K_4; \mathbb Z)$, coming from a pair of distinct $(2,2)$ cycles.
Therefore, for $x, y \in A_4$ commuting elements
the map $H_2(\mathbb Z^2;
\mathbb Z) \to H^2(A_4; \mathbb Z)$
induced by $(x,y) \mapsto x^i y^j$
will be trivial when $x,y$ are $3$-cycles but nontrivial when $x,y$ are $(2,2)$
cycles.
This implies $H_2(A_4, c) = 0$ but $H_2(A_4, c') = \mathbb Z/2 \mathbb Z$.
\end{example}

\begin{example}
	\label{example:symmetric-components}
	A similar analysis to \autoref{example:alternating-components}, using
	that $S_4$ has normal subgroup $K_4 \simeq \mathbb Z/2 \mathbb Z \times
	\mathbb Z/2 \mathbb Z$ with quotient $S_3$ shows $\mathbb Z/2 \mathbb Z
	\simeq H_0(S_3;H_2(K_4; \mathbb Z)) \simeq H_2(S_4; \mathbb Z)$, and so
	$H_2(S_4; \mathbb Z)$ is generated by the image of 
$H_2(\mathbb Z^2;
\mathbb Z) \to H^2(S_4; \mathbb Z)$
induced by $(x,y) \mapsto x^i y^j$
for $x,y$ commuting transpositions.
\end{example}

\begin{lemma}
	\label{lemma:quotient-rack}
	Suppose $G$ is a finite group and $c, c' \subset G$ are two unions of
	conjugacy classes with $c' \subset c$. If $c'$ generates $G$ then $c/c'
	\simeq c/c$.
\end{lemma}
\begin{proof}
	We have to show that if $s,t \in c$ lie in the same orbit under the $c$
	conjugation action then they lie in the same orbit under the $c'$
	conjugation action.
	It suffices to show that if $s = x \cdot t$ for some $x \in c$ then
	there is a sequence of elements $y_1 \cdots y_k \in c'$ with $s = (y_1
	\cdots y_k) \cdot t$. Indeed, since $c'$ generates $G$, we can write $x = y_1 \cdots
	y_k$ with $y_i \in c'$, which gives the desired $y_1, \ldots, y_k$.
\end{proof}
We now prove our main result toward counting the components of Hurwitz spaces
$\CHur^c$ for $c \subset S_d$ the conjugacy class of transpositions.
\begin{lemma}
	\label{lemma:partial-components}
	Suppose $c \subset G$ is a union of conjugacy classes in the
	symmetric group $G = S_d$.
	Suppose $c' \subset c$ is the conjugacy class of transpositions.
	Then the map $\pi_0(\CHur^c)[(\alpha_{c'})^{-1}] \to G
	\times_{G^{\ab}}(\pi_0(\CHur^{c/c'})[(\alpha_{c'/c'})^{-1}])
	\simeq G \times_{G^{\ab}}(\pi_0(\CHur^{c/c})[(\alpha_{c'/c'})^{-1}])$,
	given by taking boundary monodromy in the first factor and taking the
	image of $c$ in $c/c$ in the second factor,
	is a bijection.
\end{lemma}
\begin{proof}
	The later isomorphism follows from the fact that $c/c' \simeq c/c$ using
	\autoref{lemma:quotient-rack}. Therefore, we will check the composite
	map is a bijection.
	Upon identifying $\pi_0(\CHur^{c/c}) \simeq \mathbb N^{|c/c|}$, we claim the map
	from the statement is a surjection. To see this, first note the map
	$\pi_0(\CHur^c)[(c')^{-1}]\simeq \pi_0(\CHur^{c/c})[(c'/c')^{-1}]$
	is a surjection. Moreover, we can 
	modify the boundary monodromy of the source 
	(within its coset of $A_d \subset S_d$)
	while preserving the
	number of branch points
by multiplying by some product of $\alpha_g$ and
$(\alpha_h)^{-1}$ for varying $g,h \in c'$.

To conclude, it is enough to show this map is injective.
In other words, suppose we have two classes $\mu$ and $\nu$, with the same image
in the target.
Since the homology of Hurwitz spaces stabilize once one has sufficiently many of
any given conjugacy class, see
\cite[Theorem 1.4.1]{landesmanL:homological-stability-for-hurwitz},
it is enough to show they have the same image after adding sufficiently many
transpositions to the right of both words, so long as we add the same transpositions to each.
By moving the transpositions to the right, we can arrange that $\mu = [a_1]
\cdots [a_k] [b_1] \cdots [b_j]$ and $\nu = [x_1] \cdots [x_k] [y_1] \cdots
[y_j]$ where $b_1, \ldots, b_j, y_1, \ldots, y_j$ consist of transpositions,
while there are no transpositions among $a_1, \ldots, a_k, x_1, \ldots, x_k$.
Moreover, we may assume that that $a_i$ and $x_i$ lie in the same
conjugacy class.
Next, using that transpositions generate $S_d$, by possibly adding the same set
of transpositions to the right of both elements, we
can use the braid group action by moving suitable transpositions in a full twist
around $[a_1] \cdots [a_k]$ and $[x_1] \cdots [x_k]$ to ensure that $a_1 = x_1$.
Repeating this, we may assume $a_i = x_i$ for all $1 \leq i \leq k$. It only
remains to ensure that $[b_1] \cdots [b_j]$ lies in the same braid group orbit as
$[y_1] \cdots [y_j]$, provided they have the same boundary monodromy.
This then follows from \autoref{example:symmetric-group-components}, which tells
us $H_2(G, c') = 0$ and hence it follows from 
\cite[Theorem 2.5]{wood:an-algebraic-lifting-invariant} and
\cite[Theorem 3.1]{wood:an-algebraic-lifting-invariant} that 
$[b_1] \cdots [b_j]$ lies in the same braid group orbit
$[y_1] \cdots [y_j],$ provided $j$ is sufficiently large
and also that $b_1, \ldots, b_j$ generate $G$ and $y_1, \ldots, y_j$ generate
$G$.
\end{proof}

So far we have identified the relevant stable components over $\mathbb C$, and
we next wish to identify its stable homology.

For $n_1, \ldots, n_\upsilon$ integers and $R$ a ring,
we use 
$\Conf_{n_1, \ldots, n_\upsilon, B}$ to denote the multi-colored configuration
space parameterizing $0$-dimensional subschemes of $\mathbb A^1_{\spec R}$ with a degree $n_i$
divisor of color $i$, see \cite[Definition
2.2.1]{landesmanL:homological-stability-for-hurwitz} for a more formal
definition. When $R = \mathbb C$, we omit that subscript.

For the next lemma, we suggest the reader review the function $\sigma$ defined in
\autoref{definition:sigma-weighting}. Before continuing let's see a brief
example.
\begin{example}
	\label{example:}
	So, for example, if $d = 3$, let $c_1$ be the conjugacy class of transpositions,
and $c_2$ be the conjugacy class of three-cycles.
Then, we claim $\sigma(n_1, n_2)$
is $1$ if $n_1$ is odd and $2$ if $n_1$ is even.
To see this, first note that
$n_1 c_1 + n_2 c_2$ has trivial image in $S_3^{\ab}$ if and only if $n_1$
is even. The claim then follows because
transpositions are the unique conjugacy class with nontrivial projection to
$S_3^{\ab}$, while there are two conjugacy classes with trivial projection to
$S_3^{\ab}$.
\end{example}
\begin{remark}
	\label{remark:}
	It may be helpful to note that the function $\sigma$ from
	\autoref{definition:sigma-weighting} is $2$-periodic as a function
of each of the inputs $n_1, \ldots, n_\upsilon$ because 
$S_d^{\ab} \simeq \mathbb Z/2 \mathbb Z$, and it is $1$-periodic as a function
of each input corresponding to a conjugacy class lying in $A_d$.
\end{remark}

We are now prepared to identify the stable homology of the relevant Hurwitz
spaces.

\begin{lemma}
	\label{lemma:cohomology-and-components}
	If $c_1 \subset S_d$ is the conjugacy class of transpositions and $c :=
	S_d - \id$, then
	there are constants $I$ and $J$ so that if $n_1 > Ii + J$, 
	the map
	\begin{align*}
	H_i([\cphurc {n_1, \ldots, n_\upsilon} {c}/S_d]; \mathbb Z[1/d!])
	\to H_i(\Conf_{n_1, \ldots, n_\upsilon}; \mathbb
	Z[1/d!])^{\oplus \sigma(n_1, \ldots, n_\upsilon)}
	\end{align*}
	sending a cover to its branch locus (with the conjugacy classes of
	monodromy recorded)
	is an isomorphism.
\end{lemma}
\begin{proof}
	Note first that $\cphurc {n_1, \ldots, n_\upsilon} {c/c_1} \simeq
	\Conf_{n_1,\ldots, n_\upsilon}$ by \autoref{lemma:quotient-rack}.
\autoref{lemma:partial-components} shows that if $c_1 \subset S_d$ is the set of
transpositions, then the components of $\cphurc {n_1, \ldots, n_\upsilon} {c_1,
\ldots, c_\upsilon}$ over $\cphurc {n_1, \ldots, n_\upsilon} {c/c_1} \simeq
\Conf_{n_1,\ldots, n_\upsilon}$ with $n_1$ sufficiently large are in bijection with
$S_d/S_d^{\ab}$, the possible values of the boundary monodromy.
By ``possible values'' we mean that if we fix $n_2, \ldots, n_\upsilon$, then
the boundary monodromy can either take all values in $A_d$ or all values in $S_d
- A_d$, depending on the image of $n_1 c_1 + \cdots + n_\upsilon c_\upsilon$ in
$S_d^{\ab} \simeq \mathbb Z/2 \mathbb Z$.
Hence, after quotienting by the conjugation action of $S_d$, we obtain that the
number of components of 
$[\cphurc {n_1, \ldots, n_\upsilon} {c_1,
\ldots, c_\upsilon}/S_d]$
is the number of possible values of the boundary monodromy, up to conjugacy.
By definition, this is precisely $\sigma(n_1, \ldots, n_\upsilon)$.
Moreover, each component of 
$[\cphurc {n_1, \ldots, n_\upsilon} {c_1,
\ldots, c_\upsilon}/S_d]$
is isomorphic to $[\Conf_{n_1, \ldots, n_\upsilon}/S_d]$
using \autoref{lemma:quotient-rack},
which then has the same $\mathbb Z[1/d!]$ cohomology as 
$\Conf_{n_1, \ldots, n_\upsilon}$
since $S_d$ acts trivially on 
$\Conf_{n_1, \ldots, n_\upsilon}$.
The result then follows from
\autoref{theorem:one-large-stable-homology}, which identifies the stable
homology of each such component.
\end{proof}

As our final preparation for proving Bhargava's conjecture in the function field
case, we wish to identify the geometrically irreducible components of the
relevant Hurwitz spaces over $\mathbb F_q$.

\begin{notation}
	\label{notation:hurwitz-components}
	Let $q$ be a prime power relatively prime to $d!$.
	We use the notation 
$\cquohur {S_d} {S_d}{n_1, \ldots, n_\upsilon} {c} {\mathbb F_q}$
to denote the union of components of 
$\cquohur {S_d} {S_d}{n} {c} {\mathbb F_q}$
as defined in 
\cite[Definition 2.3.3]{landesmanL:homological-stability-for-hurwitz}
which are geometrically irreducible and whose base change to $\overline{\mathbb
F}_q$ lies in
$\cquohur {S_d} {S_d}{n_1, \ldots, n_\upsilon} {c} {\overline{\mathbb F}_q}$,
as defined in
\cite[Notation 2.3.7]{landesmanL:homological-stability-for-hurwitz}.
\end{notation}

\begin{lemma}
	\label{lemma:unique-component}
	With notation from \autoref{notation:hurwitz-components},
	fix $g \in S_d$ and $n_1, \ldots, n_\upsilon$ integers.
	Let $c := S_d - \id$ and suppose $c_1 \subset c$ is the conjugacy class
	of transpositions.
	For $n_1$
	sufficiently large,
	there is at most one irreducible component of 
	$\cquohur {S_d} {S_d}{n_1, \ldots, n_\upsilon} {c} {\mathbb F_q}$
with fixed values $n_1, \ldots, n_\upsilon$ and boundary monodromy in the
conjugacy class of $g$, and, moreover, that component is geometrically
irreducible.
\end{lemma}
\begin{proof}
We first show there is at most one irreducible component of 
$\cquohur {S_d} {S_d} {n_1, \ldots, n_\upsilon} {c} {\overline{\mathbb F}_q}$
with boundary monodromy in the conjugacy class of $g$,
for $n_1$ large enough.
There is bijection between components of
$\cquohur {S_d} {S_d}{n_1, \ldots, n_\upsilon} {c} {\overline{\mathbb F}_q}$
and components of
$\cquohur {S_d} {S_d}{n_1, \ldots, n_\upsilon} {c} {\mathbb C}$
as
shown in \cite[Lemma 2.3.5]{landesmanL:homological-stability-for-hurwitz}.
It then follows from \autoref{lemma:partial-components}
that, once $n_1$ is sufficiently large, there is a unique component of 
$\cphur {S_d} {n_1, \ldots, n_\upsilon} {c} {\mathbb C}$
with boundary monodromy $g$, and hence a unique component of 
$\cquohur {S_d} {S_d}{n_1, \ldots, n_\upsilon} {c} {\mathbb C}$
with boundary monodromy in the conjugacy class of $g$.

Since there is at most one irreducible component of 
$\cquohur {S_d} {S_d} {n_1, \ldots, n_\upsilon} {c} {\overline{\mathbb F}_q}$,
for $n_1$ large enough,
as shown above,
Frobenius must fix that component. Hence, for $n_1$ sufficiently large, by
\cite[Lemma 2.3.8]{landesmanL:homological-stability-for-hurwitz}
every irreducible component of 
$\cquohur {S_d} {S_d} {n_1, \ldots, n_\upsilon} {c} {\mathbb F_q}$ is geometrically
irreducible because there the action of Frobenius on geometric components is
trivial.
\end{proof}

\begin{notation}
	\label{notation:discriminant-count}
	We use
$\Delta(\mathbb F_q(t), S_d, c, q^n)$ for the number of connected $S_d$ extensions of $\mathbb F_q(t)$ of discriminant $q^n$ with
monodromy in $c$, which are geometrically connected.
We use
$\Delta(\mathbb F_q(t), A_d, c, q^n)$ for the number of connected $S_d$
extensions of $\mathbb F_q(t)$ of discriminant $q^n$ with monodromy in $c$ which
become two $A_d$ extensions over $\overline{\mathbb F}_q(t)$.
\end{notation}

With the above determination of the components of Hurwitz spaces out of the way,
we are ready to deduce a function field version of Bhargava's conjecture.
In the following statement, if $x$ is a set, we use $|x|$ to denote the
cardinality of $x$, and if $y$ is a real number, we use $\|y\|$ to denote its
absolute value.

\begin{theorem}
	\label{theorem:bhargava-analog}
	We use notation from \autoref{notation:twisted-conf} and
	\autoref{notation:discriminant-count}.
	For $c = S_d - \id$ and $c_1$ the conjugacy class of transpositions, if
	$q$ is sufficiently large depending on $d$, we
	have
\begin{align}
	\label{equation:alternating-bound}
	\Delta(\mathbb F_q(t), A_d, c, q^n) = o(q^{n})
\end{align}
and
\begin{align}
	\label{equation:symmetric-bound}
\left \| \Delta(\mathbb F_q(t), S_d, c, q^n)  -
	\sum_{\substack{n_1, \ldots, n_v \\ 
			\sum_{i=1}^\upsilon n_i
	\Delta(c_i) = n}} 
	\sigma(n_1, \ldots, n_\upsilon)
	\left| \Conf_{n_1, \ldots, n_\upsilon, \mathbb F_q}  (\mathbb
	F_q) \right| \right\|
	= o(q^{n}).
\end{align}
Hence, 
\begin{align}
	\label{equation:total-bound}
\Delta(\mathbb F_q(t),c , q^n) = 
	\sum_{\substack{n_1, \ldots, n_v \\ 
			\sum_{i=1}^\upsilon n_i
	\Delta(c_i) = n}}
	\sigma(n_1, \ldots, n_\upsilon)
	\left| \Conf_{n_1, \ldots, n_\upsilon, \mathbb F_q} (\mathbb
	F_q) \right| + o(q^{n}).
\end{align}
\end{theorem}
\begin{proof}
\eqref{equation:total-bound} follows from
	\eqref{equation:alternating-bound} and \eqref{equation:symmetric-bound} because the only two
	normal subgroups of $S_d$ with cyclic quotient are $A_d$ and $S_d$
	and 
	$\inv(\mathbb F_q(t),c , q^n) = \sum_N \inv(\mathbb F_q(t),N, c , q^n)$,
	where the sum traverses over normal subgroups of $S_d$ with cyclic
	quotient.

	First, let us explain \eqref{equation:alternating-bound}.
	In this paragraph, we will use the notation 
$a(c \cap A_d,\Delta)$ and 
$b_M(\mathbb F_{q^2}(t), A_d, (A_d - \id)_{\Delta})$ for the constants in
Malle's conjecture, defined in \cite[Notation
10.1.4]{landesmanL:homological-stability-for-hurwitz}.
Now, 
\eqref{equation:alternating-bound}
follows from \cite[Theorem
	10.1.8]{landesmanL:homological-stability-for-hurwitz}
	because $a(c \cap A_d,\Delta) = 2$, as any nontrivial element of the alternating
	group cannot fix $d-2$ elements of $\{1, \ldots, d\}$.
	(In fact, one can moreover show that the left hand side of
	\eqref{equation:alternating-bound} is bounded by $O(q^{n/2})$ using that
$b_M(\mathbb F_{q^2}(t), A_d, (A_d - \id)_{\Delta}) = 1$, though we will not
need this.)

To conclude, we verify \eqref{equation:symmetric-bound}.
We can identify
$\Delta(\mathbb F_q(t), S_d, c, q^n) t $
with 
\begin{align}
	\label{equation:delta-count-as-hur}
	\sum_{\substack{n_1, \ldots, n_v \\ \sum_{i=1}^\upsilon n_i
	\Delta(c_i) = n}} 
	\cquohur {S_d}{S_d} {n_1, \ldots, n_\upsilon} c {\mathbb F_q} (\mathbb
	F_q).
\end{align}

Hence, to conclude, it suffices to show
\begin{align}
	\label{equation:symmetric-bound-difference}
\left \| \sum_{\substack{n_1, \ldots, n_v \\ \sum_{i=1}^\upsilon n_i
	\Delta(c_i) = n
}} 
	\cquohur {S_d}{S_d} {n_1, \ldots, n_\upsilon} c {\mathbb F_q} (\mathbb
	F_q)  -
	\sum_{\substack{n_1, \ldots, n_v \\ 
			\sum_{i=1}^\upsilon n_i
	\Delta(c_i) = n 
}} 
	\sigma(n_1, \ldots, n_\upsilon)
	\left|\Conf_{n_1, \ldots, n_\upsilon, \mathbb F_q}  (\mathbb
	F_q) \right| \right \|
	= o(q^{n}).
\end{align}
We conclude by explaining why \eqref{equation:symmetric-bound-difference}
holds.

We will start by bounding
\begin{align}
	\label{equation:low-transposition-bound}
	\sum_{\substack{n_1, \ldots, n_v \\ \sum_{i=1}^\upsilon n_i
\Delta(c_i) = n \\ n_1 \leq n/2}} 
\cquohur {S_d}{S_d} {n_1, \ldots, n_\upsilon} c {\mathbb F_q} (\mathbb
F_q) = o(q^{3n/4}).
\end{align}
Let $\on{rDisc}$ denote the reduced discriminant invariant, defined precisely in
\cite[Example 10.1.3]{landesmanL:homological-stability-for-hurwitz}.
By definition $\on{rDisc}(c_i) = 1$ for all $i$.
Then, since $\Delta(c_i) \geq 2$ for any $c_i$ other than transpositions, 
the reduced Discriminant of any point of discriminant $n$ with at most $n/2$
transpositions is at most $3n/4$.
We obtain that the left hand side of \eqref{equation:low-transposition-bound}
is bounded by
$\left | \on{rDisc}(\mathbb F_q(t), S_d, c, q^{3n/4}) \right |$, which we
bounded by $O(q^{3n/4+ \epsilon})= o(q^n)$ in 
\cite[Theorem 10.1.8]{landesmanL:homological-stability-for-hurwitz}.

Now, in order to bound
\eqref{equation:symmetric-bound-difference}, using
\eqref{equation:low-transposition-bound},
if we fix values for $n_2, \ldots, n_\upsilon$,
it is enough to show there are constants $C, C', I, J$ independent of $n_2,
\ldots, n_\upsilon$ such that
\begin{equation}
	\label{equation:symmetric-bound-difference-fixed-sequence}
\begin{aligned}
&\left \| \sum_{\substack{n_1 \\ \sum_{i=1}^\upsilon n_i
	\Delta(c_i) = n \\ n_1 \geq n/2}} 
	\cquohur {S_d}{S_d} {n_1, \ldots, n_\upsilon} c {\mathbb F_q} (\mathbb
	F_q)  -
	\sum_{\substack{n_1 \\ \sum_{i=1}^\upsilon n_i
	\Delta(c_i) = n \\ n_1 \geq n/2}} 
	\sigma(n_1, \ldots, n_\upsilon)
	\left| \Conf_{n_1, \ldots, n_\upsilon, \mathbb F_q}  (\mathbb
	F_q) \right| \right \|
	\\
	&= q^n \cdot \frac{2C'}{1-\frac{C}{\sqrt{q}}} \left( \frac{C}{\sqrt{q}}
	\right)^{\frac{n-J}{I}}
\end{aligned}
\end{equation}
Once we establish \eqref{equation:symmetric-bound-difference-fixed-sequence}, we
can sum over all values of $n_2, \ldots, n_\upsilon \leq n$ and bound the left hand
side of \eqref{equation:symmetric-bound-difference} by at most $q^n \cdot n^{\upsilon - 1}
\cdot \frac{2C'}{1-\frac{C}{\sqrt{q}}} \left( \frac{C}{\sqrt{q}}
\right)^{\frac{n-J}{I}}$, which is indeed $o(q^n)$, once $q$ is sufficiently large.

To verify \eqref{equation:symmetric-bound-difference-fixed-sequence}, we will check it separately as $n_1$ ranges over odd
integers and as $n_1$ ranges over even integers. The reason for considering
these two cases depending on the parity of $n_1$ is because the value of 
$\sigma(n_1, \ldots, n_\upsilon)$ is only a function of the parity of $n_1$, for
$n_2, \ldots, n_\upsilon$ fixed.
We conclude by explaining why the above claim holds via an application of
\cite[Lemma 5.2.2]{landesmanL:the-stable-homology-of-non-splitting}.
Indeed, we just have to verify the hypotheses (1) and (2) of
\cite[Lemma 5.2.2]{landesmanL:the-stable-homology-of-non-splitting},
while showing the constants $C, C'$, $I$, and $J$ there are independent of the values of
$n_2, \ldots, n_\upsilon$.
The hypothesis $(2)$ holds with the constants $C$ and $C'$ there independent of
$n_2, \ldots, n_\upsilon$
using \cite[Lemma 8.4.2]{landesmanL:homological-stability-for-hurwitz}.
Hence, it remains to verify hypothesis $(1)$, with the additional constraint that the values of $I$ and $J$
are independent of $n_2, \ldots, n_\upsilon$.
The independence of $I$ and $J$ follows from
\autoref{theorem:one-large-stable-homology-hurwitz-space}.
Hence, it remains only to identify 
the stable trace of $\on{Frob}_q^{-1}$ (where $\on{Frob}_q$
geometric Frobenius, and stable means that $n_1$ is sufficiently large) on each component of
$\cquohur {S_d}{S_d} {n_1, \ldots, n_\upsilon} c {\mathbb F_q} \times_{\spec \mathbb
F_q} \spec \overline{\mathbb F}_q$
with the stable trace of $\on{Frob}_q^{-1}$ on 
$\Conf_{n_1, \ldots, n_\upsilon,\overline{\mathbb F}_q}$.
To make this identification, we use the composite map
\begin{align}
	\label{equation:hur-to-conf-sd}
	\cquohur {S_d}{S_d} {n_1, \ldots, n_\upsilon} c {\mathbb F_q} 
	\to \cquohur {S_d/c_1}{S_d} {n_1, \ldots, n_\upsilon} c {\mathbb F_q}
	\simeq
[\Conf_{n_1,\ldots, n_\upsilon,\mathbb F_q}/S_d]
\to
	\Conf_{n_1,
\ldots, n_\upsilon,\mathbb F_q}
\end{align}
over $\mathbb F_q$,
given by sending a cover to its branch locus, where one records the degree of
each conjugacy classes of
monodromy occurring in the branch locus in the values $n_1, \ldots, n_\upsilon$.
The existence of this map \eqref{equation:hur-to-conf-sd}
relies on the identification
$\cquohur {S_d/c_1}{S_d} {n_1, \ldots, n_\upsilon} c {\mathbb F_q}
\simeq
[\Conf_{n_1,\ldots, n_\upsilon,\mathbb F_q}/S_d]$
over $\mathbb F_q$, which stems from the fact that every conjugacy class of $S_d$ is sent to
itself under the $q$th powering map when $q$ is relatively prime to $d!$. 
We note that 
\eqref{equation:hur-to-conf-sd}
is a bijection between components of the source with fixed conjugacy class of
boundary monodromy
to components of the target using
\autoref{lemma:partial-components} and
\autoref{lemma:unique-component}.
This implies that the map
\eqref{equation:hur-to-conf-sd}, when base changed to $\overline{\mathbb F}_q$
and restricted to a single component of the source
induces a Frobenius equivariant isomorphism on stable cohomology, for $n_1$
sufficiently large.
Hence, the stable trace of Frobenius (meaning that it is stable as $n_1$ grows) on the cohomology of each component of 
$\cquohur {S_d}{S_d} {n_1, \ldots, n_\upsilon} c {\overline{\mathbb F}_q}$
is identified with the stable trace of Frobenius on the cohomology of 
$\Conf_{n_1, \ldots, n_\upsilon, \overline{\mathbb F}_q}$,
yielding \eqref{equation:symmetric-bound-difference-fixed-sequence}.
\end{proof}

\section{Representation Stability}
\label{section:representation-stability}

In this section, we prove \autoref{theorem:representation-stability} on
representation stability for homology of Hurwitz spaces.
Before taking up the proof, we begin with some remarks and complements.
Throughout this section, we freely use notation from
\autoref{definition:representation-stability}.


\begin{remark}
	\label{remark:non-splitting-stability}
	If $c$ has multiple components, it will simply be false that 
	$H_i(\CHur^c_n; \mathbb H_{\lambda,n})$ stabilizes. Indeed, even in the case
	$\lambda$ is the trivial partition of $1$, 
	$H_0(\CHur^c_n; \mathbb Q)$ grows with polynomial degree
	$|c/c|-1$.
	So only in the case $|c/c|=1$ can this multiplicity possibly stabilize.
	Similarly, if we were to use 
	$\Hur^c_n$ in place of $\CHur^c_n$, then
	$H_0(\Hur^c_n; \mathbb Q)$ would typically not stabilize, except
	in the case that $c$ satisfies the {\em non-splitting property} as
	defined in \cite[Definition
	4.1.7]{landesmanL:the-stable-homology-of-non-splitting}, which
	is equivalent to the condition that $H_0(\Hur^c_n; \mathbb Q)$
	stabilizes in $n$.
\end{remark}

We let $c$ be a finite rack with a single component. Recall our goal is to show
$\CHur^c_n$ satisfies linear representation stability.
The idea will be to define an appropriate rack such that knowing the stable
homology of certain Hurwitz spaces associated to that rack will allow us to
deduce representation stability.
We now define the relevant rack.

\begin{definition}
	\label{definition:copies-rack}
		For $j \geq 1$, let $c^{\boxtimes j}$ denote the rack of order $j|c|$ consisting of $j$
		copies of $c$, given by $c^{\boxtimes j} = c_1 \coprod \cdots
		\coprod c_j$. If $x_u \in
	c_u, y_v \in c_v$ map to $x, y \in c$ under the isomorphism $c_u \simeq
	c, c_v \simeq c$, then $x_u \triangleright y_v$ is defined to
	be $(x\triangleright
	y)_v \in c_v$.
\end{definition}

In what follows, we use $1^u$ as notation for the tuple $\underbrace{1, \ldots, 1}_{u \text{ times}}$. 
We first record an elementary consequence of the representation theory of $S_n$.

\begin{lemma}
	\label{lemma:sn-rep-theory}
	Let $c$ be a finite rack with a single component.
	For any partition $\lambda$,
	there is some value of $j \leq |\lambda|$ so that the map
	$\phi_n: \CHur^{c^{\boxtimes j}}_{1^{j-1}, n-j+1} \to
	\CHur^c_n$
	contains a copy of $\mathbb
	H_{\lambda,n} \subset (\phi_n)_* \mathbb Q$.
\end{lemma}
\begin{proof}
	Let $\std$ denote the standard representation of $S_n$, which has
	dimension $n-1$, and let $\perm$ denote the $n$-dimensional permutation
	representation representation, which has dimension $n$.
	Let $V^{n,j}$ denote the $\binom{n}{n-j}$ dimensional $S_n$
	representation obtained from the permutation action on the
	set $S^{n,j}$ consisting of the $\binom{n}{n-j}$ order $j$ subsets of $\{1, \ldots, n\}$.
	The set $S^{j,n}$ corresponds to the
	cover $\Conf_{1^j,n-j} \to \Conf_n$ in the sense that it is the kernel
	of the action $\pi_1(\Conf_n)$ on $S^{j,n}$, acting through the quotient
	$\pi_1(\Conf_n)\simeq B_n \to S_n$.
	For any given partition $\lambda$, the representation theory for the
	symmetric group implies the representation associated to the
	partition $(n-|\lambda|, \lambda_1, \ldots, \lambda_p)$
	is a subrepresentation of $\std^{\otimes |\lambda|}$.
	Therefore it is also a subrepresentation of 
	$\perm^{\otimes |\lambda|}$.
	Since $\perm^{\otimes |\lambda|}$ can be expressed as a sum of $V^{n,j}$
	for $j \leq |\lambda|$, we also find
	that $\rho_{\lambda,n}$ appears in some $V^{n,j}$ for $j \leq
	|\lambda|$. Hence, 
	there is some $j \leq |\lambda|$ so that $\mathbb V_{\lambda,n}$ appears
	as a subrepresentation of $(\psi_n)_* \mathbb Q$ for $\psi_n:
	\Conf_{1^{j-1},n-j+1} \to \Conf_n$. Pulling this back along the map $\CHur_n^c \to
	\Conf_n$ yields the result.
\end{proof}

\subsection{Proof of \autoref{theorem:representation-stability}}
\label{subsection:representation-stability-proof}

Let 
$Z'' \subset \CHur^{c^{\boxtimes n}}_{1^n}$ be a component mapping to a
component
$Z' \subset \CHur^{c^{\boxtimes j}}_{1^{j-1},n-j-1}$ which maps to a component
$Z \subset \CHur^c_n$.
Using \autoref{lemma:sn-rep-theory}, we obtain a commutative diagram
	\begin{equation}
		\label{equation:}
		\begin{tikzcd}
			H_i(Z; \mathbb H_{\lambda,n}|_Z) \ar {r} \ar {d}{\iota^Z_i} &
			H_i(Z'; \mathbb Q)
			\ar {r} \ar {d}{\iota^{Z'}_i} & H_i(Z'' ; \mathbb Q)\ar
			{d}{\iota^{Z''}_i}  \\
			H_i(\CHur^c_n; \mathbb H_{\lambda,n}) \ar {r} \ar {d}{\alpha_i} &
			H_i(\CHur^{c^{\boxtimes j}}_{1^{j-1},n-j-1}; \mathbb Q)
			\ar {r} \ar {d}{\beta_i} & H_i(\CHur^{c^{\boxtimes
			n}}_{1^n} ; \mathbb Q)\ar
				{d}{\gamma_i} \\
			H_i(\Conf_n; \mathbb V_{\lambda,n}) \ar {r} &
			H_i(\Conf_{1^{j-1},n-j-1};\mathbb Q) \ar {r} &
			H_i(\Conf_{1^n}; \mathbb Q).
	\end{tikzcd}\end{equation}
	By \cite[Theorem 2.4]{shusterman:the-tamely-ramified-geometric}
	the map 
	$\CHur^{c^{\boxtimes n}}_{1^n} \to
	\CHur^{c^{\boxtimes j}}_{1^{j-1},n-j-1} \xrightarrow{\phi_n}
	\CHur^{c}_n$ induces a bijection on components for $n$
	sufficiently large, depending on $c$.
	Since the map $\CHur^{c^{\boxtimes j}}_{1^{j-1},n-j-1} \to
	\CHur^{c}_n$ induces a bijection on components for $n$ large enough, 
the summand $\mathbb H_{\lambda,n} \subset (\phi_n)_*\mathbb Q$ restricts to a
summand $\mathbb H_{\lambda,n}|_Z \subset ((\phi_n|_{Z'})_* \mathbb Q|_{Z'}) =((\phi_n)_*\mathbb Q)|_{Z'}$. 
Recall that we are trying to show 
$\alpha_i \circ\iota^Z_i$ induces an isomorphism when $n - |\lambda| > Ii + J$
for suitable constants $I$ and $J$. Hence, by the above, it
suffices to show 
$\beta_i \circ \iota^{Z'}_i$ induces an isomorphism when $n - |\lambda| > Ii +
J$.

Note that $\beta_0 \circ \iota^{Z'}_0$ is an isomorphism by construction,
because the source and target both have a single component.
	We next explain why $\beta_i$ is also
	an isomorphism for $i > 0$. 
	Let $Z''' \subset \CHur^{c^{\boxtimes j}/c^{\boxtimes
	j}}_{1^{j-1},n-j+1}$ denote the component which $Z'$ maps to under the
	projection
$\CHur^{c^{\boxtimes j}}_{1^{j-1},n-j-1} \to \CHur^{c^{\boxtimes j}/c^{\boxtimes
	j}}_{1^{j-1},n-j+1}$.
	Since $c^{\boxtimes j}/c_j \simeq c^{\boxtimes j}/c^{\boxtimes
	j}$, 
	it follows from
\autoref{theorem:one-large-stable-homology-hurwitz-space}
that there are constants $I$ and $J'$ depending on $c$ so that for $n - j+1> Ii
+ J'$,
the map 
\begin{align*}
	\beta_i \circ \iota^{Z'}_i: H_i(Z';\mathbb Q) \simeq
	H_i(Z''';\mathbb Q)
	\simeq H_i(\Hur^{c^{\boxtimes j}/c^{\boxtimes
	j}}_{1^{j-1},n-j+1};\mathbb Q)
	\simeq H_i(\Conf_{1^{j-1},n-j+1};\mathbb Q)
\end{align*}
is an isomorphism.

We have now shown that 
$\beta_i \circ \iota^{Z'}_i$ is an isomorphism for 
$n-j > Ii + J'$.
Since $j \leq |\lambda|$ we also have that
$\beta_i \circ \iota^{Z'}_i$ is an isomorphism 
for
$n-|\lambda| > Ii+J'$. 
This completes the proof, as explained above, since it
then means that 
$\alpha_i \circ\iota^Z_i$ 
is an isomorphism 
when $n - |\lambda| > Ii + J$
for some constants $I$ and $J$, with $J$ possibly larger than $J'$ but only
depending on $c$.
\qed

\section{Further questions}
\label{section:further-questions}

The results of this paper open avenues to prove a vast collection of function
field results in arithmetic statistics, and the examples we surveyed, such as
the BKLPR heuristics and Bhargava's conjecture, only constitute a small
collection of potential applications.

Here, to prove a version of Bhargava's conjecture,
we counted degree $d$, $S_d$ extensions of $\mathbb F_q(t)$ by discriminant. 
It seems likely that the techniques of this paper could also compute the
constant determining the asymptotic count of $S_d$ extensions by other invariants. It also seems likely
one could generalize the techniques to predict what the constant should be in
T\"urkelli's version of Malle's conjecture. (In 
\cite[Theorem 10.1.10]{landesmanL:homological-stability-for-hurwitz} we showed there
are some periodic constants relevant to T\"urkelli's conjecture when one
counts by discriminant, but we did not compute them.)
We note that, in some cases, a prediction for the constant in Malle's conjecture over $\mathbb Q$ has
been made in \cite{loughranS:malles-conjecture-and-brauer-groups}.

More generally, it would be natural to predict the constant
governing the number of extensions of an arbitrary global field, instead of just
$\mathbb Q$ or $\mathbb F_q(t)$. 
It would also be natural to predict
the number of extensions with specified local conditions at a finite set
of places.
In the function field case, adding local conditions amounts
to understanding the cohomology of Hurwitz modules over punctured
curves, where one imposes local conditions at the punctures. We note that, in the related context of the Cohen-Lenstra-Martinet heuristics,
predictions have been made for the average size of torsion in class groups
over varying extensions of number fields
\cite{cohen1987class} and versions with local conditions have been given in
\cite{wood:cohen-lenstra-heuristics-and-local-conditions}.
The idea for how to make the above conjectures over function fields would be 
to phrase them
in terms 
of components of Hurwitz spaces, so that one can aim to prove them over function
fields using the techniques of this paper.
One could then try to phrase the resulting conjectures in a way so that they
could also work over number fields.

Another direction to investigate concerns whether there is a moduli
interpretation of Hurwitz spaces associated to an arbitrary rack (and not only
racks coming from unions of conjugacy classes in a group).
We conjecture that such an interpretation for a rack $c$ does exist over
$\mathbb Z[1/|G^c_c|]$,
so that it is possible to define a scheme over 
$\mathbb Z[1/|G^c_c|]$ whose restriction to $\mathbb C$ is
$\Hur^c$.
If one is able to define such a scheme, one would obtain immediate consequences
for the number of $\mathbb F_q$ points in each of its components, using the
results of this paper.
Similarly, we ask whether it is possible to define a moduli interpretation of
Hurwitz modules or bijective Hurwitz modules over $\mathbb
Z[1/N^c_S]$,
for some integer $N^c_S$ depending on the rack $c$ and the Hurwitz module $S$,
whose pullback to $\mathbb C$ is $\Hur^{c,S}$.
Perhaps $N^c_S = |G^c_S| \cdot |G^c_c|$.
Again, if this is true, one would obtain immediate consequences for the finite
field points of such a scheme using the results of this paper.

In \autoref{theorem:representation-stability}, we proved representation
stability for Hurwitz spaces. Can one also prove representation stability for
Hurwitz modules? Of course, the main results
\autoref{theorem:homology-stabilizes-intro} and
\autoref{theorem:one-large-stable-homology} are already stated in the setting of
Hurwitz modules, and it appears that one of the main obstacles to answering this
question is to generalize \cite[Theorem 2.4]{shusterman:the-tamely-ramified-geometric}
to the setting of Hurwitz modules.
Can one prove uniform representation stability for Hurwitz spaces and Hurwitz
modules, so that the constants do not depend on the partition $\lambda$?
(This has been shown in some special cases in 
\cite{himesMW:homological-stability-for-hurwitz}.)

Another direction which it seems likely these results could apply are in
computing moments of Selmer groups of semi-abelian varieties.
Here, we restricted ourselves to the setting of Selmer groups of abelian
varieties.
However, the distribution of Selmer groups of $\mathbb G_m$ in quadratic twist
families is closely related to the Cohen-Lenstra heuristics, as discussed in
\cite[Remark 1.4]{landesman:a-geometric-approach-cohen-lenstra}.
It would be interesting to find a common generalization of the Cohen-Lenstra and
BKLPR heuristics predicting the distribution of Selmer groups of
quadratic twist families of semi-abelian varieties.

Yet another direction of further possible study relates to special values of L-functions and their
moments.
Given a union of conjugacy classes $c_G$ in a group $G$, it is possible to construct a rack $c$ so that the universal
curve over $\Hur^{c_G}$ is a disjoint union of certain
components of $\Hur^c$.
Since average values of $L$ functions at the central point can be related to point counts on fibers
of the universal curve
it would be interesting to see if the results of this paper can say anything
about moments of L functions, especially along the lines of the results in
\cite{bergstromDPW:hyperelliptic-curves-the-scanning-map}.
As explained to us by Will Sawin, it seems unlikely our results could obtain the
necessary information to prove the analog of
\cite[Proposition 1.5]{millerPPR:uniform-twisted-homological-stability}
for general groups $G$, as that would seem to involve understanding something
about the unstable homology of Hurwitz spaces.

\bibliographystyle{alpha}
\bibliography{./bibliography}

\end{document}